\documentclass[11pt,oneside]{amsart}

\usepackage{vmargin}
\usepackage{amssymb}
\usepackage{mathrsfs}
\usepackage[all]{xy}
\usepackage{color}
\usepackage{soul}
\usepackage{hyperref}

\newcommand\N{{\mathbb N}}
\newcommand\R{{\mathbb R}}

\newcommand\Sp{{\mathbb S}}

\newcommand\Ee{{\mathbb E}}

\def\BB{{\mathcal B}}
\def\CC{{\mathcal C}}
\def\DD{{\mathcal D}}
\def\EE{{\mathcal E}}
\def\FF{{\mathcal F}}
\def\GG{{\mathcal G}}
\def\HH{{\mathcal H}}
\def\II{{\mathcal I}}

\def\KK{{\mathcal K}}
\def\LL{{\mathcal L}}
\def\MM{{\mathcal M}}

\def\OO{{\mathcal O}}
\def\PP{{\mathcal P}}

\def\RR{{\mathcal R}}
\def\SS{{\mathcal S}}
\def\TT{{\mathcal T}}
\def\UU{{\mathcal U}}
\def\VV{{\mathcal V}}
\def\WW{{\mathcal W}}
\def\VV{{\mathcal V}}

\def\ZZ{{\mathcal Z}}

\def\eps{{\varepsilon}}

\newcommand{\wto}{\rightharpoonup}

\def\SN{\mathfrak{S}_N}
\def\SSS{\mathfrak{S}}

\newtheorem{theo}{Theorem}

\newtheorem{lem}[theo]{Lemma}

\theoremstyle{definition}
\newtheorem{defin}[theo]{Definition}

\theoremstyle{remark}
\newtheorem{rem}[theo]{Remark}
\newtheorem{rems}[theo]{Remarks}
\newtheorem{ex}[theo]{Example}

\newcommand{\beqn}{\begin{equation}}
\newcommand{\eeqn}{\end{equation}}
\newcommand{\bear}{\begin{eqnarray}}
\newcommand{\eear}{\end{eqnarray}}
\newcommand{\bean}{\begin{eqnarray*}}
\newcommand{\eean}{\end{eqnarray*}}

\newcommand{\Black}{\color{black}}

\def\signsm{\bigskip \begin{center} {\sc St\'ephane Mischler\par\vspace{3mm}
Universit\'e Paris-Dauphine\par
CEREMADE, UMR CNRS 7534\par
Place du Mar\'echal de Lattre de Tassigny
75775 Paris Cedex 16\par
FRANCE\par\vspace{3mm}
e-mail:} \tt{mischler@ceremade.dauphine.fr} \end{center}}

\def\signcm{\bigskip \begin{center} {\sc
Cl\'ement Mouhot\par\vspace{3mm}
University of Cambridge\par
DPMMS, Centre for Mathematical Sciences\par
Wilberforce Road,
Cambridge CB3 0WA,
UK\par\vspace{3mm}
e-mail:} \tt{C.Mouhot@dpmms.cam.ac.uk} \end{center}}

\begin{document}

\title[Kac's Program in  Kinetic  Theory]
{Kac's Program in Kinetic Theory \\[1mm]
  (Version of \today)}


\author{S. Mischler}
\author{C. Mouhot}



\begin{abstract}
  This paper is devoted to the study of propagation of chaos and 
  mean-field limits for systems of indistinguable particles, undergoing
  collision processes. The prime examples we will consider are the
  many-particle jump processes of Kac and McKean
  \cite{Kac1956,McKean1967} giving rise to the Boltzmann equation. We
  solve the conjecture raised by Kac \cite{Kac1956}, motivating his
  program, on the rigorous connection between the long-time behavior
  of a collisional many-particle system and the one of its mean-field
  limit, for bounded as well as unbounded collision rates.
 
  Motivated by the inspirative paper by Gr\"unbaum \cite{Grunbaum}, we
  develop an abstract method that reduces the question of propagation
  of chaos to that of proving a purely functional estimate on
  generator operators ({\em consistency estimates}), along with
  differentiability estimates on the flow of the nonlinear limit
  equation ({\em stability estimates}). This allows us to exploit
  dissipativity at the level of the mean-field limit equation rather
  than the level of the particle system (as proposed by Kac).

  Using this method we show: (1) Quantitative estimates, that are
  uniform in time, on the
  chaoticity of a family of states. 
  (2) Propagation of {\it entropic chaoticity}, as defined in
  \cite{CCLLV}. 
  (3) Estimates on the time of relaxation to equilibrium, that
  are 
  \emph{independent of the number of particles in the system}.
  Our results cover the two main Boltzmann physical collision
  processes with unbounded collision rates: hard spheres and
  \emph{true} Maxwell molecules interactions. The proof of the
  \emph{stability estimates} for these models requires significant
  analytic efforts and new estimates.
\end{abstract}

\maketitle


  \smallskip

Jump processes.



\tableofcontents


Boltzmann is best known for the equation bearing his name in kinetic
theory \cite{Boltzmann1872,Boltzmann1896}. Inspired by Maxwell's
discovery \cite{Maxwell1867} of (what is now called) the Boltzmann
equation and its ``Maxwellian'' (i.e. Gaussian) equilibrium, Boltzmann
\cite{Boltzmann1872} discovered the ``$H$-theorem'' (the entropy must
increase under the time evolution of the equation), which explained
why the solutions should be driven towards the equilibrium of
Maxwell. In the same work he also proposed the deep idea of
``stosszahlansatz'' (molecular chaos) to explain how the irreversible
Boltzmann equation can emerge from Newton's laws of the dynamics of
the particle system. Giving a precise mathematical meaning to this
notion and proving this limit remains a tremendous open problem to
this date; the best and astonishing result so far \cite{Lanford} is
only valid for very short times.

In 1956, Kac \cite{Kac1956} proposed the simpler, and seemingly more
tractable, question of deriving the \emph{spatially homogeneous}
Boltzmann equation from a many-particle jump process. To do so, he
introduced a rigorous notion of ``molecular chaos''\footnote{Kac in
  fact called this notion ``Boltzmann's property'' in \cite{Kac1956}
  as a clear tribute to the fundamental intuition of Boltzmann.} in
this context. The ``chaoticity'' of the many-particle equilibrium with
respect to the Maxwellian distribution, i.e. the fact that the first
marginals of the uniform measure on the sphere $\Sp^{N-1}(\sqrt N)$
converges to a Gaussian function as $N$ goes to infinity, has been
known for a long time (at least since Maxwell).\footnote{We refer to
  \cite{DiaconisFreedman1987} for a bibliographic discussion, see also
  \cite{CCLLV} where \cite{mehler} is quoted as the first paper
  proving this result.}  However in \cite{Kac1956} Kac proposed the
first proof of the \emph{propagation of chaos} for a simplified
collision evolution process for which series expansion of the solution
exists, and he showed how the many-particle limit rigorously follows
from this property of propagation of chaos. This proof was later
extended to a more realistic collision model, the so-called cutoff
Maxwell molecules, by McKean \cite{McKean1967}.

Since in this setting both the many-particle system and the limit
equation are dissipative, Kac raised the natural question of relating
their asymptotic behaviors.  In his mind this program was to be
achieved by understanding dissipativity at the level of the linear
many-particle jump process and he insisted on the importance of
estimating how its relaxation rate depends on the number of
particles. This has motivated beautiful works on the ``Kac spectral
gap problem''
\cite{janvresse,maslen,CarlenCL2003,CarlenGeronimoLoss2008,CCL-preprint},
i.e. the study of this relaxation rate in a $L^2$ setting. However, so
far this linear strategy has proved to be unsuccessful in relating the
asymptotic behavior of the many-particle process and that of the limit
equation (cf. the discussion in \cite{CCLLV}).

In the time of Kac the study of nonlinear partial differential
equations was rather young and it was plausible that the study of a
linear many-dimension Markov process would be easier. However the
mathematical developpement somehow followed the reverse direction and
the theory of existence, uniqueness and relaxation to equilibrium for
the spatially homogeneous Boltzmann is now well-developed (see the
many references along this paper). This paper (together with its
companion paper \cite{MMW}) is thus an attempt to develop a
quantitative theory of mean-field limit which \emph{strongly relies on
  detailed knowledge of the nonlinear limit equation, rather than on
  detailed properties of the many-particle Markov process}.

The main outcome of our theory will be to find quantitative estimates
to the propagation of chaos that are uniform in time, as well as
propagation of entropic chaos. We also prove estimates on the
relaxation rates, measured in the Wasserstein distance and relative
entropy, that are \emph{independent of the number of particles}. All
this is done for the two important, realistic and achetypal models of
collision with unbounded collision rates, namely hard spheres and the
true (i.e. without cutoff) Maxwell molecules. This provides a first
answer to the question raised by Kac. However, our answer is an
``inverse'' answer in the sense that our methodology is ``top-down''
from the limit equation to the many-particle system rather than
``bottom-up'' as was expected by Kac.

 
\bigskip



\noindent
\textbf{Acknowledgments.} We thank the mathematics departement of
Chalmers University for the invitation in November 2008, where the
abstract method was devised and the related joint work~\cite{MMW} with
Bernt Wennberg was initiated. We thank Isma\"el Bailleul, Thierry
Bodineau, Fran\c cois Bolley, Anne Boutet de Monvel, Jos\'e Alfredo
Ca\~nizo, Eric Carlen, Nicolas Fournier, Fran\c cois Golse, Arnaud
Guillin, Maxime Hauray, Joel Lebowitz, Pierre-Louis Lions, Richard
Nickl, James Norris, Mario Pulvirenti, Judith Rousseau, Laure
Saint-Raymond and C\'edric Villani for fruitful comments and
discussions, and Amit Einav for his careful proofreading of parts of
the manuscript. We would also like to mention the inspiring courses by
Pierre-Louis Lions at Coll\`ege de France on ``Mean-Field Games'' in
2007-2008 and 2008-2009, which triggered our interest in the
functional analysis aspects of this topic. Finally we thank the
anonymous referees for helpful suggestions on the presentation.

\bigskip


\section{Introduction and main results}
\label{sec:intro}
\setcounter{equation}{0}
\setcounter{theo}{0}


\subsection{The Boltzmann equation}
\label{sec:introEB}

The Boltzmann equation (Cf. \cite{Ce88} and \cite{CIP}) describes the
behavior of a dilute gas when the only interactions taken into account
are binary collisions. It is given by
\begin{equation*}
  \label{eq:Boltzmann-complete}
  \frac{\partial f}{\partial t} + v \cdot \nabla_x f = Q(f,f), \qquad
  x \in \Omega, \quad v \in \R^d, \quad t \ge 0,
\end{equation*}
where $Q=Q(f,f)$ is the bilinear {\em collision operator} acting only
on the velocity variable, $\Omega$ is the spatial domain and $d \ge 2$
is the dimension. Some appropriate boundary conditions need to be
imposed.

In the case where the distribution function is assumed to be
independent of the position $x$, we obtain the so-called {\it
  spatially homogeneous Boltzmann equation}:
 \begin{equation}\label{el}
   \frac{\partial f}{\partial t}(t,v) 
   = Q(f,f)(t,v), \qquad  v \in \R^d, \quad t \geq 0,
 \end{equation}
which will be studied in this paper. 

Let us now focus on the collision operator $Q$. It is defined by the
bilinear symmetrized form
 \begin{equation}\label{eq:collop}
 Q(g,f)(v) = \frac12\,\int _{\R^d \times \mathbb{S}^{d-1}} B(|v-v_*|, \cos \theta)
         \left(g'_* f' + g' f_* '- g_* f - g f_* \right) \, {\rm d}v_* \, {\rm d}\sigma,
 \end{equation}
where we use the shorthands $f=f(v)$, $f'=f(v')$, $g_*=g(v_*)$ and
$g'_*=g(v'_*)$. Moreover, $v'$ and $v'_*$ are parametrized by
 \begin{equation}\label{eq:rel:vit}
   v' = \frac{v+v_*}2 + \frac{|v-v_*|}2 \, \sigma, \qquad 
   v'_* = \frac{v+v_*}2 - \frac{|v-v_*|}2 \, \sigma, \qquad 
   \sigma \in \mathbb{S}^{d-1}. 
 \end{equation}
 Finally, $\theta\in [0,\pi]$ is the \emph{deviation angle} between $v'-v'_*$
 and $v-v_*$ defined by 
\[
\cos \theta = \sigma \cdot \hat u, \quad u = v-v_*, \quad \hat u =
\frac{u}{|u|},
\] 
and $B$ is the \emph{collision kernel} determined by the physical
context of the problem. 

The Boltzmann equation has the following fundamental
formal properties: first, it conserves mass, momentum and energy, i.e. 
  \begin{equation*}
  \frac{{\rm d}}{{\rm d}t} \int_{\R^d} f \, \phi(v) \, {\rm d}v =  \int_{\R^d}Q(f,f) \, \phi(v)\,{\rm d}v = 0, \qquad
  \phi(v)=1,v,|v|^2. \label{CON}
\end{equation*}
Second, it satisfies Boltzmann's celebrated $H$-theorem:
\begin{equation*} 
  - \frac{{\rm d}}{{\rm d}t} \int_{\R^d} f \, \ln f \, {\rm d}v = 
  - \int_{\R^d} Q(f,f)\, \ln(f) \, {\rm d}v \geq 0.
\end{equation*}

We shall consider collision kernels of the form
\[
B=\Gamma(|v-v_{*}|) \, b(\cos \theta) 
\]
where $\Gamma,b$ are nonnegative functions.
In dimension $d=3$, we give a short overview of the main collision
kernels appearing in physics, highlighting the key models we
consider in this paper.
 \begin{itemize}
 \item[(1)] Short (finite) range interaction are usually modeled by
   the {\em hard spheres collision kernel}
   \begin{equation}\label{model:hs}
     {\bf (HS)}  \qquad B(|v-v_*|, \cos \theta)= C \,
     |v-v_*|, \quad C>0. 
 \end{equation}
\item[(2)] Long-range interactions are usually modeled by collision
   kernels derived from interaction potentials 
   \[
   V(r) = C \, r^{-s}, \quad s >2, \ C>0. 
   \]
   They satisfy the formula 
   \[
   \Gamma(z)=|z|^\gamma \ \mbox{with } \ \gamma = (s-4)/s 
   \]
   and
   \[
   b(\cos \theta) \sim_{\theta \sim 0} C_b \,
   \theta^{-2-\nu} \mbox{ with } \nu = 2/s, \ C_b >0
   \]
   (b is in $L^1$ away from $\theta = 0$). More informations about
   this type of interactions can be found in~\cite{Ce88}.  This
   general class of collision kernels includes the {\em true Maxwell
     molecules collision kernel} when $\gamma=0$ and $\nu = 1/2$, i.e.
   \begin{equation}\label{model:tmm}
   {\bf (tMM)} \qquad B(|v-v_*|, \cos \theta)= b(\cos \theta)
   \sim_{\theta \sim 0} C_b \, \theta^{-5/2}.
 \end{equation}
 
   It also includes the so-called {\em Grad's cutoff Maxwell
     molecules} when the singularity in the $\theta$ variable is
   removed. For simplicity we will consider the model where
   \begin{equation}\label{model:gmm}
   {\bf (GMM)} \qquad B(|v-v_*|, \cos \theta)=1
   \end{equation}
which is an archetype of such collision kernels. 
 \end{itemize}
 

\subsection{Deriving the Boltzmann equation from many-particle
   systems}
\label{sec:questionderivation}

The question of deriving the Boltzmann equation from particle systems
(interacting via Newton's laws) is a famous problem. It is
related to the so-called $6$-th Hilbert problem proposed by Hilbert at
the International Congress of Mathematics at Paris in 1900: axiomatize
mechanics by ``\textit{developing mathematically the limiting processes
[\dots] which lead from the atomistic view to the laws of motion of
continua}''.

At least at the formal level, the correct limiting procedure has been
identified by Grad~\cite{Grad1949} in the late fourties and a clear
mathematical formulation of the open problem was proposed in
~\cite{Cerc1972} in the early seventies. It is now called the {\em
  Boltzmann-Grad} or {\em low density} limit. However the original
question of Hilbert remains largely open, in spite of a striking
breakthrough due to Lanford~\cite{Lanford}, who proved the limit for
short times (see also Illner and Pulvirenti \cite{MR849204} for a
close-to-vacuum result). The tremendous difficulty underlying this
limit is the {\em irreversibility} of the Boltzmann equation, whereas
the particle system interacting via Newton's laws is a \emph{reversible}
Hamiltonian system.

In 1954-1955, Kac~\cite{Kac1956} proposed a simpler and more tractable
problem: start from the Markov process corresponding to collisions
only, and try to prove the limit towards the {\em spatially
  homogeneous} Boltzmann equation. {\em Kac's jump process} runs as
follows: consider $N$ particles with velocities $v_1,\dots,v_N \in
\R^d$. Assign a random time for each pair of particles $(v_i,v_j)$
following an exponential law with parameter $\Gamma(|v_i-v_j|)$, and
take the smallest. Next, perform a collision $(v_i,v_j) \to (v_i ^*,
v_j ^*)$ according to a random choice of direction parameter, whose
law is related to $b(\cos \theta)$, and start again with the post
collision velocities. This process can be considered on $\R^{dN}$;
however it has some invariant submanifolds of $\R^{dN}$ (depending on
the number of quantities conserved under the collision), and can be
restricted to them. For instance, in the original simplified model of
Kac (scalar velocities, i.e. $d=1$) the process can be restricted to
$\Sp^{N-1}(\sqrt{ \EE N})$, the sphere with radius $\sqrt{\EE N}$,
where $\EE$ is the energy. In the more realistic models of the hard
spheres or Maxwell molecules (when $d=3$) the process can be
restricted to the following submanifold of $\R^{dN}$ associated to
elastic collisions invariants for $\MM \in \R^d$, $\EE \ge 0$:
\begin{equation}\label{BolSphere}
\mathcal S^N(\MM,\EE) := \left\{ \frac{\left|v_1-\MM\right|^2 + \dots +
  \left|v_N-\MM\right|^2}{N} =\EE \right\}
\cap \left\{ \frac{v_1 + \dots + v_N}{N} =\MM \right\}.
\end{equation}
Without loss of generality, we will consider the case $\MM=0$, using
Galilean invariance. We will denote by $\SS^N(\EE):=\SS^N(0,\EE)$ and
refer to these submanifolds as {\em Boltzmann spheres}.

Kac then formulated the following notions of {\em chaos} and {\em
  propagation of chaos}: Consider a sequence $(f^N)_{N \ge 1}$ of
probabilities on $E^N$, where $E$ is some given Polish space (e.g. $E
= \R^d$): the sequence is said to be $f$-\textit{chaotic} if
\[ 
f^N \sim f^{\otimes N}\quad\hbox{when}\quad N \to \infty
\]
for some given one-particle probability $f$ on $E$. The convergence is
to be understood as the convergence of the $\ell$-th marginal of $f^N$
to $f^{\otimes \ell}$, for any $\ell \in \N^*$, in the weak measure
topology. This is a \emph{low correlation} assumption.

It is an elementary calculation to see that if the probability
densities $f^N _t$ of the $N$-particle system were perfectly
tensorized \emph{during some time interval $t \in [0,T]$} (i.e. had
the form of an $N$-fold tensor product of a one particle probability
density $f_t$), then the latter $f_t$ would satisfy the nonlinear
limit Boltzmann equation during that time interval. But generally
interactions between particles instantaneously destroy this
``tensorization'' property and leave no hope to show its propagation
in time. Nevertheless, following Boltzmann's idea of molecular chaos,
Kac suggested that the weaker property of chaoticity can be expected
to propagate in time, in the correct scaling limit.

This framework is our starting point. Let us emphasize that the limit
performed in this setting is different from the Boltzmann-Grad
limit. It is in fact a {\em mean-field limit}. This limiting process
is most well-known for deriving Vlasov-like equations. In a companion
paper~\cite{MMW} we develop systematically our new functional approach
to the study of mean-field limits for Vlasov equations, McKean-Vlasov
equations, and granular gases equations.

\subsection{The notion of chaos and how to measure it}
\label{sec:notion-chaos-how}

Our goal in this paper is to set up a general and robust method for
proving the propagation of chaos with {\em quantitative rates} in terms
of the number of particles $N$ and of the final time of observation
$T$.  

Let us briefly discuss the notion of \emph{chaoticity}. The original
formulation in \cite{Kac1956} is: 
  A sequence $f^N \in P_{\mbox{{\tiny sym}}}(E^N)$ of symmetric\footnote{i.e. invariant
  according to permutations of the particles.}
  probabilities on $E^N$ is $f$-chaotic, for a given probability $f
  \in P(E)$, if for any $\ell \in \N^*$ and any $\varphi \in
  C_b(E)^{\otimes \ell}$ there holds
  \[
  \lim_{N \to \infty, N \ge \ell} \left\langle f^N, \varphi \otimes {\bf
      1}^{N-\ell} \right\rangle = \left\langle f^{\otimes \ell},
    \varphi \right\rangle.
  \]
Together with additional assumptions on the moments, this weak
convergence can be expressed for instance in Wasserstein distance as:
\[
\forall \, \ell \ge 1, \quad \lim_{N \to \infty, N \ge \ell} W_1 \left( \Pi_\ell \left( f^N\right)
    , f^{\otimes \ell} \right) =0
\]
where $\Pi_\ell$ denotes the marginal on the $\ell$ first
variables. We call this notion {\em finite-dimensional chaos}. 

In contrast with most previous works on this topic, we are interested
here in {\em quantitative chaos}: 
  Namely, we say that $f^N$ is $f$-chaotic with rate $\eps(N)$, where
  $\eps(N) \to 0$ as $N \to\infty$ (typically $\eps(N) \sim N^{-r}$, $r
  >0$ or $\eps(N) \sim (\ln N)^{-r}$, $r>0$), if for any $\ell \in \N^*$
  there exists $K_\ell \in (0,\infty)$ such that
  \begin{equation}\label{eq:finite-dim-chaos}
  W_1 \left( \Pi_\ell \left( f^N\right) , f^{\otimes \ell} \right) \le
  K_\ell \, \eps(N).
\end{equation}

Similar statements can be formulated for other metrics. For
instance, a convenient way to measure chaoticity is through duality:
for some normed space of {\em smooth functions} $\FF \subset C_b(E)$
(to be specified) and for any $\ell \in \N^*$ there exists $K_\ell \in
(0,\infty)$ such that for any $\varphi \in \FF^{\otimes\ell}$, $\|
\varphi \|_\FF \le 1$, there holds
\begin{equation*}
  \label{eq:chaos}
  \big| \left\langle \Pi_\ell \left[   f^N  \right] - f ^{\otimes \ell}, 
  \varphi  \right\rangle \big| 
  \le K_\ell \, \eps(N).
\end{equation*}
The Wasserstein distance $W_1$ is recovered when $\FF$ is the
space of Lipschitz functions.

Observe that in the latter statements the number of variables $\ell$
considered in the marginal is kept fixed as $N$ goes to infinity. It
is a natural question to know how the rate depends on $\ell$. As we
will see, the answer to this question is essential to the estimation
of a relaxation time that will be uniform in the number of
particles. We therefore introduce a stronger notion of {\em
  infinite-dimensional chaos}, based on {\em extensive}\footnote{By
  extensivity we mean here that the functional measuring the distance
  between two distributions should behave additively with respect to
  the tensor product.}  functionals. We
consider 
the following definition:
\[
\lim_{N \to \infty} \frac{W_1 \left( f^N, f^{\otimes N} \right)}{N} =0
\]
(with corresponding quantitative formulations\dots).  This amounts to
say that one can prove a sublinear control on $K_\ell$ in terms of
$\ell$ in \eqref{eq:finite-dim-chaos}. Variants for other extensive
metrics could easily be considered as well.

Finally one can formulate an even stronger notion of
{\em (infinite-dimensional) entropic chaos} (see \cite{CCLLV} and
definition~\eqref{def:RelativH} of the relative entropy below):
\[
\frac1N \, H\left( f^N \right | \gamma^N) \xrightarrow[]{N \to \infty} H \left(
  f  | \gamma\right), 
\]
where $\gamma^N$ is an invariant measure of the $N$-particle system,
which is $\gamma$-chaotic with $\gamma$ an invariant measure of the
limit equation. This notion of chaos obviously admits quantitative
versions as well. Moreover, it is particularly interesting as it
corresponds to the derivation of Boltzmann's entropy and $H$-theorem
from the entropies of the many-particle system. We shall come back to this
point.

Now, considering a sequence of symmetric $N$-particle densities
\[
f^N \in C([0,\infty);P_{\mbox{{\tiny sym}}}(E^N))
\]
and a $1$-particle density of the expected mean field limit 
\[
f \in C([0,\infty);P(E)),
\]
we say that there is {\em propagation of chaos} on some time interval
$[0,T)$, $T \in (0,+\infty]$, if the $f_0$-chaoticity of the initial
family $f^N_0$ implies the $f_t$-chaoticity of the family $f_t ^N$ for
any time $t \in [0,T)$, \emph{according to one of the definitions of
  chaoticity above.}

Note that the \emph{support} of the $N$-particle distributions
matters. Indeed the energy conservation implies that the evolution is
entirely decoupled on the different subspaces associated with this
invariant, e.g. each $\SS^N(\EE)$ for the different values of $\EE$
(we consider here centered distributions). On each such subspace the
$N$-particle process is ergodic and admits a unique invariant measure
$\gamma^N$, which is constant. However when considered on $\R^{dN}$
the $N$-particle process admits infinitely many invariant
measures. Therefore, in the study of the long-time behavior we will
often consider $N$-particle distributions that are supported on
$\SS^N(\EE)$ for appropriate energy $\EE$. We shall discuss the
construction of such chaotic initial data in Section~\ref{sec:chaos}.


\subsection{Kac's program}
\label{sec:kacs-program}

As was mentioned before, Kac proposed to derive the spatially
homogeneous Boltzmann equation from a many-particle Markov jump
process with binary collisions, via its master equation (the equation
on the law of the process). Intuitively, this amounts to considering
the spatial variable as a hidden variable inducing ergodicity and
markovian properties on the velocity variable. Although the latter
point has not been proved so far to our knowledge, it is worth noting
that it is a very natural guess and an extremely interesting open
problem (and possibly a very difficult one). Hence Kac introduced
artificial stochasticity as compared to the initial Hamiltonian
evolution, and raised a fascinating question: if we have to introduce
stochasticity, at least \emph{can we keep it under control through the
  process of deriving the Boltzmann equation and relate it to the
  dissipativity of the limit equation?}  \smallskip

Let us briefly summarize the main questions raised in or motivated by \cite{Kac1956}:
\begin{enumerate}
\item Kac's combinatorical proof (later to be extended to collision
  processes that preserve momentum and energy \cite{McKean1967}) had
  the unfortunate non-physical restriction that the collision kernel
  is bounded. These proofs were based on an infinite series ``tree''
  representation of the solution according to the collision history of
  the particles, as well as some sort of Leibniz formula for the
  $N$-particle operator acting on tensor product. The first open
  problem raised was: \emph{can one prove propagation of chaos for the
    hard spheres collision process?}


\item Following closely the spirit of the previous question, it is
  natural to ask whether {\em one can prove propagation of chaos for
    the true Maxwell molecules collision process}? (this is the other
  main physical model showing an unbounded collision kernel).  The
  difficulty here lies in the fact that the particle system can
  undergo infinite number of collisions in a finite time interval, and
  no ``tree'' representation of solutions is available. This is
  related to the interesting physical situation of \emph{long-range
  interaction}, as well as the interesting mathematical framework of
  \emph{fractional derivative operators} and \emph{L\'evy walk}.

\item Kac conjectures \cite[Eq.~(6.39)]{Kac1956} that (in our
  notations) the convergence
\[
\frac{H(f^N _t|\gamma^N)}{N} \to H(f_t|\gamma)
\]
is propagated in time, which would imply Boltzmann's $H$-theorem
from the monotonicity of $H(f^N|\gamma^N)/N$ for the Markov
process. He concludes with: ``{\it If the above steps could be made
  rigorous we would have a thoroughly satisfactory justification of
  Boltzmann's $H$-theorem.}''  In our notation the question is: {\em
  can one prove propagation of entropic chaos in time}?

\item Eventually, Kac discusses the relaxation times, with the goal of
  deriving the relaxation times of the limit equation from those of
  the many-particle system. This requires the estimations to be
  \emph{independent of the number of particles} on this relaxation
  times. As a first natural step he raises the question of obtaining
  such uniform estimates on the $L^2$ spectral gap of the Markov
  process. 
  This question has
  triggered many beautiful works (see the next subsection), however it
  is easy to convince oneself (see the discussion in \cite{CCLLV} for
  instance) that there is no hope of passing to the limit $N \to
  \infty$ in this spectral gap estimate, even if the spectral gap is
  independent of $N$. The $L^2$ norm is catastrophic in infinite
  dimension. Therefore 
  we reframe this question in a setting which ``tensorizes correctly in
  the limit $N \to \infty$'', that is in our notation: {\em can one
    prove relaxation times {\it independent of the number of
      particles} on the normalized Wasserstein distance
    $W_1(f^N,\gamma^N)/N$ or on the normalized relative entropy
    $H(f^N|\gamma^N)/N$?} 
\end{enumerate}

This paper is concerned with solving these four questions.


\subsection{Review of known results}
\label{sec:review}

Kac \cite{Kac1956}-\cite{Kac1957} has proved point (1) in the case of
his one-dimensional toy model. The key point in his analysis is a
clever combinatorial use of an exponential formula for the solution,
expressing it in terms of an abstract derivation operator (reminiscent
of Wild sums~\cite{Wild}). It was generalized by McKean
\cite{McKean1967} to the Boltzmann collision operator but only for
``Maxwell molecules with cutoff'', i.e. roughly when the collision
kernel $B$ is constant. In this case the convergence of the
exponential formula is easily proved and this combinatorial argument
can be extended. Kac raised in \cite{Kac1956} the question of proving
propagation of chaos in the case of hard spheres and more generally
unbounded collision kernels, although his method seemed impossible to
extend (the problem is the convergence of this exponential formula, as
discussed in \cite{McKean1967} for instance).

During the seventies, in a very abstract and compact paper
\cite{Grunbaum}, Gr\"unbaum proposed another method for dealing with
the hard spheres model, based on the Trotter-Kato formula for
semigroups and a clever functional framework (partially reminiscent of
the tools used for mean-field limits for McKean-Vlasov equations).
Unfortunately this paper was incomplete for two reasons: (1) It was
based on two ``unproved assumptions on the Boltzmann flow'' (page
328): (a) existence and uniqueness for measure solutions and (b) a
smoothness assumption. Assumption (a) was recently proved in
\cite{Fo-Mo} using Wasserstein metrics techniques and in \cite{EM},
adapting the classical DiBlasio trick \cite{DiB74}. Assumption (b),
while inspired by the cutoff Maxwell molecules (for which it is true),
fails for the hard spheres model (cf. the counterexample constructed
by Lu and Wennberg \cite{LuW02}) and is somehow too "rough" in this
case.  (2) A key part in the proof in this paper is the expansion of
the ``$H_f$'' function, which is a clever idea by Gr\"unbaum ---
however, it is again too coarse for the hard spheres case (though
adaptable to the cutoff Maxwell molecules). Nevertheless it is the
starting point for our idea of developing a differential calculus in
spaces of probability measures in order to control fluctuations.

A completely different approach was undertaken by Sznitman in the
eighties \cite{S1} (see also Tanaka \cite{T2} for partial results
concerning non-cutoff Maxwell molecules). Starting from the
observation that Gr\"unbaum's proof was incomplete, he took on to give
a full proof of propagation of chaos for hard spheres by another
approach. His work was based on: (1) a new uniqueness result for
measures for the hard spheres Boltzmann equation (based on a
probabilistic reasoning on an enlarged space of ``trajectories''); (2)
an idea already present in Gr\"unbaum's approach: reduce the question
of the propagation of chaos to a law of large numbers on measures by
using a combinatorical argument on symmetric probabilities; (3) a new
compactness result at the level of the empirical measures; (4) the
identification of the limit by an ``abstract test function''
construction showing that the (infinite particle) system has
trajectories included in the chaotic ones. While Sznitman's method
proves propagation of chaos for the hard spheres, it doesn't provide
any rate for chaoticity as defined previously.

Let us also emphasize that in \cite{McK4} McKean studied fluctuations
around deterministic limit for 2-speed Maxwellian gas and for the
usual hard spheres gas. In \cite{GM} Graham and M\'el\'eard obtained a
rate of convergence (of order $1/ N$ for the $\ell$-th marginal) on
any bounded finite interval of the $N$-particle system to the
deterministic Boltzmann dynamics in the case of Maxwell molecules
under Grad's cut-off hypothesis. Lastly, in \cite{FM7,FM10} Fournier
and M\'el\'eard obtained the convergence of the Monte-Carlo
approximation (with numerical cutoff) of the Boltzmann equation for
true Maxwell molecules with a rate of convergence (depending on the
numerical cutoff and on the number $N$ of particles).

Kac was raising the question of how to control the asymptotic behavior
of the particle system in the many-particle limit. As a first step, he
suggested to study the behavior of the spectral gap in $L^2(\mathbb
S^{N-1}(\sqrt N))$ of the Markov process as $N$ goes to infinity and
conjectured it to be bounded away from zero uniformly in terms of
$N$. This question has been answered only recently in
\cite{janvresse,CarlenCL2003} (see also
\cite{maslen,CarlenGeronimoLoss2008}). However the $L^2$ norm behaves
geometrically in terms of $N$ for tensorized data; this leave no hope
to use it for estimating the long-time behavior as $N$ goes to
infinity, as the time-decay estimates degenerate beyond times of order
$O(1/N)$. In the paper \cite{CCLLV}, the authors suggested to make use
of the relative entropy for estimating the relaxation to equilibrium,
and replace the $L^2$ spectral gap by a linear inequality between the
entropy and the entropy production. They constructed entropically
chaotic initial data, following the definition mentioned above, but
did not succeed in proving the propagation in time of this chaoticity
property. Moreover the linear inequality between the entropy and the
entropy production is conjectured to be wrong for any physical
collision kernels in \cite{Villani2003}, see for instance
\cite{MR2786394,Einav-preprint} for partial confirmations to this
conjecture.

After we had finished writing our paper, we were told about the recent
book \cite{Kolokoltsov} by Kolokoltsov. This book focuses on
fluctuation estimates of central limit theorem type. It does not prove
quantitative propagation of chaos but weaker estimates (and on finite
time intervals), and moreover we were not able to extract from it a
full proof for the cases with unbounded collision kernels (e.g. hard
spheres). However the comparison of generators for the many-particle
and the limit semigroup is reminiscent of our work.

\subsection{The abstract method}

Our initial inspiration was Gr\"unbaum's paper \cite{Grunbaum}. Our
original goal was to construct a general and robust method able to
deal with mixture of jump and diffusion processes, as it occurs in
granular gases (see our companion paper \cite{MMW} for results in this
direction). This lead us to develop a new theory, inspired from more
recent tools such as the course of Lions on ``Mean-field games'' at
Coll\`ege de France, and the master courses of M\'el\'eard
\cite{Meleard1996} and Villani \cite{VillaniMF} on mean-field
limits. One of the byproduct of our paper is that we make Gr\"unbaum's
original intuition fully rigorous in order to prove propagation of
chaos for the Boltzmann velocities jump process for hard spheres.

Like Gr\"unbaum \cite{Grunbaum}, we use a duality argument.  We
introduce the semigroup $S^N_t$ associated to the flow of the
$N$-particle system in $P(E^N)$, and the semigroup $T^N_t$ in
$C_b(E^N)$ in duality with it. We also introduce the (nonlinear)
semigroup $S^{N \! L}_t$ in $P(E)$ associated to the mean-field
dynamics (the exponent ``NL'' signifies the nonlinearity of the limit
semigroup, due to the interaction) as well as the associated (linear)
{\em ``pullback'' semigroup} $T^\infty_t$ in $C_b(P(E))$ (see
Subsection~\ref{sec:mean-field-limiting} for the definition). Then we
prove stability and convergence estimates between the linear
semigroups $T^N_t$ and $T^\infty_t$ as $N$ goes to infinity.

The preliminary step consists of defining a common functional
framework in which the $N$-particle dynamics and the limit dynamics
make sense so that we can compare them. Hence we work at the level of
the ``full'' limit space $P(P(E))$ and, at the dual level,
$C_b(P(E))$. Then we identify the regularity required in order to
prove the ``\emph{consistency estimate}'' between the generators $G^N$
and $G^\infty$ of the dual semigroups $T^N _t$ and $T^\infty _t$, and
finally prove a corresponding ``\emph{stability estimate}'' on
$T^\infty _t$, based on stability estimates at the level of the limit
semigroup $S^{N \!  L}_t$.

The latter crucial step leads us to introduce an abstract differential
calculus for functions acting on measures endowed with various metrics
associated with the weak or strong topologies. More precisely, we shall define
functions of class $C^{1,\eta}$ on the space of probability measures
by working on subspaces of the space of probability measures, for
which the tangent space has a Banach space structure. This
``stratification'' of subspaces is related to the conservation
properties of the flows $S^N_t$ and $S^{N \!  L}_t$. This notion is
related but different from the notions of differentiability developed
in the theory of gradient flow by Ambrosio, Otto, Villani and
co-authors in \cite{AmbrosioBook,OttoJK1998,OttoV2000}, and from the
one introduced by Lions in \cite{PLL-cours}.

Another viewpoint on our method is to consider it as some kind of
accurate version (in the sense that it establishes a rate of
convergence) of the BBGKY hierarchy method for proving propagation of
chaos and mean-field limit on statistical solutions. This viewpoint is
extensively explored and made rigorous in Section~\ref{sec:bbgky}
where we revisit the BBGKY hierarchy method in the case of a
collisional many-particle system, as was considered for instance in
\cite{ArkerydCI99}. The proof of uniqueness for statistical solutions
to the hierarchy becomes straightforward within our framework by using
differentiability of the limit semigroup as a function acting on
probabilities.

This general method is first exposed at an abstract level in
Section~\ref{sec:abstract-theo} (see in particular
Theorem~\ref{theo:abstract}). This method is, we hope, interesting per
se for several reasons: (1) it is fully quantitative, (2) it is highly
flexible in terms of the functional spaces used in the proof, (3) it
requires a minimal amount of informations on the $N$-particle systems
but more stability informations on the limit PDE (we intentionally
presented the assumptions in a way resembling the proof of the
convergence of a numerical scheme, which was our ``methodological
model''), (4) the ``differential stability'' conditions that are
required on the limit PDE seem (to our knowledge) new, at least for
the Boltzmann equation.

\subsection{Main results}
\label{sec:intromainresults}

Let us give some simplified versions of the main results in this
paper. All the abstract objects will be fully introduced in the next
sections with more details.

\begin{theo}[Summary of the main results]\label{theo:main}
  Consider some centered initial distribution $f_0 \in P(\R^d) \cap
  L^\infty(\R^d)$ with finite energy $\EE$, and with compact support
  or moment bounds. Consider the corresponding solution
  $f_t$ to the spatially homogeneous Boltzmann equation for hard
  spheres or Maxwell molecules, and the solution $f^N_t$ of the
  corresponding Kac's jump process starting either from the $N$-fold
  tensorization of $f_0$ or the latter conditioned to $\mathcal S^N(\EE)$.

  One can classify the results established into three main statements:
  \begin{enumerate}
  \item {\bf Quantitative uniform in time propagation of chaos, finite
      or infinite dimensional, in weak measure distance}
    (cf. Theorems~\ref{theo:tMM}-\ref{theo:max-wasserstein}-\ref{theo:HS}-\ref{theo:hs-wasserstein}):
\begin{equation*}
  \forall \, N \ge 1, \ \forall \, 1 \le \ell  \le N, \quad 
  \sup_{t \ge 0} {W_1 \left( \Pi_{\ell} f^N
      _t, \left( f_t ^{\otimes \ell} \right) \right) \over \ell} \le \alpha(N)
\end{equation*}
for some $\alpha(N) \to 0$ as $N \to \infty$. 
In the hard spheres case, the uniformity in time of this estimate is
only proved when the distribution is conditioned to $\mathcal
S^N(\EE)$. Moreover, the proof provides explicit estimates on the rate
$\alpha$, which are controlled by a power law for Maxwell molecules, and
by a power of a logarithm for hard spheres.

\item {\bf Propagation of entropic chaos}
  (cf. Theorem~\ref{theo:entropy}-(i)): Consider the case where the
  initial datum of the many-particle system has support included in
  $\mathcal S^N(\EE)$. If the initial datum is entropically chaotic in
  the sense
$$
\frac1N \, H\left( f^N _0 | \gamma^N \right) \xrightarrow[]{N \to
  +\infty} H\left(f_0 | \gamma \right) 
$$
with 
\begin{equation}\label{def:RelativH} 
H\left( f^N _0 | \gamma^N \right) :=
\int_{\SS^N(\EE)} \ln \frac{{\rm d}f^N _0}{{\rm d}\gamma^N} \,
{\rm d}f^N _0 (v) \ \mbox{ and } \ H\left(f_0 | \gamma \right) :=
\int_{\R^d} f_0 \, \ln \frac{f_0}{\gamma} \, {\rm d}v 
\end{equation}
where $\gamma$ is the Gaussian equilibrium with energy $\EE$ and
$\gamma^N$ is the uniform probability measure on $\mathcal S^N(\EE)$, then
the solution is also entropically chaotic for any later time:
$$
\forall \, t \ge 0, \quad \frac1N \, H\left( f^N _t | \gamma^N \right)
\xrightarrow[]{N \to +\infty} H\left(f_t | \gamma \right).
$$
This proves the derivation of the $H$-theorem this context, i.e. the
monotonic decay in time of $H(f_t|\gamma)$, since for any $N
\ge 2$, the functional $H(f^N_t | \gamma^N)$ is monotone decreasing
in time for the Markov process. 

\item {\bf Quantitative estimates on relaxation times, independent of
    the number of particles}
  (cf. Theorems~\ref{theo:max-wasserstein}-\ref{theo:hs-wasserstein}
  and Theorem~\ref{theo:entropy}-(ii)): Consider the case where the
  initial datum of the many-particle system has support included in
  $\mathcal S^N(\EE)$. Then we have
\begin{equation*}
  \forall \, N \ge 1, \ \forall \, 1 \le \ell  \le N, \  \forall \, t
  \ge 0, \quad  
    {W_1 \left( \Pi_{\ell} f^N
      _t, \Pi_\ell \left( \gamma^N \right) \right) \over \ell} \le
  \beta(t) 
\end{equation*}
for some $\beta(t) \to 0$ as $t \to \infty$. Moreover, the proof
provides explicit estimates on the rate $\beta$, which are controlled
by a power law for Maxwell molecules, and by a power of a logarithm
for hard spheres.

Finally in the case of Maxwell molecules, if we assume furthermore that
the Fisher information of the initial datum $f_0$ is finite: 
\begin{equation}
\label{def:infoFisher}
I(f_0) := 
  \int_{\R^d} \frac{\left| \nabla_v f_0 \right|^2}{f_0} \, {\rm d}v <
  +\infty,
\end{equation}
then the following estimate also holds:
$$ 
\forall \, N \ge 1, \quad \frac1N \, H\left( f^N _t | \gamma^N \right)
\le \beta(t) 
$$ 
for some function $\beta(t) \to 0$ as $t \to \infty$, with the same
kind of estimates.
  \end{enumerate}
\end{theo}

\subsection{Some open questions}
\label{sec:open-question}

Here are a few questions among those raised by this work: 
\begin{enumerate}
\item What about the optimal rate in the chaoticity estimates along
  time? Our method reduces this question to the chaoticity estimates
  at initial time, and therefore to the optimal rate in the
  quantitative law of large numbers for measures according to various
  weak measure distances.
\item What about the optimal rate in the relaxation times (uniformly
  in the number of particles)? Spectral gap studies predict
  exponential rates, both for the many-particle system and for the
  limit system, however our rates are far from it!
\item Can uniform in time propagation of chaos be proved for non-reversible
  jump processes (such as inelastic collision processes) for which the
  invariants measures $\gamma^N$ and $\gamma$ are not explicitely
  known (e.g. granular gases)?
\end{enumerate}

\subsection{Plan of the paper}
\label{sec:plan}

In Section~\ref{sec:abstract-setting} we set the abstract functional
framework together with the general assumption and in
Section~\ref{sec:abstract-theo} we state and prove our main abstract
theorem (Theorem~\ref{theo:abstract}). In Section~\ref{sec:chaos} we
present some tools and results on weak measure distances, on the
construction of initial data with support on the Boltzmann sphere
$\SS^N(\EE)$
for the $N$-particle system, and on the sampling process of the limit
distribution by empirical measures. In Section~\ref{sec:BddBoltzmann}
we apply the method to (true) Maxwell molecules
(Theorems~\ref{theo:tMM} and \ref{theo:max-wasserstein}). In
Section~\ref{sec:hardspheres} we apply the method to hard spheres
(Theorems~\ref{theo:HS} and
\ref{theo:hs-wasserstein}). Section~\ref{sec:h-theorem-entropic} is
devoted to the study of entropic chaos. Lastly, in
Section~\ref{sec:bbgky} we revisit the BBGKY hierarchy method for the
spatially homogeneous Boltzmann equation in the light of our
framework.


\section{The abstract setting}
\label{sec:abstract-setting}
\setcounter{equation}{0}
\setcounter{theo}{0}


In this section we shall state and prove the key abstract result. This will
motivate the introduction of a general functional framework.

\subsection{The general functional framework of the duality approach}
\label{framework}

Let us set up the framework. Here is a diagram which sums up the
duality approach (norms and duality brackets shall be specified in
Subsections~\ref{sec:funct-set}):

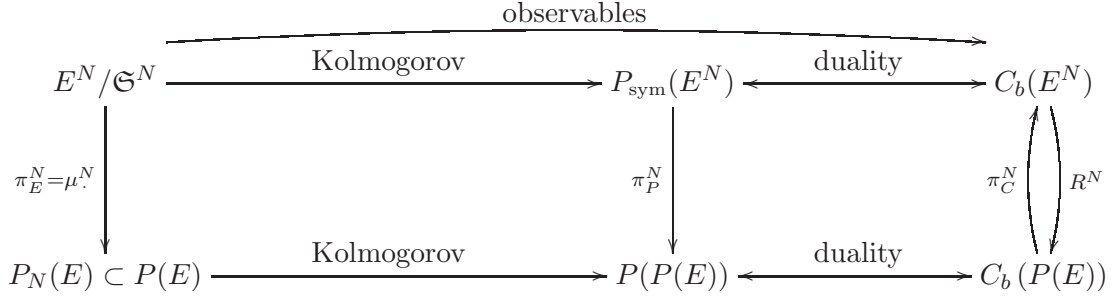
\begin{displaymath}
  \xymatrix{
    E^N / \SSS^N \ar[dd]_{\pi^N_E = \mu^N _{\, \cdot}} 
    \ar[rrrrr]^{\mbox{
    Kolmogorov }} 
    \ar@/^/@<+3ex>[rrrrrrrr]^{\mbox{observables}}
    & & & & & P_{\mbox{{\tiny sym}}}(E^N) \ar[dd]_{\pi^N_P}
    \ar[rrr]^{\mbox{duality}} 
    & & &  C_b(E^N) \ar[lll] \ar@/^/[dd]^{R^N} \\
    & & & & & & & \\
    P_N(E) \subset P(E)  
    \ar[rrrrr]^{\mbox{ 
    Kolmogorov }} & & & & & 
    P(P(E)) \ar[rrr]^{\mbox{duality}} & & &  
    C_b\left(P(E)\right) \ar[lll] \ar@/^/[uu]^{\pi^N_C} 
  }
\end{displaymath}
\medskip

In this diagram:
\smallskip

\noindent
- $E$ denotes a \emph{Polish space}:
\begin{quote} 
  This is a separable completely metrizable topological space. We
  shall denote by $d_E$ the distance on this space in the sequel. 
\end{quote}
\smallskip

\noindent
- $\SSS^N$ denotes the \emph{$N$-permutation group}.
\smallskip

\noindent
- $P_{\mbox{{\tiny sym}}}(E^N)$ denotes the set of \emph{symmetric
  probabilities} on $E^N$: 

\begin{quote} Given a permutation $\sigma \in \SSS^N$, a vector 
\[
V = (v_1, \dots, v_N) \in E^N,
\] a function $\varphi \in C_b(E^N)$ and a probability $f^N \in
P(E^N)$, we successively define 
\[
V_\sigma = (v_{\sigma(1)}, \dots, v_{\sigma(N)}) \in E^N,
\]
and 
\[
\varphi_\sigma \in C_b(E^N) \ \mbox{ by setting } \ \varphi_\sigma(V) =
\varphi(V_\sigma)
\] 
and finally 
\[
f^N_\sigma \in P(E^N) \ \mbox{ by setting } \ \left\langle
f^N_\sigma, \varphi \right\rangle = \left\langle f^N, \varphi_\sigma \right\rangle.
\]
We then say that a probability $f^N$ on $E^N$ is symmetric if it is
invariant under permutations: 
\[
\forall \, \sigma \in \SSS^N, \quad f^N_\sigma = f^N.
\]
\end{quote}
\smallskip

\noindent
- The probability measure $\mu^N _V$ denotes the {\em empirical
  measure}:
\begin{quote}
  \[
  \mu^N _V := \frac1N \, \sum_{i=1} ^N \delta_{v_i},
  \quad V=(v_1,\dots, v_N) \in E^N
  \]
  where $\delta_{v_i}$ denotes the Dirac mass on $E$ at point $v_i \in
  E$.
\end{quote}
\smallskip

\noindent
- $P_N(E)$ denotes the subset $\{\mu^N_V, \, V \in E^N \}$ of
empirical measures of $P(E)$.
\smallskip

\noindent
- $P(P(E))$ denotes the space of probabilities on the Polish space
$P(E)$ (endowed for instance with the Prokhorov distance), and this is
again a Polish space.
\smallskip

\noindent
- $C_b\left(P(E)\right)$ denotes the space of continuous and bounded
  functions on $P(E)$:
\begin{quote}
This space shall be endowed with either the weak or
  strong topologies (see Subsection~\ref{sec:funct-set}), and later
  with some metric differential structure. 
\end{quote}
\smallskip

\noindent
- The map $\pi^N_E$ from $E^N / \SSS^N$ to $P_N(E)$ is defined by
  \[
  \forall \, V \in E^N / \SSS^N, \quad \pi^N_E(V)  := \mu^N _V.
  \]
\smallskip

\noindent
- The map $\pi^N_C$ from $C_b(P(E))$ to $C_b(E^N)$ is defined by
  \[ 
  \forall \, \Phi \in C_b\left(P(E)\right), \ \forall \, V \in E^N, \
  \left(\pi^N_C\Phi\right)(V) := \Phi\left( \mu^N _V \right).
  \]
\smallskip

\noindent
- The map $R^N$ from $C_b(E^N)$ to $C_b(P(E))$ is defined by:
\begin{quote}
 \[ 
  \forall \, \varphi \in C_b(E^N), \ \forall \, \rho \in P(E), \quad
  R^N_\varphi (\rho) := \left\langle \rho^{\otimes N}, \varphi
  \right\rangle.
  \]
\end{quote}
\smallskip

\noindent 
- The map $\pi^N_P$ from $P_{\mbox{{\tiny sym}}}(E^N)$ to $P(P(E))$
is defined by: 
\begin{quote}
  \[
  \left\langle \pi^N_P f^N,\Phi
  \right\rangle = \left\langle f^N,\pi^N_C \Phi \right\rangle =
    \int_{E^N} \Phi\left( \mu^N _V \right) \, {\rm d}f^N(v)
  \]
  for any $f^N \in P_{\mbox{{\tiny sym}}}(E^N)$ and any $\Phi \in
  C_b(P(E))$, where the first bracket means $\langle \cdot , \cdot
  \rangle_{P(P(E)),C_b(P(E))}$ and the second bracket means $\langle
  \cdot , \cdot \rangle_{P(E^N),C_b(E^N)}$.
\end{quote}
\medskip

Let us now discuss the ``horizontal'' arrows: 
\medskip

\noindent
- The arrows pointing from the first column to the second one consists
in writing the {\em Kolmogorov} equation associated with the
many-particle stochastic Markov process.
\medskip

\noindent
- The arrows pointing from the second column to the third column
consists in writing the dual evolution semigroup (note that the
$N$-particle dynamics is linear). As we shall discuss later the dual
spaces of the spaces of probabilities on the phase space can be
interpreted as the spaces of observables on the original systems.
\medskip


Our functional framework shall be applied to weighted spaces of
probability measures rather than directly in $P(E)$. More precisely,
for a given weight function $m : E \to \R_+$ we shall use affine
subsets of the weighted space of probability measures
  \begin{equation*}\label{def:Mmrho}
    \left\{ f \in P(E); \ 
      M_m(f) := \langle f, m \rangle < \infty \right\}
  \end{equation*}
  as our basis functional spaces. Typical examples are 
  $m(v) := \tilde m (d_E (v,v_0))$ for some fixed $v_0 \in E$ with
  $\tilde m(z) = z^k$ or $\tilde m(z) = e^{a \, z^k}$, $a,k>0$. More
  specifically when $E= \R^d$, we shall use $m(v) = \langle v
  \rangle^k := (1 + |v|^2)^{k/2}$ or $m(v) = e^{a \, |z|^k}$, $a, k >
  0$.

  We shall sometimes abuse notation by writing $M_k$ for $M_m$ when
  $\tilde m(z) = z^k$ or $m(v)=\langle v \rangle^k$ in the examples
  above. We shall denote by $M^1$ the space of finite signed measures
  endowed with the total variation norm, and $M^1_m$ the space of
  finite signed measures $h$ whose variation $|h|$ satisfies $\langle
  |h|, m \rangle <+\infty$, and endowed with the total variation
  norm. Again we contract the notation as $M^1_k$ when $\tilde m(z) =
  z^k$ or $m(v) = \langle v \rangle^k$.

\subsection{The $N$-particle semigroups}
\label{sec:semigroups}

Let us introduce the mathematical semigroups describing the evolution
of objects living in these spaces, for any $N \ge 1$.  \medskip

\noindent {\em Step~1.} Consider the trajectories $\VV^N_t \in E^N$,
$t \ge 0$, of the particles (Markov process viewpoint).
We make the further assumption that this flow commutes
with permutations:
\begin{quote}
  For any $\sigma \in \SSS^N$, the solution at time $t$ starting from
  $\left(\VV^N_0\right)_\sigma$ is $\left(\VV^N_t\right)_\sigma$.
\end{quote}
This mathematically reflects the fact that particles are
indistinguishable.  \medskip

\noindent {\em Step~2.} This flow on $E^N$ yields a corresponding
semigroup $S^N _t$ acting on $P_{\mbox{{\tiny sym}}}(E^N)$ for the
probability density of particles in the phase space $E^N$ (statistical
viewpoint), defined through the formula 
\begin{equation*}
  \forall \, f^N_0 \in P_{\mbox{{\tiny sym}}}(E^N), \ \varphi \in C_b(E^N),
  \quad  \left \langle S^N _t (f^N_0), \varphi \right \rangle 
  = {\bf E} \left(\varphi\left(\VV^N _t\right)\right)
\end{equation*}
where the bracket obviously denotes the duality bracket between
$P(E^N)$ and $C_b(E^N)$ and ${\bf E}$ denotes the expectation
associated to the space of probability measures in which the process
$\VV^N_t$ is built.  In other words, $f^N_t := S^N _t (f^N_0)$ is nothing but the law
of $\VV_t ^N$. Since the flow $(\VV^N_t)$ commutes with
permutation, the semigroup $S^N _t$ acts on $P_{\mbox{{\tiny
      sym}}}(E^N)$:  if the law $f^N_0$ of $\VV^N_0$ belongs to $P_{\mbox{{\tiny sym}}}(E^N)$, 
then for later times the law $f^N_t$ of $\VV^N_t$ also belongs to $P_{\mbox{{\tiny sym}}}(E^N)$. 
To the $C_0$-semigroup $S^N_t$ on $P_{\mbox{{\tiny sym}}}(E^N)$ one can associate a
linear evolution equation with a generator denoted by $A_N$:
\begin{equation*}
  \partial_t f^N = A^N f^N, \qquad f^N \in P_{\mbox{{\tiny sym}}}(E^N), 
\end{equation*}
which is the \emph{forward Kolmogorov (or Master) equation} on the law of
$\VV^N_t$.  \medskip

\noindent {\em Step~3.} We also consider the Markov semigroup $T^N_t$ acting on the
functions space $ C_b(E^N)$ of {\em observables} on the evolution
system $(\VV_t ^N)$ on $E^N$ (see the discussion in the next remark), 
which is {\em in duality} with the semigroup $S^N _t$, in the sense that:
\[ 
\forall \, f^N \in P(E^N), \ \varphi \in C_b(E^N), \quad \left\langle
  f^N, T^N _t (\varphi) \right\rangle = \left\langle S^N_t(f^N), \varphi
\right\rangle.
\] 
To the  $C_0$-semigroup $T^N_t$ on $C_b(E^N)$ we can
associate  the following {\em linear} evolution equation with
a generator denoted by $G^N$:
\begin{equation*}
  \partial_t \varphi = G^N(\varphi), \qquad \varphi \in C_b(E^N),
\end{equation*}
which is the \emph{backward Kolmogorov equation}.

\bigskip

\subsection{The mean-field limit semigroup}
\label{sec:mean-field-limiting}
We now define the evolution of the limit mean-field equation.
\medskip

\noindent {\em Step~1.} Consider a semigroup $S^{N \!  L}_t$ acting on
$P(E)$ associated with an evolution equation and some operator $Q$:
\begin{quote}
  For any $f_0 \in P(E)$ (assuming possibly some additional moment
  bounds), then $S^{N \!  L}_t(f_0) := f_t$ where $f_t \in
  C(\R_+, P(E))$ is the solution to
  \begin{equation}\label{eq:limit}
    \partial_t f_t  = Q(f_t), \quad f_{|t=0} = f_0.
  \end{equation}
  This semigroup and the operator $Q$ are typically \emph{nonlinear}
  for mean-field models, namely bilinear in case of Boltzmann's
  collisions interactions.
\end{quote}

\medskip

\noindent {\em Step~2.}
Consider then the associated {\em pullback semigroup} $T^\infty _t$ acting on
$C_b(P(E))$:
\[ 
\forall \, f \in P(E), \ \Phi \in C_b(P(E)), \quad T^\infty _t
[\Phi](f) := \Phi\left( S^{N\! L}_t(f)\right).
\]
(Again additional moment bounds can be required on $f$ in order to
make this definition rigorous.)  Note carefully that $T_t ^\infty$ is
always {\em linear} as a function of $\Phi$, although of course $T_t
^\infty[\Phi](f)$ is not linear in general as a function of $f$. We
shall associate (when possible) the following {\em linear} evolution
equation on $C_b(P(E))$ with some generator denoted by $G^\infty$:
\begin{equation*}
  \partial_t \Phi = G^\infty(\Phi).
\end{equation*}

\begin{rem}\label{edo-edp}
  The semigroup $T_t ^\infty$ can be interpreted physically as the
  semigroup of the evolution of {\em observables} of the nonlinear
  equation \eqref{eq:limit}. Let us give a short heuristic
  explanation. Consider a nonlinear ordinary differential equation
  \[
  \frac{{\rm d}v}{dt} =F(V) \ \mbox{ on } \ \R^d \ \mbox{ with } \ \nabla_v
  \cdot F \equiv 0 \ \mbox{ and } \ V_{| t=0} = v
  \] 
  with divergence-free vector field for simplicity. One can then define
  formally the {\em linear} Liouville transport partial differential
  equation
  \[ 
  \partial_t f + \nabla_v \cdot (F \, f) =0,
  \]
  where $f=f_t(v)$ is a time-dependent probability density over the
  phase space $\R^d$, whose solution is given (at least formally) by
  following the characteristics backward $f_t(v) =
  f_0(V_{-t}(v))$. 
   Now, instead of the Liouville viewpoint, one can adopt the viewpoint
  of {\em observables}, that is functions depending on the position of
  the system in the phase space (e.g. energy, momentum, etc ...). For
  some observable function $\varphi_0$ defined on $\R^d$, the
  evolution of the value of this observable along the trajectory is
  given by $\varphi_t(v)=\varphi_0(V_t(v))$ and $\varphi_t$ is
  solution to the following {\em dual} linear PDE
  \[
  \partial_t \varphi - F \cdot \nabla_v \varphi = 0.
  \]

  Now let us consider a nonlinear evolution system 
\[
\frac{d f}{dt} =Q(f) \  \mbox{ in an abstract space } \ f \in \HH.
\] 
By  analogy 
we define two linear evolution systems on the larger functional spaces
$P(\HH)$ and $C_b(\HH)$: first the abstract Liouville equation
\[ 
\partial_t \pi + {\delta \over \delta f} \cdot (Q(f) \, \pi) = 0,
\qquad \pi \in P(\HH)
\]
and second the abstract equation for the evolution of observables
\[
\partial_t \Phi - Q(f) \cdot {\delta \Phi \over \delta f}(f)
= 0, \qquad \Phi \in C_b(\HH).
\]

However in order to make sense of this heuristic, the scalar product
have to be defined correctly as duality brackets, and, most
importantly, a differential calculus on $\HH$ has to be defined
rigorously. Taking $\HH=P(E)$, this provides an intuition for our
functional construction, as well as for the formula of the generator
$G^\infty$ below (compare the previous equation with formula
\eqref{eq:formulaGinfty}). Be careful that when $\HH=P(E)$, the
abstract Liouville and observable equations refers to trajectories
{\em in the space of probabilities} $P(E)$ (i.e. solutions to the
nonlinear equation \eqref{eq:limit}), and not trajectories of a
particle in $E$.  Note also that for a {\em dissipative equation} at
the level of $\HH$ (such as the Boltzmann equation), it seems more
convenient to use the observable equation rather than the Liouville
equation since ``forward characteristics'' can be readily used in
order to construct the solutions to this observable
equation.  
\end{rem}

Summing up we obtain the following picture for the semigroups:
\smallskip
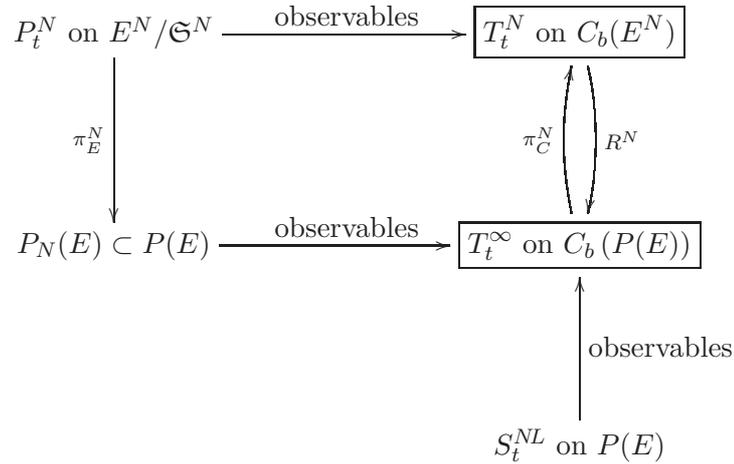
\begin{displaymath}
  \xymatrix{
    P^N_t \ \mbox{on} \ E^N / \SSS^N \ar[dd]_{\pi^N_E} 
    \ar[rrr]^{\mbox{observables}} 
    & & & \boxed{T^N _t \ \mbox{on} \ C_b(E^N) \ar@/^/[dd]^{R^N}} \\
    & & &  \\
    P_N(E) \subset P(E)  
    \ar[rrr]^{\mbox{observables}} 
    & & & \boxed{T^\infty _t \ \mbox{on} \ C_b\left(P(E)\right)} 
    \ar@/^/[uu]^{\pi^N _C}
    \\
    & & & \\
    & & & S^{N\! L}_t\ \mbox{on} \ P(E) \ar[uu]_{\mbox{observables}} 
  }
\end{displaymath}
\smallskip

Hence a key point of our construction is that, through the evolution
of {\em observables}, we shall ``interface'' the two evolution systems
(the nonlinear limit equation and the $N$-particle system) via the
applications $\pi^N _C$ and $R^N$. From now on we shall write $\pi^N
= \pi^N _C$.


\subsection{The metric issue}
\label{sec:funct-set}

$C_b(P(E))$ is our fundamental space in which we  shall compare
(through their observables) the semigroups of the $N$-particle system
and the limit mean-field equation. Let us make the topological and
metric structures used on $P(E)$ more precise. At the topological
level there are two canonical choices (which determine two different
sets $C_b(P(E))$):
\begin{enumerate}
\item[(1)] The strong topology which is associated to the total
  variation norm, denoted by $\|\cdot \|_{M^1}$; the corresponding set
  shall be denoted by $C_b(P(E),TV)$. 
\item[(2)] The weak topology, i.e. the trace on $P(E)$ of the weak
  topology on $M^1(E)$ (the space of Radon measures on $E$ with finite
  mass) induced by $C_b(E)$; the corresponding set shall be denoted by
  $C_b(P(E),w)$.
\end{enumerate}

It is clear that 
\[
C_b(P(E),w) \subset C_b(P(E),TV).
\]
The supremum norm $\| \Phi \|_{L^\infty(P(E))}$ does {\em not} depend on
the choice of topology on $P(E)$, and induces a Banach topology on the
space $C_b(P(E))$. The transformations $\pi^N$ and $R^N$ satisfy:
\begin{equation}\label{eq:compat:infty}
  \left\| \pi^N \Phi \right\|_{L^\infty(E^N)} \le \| \Phi
  \|_{L^\infty(P(E))} \ \mbox{ and } \ \left\| R^N _\phi \right\|_{L^\infty(P(E))} \le
  \| \phi \|_{L^\infty(E^N)}.
\end{equation}

The transformation $\pi^N$ is well defined from $C_b(P(E),w)$ to
$C_b(E^N)$, but in general, it does not map $C_b(P(E),TV)$ into
$C_b(E^N)$ since the map 
\[
V \in E^N \mapsto \mu^N_V \in (P(E),TV)
\]
is not continuous.

In the other way round, the transformation $R^N$ is well defined from
$C_b(E^N)$ to $C_b(P(E),w)$, and therefore also from $C_b(E^N)$ to
$C_b(P(E),TV)$: for any $\phi \in C_b(E^N)$ and for any sequence
$f_k \wto f$ weakly, we have $f_k^{\otimes N} \wto
f^{\otimes N}$ weakly, and then $R^N[\phi](f_k) \to
R^N[\phi](f)$.

\smallskip There are many different possible metric structures
inducing the weak topology on $C_b(P(E),w)$. The mere notion of
continuity does not require discussing these metrics, but any subspace
of $C_b(P(E),w)$ with differential regularity shall strongly depend on
this choice, which motivates the following definitions.


\begin{defin}\label{defGG1} 
For a given \emph{weight function} $m : E \to \R_+$, 
we define the subspaces of probabilities:
 \begin{equation*}
  \mathcal P_{m} := \left\{ f \in P(E); \,\, \langle f, m \rangle < \infty \right\}.
  \end{equation*}
  As usual we contract the notation as $\PP_k $ when $E = \R^d$ and
  $m(v) = \langle v \rangle^k := (1+|v|^2)^{k/2}$, $k \in \R$, $v \in
  \R^d$.

  We also define the corresponding bounded subsets for $a > 0$
  \[
  \mathcal{B P}_{m,a} := \{ f \in P_m; \,\, \langle f, m
  \rangle \le a \}.
  \]
 
  For a given \emph{constraint function} ${\bf m} : E \to \R^D$ such
  that $\langle f, {\bf m} \rangle$ is well defined for any $f \in
  \mathcal P_m$ and a given 
  \emph{space of constraints} $\RR_{m,{\bf m}} \subset \R^D$, 
  for any ${\bf r} \in \mathcal R_{m,{\bf m}}$, we define the
  corresponding (possibly empty) \emph{constrained subsets}
  $$
  \mathcal P_{m,{\bf m},\hbox{\small\bf r}} := \left\{ f \in \mathcal P_m; \,\, \langle
  f, {\bf m} \rangle = {\bf r}\right\},
  $$ 
  and the corresponding (possibly empty) \emph{bounded constrained
    subsets}
  $$
  \mathcal{B P}_{m,{\bf m},a,\hbox{\small\bf r}} := \left\{ f \in
  \mathcal{B P}_{m,a}; \,\, \langle  f, {\bf m} \rangle = {\bf r}  \right\}.
  $$

We also define the corresponding \emph{space of increments}
  $$
  \mathcal{I P}_{m,{\bf m},\RR_{m,{\bf m}}} := \left\{ f_2 - f_1; \,\, \exists \, {\bf
      r} \in \RR_{m,{\bf m}} \,\,\hbox{s.t.}\,\, f_1,f_2 \in \mathcal P_{m,{\bf m},{\bf r}}
  \right\}.
  $$
\end{defin}
\smallskip

Be careful that the space of increments is not a vector space in
general.  Let us now define the notion of distances over probabilities
that we shall consider.

\begin{defin}\label{defGG2}
  Consider a weight function $m_\GG$, a constraint function ${\bf
    m}_\GG$ and a set of constraints $\RR_\GG$.  We shall use for the
  associated spaces of the previous definition the following
  simplified contracted notation: $\mathcal P_\GG$ for $\mathcal P_m$,
  $\mathcal{BP}_{\GG,a}$ for $\mathcal {BP}_{m,a}$,
  $\mathcal{R}_{\GG}$ for $\mathcal{R}_{m,{\bf m}}$, $\mathcal
  P_{\GG,{\bf r}}$ for $\mathcal P_{m,{\bf m},{\bf r}}$, $\mathcal{B
    P}_{\GG,a,{\bf r}}$ for $\mathcal{B P}_{m,{\bf m},a,{\bf r}}$ and
  $\mathcal{I P}_{\GG}$ for $\mathcal{I P}_{m,{\bf m},\RR_{m,{\bf m}}}$.

\smallskip
  We shall consider a distance $d_\GG$ which
  \begin{quote}
    (1) either is defined on the whole space $\mathcal P_\GG$
    (i.e. whatever the values
    of the constraints), \\
    (2) or such that there is a Banach space $\GG \supset \mathcal{I}
    \mathcal P_\GG$ endowed with a norm $\| \cdot \|_\GG$ such that
    $d_\GG$ is defined for any ${\bf r} \in \RR_\GG$ on $\mathcal
    P_{\GG,\hbox{\small\bf r}}$, by setting
  $$
  \forall \, f, \, g \in \mathcal P_{\GG,\hbox{\small\bf r}}, \quad d_\GG(f,g)
  := \| g- f \|_\GG.
  $$
\end{quote}
\end{defin}
\smallskip

Let us finally define a quantitative H\"older notion of equivalence for the
distances over probabilities.

\begin{defin}\label{holderequiv}
  Consider some weight and constraint functions $m_\GG$, ${\bf
    m}_\GG$. We say that two metrics $d_0$ and $d_1$ defined on $\mathcal P_\GG$ are
  {\it H\"older equivalent on bounded sets} if there
  exists $\kappa \in (0,\infty)$ and, for any $a \in (0,\infty)$, there
  exists $C_a \in (0,\infty)$ such that
  $$
  \forall \, f, \, g \in \BB \mathcal P_{\GG,a}, \quad
  d_0(f,g) \le C_a \, \left[d_1(f,g)\right]^\kappa, \quad 
  d_1(f,g) \le C_a \, \left[d_0(f,g)\right]^\kappa
  $$
  for some constant $C_a$ depending on $a >0$. 

  If $d_0$ and $d_1$ are, as in the previous definition, only defined
  on $\mathcal P_{\GG,{\bf r}}$ for given values of the constraints
  ${\bf r} \in \bf \RR_\GG$, we modify this definition as follows: 
$$
\forall \, {\bf r} \in \RR_\GG, \ \forall \, f, \, g \in \BB \mathcal P_{\GG,{\bf
    r}, a}, \quad d_0(f,g) \le C_a \, \left[d_1(f,g)\right]^\kappa,
\quad d_1(f,g) \le C_a \, \left[d_0(f,g)\right]^\kappa
  $$
  for some $\kappa \in (0,\infty)$ and some constant $C_a$ depending
  on $a >0$.  
  \end{defin}
\smallskip

\begin{ex}
  The choice 
\[
m_\GG := 1, \quad {\bf m}_\GG := 0,\quad \RR_\GG := \{ 0 \}, \quad \| \cdot \|_\GG
  := \|\cdot \|_{M^1}
\]
recovers $\mathcal P_\GG(E) = P(E)$. More generally on can choose 
\[
m_{\GG_k}(v) := d_E(v,v_0)^{k}, \quad {\bf m}_{\GG_k} := 0, \quad \RR_{\GG_k} := \{ 0 \}, \quad \|
\cdot \|_{\GG_k} := \left\|\cdot \, d_E(v,v_0)^{k} \right\|_{M^1}.
\]
For $k_1 > k_2, k_3 \ge 0 $, the distances $d_{\GG_{k_2}}$ and
$d_{\GG_{k_3}}$ are H\"older equivalent on bounded sets of $\mathcal
P_{\GG_{k_1}}$.
\end{ex}

\begin{ex}
  There are many distances on $P(E)$ which induce the weak topology,
  see for instance \cite{BookRachev}. In the next section, we present
  some of them which have a practical interest for us, and which are
  all topologically uniformly equivalent on bounded sets of $P(E)$ in
  the sense of the previous definition, with the choice of a
  convenient (strong enough) weight function.
\end{ex}

 \subsection{Distances on probabilities}
 \label{subsec:ExpleMetrics} 
 Let us discuss some well-known distances on $P(\R^d)$ (or defined on
 subsets of $P(\R^d)$), which shall be useful in the sequel. These
 distances are all topologically equivalent to the weak topology
 $\sigma (P(E),C_b(E))$ on the sets $\BB \PP_{k,a}(E)$ for $k$ large
 enough and for any $a \in (0,\infty)$, and they are all uniformly
 topologically equivalent (see \cite{TV,coursCT} and
 section~\ref{subsect:ComparisonDistance}).  We refer to
 \cite{BookRachev,VillaniTOT,coursCT} and the references therein for
   more details on these distances.

\subsubsection{Dual-H\"older --or Zolotarev's-- Distances}\label{expleZolotarev} 
  Denote by $d_E$ a distance on $E$ and let us fix $v_0 \in
  E$ (e.g. $v_0=0$ when $E = \R^d$ in the sequel). Denote by
  $C^{0,s}_0(E)$, $s \in (0,1)$ (resp. $\mbox{Lip}_0(E)$) the set of
  $s$-H\"older functions (resp. Lipschitz functions) on $E$ vanishing
  at one arbitrary point $v_0 \in E$ endowed with the norm
  $$
  [\varphi ]_{s} := \sup_{x,y \in E} {|\varphi(y) - \varphi(x)| \over
    d_E(x,y)^s}, s \in (0,1], 
  \qquad  [\varphi ]_{\mbox{{\scriptsize Lip}}} := [\varphi ]_1.
  $$
  We then define the dual norm: take $m_\GG := 1$, ${\bf m}_\GG := 0$,
  $\RR_\GG := \{ 0 \}$, and $\mathcal P_\GG(E)$ endowed with
 \begin{equation*}
   \forall \, f,g \in \mathcal P_\GG,
  \quad [g-f]^*_s := \sup_{\varphi \in C^{0,s}_0(E)} {\langle
    g - f, \varphi \rangle \over [\varphi]_s }. 
  \end{equation*}

\subsubsection{Wasserstein distances}\label{expleWp} 
  Given $q \in [1,\infty)$, define $W_q$ on
  $$
  \mathcal P_\GG (E) = \PP_q(E):= \left\{ f \in P(E); \,\, \left \langle f,
    d_E(\cdot,v_0)^q \right \rangle < \infty \right\}
  $$
  by
  \begin{eqnarray*}
    \forall \, f,g \in \PP_q(E), \quad  W_q(f,g) 
    :=  \inf_{\pi \in \Pi(f,g)} \Bigl( \int_{E\times E} d_E(x,y)^q \,
    {\rm d} \pi(x,y)\Bigr)^{1/q},
\end{eqnarray*}
where $\Pi(f,g)$ denotes the set of probability measures $\pi \in
P(E \times E)$ with marginals $f$ and $g$: 
\[
\pi(A,E) = f (A) \ \mbox{ and } \ \pi(E,A) = g (A) \ \mbox{ for any
  Borel set } A \subset E.
\]
Note that for $V_1,V_2 \in E^N$ and any $q \in [1,\infty)$, one has
\begin{equation}\label{Wqellq} 
  W_q\left(\mu^N_{V_1},\mu^N_{V_2}\right) =
  d_{\ell^q(E^N/\SN)}
  \left(V_1, V_2\right) := \min_{\sigma \in \SN} \left( {1 \over N} \sum_{i=1}^N
    d_E\left((V_1)_i,(V_2)_{\sigma(i)}\right)^q \right)^{1/q}, 
\end{equation} 
and that
\begin{equation*}
  \forall \, f, \, g \in
  \PP_1(E) , \quad W_1 (f,g) 
  = [f-g]^*_1 = \sup_{\varphi \in \mbox{\tiny Lip}_0(E)}\, \left \langle f-g,
  \varphi \right \rangle.  
\end{equation*}

\subsubsection{Fourier-based norms} \label{expleFourier} Given
$E=\R^d$, $m_{\GG_1} := |v|$, ${\bf m}_{\GG_1} :=0$, $\RR_{\GG_1} :=
\{ 0 \}$, let us define
  \[ 
  \forall \, f \in \mathcal{I} \mathcal P_{\GG_1}, \quad \| f
  \|_{\GG_1} = |f|_s := \sup_{\xi \in \R^d} \frac{|\hat
    f(\xi)|}{|\xi|^s}, \quad s \in (0,1],
  \]
  where $\hat f$ denotes the Fourier transform of $f$ defined through
  the expression 
  $$
  \hat f (\xi) = (\FF f) (\xi) := \int_{\R^d}  e^{- i \, x \cdot \xi} \, {\rm d}f(x).
  $$
  Similarly, given $E=\R^d$, $m_{\GG_2} := |v|^2$, ${\bf m}_{\GG_2}
  :=v$, $\RR_{\GG_2} := \R^d$, we define
  \begin{equation}\label{eq:defToscani}
  \forall \, f \in \mathcal{I}  \mathcal \PP_{\GG_2}, \quad \| f \|_{\GG_2} = |f|_s
  := \sup_{\xi \in \R^d} \frac{|\hat f(\xi)|}{|\xi|^s}, \quad s \in (1,2].
\end{equation}
Obviously higher-order versions of this norm could be defined
similarly by increasing the number of constraints. However we shall
see in the next subsubsection how to extend this notion of distance
without constraints.

  For any given $a > 0$ and any constraint
  $$
  {\bf r} \in \RR_{\GG_1,a} := \left\{ {\bf r} \in \R^d; \,\, \exists \, f \in \PP_2(\R^d), \,\, \langle f,|v|^2 \rangle \le a,
  \,\, \langle f, v \rangle = {\bf r} \right\} = \left\{ {\bf r} \in \R^d, \,\, | {\bf r } |^2 \le a \right\},
  $$
  we observe that on the set 
  $$
  \BB \mathcal P_{\GG_1,a,{\bf r}} := \left\{ f \in \PP_2(\R^d), \,\, \langle f,|v|^2 \rangle \le a,
  \,\, \langle f, v \rangle = {\bf r} \right\},
  $$
  the distance $d_{\GG_2}$ is bounded, so that the diameter of $\BB
  \mathcal P_{\GG_1,a,{\bf r}}$ is bounded: for any $f_1, f_2 \in \BB
  \mathcal P_{\GG_1,a,{\bf r}}$ there holds
  \begin{multline*}
  \left\|f_1-f_2 \right\|_{\GG_2} = \sup_{\xi \in \R^d} {1 \over
    |\xi|^2} \left| \int_{\R^d} (e^{-i \, \xi \cdot v} - 1 + i \, \xi
    \cdot v ) \, ({\rm d}f_1-{\rm d}f_2)(v) \right| \\ \le {1 \over 2}
  \int_{\R^d} |v|^2 \, \left|({\rm d}f_1-{\rm d}f_2)(v)\right| \le a.
\end{multline*}

\subsubsection{More Fourier-based norms} \label{expleFourierGal} 
More generally, given $E=\R^d$ and $k \in \N^*$, we set 
\[
m_{\GG} := |v|^k,\quad {\bf m}_{\GG} := \left(v^\alpha\right)_{\alpha
  \in \N^d, \, |\alpha| \le k-1}, \quad \RR_\GG := \R^D, \quad D
:= d + \dots + (k-1) \, d,
\]
with $|\alpha| = \alpha_1 + \dots + \alpha_d$ and
\[
v^\alpha = \left(v_1^{\alpha_1},\dots, v_d^{j_d}\right), \quad \alpha
= \left(\alpha_1, \dots, \alpha_d\right) \in \N^d,
\]
and we define
  \[ 
  \forall \, f \in \mathcal{I}  \mathcal P_{\GG}, \quad \| f \|_{\GG} = |f|_s
  := \sup_{\xi \in \R^d} \frac{|\hat f(\xi)|}{|\xi|^s}, \quad s \in (0,k].
  \]
  In fact, we may extend the above norm to $M^1_k(\R^d)$ in the
  following way.  We first define for
\[
f \in M^1_{k-1}(\R^d) \ \mbox{ and } \ \alpha \in \N^d, \,
|\alpha|=\alpha_1 + \dots +\alpha_d \le k-1
\]
the following moment: 
$$
M_\alpha[f] := \int_{\R^d} v^\alpha\, {\rm d} f(v).
$$
Consider a fixed (once for all) function $\chi =\chi(\xi) \in C^\infty _c(\R^d)$
(smooth with compact support), such that $\chi \equiv 1$ on the set
$\{ \xi \in \R^d, |\xi| \le 1 \}$. This implies in particular
\[
\int_{\R^d} \FF^{-1} (\chi) (v) \, {\rm d}v = \chi(0) = 1.
\]
Then we define the following function $\MM_k[f]$ through its Fourier
transform
\[ 
\hat{\mathcal{M}_k}[f](\xi) := \chi(\xi) \, \left( \sum_{|\alpha| \le k-1}
  M_\alpha[f] \, \frac{\xi^\alpha}{\alpha!} 
\right), \quad \alpha ! : = \alpha_1 ! \dots \alpha_d !
\]
Note that this is a mollified version of the $(k-1)$-Taylor expansion of $\hat f$ at
$\xi=0$. Then we may define the norm
\beqn\label{eq:defToscanimodif}
||| f |||_k := |f - \MM_k[f]|_k + \sum_{\alpha \in \N^d, \, |\alpha| \le k-1} \left|M_\alpha[f] \right|
\eeqn
where 
\begin{equation*}
  |g|_k := \sup_{\xi \in \R^d} \frac{\left| \hat g(\xi)  \right|}{|\xi|^{k}}
\end{equation*}
is defined for a signed measure whose Fourier transform Taylor
expansion at zero cancels up to the order $k-1$. 

\subsubsection{Negative Sobolev norms}\label{expleH-s} 
  Given $s \in (d/2,d/2+1/2)$ and 
\[
E=\R^d, \quad m_{\GG_1} := |v|, \quad {\bf m}_{\GG_1} :=0, \quad \RR_{\GG_1} := \{ 0 \}, 
\]
we define the following negative homogeneous Sobolev norm
  \[ 
  \forall \, f \in \mathcal{I}  \mathcal P_{\GG_1}, \quad \| f \|_{\GG_1}
  = \| f \|_{\dot H^{-s} (\R^d)} := \left\|
    \frac{\hat f(\xi)}{|\xi|^s} \right\|_{L^2}.
  \]

Similarly, given $s \in [d/2+1/2,d/2+1)$ and 
\[
E=\R^d, \quad m_{\GG_2} := |v|^2, \quad {\bf m}_{\GG_2} :=v, \quad
\RR_{\GG_2} := \R^d, 
\]
we define
\[
  \forall \, f \in \mathcal{I}  \mathcal P_{\GG_2}, \quad \| f \|_{\GG_2}
  = \| f \|_{\dot H^{-s} (\R^d)} := \left\|
    \frac{\hat f(\xi)}{|\xi|^s} \right\|_{L^2(\R^d)}.
  \]
  
  It is also possible to use the non-homogeneous Sobolev space
  $H^{-s}(\R^d)$. Observe that probabilities are included in $H^{-s}(\R^d)$
  as soon as $s>d/2$.

 \subsubsection{Comparison of distances when $E=\R^d$}
\label{subsect:ComparisonDistance}

All the previous distances are \emph{H\"older equivalent on bounded
  sets} in the sense of Definition~\ref{holderequiv}. Precise
quantitative statements of these equivalences are given in
Lemma~\ref{lem:ComparDistances} in Section~\ref{app:distances}.


\subsection{Differential calculus in the space of probability measures}
\label{sec:diff-measures}

We start with a purely metric definition in the case of usual H\"older
regularity.

 \begin{defin}\label{def:Holdercalculus}
   Given some metric spaces $\tilde\GG_1$ and $\tilde\GG_2$, some
   weight function
\[
\Lambda : \tilde\GG_1 \mapsto [1,+\infty),
\]
 we denote by 
\[
UC_\Lambda(\tilde \GG_1,\tilde\GG_2)
\]
the weighted space of uniformly continuous functions from $\tilde
\GG_1$ to $\tilde\GG_2$, that is the set of functions $\SS : \tilde
\GG_1 \to \tilde\GG_2$ such that there exists a modulus of continuity
$\omega$ so that
  \begin{equation}
    \label{eq:devphi}
    \forall \, f_1, \, f_2 \in \tilde \GG_1, \quad 
    d_{\GG_2} \left(\SS(f_1) , \SS(f_2) \right) 
    \le  \Lambda(f_1,f_2) \, \omega 
    \left(  d_{\GG_1} \left(f_1, f_2\right) \right),
  \end{equation}
  with 
  \[
  \Lambda\left(f_1,f_2\right) := \max \left\{ \Lambda\left(f_1\right),
    \Lambda\left(f_2\right) \right\}
  \]
  and where $d_{\GG_k}$ denotes the metric of $\tilde \GG_k$. Note
  that the tilde sign in the notation of the distance has been removed
  in order to present unified notation with the next definition.
  
  For any $\eta \in (0,1]$, we denote by 
\[
C^{0,\eta}_\Lambda(\tilde \GG_1,\tilde\GG_2)
\]
the weighted space of functions from $\tilde \GG_1$ to $\tilde\GG_2$
with $\eta$-H\"older regularity, that are the uniformly continuous
functions for which the modulus of continuity satisfies $\omega(s) \le
C \, s^\eta$ for some constant $C >0$.  We then define the
\emph{semi-norm}
\[
[ S ]_{C_\Lambda^{0,\eta}(\tilde \GG_1,\tilde\GG_2)} \ \mbox{ for } \ S \in
  C_\Lambda^{0,\eta}(\tilde \GG_1,\tilde\GG_2)
\]
as the infimum of the constants $C > 0$ such that \eqref{eq:devphi}
holds with $\omega(s) =  C \, s^\eta$. 
 \end{defin}

\smallskip

We now define a first order differential calculus, for which we
require a norm structure on the functional spaces.

\begin{defin}\label{def:C1kcalculus} 
  Given some \emph{Banach} spaces $\GG_1, \GG_2$ and some
  \emph{metric} sets $\tilde \GG_1, \tilde \GG_2$ such that 
  \[
  \II \GG_i := \tilde \GG_i - \tilde \GG_i \subset \GG_i, \quad i =1, 2,
  \]
  and where all vectorial lines of $\GG_i$ intersect $\II \GG_i$, some
  weight function
  \[ 
  \Lambda : \tilde\GG_1 \mapsto [1,\infty),
  \]
we define
  \[
  UC_\Lambda^{1}\left(\tilde \GG_1, \GG_1; \tilde \GG_2, \GG_2\right)
  \]
  (later simply denoted by $UC_\Lambda^{1}(\tilde \GG_1; \tilde
  \GG_2)$), the space of continuously differentiable functions from
  $\tilde \GG_1$ to $\tilde \GG_2$, whose derivative satisfies some
  weighted uniform continuity.

  In a more explicit way, this is the set of uniformly continuous
  functions
  \[
  \SS : \tilde \GG_1 \to \tilde \GG_2
  \]
  such that there exists a map 
  \[
  D \SS : \tilde \GG_1 \to \LL(\GG_1,\GG_2)
  \]
  (where $\LL(\GG_1,\GG_2)$ denotes the space of linear
  applications from $\GG_1$ to $\GG_2$), 
  some modulus of continuity 
\[
\Omega_c : \R_+ \to \R_+, \quad \Omega_c(s) \to 0 \quad \mbox{ as } \quad
s \to 0
\] 
and some modulus of differentiability 
\[
\Omega_d : \R_+ \to \R_+, \quad \frac{\Omega_d(s)}{s} \to 0 \quad
\mbox{ as } \quad s \to 0
\]
so that for any $f_1, \, f_2 \in \tilde \GG_1$:
  \begin{eqnarray}
    \label{eq:devdist1}\qquad
    \left\| \SS(f_2) - \SS(f_1)\right\|_{\GG_2} \!\!&\le&\!\!   
    \Lambda(f_1,f_2)\, \Omega_c\left( \|f_2 - f_1 \|_{\GG_1} \right)
    \\
    \label{eq:devdist3}\qquad
    \left\| \SS(f_2) - \SS(f_1) - \left \langle D \SS[f_1] ,
        f_2 - f_1 \right \rangle \right\|_{\GG_2} \!\!&\le&\!\!  
    \, \Lambda(f_1,f_2) \, \Omega_d \left( \|f_2 - f_1 \|_{\GG_1} \right).
    \end{eqnarray} 


   
     For any $\eta \in (0,1]$, we also denote by 
      \[
      C_\Lambda^{1,\eta}\left(\tilde \GG_1, \GG_1; \tilde \GG_2, \GG_2\right)
  \]
  (later simply denoted by $C_\Lambda^{1,\eta}(\tilde \GG_1; \tilde
  \GG_2)$), the space of continuously differentiable functions from
  $\tilde \GG_1$ to $\tilde \GG_2$, so that 
\[
\Omega_c (s) = C_c \, s^{\eta'} \quad \mbox{ and} \quad  
\Omega_d(s) = C_d \, s^{1+\eta}
\]
for some constants $C_c, C_d >0$ and $\eta' \in [\eta,1]$. 
%

     We define respectively $C_c ^\SS$, 
     $C_d ^\SS$, as
    the infimum of the constants $C_c, 
    C_d > 0$ such that
    respectively \eqref{eq:devdist1}, 
    \eqref{eq:devdist3} holds with the above choice of modulus $\Omega_i$. 
    We then define the \emph{semi-norms}
    \[ 
    [\SS ]_{C^{0,\eta'}_\Lambda} := C_c^\SS, 
    \quad [\SS ]_{C^{1,\eta}_\Lambda}
    := C_d ^\SS
\]
and the norm
\[
\| \SS \|_{C^{1,\eta}_\Lambda} := C_c^\SS 
+ C_d^\SS.
  \]

  In the sequel we omit the subscript $\Lambda$ or we replace it by
  the subscript $b$ in the case when $\Lambda \equiv 1$. We also omit
  the second space when it is $\R$: $C_\Lambda^{1,\eta}(\tilde \GG):=
  C_\Lambda^{1,\eta}(\tilde \GG; \R)$.
 \end{defin}
 
 \smallskip

 \begin{rems} 
\begin{enumerate} 
\item Due to the different notions of distances used for $\GG_1$ and
  $\GG_2$ on the one hand, and the lack of a vector space structure on
  $\II \GG_1 = \tilde \GG_1 - \tilde \GG_1$ on the other hand, our
  definition differs from the usual one, in the sense that it does not
  imply the Lipschitz property or the boundedness of $D \SS[f_1]$ for
  instance.
\item In the sequel, we shall apply this abstract differential
   calculus with some suitable subspaces $\tilde \GG_i \subset P(E)$,
   i.e. in our applications sets of probabilities some
   \emph{prescribed} moments and some \emph{moment bounds}. This
   choice of subspaces is crucial in order to make rigorous the
   intuition of Gr\"unbaum \cite{Grunbaum} (see the --- unjustified
   --- expansion of $H_f$ in \cite{Grunbaum}). 
\item It is worth emphasizing
   that our differential calculus is based on the idea of considering
   $P(E)$ (or subsets of $P(E)$) as ``plunged sub-manifolds'' of some
   larger normed spaces $\GG_i$. We hence develop a differential
   calculus in the space of probability measures into a simple and
   robust framework, well suited to deal with the different objects we
   have to manipulate ($1$-particle semigroup, polynomial,
   generators\dots). Our approach thus differs from the approach of
   P.-L.  Lions recently developed in his course at Coll\`ege de
   France \cite{PLL-cours} or the one developed by L. Ambrosio
   \textit{et al.} \cite{AmbrosioBook} in order to deal with gradient
   flows PDEs in spaces of probability measures associated with the
   Wasserstein metric, as introduced by Otto \textit{et al.}
   \cite{OttoJK1998,MR1842429}. 
\item One novelty of our work is the use of
   this differential calculus in order to state some ``differential''
   stability conditions on the limit semigroup. Roughly speaking the
   latter estimates measure how this limit semigroup handles
   fluctuations departing from chaoticity, they are the corner stone
   of our analysis.
 \end{enumerate}
\end{rems}

This differential calculus behaves well for composition in the sense that 
for any given \[
\UU \in C^{1,\eta}_{\Lambda_\UU} (\tilde \GG_1; \tilde \GG_2) \  
\mbox{ and } \ \VV \in  C^{1,\eta}_{\Lambda_\VV} (\tilde \GG_2; \tilde
\GG_3)
\]
there holds 
\[
\SS:= \VV \circ \UU \in C^{1,\eta} _{\Lambda_\SS} (\tilde \GG_1;
\tilde \GG_3)
\]
for some appropriate weight function $\Lambda_\SS$.  We conclude the
section by stating a precise result adapted to our applications.

\begin{lem}\label{lem:DL} For any given
\[
\UU \in C^{1,\eta} _{\Lambda} (\tilde \GG_1; \tilde \GG_2) \cap C^{0,(1+\eta)/2} _{\Lambda} (\tilde \GG_1; \tilde \GG_2), \quad \eta \in (0,1], \ \mbox{
  and } \ \VV \in C^{1,1} (\tilde \GG_2; \tilde \GG_3), 
\]
there holds
  \[
  \SS:= \VV \circ \UU \in C^{1,\eta} _{\Lambda^2} (\tilde \GG_1;
  \tilde \GG_3) \quad \mbox{ and } \quad 
  D \SS [f] = D \VV [\UU(f)] \circ D \UU[f].
  \]
More precisely, there holds
 \begin{eqnarray*}
    [ \SS ]_{C^{0,1}_{\Lambda}} \le [\VV
    ]_{C^{0,1}} \, [\UU]_{C^{0,1}_\Lambda}, 
  \end{eqnarray*}
  and 
  \[
   [\SS ]_{C^{1,\eta}_{\Lambda^2}} \le \| \VV
    \| _{C^{1,1}} \, [\UU ]_{C^{1,\eta}_\Lambda} + [\VV ]_{C^{1,1}} \,
    [\UU ]_{C^{0,(1+\eta)/2}_\Lambda}^{2}.
  \]
  \end{lem}

\begin{proof}[Proof of Lemma~\ref{lem:DL}] 
For $f_1, f_2 \in \tilde\GG_1$,    we have
\[
\UU(f_2) = \UU(f_1) + 
\left\langle D \UU [f_1],
 f_2-f_1  \right\rangle + \TT_\UU\left(f_1,f_2\right)
\]
with 
\[
\left\| \TT_\UU\left(f_1,f_2\right) \right\|_{\GG_2} \le
[\UU]_{C^{1,\eta}_\Lambda} \, \Lambda\left( f_1, f_2 \right) \, \left\|
f_2 - f_1 \right\|_{\GG_1}^{1+\eta},
\]
as well as 
\[
\left\| \UU(f_2) - \UU(f_1) \right\|_{\GG_2} \le 
[\UU]_{C^{0,(1+\eta)/2}_\Lambda} \, \Lambda(f_1,f_2) \, \left\| f_2 -
  f_1 \right\|_{\GG_1}^{(1+\eta)/2}.
\]

A similar Taylor expansion holds for $\VV$: for $g_1, g_2 \in
\tilde\GG_2$
\[
\VV(g_2) = \VV(g_1) + 
\left\langle D \VV [g_1],
  g_2-g_1 \right\rangle + \TT_\VV\left(g_1,g_2\right)
\]
with
\[
\left\| \TT_\VV\left(g_1,g_2\right) \right\|_{\GG_3} \le
[\VV]_{C^{1,1}} \, \left\| g_2 - g_1 \right\|_{\GG_2}^2.
\]

We then write 
\[
\left\| \SS (f_2) - \SS(f_1) \right\|_{\GG_3} \le
[\VV]_{C^{0,1}} \, \left\| \UU (f_2) - \UU(f_1) \right\|_{\GG_2} \le 
[\VV]_{C^{0,1}} \, [\UU]_{C^{0,1}_\Lambda} \, \Lambda\left( f_1, f_2
\right) \, \| f_2 - f_2 \|_{\GG_1}
\]
which implies 
\[
[\SS]_{C^{0,1}_\Lambda} \le [\VV]_{C^{0,1}} \, [\UU]_{C^{0,1}_\Lambda},
\]
and
\begin{eqnarray*}
\SS(f_2) = (\VV \circ \UU) (f_2)
  &=& \VV \Big( \UU(f_1) + 
  \left\langle D \UU [f_1],
 f_2 - f_1 \right\rangle +  \TT_\UU(f_1,f_2) \Big) \\
  &=& \VV \left( \UU\left(f_1\right) \right) + \TT_\VV
  \left(\UU\left(f_2\right), \UU\left(f_1\right) \right) \\
  &+&  \Big\langle D \VV[\UU(f_1)], 
  \left\langle D \UU [f_1],  f_2-f_1 \right\rangle 
  +  \TT_\UU\left(f_1,f_2\right)  \Big\rangle
\end{eqnarray*}
which implies 
\[
\left\langle D \SS [f_1], f_2-f_1 \right
\rangle 
= \Big\langle D \VV[\UU(f_1)], 
    \big(  \left\langle D \UU [f_1],
  f_2-f_1 \right\rangle 
   \big) \Big\rangle.
\]
Observe that since $\VV \in C^{1,1} (\tilde \GG_2; \tilde \GG_3)$, one
has 
\[
\forall \, h \in \II \GG_2, \quad \left\| \Big\langle D \VV[\UU(f_1)],
  h \Big\rangle \right\|_{\GG_3} \le \left( [\VV]_{C^{0,1}} +
  [\VV]_{C^{1,1}} \right) \, \| h \|_{\GG_2}
\]
and therefore by scaling (and using that all vectorial lines of $\II
\GG_2$ intersects $\II \GG_2$ at a non-zero point) it extends to 
\[
\forall \, h \in \GG_2, \quad \left\| \Big\langle D \VV[\UU(f_1)], h
  \Big\rangle \right\|_{\GG_3} \le \left( [\VV]_{C^{0,1}} +
  [\VV]_{C^{1,1}} \right) \, \| h \|_{\GG_2}.
\]

Finally we estimate the remaining term:
\begin{eqnarray*}
  && \left\|  \SS(f_2) - \SS(f_1) - \left\langle D \SS[f_1],
      f_2-f_1 \right\rangle \right\|_{\GG_3}  \\
  &&\qquad = 
  \left\| \TT_\VV(\UU(f_2),\UU(f_1)) + \langle D\VV[\UU(f_1)],
    \TT_\UU(f_1,f_2)\rangle \right\|_{\GG_3}
\\
  &&\qquad
  \le [\VV]_{C^{1,1}}\, \left\|  \UU(f_1) - \UU(f_2)\right\|_{\GG_2}
  ^2 +\| \VV \|_{C^{1,1}} \, \left\| \TT_\UU (f_1,f_2) \right\|_{\GG_2},
  \\
  &&\qquad
  \le [\VV]_{C^{1,1}} \left(   [\UU]_{C^{0,(1+\eta)/2}_\Lambda} \,
    \Lambda \left(f_1,f_2\right) \, 
    \| f_2 - f_1 \|_{\GG_1} ^{(1+\eta)/2} \right)^2  \,
  \\
  &&\qquad \qquad \qquad +\| \VV \|_{C^{1,1}} \,  \Lambda\left( f_1, f_2 \right)
  \, 
  [\UU]_{C^{1,\eta}_\Lambda}  \, \| f_2 - f_1 \|_{\GG_1} ^{1+\eta},
  \end{eqnarray*}
and we conclude by recalling that $\Lambda \ge 1$. 
\end{proof}

\subsection{The pullback generator}
\label{sec:calculus-gen}

As a first application of this differential calculus, let us compute
the generator of the pullback limit semigroup.  

\begin{lem}\label{lem:H0} 
  Given some Banach space $\GG$ and some space of probability measures
  $\mathcal P_{\GG}(E)$ (see Definitions~\ref{defGG1}-\ref{defGG2})
  associated to a weight function $m$ and constraint function ${\bf
    m}$, and endowed with the metric induced from $\GG$, then for some
  $\zeta \in (0,1]$ and some $\bar a\in (0,\infty)$ we assume that for
  any $a\in (\bar a,\infty)$ and any choice of constraints ${\bf r}
  \in \RR_\GG$,
there holds:
  \begin{itemize}
  \item[(i)] The equation \eqref{eq:limit} generates a semigroup 
    \[
    S^{N\! L}_t : \BB \mathcal P_{\GG,a,{\bf r}} \to \BB \mathcal P_{\GG,a,{\bf r}}
    \]
    which is $\zeta$-H\"oder continuous locally uniformly in time, in
    the sense that for any $\tau \in (0,\infty)$ there exists $ C_\tau
    \in (0,\infty)$ such that
    \begin{equation*}
      \forall \, f,g
      \in \BB \mathcal P_{\GG,a,{\bf r}}, \quad \sup_{t \in [0,\tau]}
      \left\| S^{N\! L}_t f -
      S^{N\! L}_t g \right\|_{\GG_1} \le C_\tau \, \, \| f - g \|^\zeta_{\GG_1}.
     \end{equation*}
   \item[(ii)] The application $Q$ is bounded and $\zeta$-H\"older continuous from
    $\BB \mathcal P_{\GG,a,{\bf r}}$ into $\GG$.
  \end{itemize}
  
  Then for any $a \in (\bar a,\infty)$, ${\bf r} \in \RR_{\GG}$ the
  pullback semigroup $T^\infty_t$ defined by 
  \[ 
  \forall \, f \in \BB \mathcal P_{\GG,a,{\bf r}}(E), \ \Phi \in  U C_b( \BB
  \mathcal P_{\GG,a,{\bf r}}(E)), \quad T^\infty _t [\Phi](f) := \Phi\left( S^{N\!
      L}_t(f)\right)
  \]
  is a $C_0$-semigroup of contractions on the Banach space $U C_b( \BB
  \mathcal P_{\GG,{\bf r},a}(E))$. 

  Its generator $G^\infty$ is an unbounded linear operator on $U C_b(
  \BB \mathcal P_{\GG,a,{\bf r}}(E))$ with domain denoted by
  $\hbox{{\em Dom}}(G^\infty)$ and containing 
  $U C^1_b( \BB \mathcal P_{\GG,a,{\bf r}}(E))$. On the latter space,
  it is defined by the formula
  \begin{equation}
    \label{eq:formulaGinfty}
    \forall \, \Phi \in  U C^1_b( \BB \mathcal P_{\GG,a,{\bf r}}(E)), \ 
    \forall \, f \in \BB \mathcal P_{\GG,a,{\bf r}}(E), \quad 
    \left( G^\infty \Phi \right) (f) :=
    \left \langle D\Phi[f], Q(f)\right\rangle.
  \end{equation}
\end{lem}


\begin{rem}
  Note that the restriction to uniformly continuous functions $\Phi$
  on space of probability measures will be harmless in the sequel for two
  reasons: first in most cases our choice of weight, constraints and
  distance yields a compact space $\BB \mathcal P_{\GG,a}(E)$, and second
  and most importantly we shall only manipulate this pullback
  semigroup for functions $\Phi$ having at least uniform H\"older regularity.
\end{rem}


\begin{proof}[Proof of Lemma~\ref{lem:H0}.]
The proof is split  in several steps. 

\bigskip \noindent {\sl Step 1. H\"older regularity in time for
  $S^{N\! L}_t(f^0)$.}  We claim that for any ${\bf r} \in
\RR_{\GG}$ and $f_0 \in \BB \mathcal P_{\GG,a,{\bf r}}(E)$ and $\tau
> 0$ the application
\[
\SS\left(f_0\right) : [0,\tau) \to \mathcal P_{\GG}, \quad t \mapsto \SS^{N \!
  L}_t \left(f_0\right)
\]
is right differentiable in $t=0$ with 
\[
\SS\left(f_0\right)'(0^+) = Q\left(f_0\right).
\]
Let us write $f_t := S^{N \! L}_t f_0$.  First, since $f_t \in \BB \mathcal
P_{\GG,a,{\bf r}}$ for any $t \in [0,\tau]$ (assumption (i)) and $Q$
is bounded on $ \BB \mathcal P_{\GG,a,{\bf r}}(E)$ (assumption (ii)),
we deduce that 
\begin{equation}\label{eq:ft-f0A2} 
  \left\| f_t - f_0
  \right\|_{\GG} = \left\| \int_0^t Q(f_s) \, ds \right\|_{\GG} \le K \,
  t
 \end{equation}
 for a constant $K$ which is uniform according to $f_0 \in \BB
 \mathcal P_{\GG,a,{\bf r}}(E)$.

We then use the previous inequality together with the fact that $Q$ is
$\zeta$-H\"older continuous (assumption (ii) again), to get
\begin{eqnarray*} 
  \left\| f_t - f_0 - t \,
  Q\left(f_0\right)\right\|_{\GG} &=& 
  \left\| \int_0^t \left(Q\left(f_s\right) - Q\left(f_0\right)\right) \, ds
  \right\|_{\GG}
  \\
  &=& L \, \int_0^t \left\| f_s - f_0 \right\|_{\GG}^\zeta \, ds
  \\
  &\le& L \, \int_0^t (K \, s)^\zeta \, ds = L \, K^\zeta \,
  {t^{1+\zeta} \over 1+\zeta}, 
\end{eqnarray*}
which implies the claim. 

Then the semigroup property of $(S^{N\!  L}_t)$ implies that $t
\mapsto f_t$ is continuous from $\R_+$ into $\mathcal P_{\GG}(E)$ and right
differentiable at any point.

\bigskip \noindent {\sl Step 2. Contraction property for $T^\infty_t$.} We claim that $(T^\infty_t)$ is a $C_0$-semigroup of
contractions on $UC_b(\BB \mathcal P_{\GG,a,{\bf r}}(E))$. 
First for any $\Phi \in U C_b( \BB \mathcal P_{\GG,a,{\bf r}}(E))$, we denote by
$\omega_\Phi$ the modulus of continuity of $\Phi$. We have thanks to
the assumption (i):
\begin{eqnarray*}
  \forall \, t \in [0,\tau], \quad  
  \left|\left(T^\infty_t \Phi\right)(g) - \left(T^\infty_t \Phi\right)(f)\right| 
  &=& \left| \Phi\left(S^{N\!L}_t(g)\right) - \Phi\left(S^{N\!L}_t(f)\right) \right| \\
  &\le& \omega_\Phi \left( \left\| S^{N\!L}_t(g) - S^{N\!L}_t(f)  \right\|_{\GG_1} \right)
  \\
  &\le& \omega_\Phi \left(C_\tau \, \| g-f \|_{\GG_1}^\zeta\right),
\end{eqnarray*}
so that $T^\infty_t \Phi \in UC_b( \BB \mathcal P_{\GG_1,a,{\bf r}}(E))$ for any $t \in
[0,\tau]$, and then, by iteration, for any $t \ge 0$.  Next, we have
$$
\left\| T^\infty_t \right\| = \sup_{\| \Phi \| \le 1} \left\| T^\infty_t \Phi \right\|  =
\sup_{\| \Phi \| \le 1} \sup_{f \in \BB \mathcal P_{\GG,a,{\bf r}}(E) } \left|
\Phi\left(S^{N\!L}_t(f)\right)\right| \le 1
$$
since 
$$
\| \Phi \| = \sup_{f \in \BB \mathcal P_{\GG,a,{\bf r}}(E)} |\Phi(f)|
$$
and $S^{N\! L}_t$ maps $\BB \mathcal P_{\GG,a,{\bf r}}(E)$ to itself.

Finally, from \eqref{eq:ft-f0A2}, for any $\Phi \in UC_b(\BB
\mathcal P_{\GG,a,{\bf r}}(E))$, we have
$$
\left\| T^\infty_t \Phi - \Phi \right\| = \sup_{f \in \BB
  \mathcal P_{\GG,a,{\bf r}}(E)} \left| \Phi(S^{N\!L}_t(f)) - \Phi(f)\right|
\le \omega_\Phi (K \, t) \to 0. 
$$
As a consequence $(T^\infty_t)$ has a closed generator $G^\infty$ with
dense domain 
\[
\mbox{Dom}(G^\infty) \subset UC_b( \BB \mathcal
P_{\GG,a,{\bf r}}(E)), \quad \overline{\mbox{Dom}(G^\infty)} = UC_b( \BB \mathcal
P_{\GG,a,{\bf r}}(E))
\] 
(see for instance \cite[Chapter~1, Corollary~2.5]{pazy}).

\bigskip \noindent {\sl Step 3. } We shall now identify this generator
on a subset of its domain. 
Let us construct a natural candidate provided by the heuristic of
Remark~\ref{edo-edp}. Let us define $\tilde G^\infty \Phi $ by
$$ 
\forall \, \Phi \in  U C^1_b( \BB  \mathcal P_{\GG,a}(E)), \ \forall \,
f \in \BB \mathcal P_{\GG,a}(E), \quad ( \tilde G^\infty \Phi ) (f) :=
\left \langle D\Phi[f], Q(f)\right\rangle.
$$
The right-hand side is well defined since 
\[
D\Phi[f] \in \LL(\GG,\R) = \GG' \ \mbox{ and } \ Q(f) \in
\GG.
\]
Moreover, since both applications 
\[
f \mapsto D\Phi[f] \ \mbox{ and } \ f \mapsto Q(f)
\]
are uniformly continuous on $\BB \mathcal P_{\GG,a,{\bf r}}(E)$, so is the
application 
\[
f \mapsto (\tilde G^\infty \Phi) (f).
\]
Hence $\tilde G^\infty \Phi \in UC_b(\BB \mathcal P_{\GG,a,{\bf r}}(E))$.

\medskip \noindent {\sl Step 4. } Finally, by composition, for any
fixed $\Phi \in U C^1_b( \BB  \mathcal P_{\GG,a,{\bf r}}(E))$ and $f \in \BB
\mathcal P_{\GG,a,{\bf r}}(E)$, the map 
\[
t \mapsto T^\infty _t \Phi (f) = \Phi \circ S^{N\! L}_t (f)
\]
is right differentiable in $t=0$ and
\begin{eqnarray*} 
  \left( \frac{{\rm d}}{{\rm d}t} (T^\infty_t \Phi) (f) \right)_{\big|_{t=0}}
  &:=&  \left( \frac{{\rm d}}{{\rm d}t} (\Phi \circ S^{N\! L}_t (f)(t)) \right)_{\big|_{t=0}} \\
  &=& \left\langle D\Phi [S^{N\! L}_0 (f)], \left( \frac{{\rm d}}{{\rm d}t} S^{N\!
        L}_t(f) \right)_{\big|_{t=0}} \right\rangle \\
  &=& \left\langle D\Phi[f], Q(f) \right\rangle = \left( \tilde G^\infty \Phi
  \right) (f),
\end{eqnarray*}
which precisely means that $\Phi \in \hbox{Dom}(G^\infty)$ and that
\eqref{eq:formulaGinfty} holds.  
\end{proof}

\subsection{Duality inequalities}
\label{sec:compatibility}

Our transformations $\pi^N$ and $R^N$ behave nicely for the supremum
norm on $C_b(P(E),TV)$, see~\eqref{eq:compat:infty}.  More generally
we shall consider ``\emph{duality pairs}'' of metric spaces as follows:
\begin{defin}
  We say that a pair $(\FF,\mathcal P_\GG)$ of a normed vectorial
  space $\FF \subset C_b(E)$ endowed with the norm $\| \cdot \|_\FF$
  and a space of probability measures $\mathcal P_\GG \subset P(E)$
  endowed with a metric $d_\GG$ \emph{satisfy a duality inequality} if
\begin{equation} \label{eq:dualiteFGbis} \forall \, f,g  \in \mathcal P_\GG,
  \,\, \forall \, \varphi \in \FF, \quad |\langle g-f, \varphi \rangle
  | \le C \, d_\GG(f,g)\, \| \varphi \|_\FF,
\end{equation} 
where here $\langle \cdot, \cdot \rangle$ stands fot the usual duality
brackets between probabilities and continuous functions.   
In the case where the distance $d_\GG$ is associated with a normed
vector space $\GG$, this amounts to the usual duality
inequality 
  $|\langle h, \varphi \rangle |\le \|h
  \|_\GG \, \| \varphi \|_\FF$.
\end{defin}
\smallskip

The ``compatibility'' of the transformation $R^N$ for any such pair
follows from the multilinearity: if $\FF$ and $\GG$ are in duality,
$\FF \subset C_b(E)$ and $\mathcal P_\GG$ is endowed with the metric associated
to $\|\cdot \|_\GG$, then for any 
\[
\varphi = \varphi_1 \times \dots \times \varphi_N \in \FF^{\otimes N},
\]
the polynomial function $R^N_\varphi$ in $C_b(P(E))$ is
$C^{1,1}(\mathcal P_\GG,\R)$. Indeed, given
$f_1, f_2 \in \mathcal P_{\GG_1}$, we define
\[ 
\GG \to \R, \quad h \mapsto DR^\ell_\varphi \left[f_1\right] (h) :=
\sum_{i=1}^N \left( \prod_{j \not= i } \left\langle f_1, \varphi_j
  \right\rangle \right) \left\langle h, \varphi_i \right\rangle,
\]
and we have
\begin{eqnarray*}
  R^N_\varphi(f_2) - R^N_\varphi(f_1) &=& \sum_{i=1}^N
  \left( \prod_{1 \le k < i } \left\langle f_2, \varphi_k
    \right\rangle \right) \, 
  \left\langle f_2 - f_1, \varphi_i \right\rangle \, \left( \prod_{i < k \le \ell }
  \langle f_1,  \varphi_k \rangle \right), 
\end{eqnarray*}
and
\begin{multline}\nonumber
  R^N _\varphi(f_2) - R^N _\varphi\left(f_1\right) -
  DR^N _\varphi \left[f_1\right] \left(f_2 - f_1\right) =
  \\
  = \sum_{1\le j < i \le N} \left( \prod_{1 \le k < j}
    \left\langle f_2, \varphi_k \right\rangle\right) \, \left\langle
    f_2 - f_1, \varphi_j \right\rangle \, \left( \prod_{j < k <i}
    \left\langle f_1,\varphi_k \right\rangle \right) \, \left\langle
   f_2 - f_1, \varphi_i \right\rangle \left( \prod_{i < k \le \ell } \left\langle
  f_1, \varphi_k \right\rangle \right).
\end{multline} \Black
Hence for instance $R^N _\varphi \in C^{1,1}(\mathcal P_\GG;\R)$ with
\begin{eqnarray*} 
  \left|R^N_\varphi(f_2) - R^N_\varphi(f_1) \right| &\le& N \, \| \varphi
  \|_{\FF \otimes (L^\infty) ^{N-1}} \, \left\| f_2 - f_1 \right\|_\GG , \\
  \left| DR^N _\varphi [f_1] (h) \right| &\le& N \, \left\| \varphi \right\|_{\FF \otimes
    (L^\infty) ^{N-1}} \, \| h \|_\GG, 
\end{eqnarray*}
and 
\begin{equation}
\label{eq:Ck1polyk}
  \left|R^N _\varphi(f_2) - R^N _\varphi(f_1) -
  DR^N _\varphi [f_1] (f_2 - f_1) \right| \le \frac{N (N -1)}2
  \| \varphi \|_{\FF^2 \otimes (L^\infty) ^{N-2}} \, \left\| f_2 - f_1 \right\|^2_\GG ,  
\end{equation}
where we have defined 
\begin{eqnarray*}
  \| \varphi \|_{\FF^k \otimes (L^\infty) ^{N-k}} 
  &:=& \max_{i_1, \dots,  i_k \mbox{ {\tiny distincts in }}
    [|1,N|]}  \left(
  \| \varphi_{i_1} \|_{\FF} \, \dots \| \varphi_{i_k} \|_{\FF} \! 
  \prod_{j \neq (i_1,\dots,i_k)} \| \varphi_j \|_{L^\infty(E)} \right).
 \end{eqnarray*}
 
\begin{rems}\label{rem:RphiC11} 
   \begin{enumerate}
  \item It is easily seen in this computation that the assumption that
    $\varphi$ is tensor product is not necessary. In fact it is likely
     that this assumption could be relaxed all along our proof. 
   \item The assumption $\FF \subset C_b(E)$ could also be
     relaxed. For instance, when 
     \[
     \FF := \hbox{\rm Lip}_0(E)
     \]
     is the space of Lipschitz function which vanishes in some fixed
     point $x_0 \in E$, $\GG$ is its dual space, and
     \[
     \mathcal P_\GG := \{ f \in \PP_1(E); \,\, \langle f, \hbox{\rm dist}_E(\cdot
     ,x_0) \rangle \le a \}
     \]
     for some fixed $a > 0$, we have $R^N _\varphi \in
     C^{1,1}(\PP_1(E);\R)$ with 
     $$
     \left[R^N _\varphi\right]_{C^{0,1}} \le N \, a^{N-2} \, \| \varphi
     \|_{\FF^{\otimes N}}, \quad \left[R^N _\varphi\right]_{C^{1,1}} \le
     \frac{N (N -1)}2 \, a^{N-1} \, \| \varphi \|_{\FF^{\otimes N}},
     $$
     or equivalently $R^N _\varphi \in C^{1,1}_\Lambda(\PP_1(E);\R)$ with
     $\Lambda(f) := \| f \|_{M^1_1}^{N-2}$. 
\end{enumerate}
\end{rems}

%





\smallskip In the other way round, for the projection $\pi^N$ it is
clear that if the empirical measure map 
\[
X \in E^N \mapsto \mu^N_X \in P(E)
\] 
belongs to $C^{k,\eta}(E^N,\mathcal P_\GG)$ for some norm structure $\GG$,
then by composition one has
\begin{equation}\label{eq:compat-pi-F}
  \left\| \pi^N (\Phi) \right\|_{C^{k,\eta}(E ^N;\R)} 
  \le C_\pi \, \left\| \Phi \right\|_{C^{k,\eta}(\mathcal P_\GG)}.
\end{equation}
However the regularity of the empirical measure of course heavily
depends on the choice of the metric $\GG$.

\begin{ex}\label{expleTVbis} 
  In the case $\FF = (C_b(E),L^\infty)$ and $\GG = (M^1(E),TV)$,
  \eqref{eq:compat-pi-F} trivially holds with $C^{k,\eta}$ replaced by
  $C_b$. 
\end{ex}

\begin{ex}\label{expleWpbis}  When $\FF = \mbox{Lip}_0(E)$ 
endowed with the norm
  $\|\phi \|_{\mbox{{\scriptsize Lip}}}$ and $\mathcal P_\GG(E)$ (constructed in
  Subsubsection~\ref{expleWp}) is endowed with the Wasserstein distance
  $W_1$ with linear cost, one has \eqref{eq:compat-pi-F} with $k=0$,
  $\eta=1$:
  \[
  \left| \Phi\left(\mu^N _X \right) - \Phi\left(\mu^N _Y \right)
  \right| \le \| \Phi \|_{C^{0,1}(\mathcal P_\GG)} \, W_1\left(\mu^N _X ,\hat
    \mu^N _Y \right) \le \| \Phi \|_{C^{0,1}(\mathcal P_\GG)} \, \|X-Y
  \|_{\ell^1},
  \]
  where we use (\ref{Wqellq}), which proves that
  \[
  \| \pi^N (\Phi) \|_{C^{0,1} (E^N)} \le \| \Phi \|_{C^{0,1}(\mathcal P_\GG)},
  \]
  when $E^N$ is endowed with the $\ell^1$ distance defined in
  (\ref{Wqellq}).
\end{ex}

\section{The abstract theorem}
\label{sec:abstract-theo}
\setcounter{equation}{0}
\setcounter{theo}{0}

\subsection{Assumptions of the abstract theorem}
\label{sec:hyp-evol}

Let us list the assumptions that we need for our main abstract
theorem. 
\bigskip

\begin{quote}
\begin{center}{\bf (A1) Assumptions on the $N$-particle system.} 
\end{center}
\smallskip

$G^N$ and $T^N_t$ are well defined on $C_b(E^N)$ and are
invariant under permutation so that the evolution $f^N_t$ is well
defined. We moreover assume the following moment conditions:
  \begin{itemize}
  \item[(i)] {\it Conservation constraint}: There exists a constraint
    function ${\bf m}_{\GG_1} : E \to \R^D$ and a subset $\RR_{\GG_1}
    \subset \R^D$ such that defining the sequence of constraint sets
  $$
  \Ee_N := \left\{ V \in E^N; \,\, \langle \mu^N_V, {\bf
      m}_{\GG_1} \rangle \in \RR_{\GG_1} \right\}
  $$
  there holds 
  \begin{equation*}
    \forall \, t \in [0,T), \quad  \hbox{Supp}
      \, f^N_t \subset \Ee_N . 
    \end{equation*}
%
\Black
  \item[(ii)] {\it Propagation of an integral moment bound}: There exists
    a weight function $m_{\GG_1}$, a time $T \in (0,\infty]$ and a
    constant $C_{1,T} \in (0,\infty)$, possibly depending on $T$ and
    $m_{\GG_1}$, but not on the number of particles $N$, such that
    \begin{equation*}
     \forall \, N \ge 1, \quad  \sup_{0 \le t < T} \left\langle f^N
      _t, M^N _{m_{\GG_1}} \right \rangle \le C_{m_{\GG_1},T},
  \end{equation*}
  where for any weight function $m$ on $E$ we define the
    \emph{mean weight function} $M^N_m$ on $E^N$ by
\[
\forall \, V \in E^N, \quad M^N _{m} (V) := \frac1N \, \sum_{i=1} ^N m (v_i) = M_m(\mu^N_V) = \langle \mu^N_V,m \rangle.
\]
  
\item[(iii)] {\it Support moment bound at initial time}: There exists a
  weight function $m_{\GG_3}$ and a constant $C^N_{3} \in
  (0,+\infty)$, possibly depending on the number of particles $N$,
  such that
    \begin{equation*}
      \hbox{Supp}
      \, f^N_0 \subset \left\{V \in E^N; M^N_{m_{\GG_3}} (V) \le C^N_{m_{\GG_3},0} \right\}.
    \end{equation*}
  \end{itemize}
\end{quote}
\bigskip

Note that the name of the weights functions $m_{\GG_1}$ and
$m_{\GG_3}$ and of the constraint function ${\bf m}_{\GG_1}$ in {\bf
  (A1)} above are chosen bearing in mind the coherence with the
functional framework introduced in Definition~\ref{defGG1} (in
particular $\langle f, {\bf m}_{\GG_1} \rangle$ is well defined as a
vector of $\R^D$ for any $f \in \PP_{\GG_1}$) and with the
functional spaces in the other assumptions. Observe moreover that
the support condition $\Ee^N$ will be useful for controlling the
moments of the empirical measures sampled out of the distribution
$f^N$: $V \in \Ee_N$ implies $\mu^N_V \in \RR_{\GG_1}$,
  where this set of constraints is defined in Definitions~\ref{defGG1}
  and ~\ref{defGG2}.


\bigskip

\begin{quote}
  \begin{center} {\bf (A2) Assumptions for the existence of the
      pullback semigroup.}
\end{center}
\smallskip

We consider a weight function $m_{\GG_1}$, a constraint function ${\bf
  m_{\GG_1}} : E \to \R^D$ and a set of constraints $\RR_{\GG_1}
\subset \R^D$ as in {\bf (A1)-(i)} and {\bf (A1)-(ii)}.  We then
consider the corresponding space of probability measures $\mathcal
P_{\GG_1}(E)$
and the corresponding vectorial \emph{space of increments}
$\mathcal {IP}_{\GG_1}$ according to Definition~\ref{defGG1}.  We
finally consider a Banach space $\GG_1 \supset \mathcal {IP}_{\GG_1}$
so that $\mathcal P_{\GG_1,{\bf r}}(E)$ is endowed with the distance
$d_{\GG_1}$ induced by the norm $\|Ê\cdot \|_{\GG_1}$ for any
constraint vector ${\bf r} \in {\RR}_{\GG_1}$.


%
Then we assume the following: there are $\zeta \in (0,1]$ and $\bar a
\in (0,\infty)$ such that for any $a\in (\bar a,\infty)$ and $ {\bf r}
\in \RR_{\GG_1}$ one has
  \begin{itemize}
  \item[(i)] The equation \eqref{eq:limit} generates a semigroup 
    \[
    S^{N\! L}_t : \BB \mathcal P_{\GG_1,a,{\bf r}} \to \BB \mathcal P_{\GG_1,a,{\bf r}}  
    \]
    which is $\zeta$-H\"oder continuous locally uniformly in time, in
    the sense that for any $\tau \in (0,\infty)$ there exists $ C_\tau
    \in (0,\infty)$ such that
    \begin{equation*}
      \forall \, f,g
      \in \BB \mathcal P_{\GG_1,a, {\bf r}}, \quad \sup_{t \in
        [0,\tau]} \left\| S^{N\! L}_t f -
      S^{N\! L}_t g \right\|_{\GG_1} \le C_\tau \, \, \| f - g \|^\zeta_{\GG_1}.
     \end{equation*}
   \item[(ii)] The application $Q$ is bounded and $\zeta$-H\"older continuous from
    $\BB \mathcal P_{\GG_1,a,{\bf r}}$ into $\GG_1$.
%
\end{itemize}
\end{quote}
\bigskip

The important consequence of this assumption is that the semigroups
$S^{N\!L}_t$  and $T^\infty_t$ are well defined as
well as the generators $G^N$ and $G^\infty$ thanks to Lemma~\ref{lem:H0}.
\smallskip

We then need the key following \emph{consistency} assumption. It
intuitively states that the $N$-particle \emph{approximation} of the limit
mean-field equation is consistent. More rigorously this means a
convergence of the generators of the $N$-particle approximation
towards the generator of the limit pullback semigroup within the
abstract functional framework we have introduced. 
\bigskip

\begin{quote}
\begin{center}
{\bf (A3) Convergence of the generators.} 
\end{center}
\smallskip

We consider $\mathcal P_{\GG_1}$, $m_{\GG_1}$ and ${\bf m}_{\GG_1}$
introduced in {\bf (A2)}, and we also define a weight function
\[
1 \le m' _{\GG_1} \le C \, m_{\GG_1}
\]
possibly weaker than $m_{\GG_1}$ and we define the associated weight
on the distribution:
\[
\Lambda_1 (f) := \left\langle f,m' _{\GG_1} \right\rangle.
\]

Then we assume that for some function 
\[
\eps(N) \to 0 \ \mbox{ as } \ N \to \infty \quad \hbox{and}\quad \eta \in (0,1]
\]
the generators $G^N$ and $G^\infty$ satisfy
  \begin{multline*}
     \forall \,  \Phi \in \bigcap_{{\bf r} \in   \RR_{\GG_1} } 
    C_{\Lambda_1}^{1,\eta}(\mathcal P_{\GG_1,{\bf r}}),
   \qquad 
   \left\|  \left( M^N _{m_{\GG_1}} \right)^{-1} \!\!
     \left( G^N \, \pi_N - \pi_N \, G^\infty \right)  \,  
     \Phi \right\|_{L^\infty(\Ee_N)} 
 \\ \le \eps(N) \, 
   \sup_{{\bf r} \in   \RR_{\GG_1}} [ \Phi ]_{C^{1,\eta} _{\Lambda_1}(\mathcal P_{\GG_1,{\bf r}})}.
\end{multline*}
\end{quote}
\bigskip

Note the following aspect, which shall be a source of difficulties in
the applications of the abstract theorem: {\em the loss of
  weight in the consistency estimate {\bf (A3)} has to be matched by
  the support constraints on the $N$-particle system in the assumption
  {\bf (A1)-(i)} and the moment bounds propagated on the $N$-particle
  system in the assumption {\bf (A1)-(ii)}. } (In fact the loss of
weight $m_{\GG_1}$ in the consistency estimate can be even slighlty
higher than the weight $m_{\GG_1}'$ for which the stability estimates
are performed, due to the energy conservation, see later.)

Moreover, the best we are able to prove uniformly in $N$ on the
$N$-particle system are \emph{polynomial} moment bounds. This thus
constraints the kind of loss of weight we can afford in the following
stability
estimate.  
\medskip

We now state the second key \emph{stability} assumption. Intuitively
this corresponds to the abstract regularity that needs to be
transported along the flow of the limit mean-field equation so that
the fluctuations around chaoticity can be controlled. More rigorously
this means some differential regularity on the pullback limit
semigroup, which corresponds to some differential regularity on the
limit nonlinear semigroup \emph{according to the initial data} and
\emph{in the space of probability measures}. 
\bigskip

\begin{quote}
\begin{center}
{\bf (A4) Differential stability of the limit semigroup.} 
\end{center}  
\smallskip

We consider some Banach space $\GG_2 \supset \GG_1$ (where $\GG_1$ was
defined in {\bf (A2)}) and the corresponding space of probability measures $\mathcal P_{\GG_2}(E)$ (see
Definitions~\ref{defGG1}-\ref{defGG2}) with the weight function
$m_{\GG_2}$ and the constraint function ${\bf m_{\GG_2}}$, and endowed
with the metric induced from $\GG_2$.

We assume that the flow $S^{N \! L}_t$ is
$C^{1,\eta}_{\Lambda_2}(\mathcal P_{\GG_1, {\bf r}},\mathcal
P_{\GG_2})$ for any ${\bf r} \in \RR_{\GG_1}$, in the sense that there exists $C_T ^\infty >0$ such that
  \begin{equation*}
    \sup_{{\bf r} \in   \RR_{\GG_1}}  \int_0^T  
    \left( \left[ S^{N \! L}_t \right]_{C^{1,\eta}_{\Lambda_2 }(\mathcal P_{\GG_1, {\bf r} },\mathcal P_{\GG_2})} 
      + \left[ S^{N \!
        L}_t\right]^2_{C^{0,(1+\eta)/2}_{\Lambda_2}(\mathcal P_{\GG_1,  {\bf r} },\mathcal P_{\GG_2})} \right) 
    \, {\rm d}t  \le C_T ^\infty,
  \end{equation*}
  with  $\Lambda_2 :=  \Lambda_1^{1/2}$, where $\eta \in (0,1)$ and $\Lambda_1$ are the
  same as in {\bf (A3)}.
\end{quote}
\bigskip

We finally state a weaker stability assumption on the limit
semigroup. It shall be used intuitively for proving that the initial
error made in the law of large numbers when approximating a
probability by empirical measures is propagated by the limit
semigroup. The reason for dissociating this assumption from the
previous one is because we need flexibility with different choices of
distances for them.  \bigskip

\begin{quote}
\begin{center}
{\bf (A5) Weak stability of the limit semigroup.} 
\end{center}
\smallskip

We assume that, for some probabilistic space $\mathcal P_{\GG_3}(E)$
associated to the weight function $m_{\GG_3}$ (as in {\bf
  (A1)-(iii)}), a constraint function ${\bf m}_{\GG_3}$, a set of
contraints $\RR_{\GG_3}$ and some metric structure $d_{\GG_3}$, for
any $a,T > 0$ there exists a concave modulus of continuity
$\Theta_{a,T}$ (i.e. $\Theta_{a,T} : \R_+ \to \R_+$ continuous concave
with $\Theta_{a,T}(0) =0$) such that we have
 \begin{equation*}
   \forall \, {\bf r} \in \RR_{\GG_3}, \ \forall  \, f_1, f_2 \in
   \BB \mathcal P_{\GG_3,a,{\bf r}}(E), 
   \quad   \sup_{[0,T)} d_{\GG_3} 
   \left( S^{N \! L}_t(f_1), S^{N \! L}_t(f_2) \right) \le \Theta_{a,T} \left( 
     d_{\GG_3} \left(f_1,f_2\right) \right).
  \end{equation*}
\end{quote}
\bigskip

Observe that in the latter assumption, we require that $f_1, f_2 \in
\BB \mathcal P_{\GG_3,a}(E)$ which requires in particular the bounds
\[
M_{m_{\GG_3}} \left(f_i\right) = \int_{\R^d} f_i \, m_{\GG_3} \, {\rm d}v \le a, 
\quad i=1,2.
\]
When applying the assumption to some empirical measure for one of the
argument $f_1$ or $f_2$, this requires the pointwise control of terms
like $M^N _{m_{\GG_3}}$. This is the reason for the assumption {\bf
  (A1)-(iii)}.

\subsection{Statement of the result}
\label{sec:abstract-result}

We are now in position to state the main abstract result. This result
can be considered intuitively as a \emph{convergence} in approximation
theory, in the sense of proving that approximation errors between the
$N$-particle system and the limit mean-field system are propagated
along time without instability amplification mechanism. More
specifically the approximation error means in the present context some
kind of distance between the discrete $N$-particle system and the
limit mean-field system, within our abstract functional
framework. This result implies in particular the propagation of
chaos. In terms of method, we aim at treating the $N$-particle system
as a perturbation (in a very degenerated sense) of the limit problem,
and minimizing assumptions on the many-particle systems in order to
avoid complications of many dimensions dynamics.

\begin{theo}\label{theo:abstract}
  Consider a family of $N$-particle initial conditions 
  \[
  f^N_0 \in  P_{\mbox{{\tiny {\em sym}}}}(E^N), \quad N \ge 1, 
  \]
  and the associated solutions 
  \[
  f_t^N = S^N_t \left( f_0^N \right).
  \]

  Consider a $1$-particle initial condition $f_0 \in P(E)$ and the associated solution 
\[
f_t = S^{N \! L}_t \left(f_0\right)
\]
of the limit mean-field equation.

  Assume that the assumptions {\bf (A1)-(A2)-(A3)-(A4)-(A5)} hold for some spaces
  $\mathcal P_{\GG_k}$, $\GG_k$ and $\FF_k$, $k=1,2,3$ with $\FF_k \subset
  C_b(E)$, and where $\FF_k$ and $\GG_k$ are in
  duality (that is \eqref{eq:dualiteFGbis} holds). 

Assume also that the $1$-particle distribution satisfies the moment
bound 
\[
M_{m_{\GG_3}}\left( f_0 \right) = \left\langle f_0, m_{\GG_3} \right\rangle <
+\infty.
\]

  Then there is an explicit absolute constant $C \in (0,\infty)$ such
  that for any $N, \ell \in \N^*$, with $N \ge 2 \ell$, and for any
  \[ \varphi = \varphi_1 \otimes \varphi_2 \otimes \dots \otimes \,
  \varphi_\ell \in (\FF_1 \cap \FF_2 \cap \FF_3)^{\otimes \ell} \]
  we have
  \begin{eqnarray}
  \label{eq:cvgabstract1}
  &&\quad \sup_{[0,T)}\left| \left \langle \left( S^N_t(f_0 ^N) - \left(
          S^{N \! L}_t(f_0) \right)^{\otimes N} \right), \varphi 
    \right\rangle \right| 
  \\ \nonumber 
  &&\quad 
  \le C \, \Bigg[ \ell^2 \, \frac{\|\varphi\|_\infty}{N} 
  + C_{m_{\GG_1},T} \, C_T^\infty \, \varepsilon(N) \, \ell^2 \, 
  \|\varphi\|_{\FF_2 ^2 \otimes (L^\infty)^{\ell-2}} 
  \\ \nonumber
  &&\qquad\qquad\qquad 
  + \ell \, \, \|\varphi\|_{\FF_3 \otimes (L^\infty)^{\ell-1}} \, 
  \Theta_{a^N,T}  \left(  \WW_{d_{\GG_3}} \!\!  \left(\pi^N_P
      f^N_0,\delta_{f_0}\right) 
  \right) \Bigg],
  \end{eqnarray} 
  where $a^N >0$ depends on $C^N _{\GG_3,0}$ and $M_{\GG_3}(f_0)$, and
  where $\WW_{d_{\GG_3}}$ stands for an abstract Monge-Kantorovich
  distance in $P(\mathcal P_{\GG_3}(E))$ (see the third point in the next
  remark)
  \begin{equation}\label{eq:defWWd3} 
    \WW_{d_{\GG_3}} \!\!  \left(\pi^N_P f^N_0,\delta_{f_0}\right) = 
    \int_{E^N} d_{\GG_3}(\mu^N_V,f_0) \, {\rm d}f^N_0 (v).
   \end{equation}
\end{theo}

\begin{rem}
  In the applications the worst decay rate in the right-hand side of
  \eqref{eq:cvgabstract1} is always the last one. This last term
  controls two kind of errors: (1) the chaoticity of the initial data,
  that is how well $f^N _0 \sim f_0 ^{\otimes N}$, (2) the rate of
  convergence in the law of large numbers for measures in the distance
  $d_{\GG_3}$. 

  Let us discuss more the meaning of this last term and the related
  issue of sampling by empirical measures in statistics (see also
  Section~\ref{sec:chaos}).  Following the abstract definition of the
  optimal transport Wasserstein distance we define
$$
\forall \, \mu_1,\, \mu_2 \in P\left(\mathcal P_{\GG_3}\right), \quad
\WW_{d_{\GG_3}} \left(\mu_1,\mu_2\right) = \inf_{\pi \in
  \Pi \left(\mu_1, \mu_2\right)} \int_{\mathcal P_{\GG_3} \times \mathcal P_{\GG_3}}
d_{\GG_3}\left(\rho_1,\rho_2\right) \, {\rm d}\pi(\rho_1,\rho_2),
$$
where $\Pi(\mu_1,\mu_2)$ denotes the set of probability measures on
the product space $\mathcal P_{\GG_3} \times \mathcal P_{\GG_3}$ with
first marginal $\mu_1$ and second marginal $\mu_2$. In the case when
$\mu_2 = \delta_{f_0}$ then
\[
\Pi (\mu_1 , \delta_{f_0}) = \left\{ \mu_1 \otimes \delta_{f_0} \right\} 
\]
has only one element, and therefore
\begin{eqnarray*}
  \WW_{d_{\GG_3}}  \left(\pi^N_P f_0^N , \delta_{f_0}\right) 
  &=& \inf_{\pi \in \Pi \left(\pi^N_P f_0^N , \delta_{f_0}\right)} 
  \int_{\mathcal P_{\GG_3} \times \mathcal P_{\GG_3}} 
  d_{\GG_3} \left(\rho_1,\rho_2\right) \, {\rm d}\pi\left(\rho_1,\rho_2\right) \\
  &=&  \int_{E^{N}}  d_{\GG_3} \left(\mu^N_V,f_0\right)   \, {\rm d} f_0^N
  (v),
\end{eqnarray*}
which explains the notation \eqref{eq:defWWd3}.  We simply write in
the tensorized case:
\begin{equation}\label{eq:defOmegaN}
\WW_{d_{\GG_3}}  ^N \left(f_0 \right) := 
\WW_{d_{\GG_3}}  \left(\pi^N_P f_0^{\otimes N} , \delta_{f_0}\right)
\end{equation}
Comparisons of the $\WW^N_d$ functionals and estimates on the rate
\[
\WW_{d} ^N (f)  \to 0 \ \mbox{ as } \ N \to \infty
\]
depending on the choice of the distance $d$ are discussed in 
Subsection~\ref{app:W}. 
\end{rem}

\subsection{Proof of the abstract theorem} 
For a given function 
\[
\varphi \in (\FF_1 \cap \FF_2 \cap \FF_3)^{\otimes \ell},
\]
we break up the term to be estimated into three parts:
\begin{eqnarray*}
  &&\left| \left \langle \left( S^N_t(f_0^{N}) - \left( S _t ^{N\! L}
          (f_0)
        \right)^{\otimes N} \right),  \varphi  
      \otimes 1^{\otimes N-\ell} \right\rangle \right| \le
  \\
  &&\le \left| \left\langle S^N_t(f_0^N), 
      \varphi  \otimes
      1^{\otimes N-\ell} \right\rangle -
    \left \langle S^N_t(f_0^N), 
      R^\ell_\varphi \circ \mu^N_V \right\rangle  \right| \quad \qquad\qquad (=:  \TT_ 1 )
  \\
  &&+ \left| \left\langle f_0^N, T^N_t ( R^\ell_\varphi \circ
      \mu^N_V) \right\rangle
    - \left\langle f_0^N, 
      (T_t ^\infty R^\ell_\varphi ) \circ \mu^N_V) \right\rangle
  \right| \quad \qquad\qquad (=:  \TT_ 2 )
  \\  
  &&+ \left| \left\langle f_0^N, 
      (T_t ^\infty R^\ell_\varphi ) \circ \mu^N_V) \right\rangle
    -  \left\langle (S _t ^{N\! L} (f_0))^{\otimes \ell} ,
      \varphi \right\rangle \right| \qquad\qquad\qquad (=:  \TT_ 3 ).
\end{eqnarray*} 

We deal separately with each part step by step: 
\begin{itemize} 

\item $\TT_1$ is controlled by a purely combinatorial arguments
  introduced in \cite{Grunbaum}. Roughly speaking it is the
  combinatorial price we have to pay when we use the injection
  $\pi^N_E$ based on empirical measures.

\item $\TT_2$ is controlled thanks to the consistency estimate {\bf
    (A3)} on the generators which are well defined thanks to assumption {\bf (A2)} and Lemma~\ref{lem:H0}, the differential stability assumption
  {\bf (A4)} on the limit semigroup, the support constraint {\bf (A1)-(i)} and the propagation of integral
  moment bounds {\bf (A1)-(ii)}.

\item $\TT_3$ is controlled in terms of the chaoticity of the initial
  data thanks to the weak stability assumption {\bf (A5)} on the
  limit semigroup and the support moment bounds at initial time
  {\bf (A1)-(iii)}.
\end{itemize}

\bigskip

\noindent {\bf Step~1: Estimate of the first term $\TT_1$.} 
Let us prove that for any $t \ge 0$ and any $N \ge 2 \ell$ there holds
\begin{equation}\label{estim:T1}
  \TT_1 := \left| \left\langle S^N_t(f_0^N), 
      \varphi \otimes
      1^{\otimes N-\ell} \right\rangle - \left \langle S^N_t(f_0^N),
      R^\ell_\varphi \circ \mu^N_V \right\rangle \right| \le \frac{2 \,
    \ell^2 \, \| \varphi \|_{L^\infty(E^\ell)}}{N}.
\end{equation} 
Since $S^N_t(f_0^N)$ is a symmetric probability measure, estimate
(\ref{estim:T1}) is a direct consequence of the following lemma.

\begin{lem} \label{lem:symmetrization}
For any $\varphi \in C_b(E^\ell)$ we have
\begin{equation}\label{estim:symmetrization1}
  \forall \, N \ge 2 \ell, \quad  
  \left| \left( \varphi 
      \otimes {\bf 1}^{\otimes N-\ell} \right)_{\mbox{{\tiny {\em sym}}}} -
    \pi_N R^\ell_\varphi \right| 
  \le \frac{2 \, \ell^2 \, \| \varphi \|_{L^\infty(E^\ell)}}{N}
\end{equation}
where for a function $\phi \in C_b(E^N)$, we define its symmetrized
version $\phi_{\mbox{{\tiny {\em sym}}}}$ as:
\begin{equation*}
  \phi_{\mbox{{\tiny {\em sym}}}} = 
  \frac{1}{|\SSS^N|} \, \sum_{\sigma \in \SSS^N} \phi_\sigma
\end{equation*}
where we recall that $\SSS^N$ is the set of $N$-permutations. 

As a consequence for any symmetric measure we have 
\begin{equation} \label{eq:mNRphi} 
\forall \, f^N \in P_{\mbox{{\tiny
      {\em sym}}}}(E^N), \quad \left| \langle f^N,
    R^\ell_\varphi(\mu^N _V) \rangle - \langle f^N, \varphi \rangle
  \right| \le {2 \, \ell^2 \, \| \varphi \|_{L^\infty(E^\ell)} \over
    N}.
\end{equation}
\end{lem}

\begin{proof}[Proof of Lemma~\ref{lem:symmetrization}] 
  This lemma is a simple and classical combinatorial computation. We
  briefly sketch the proof for the sake of completeness.

  For a given $\ell \le N/2$ we introduce
\[ 
A_{N,\ell} := \left\{ (i_1, \dots, i_\ell) \in [|1,N|]^\ell \, : \
  \forall \, k \not= k' \in [|1,\ell|], \ i_k \not= i_{k'} \ \right\} 
\]
and 
\[
B_{N,\ell} := A_{N,\ell}^c = \left\{ (i_1, \dots, i_\ell) \in
  [|1,N|]^\ell \right\} \setminus A_{N,\ell}.
\]
Since there are $N !/(N-\ell)!$ ways of choosing $\ell$ distinct
indices among $\{ 1, \dots, N \}$, we get
\begin{eqnarray*}
  {\left|B_{N,\ell}\right|  \over N^\ell} &=& 
  1- \frac{N! }{(N-\ell)! \, N^\ell} \\
  &=& 
  1 - \left(1 - {1 \over N}\right) \, \cdots \, 
  \left(1 - {\ell-1 \over N}\right) 
  = 1 - \exp \left( \sum_{i = 0}^{\ell-1} 
    \ln \left(1 - \frac{i}N \right) \right) \\
  &\le& 1 - \exp \left( - 2 \sum_{i = 0}^{\ell-1} \frac{i}N \right) 
  \le 2 \sum_{i = 0}^{\ell-1} \frac{i}N \le {\ell^2 \over  N},
\end{eqnarray*}
where we have used 
\[
\forall \, x \in [0,1/2], \quad \ln ( 1 - x) \ge - 2 \, x \qquad \mbox{and}
 \qquad \forall \, x \in \R, \quad e^{-x} \ge 1 - x.
\]

Then we compute
\begin{eqnarray*} &&R^\ell_\varphi\left(\mu^N_V\right) =
  {1 \over N^\ell} \sum_{i_1, \dots, i_\ell = 1}^N 
 \varphi \left(v_{i_1}, \dots, v_{i_\ell}\right) \\
  &&= {1 \over N^\ell} \sum_{(i_1, \dots, i_\ell) \in A_{N,\ell}} 
  \varphi \left(v_{i_1}, \dots, v_{i_\ell}\right)
  +  {1 \over N^\ell} \sum_{(i_1, \dots, i_\ell) \in B_{N,\ell}}  
   \varphi \left(v_{i_1}, \dots, v_{i_\ell}\right) \\
  &&= {1 \over N^\ell} \, {1 \over (N-\ell)!} \sum_{\sigma \in \SSS_N}
  \varphi \left(v_{\sigma(1)}, \dots, v_{\sigma(\ell)}\right)
  +  {\mathcal O} \left(  {\ell^2 \over N} \, \|\varphi \|_{L^\infty}
  \right).
\end{eqnarray*}
We now use the same estimate 
\[
1 - \frac{N! }{(N-\ell)! \, N^\ell} \le {\ell^2 \over  N}
\]
as above to get
\begin{multline*}
R^\ell_\varphi\left(\mu^N_V\right) = {1 \over N!} \sum_{\sigma \in \SSS_N} 
  \varphi \left(v_{\sigma(1)}, \dots , v_{\sigma(\ell)}\right)
  +  {\mathcal O} \left(  {2 \, \ell^2 \over N} \, 
    \| \varphi \|_{L^\infty} \right)  \\ = \left( \varphi 
      \otimes {\bf 1}^{\otimes N-\ell} \right)_{\mbox{{\tiny sym}}}
  +  {\mathcal O} \left(  {2 \, \ell^2 \over N} \, 
    \| \varphi \|_{L^\infty} \right)
\end{multline*}
and the proof of (\ref{estim:symmetrization1}) is complete. 

Next for any $f^N \in P_{\mbox{{\tiny sym}}}(E^N)$ we have 
$$
\left\langle f^N, \varphi \right\rangle = \left\langle f^N,
  \left(\varphi \otimes {\bf
      1}^{\otimes N-\ell} \right)_{\mbox{{\tiny sym}}} \right\rangle,
 $$
 and \eqref{eq:mNRphi} trivially follows from
 (\ref{estim:symmetrization1}). 
\end{proof}

\bigskip\noindent {\bf Step~2: Estimate of the second term $\TT_2$. }
Let us prove that for any $t \in [0,T)$ and any $N \ge 2 \ell$ there
holds 
\begin{eqnarray}\label{estim:T2}
  \quad \TT_ 2 &:= & \left| \left\langle
      f_0^N, T^N_t \left( R^\ell_\varphi \circ \mu^N_V \right) \right\rangle -
    \left\langle f_0^N, \left(T_t ^\infty R^\ell_\varphi \right) \circ
        \mu^N_V \right\rangle \right| \\ \nonumber &\le&
  C _{m_{\GG_1},T} \, C_T ^\infty \, \eps(N) \, \ell^2 \, \| \varphi \|_{\FF_2^2 \otimes
    (L^\infty)^{\ell-2}}.
\end{eqnarray} 
Observing that the semigroup $T^\infty_t$ and its generator $G^\infty$ are well defined thanks to assumption {\bf (A2)} and 
Lemma~\ref{lem:H0}, we start with the following  identity
$$
T^N_t \pi_N - \pi_N T^\infty _t = - \int_0 ^t \frac{{\rm d}}{{\rm d}s} \left( T^N
  _{t-s} \, \pi_N \, T^\infty _s \right) \, {\rm d}s = \int_0 ^t T^N _{t-s}
\, \left[ G^N \pi_N - \pi_N G^\infty \right] \, T^\infty _s \, {\rm d}s.
$$

We then use assumptions {\bf (A1)-(i)}, {\bf (A1)-(ii)} and {\bf (A3)}
and we get for any $t \in [0,T)$
\begin{eqnarray}\nonumber
  && \left|  \left\langle f_0^N, 
      T^N_t \left( R^\ell_\varphi \circ \mu^N_V \right) \right\rangle 
    - \left\langle f_0^N, \left(T_t ^\infty R^\ell_\varphi \right) \circ 
      \mu^N_V  \right\rangle \right|  \\
  \nonumber
  &&\qquad \le \int_0 ^T \left| \left\langle M^N_{m_{\GG_1}}  \, 
      S^N _{t-s}\left(f_0 ^N\right), 
      \left( M^N _{m_{\GG_1}}\right)^{-1}  \, 
      \left[ G^N \pi_N - \pi_N G^\infty \right] \, 
      \left(T^\infty_s R^\ell_\varphi\right) \right\rangle \right| 
  \, {\rm d}s \\ \nonumber
  && \qquad \le
  \left( \sup_{0 \le t < T} \left\langle f^N _t, M^N _{m_{\GG_1}} \right 
    \rangle \right) \times \\  \nonumber
  && \qquad \qquad \qquad \qquad 
  \left( \int_0^T \left\| \left( M^N _{m_{\GG_1}}\right)^{-1} \, 
      \left[ G^N \pi_N - \pi_N G^\infty \right] \, 
      \left(T^\infty_s R^\ell_\varphi\right)
    \right\|_{L^\infty(\mathbb{E}_N)} \, {\rm d}s \right) \\
  \label{eq:trotter}
  &&\qquad \qquad  
  \le  C_{m_{\GG_1},T} \, \eps(N) \, \sup_{{\bf r} \in \RR_{\GG_1} } \int_0^T 
  \left[ T^\infty_s R^\ell_\varphi\right]_{C^{1,\eta}
    _{\Lambda_1}(\mathcal P_{\GG_1},{\bf r} )} \, {\rm d}s. 
\end{eqnarray}

Now, let us fix ${\bf r} \in \RR_{\GG_1}$. Since
\[
T^\infty_t(R^\ell_\varphi) = R^\ell_\varphi \circ S^{N \!  L}_t \
\mbox{ with } S^{N \! L}_t \in C^{1,\eta}_{\Lambda_2} (\mathcal P_{\GG_1,{\bf
    r}};\mathcal P_{\GG_2})
\]
and $R^\ell_\varphi \in C^{1,1}(\mathcal P_{\GG_2})$ because $\varphi \in
\FF_2^{\otimes\ell}$ (see subsection~\ref{sec:compatibility}), we can
apply Lemma~\ref{lem:DL} and use assumption {\bf (A4)} to obtain
\[
T^\infty_t(R^\ell_\varphi) \in
C^{1,\eta}_{\Lambda_1}(\mathcal P_{\GG_1,{\bf r}})
\]
with
\[
\left[ T^\infty_s\left(R^\ell_\varphi\right) \right]
_{C^{1,\eta} _{\Lambda_1} (\mathcal P_{\GG_1,{\bf r}})}
  \le \left( \left[ S^{N \! L}_t\right]_{C^{1,\eta}_{\Lambda_2}(\mathcal P_{\GG_1, {\bf r}},\mathcal P_{\GG_2})} 
  + \left[ S^{N \! L}_t\right]^2_{C^{0,(1+\eta)/2}_{\Lambda_2}(\mathcal P_{\GG_1, {\bf r}},\mathcal P_{\GG_2})} \right)
  \, \left\| R^\ell_\varphi \right\|_{C^{1, 1}(\mathcal P_{\GG_2})}
\]
and  $\Lambda_2 = \Lambda_1^{1/2}$.
%
We then deduce thanks to \eqref{eq:Ck1polyk} and assumption {\bf (A4)}:
\begin{equation}
  \label{estim:TtPhi}
  \int_0^T [ T^\infty_s(R^\ell_\varphi) ]_{C^{1,\eta} _{\Lambda_1} (\mathcal P_{\GG_1,{\bf r}})} \, {\rm d}s
  \le C^\infty_T \,\ell^2 \,  \left(\| \varphi \|_{\FF_2 \otimes
      (L^\infty)^{\ell-2}} 
    + \| \varphi \|_{\FF_2 \otimes (L^\infty)^{\ell-1}} \right).
\end{equation}
Then we go back to the computation \eqref{eq:trotter}, and plugging
(\ref{estim:TtPhi}) we deduce (\ref{estim:T2}).  

\bigskip\noindent {\bf Step~3: Estimate of the third term $\TT_3$. }
Let us prove that for any $t \ge 0$ and $N \ge \ell$ we have 
\bean
  \TT_3 &:=& \left| \left\langle f_0^N, \left(T_t ^\infty
        R^\ell_\varphi \right) \circ \mu^N_V \right\rangle -
    \left\langle \Big(S _t ^{NL} (f_0)\Big)^{\otimes \ell} ,
      \varphi \right\rangle \right|  \\
 & \le & \left[R_\varphi \right]_{C^{0,1}}  \,   
\Theta_{a^N,T}  \left( \WW_{1,\mathcal P_{\GG_3}} \left(\pi^N_P f^N_0,\delta_{f_0}\right) \right)
\eean
where $\Theta_{a,T}$ was introduced in assumption {\bf
  (A5)}, and $a=a^N$ is defined by 
\[
a^N := \max \left\{ M_{\GG_3}(f_0) \ , \ C_{m_{\GG_3},0} ^N \right\} , 
\]
where $C_{m_{\GG_3},0} ^N$ was introduced in assumption {\bf (A1)-(iii)}.

Assumption {\bf (A1)-(iii)} indeed implies that 
\[
\mbox{Supp} \, f_0^N \subset \KK := \left\{ V \in \R^{dN} \ \mbox{ s. t. }
\ M^N _{m_{\GG_3}} \left( \mu^N _V \right) = 
\frac1N \, \sum_{i=1} ^N m_{\GG_3}(v_i)  \le C^N_{m_{\GG_3},0} \right\}.
\]
Hence we are in position to apply {\bf (A5)} for the functions $f_0$
and $\mu^N _V$ \emph{on the support of $f^N_0$} since
$M_{m_{\GG_3}}(f_0)$ is bounded by assumption, and $M_{\GG_3}(\mu^N
_V)$ is bounded by $C_{m_{\GG_3},0}^N$ when restricting to $V \in \KK$
thanks to the previous equation.

Let us also recall that $R^\ell_\varphi \in C^{0,1}(\mathcal P_{\GG_3},\R)$
because $\varphi \in \FF_3^{\otimes\ell}$.

We then write 
\begin{eqnarray*}
  \TT_3 
  &=& \left| \left\langle f^N_0, 
      R^\ell_\varphi \left(S^{N\!L}_t (\mu^N_V)\right) 
    \right\rangle -  \left \langle f_0 ^N, R^\ell_\varphi
      \left(S^{N\!L}_t(f_0 )\right) \right \rangle \right| 
  \\
  &=& \left| \left\langle f^N_0, R^\ell_\varphi \left(S^{N\!L}_t (\mu^N_V)\right) 
      - R^\ell_\varphi \left(S^{N\!L}_t(f_0 )\right) \right\rangle \right| 
  \\
  &\le&  \big[R_\varphi\big]_{C^{0,1}(\mathcal P_{\GG_3})} \, \left\langle f^N_0, \, 
   d_{\GG_3} 
    \left( S^{N \! L}_t(f_0), S^{N \! L}_t( \mu^N_V) \right) \right\rangle.
\end{eqnarray*}

We now apply {\bf (A5)} to get 
\[
\forall \, t \in [0,T], \quad d_{\GG_3} \left( S^{N \! L}_t(f_0), S^{N
    \! L}_t( \mu^N_V) \right) \le \Theta_{a^N,T} \left( d_{\GG_3}
\left( f_0, \mu^N _V \right) \right)
\]
on the support of $f^N _0$, and therefore
\[ 
\TT_3 \le \big[R_\varphi\big]_{C^{0,1}(\mathcal P_{\GG_3})} \, \left\langle
  f^N_0, \, \Theta_{a^N,T} \left( \mbox{d}_{\GG_3} \left(f_0 ,
      \mu^N_V\right) \right) \right\rangle.
\]
We then obtain from the concavity of the $\Theta_{a^N,T}$ function:
\[
\TT_3 \le \big[R_\varphi\big]_{C^{0,1}(\mathcal P_{\GG_3})} \,
\Theta_{a^N,T} \left( \left\langle f^N_0, \, \mbox{d}_{\GG_3}
    \left(f_0 , \mu^N_V\right) \right\rangle \right)
\]
which concludes the proof of this step. 

\bigskip

The proof of Theorem~\ref{theo:abstract} is complete by combining the
previous steps.


\section{The $N$-particle approximation at initial time}
\label{sec:chaos}
\setcounter{equation}{0}
\setcounter{theo}{0}

\subsection{Comparison of distances on probabilities}
\label{app:distances}
  In the following lemma we compare the different
  metrics and norms defined Subsection~\ref{subsec:ExpleMetrics}. 
Let us write
$$
M_k(f,g) := \max \left\{ \big\langle f,\, \langle v \rangle^k
  \big\rangle \, ; \  \big\langle g,\, \langle v \rangle^k
  \big\rangle \right\}.
$$

\begin{lem}\label{lem:ComparDistances} Let $f,g \in P(\R^d)$ and $k
  \in (0,\infty)$, then the following estimates hold:
\begin{enumerate}
\item[(i)] For any $q \in (1,+\infty)$ and any $k \in [q-1,\infty)$: 
\begin{equation}\label{estim:W1Wq}
   W_1(f,g) \le W_q(f,g) \le 2^{\frac{k+1}{q}} \,  
   M_{k+1}(f,g)^{\frac{q-1}{q k}} \, W_1(f,g)^{\frac1q \left( 1-
       \frac{q-1}{k} \right)}.
\end{equation}
\item[(ii)] For any $s \in (0,1]$, 
\begin{equation}\label{estim:dsWs}
  | f - g |_s \le 2^{(1-s)} \, W_s(f,g) \le 2^{(1-s)} \,
  W_1(f,g)^s.
\end{equation}
\item[(iii)] For any $s \in (d/2,d/2+1)$, 
\begin{equation}\label{estim:H-kd1}
  \| f-g \|_{\dot H^{-s}}^2
  \le  \frac{8 \, \left|\mathbb
    S^{d-1}\right|}{(2s -d)} \, \left(
  \frac {(2s-d)}{4 ( d+2 - 2s)} \right)^{s-\frac{d}2} \, |f-g |_1^{2s-d}.
\end{equation}
\item[(iv)] For any $s>0$ and $k>0$ we have 
\begin{equation}\label{estim:*1s} 
[f-g]^*_1 \le C(d,s,k) \,
M_{k+1}(f,g)^{\frac{d}{d+k(d+s)}} \, | f-g |_s^{\frac{k}{d+k(d+s)}}
\end{equation}
for some constant $C(d,s,k)>0$ depending on $d$, $s$ and $k$. 
\item[(v)] For any 
\[
s \in \left(\max \left\{\frac{d}2 \, ; \, 1\right\} ,
  \frac{d}2+1 \right)
\]
and $k>0$ we have
\begin{equation}\label{estim:W1H-k}
  [f-g]_1 ^* \le C(d,s,k) \,  M_{k+1}(f,g) ^{\frac{d}{d+ 2 k s}} \,  \| f-g
  \|_{ \dot H^{-s}}^{\frac{2k}{d+ 2 k s}}
\end{equation}
for some constant $C(d,s,k)>0$ depending on $d$, $s$ and $k$.
\end{enumerate}
\end{lem}

\begin{proof}[Proof of Lemma~\ref{lem:ComparDistances}]  Let us consider
 each inequality one by one. 
\medskip

\noindent {\it Point (i).} The inequality \eqref{estim:W1Wq} is
well-known in optimal transport theory, we refer for instance to
\cite{TV,coursCT}.  \medskip

\noindent {\it Point (ii).} Let us prove inequality
\eqref{estim:dsWs}.  Let $\pi \in \Pi(f,g)$. We write
\begin{eqnarray*}
  | \hat f (\xi) -  \hat g (\xi) | 
  &=&
  \left|  \int_{\R^d \times \R^d} \left(e^{-i \, v \cdot \xi} - e^{-i
        \, w \cdot \xi}\right) \, 
    {\rm d}\pi (v,w) \right|
  \\
  &\le&
  \int_{\R^d \times \R^d} \left|e^{-i \, v \cdot \xi} - e^{-i \, w \cdot \xi}\right| \, {\rm d}\pi (v,w) 
  \\
  &\le&
  2^{(1-s)} \, \int_{\R^d \times \R^d} |v-w|^s \, |\xi|^s \, {\rm d}\pi (v,w),
\end{eqnarray*}
which yields the first inequality in \eqref{estim:dsWs} by taking the
supremum in $\xi \in \R^d$ and the infimum in $\pi \in \Pi(f,g)$. The
second inequality is then immediate from the concavity estimate
\[
W_s(f,g) \le \left( W_1(f,g) \right)^s. 
\]
 
\medskip\noindent{\sl Point (iii).} Let us prove the inequality
\eqref{estim:H-kd1}. Consider $R > 0$ and the ball 
\[
B_R :=\left\{ x \in \R^d \ ; \ |x|\le R\right\},
\]
and write
\begin{eqnarray*}
  \| f-g \|_{\dot H^{-s}}^2
  &=& \int_{B_R} { | \hat f (\xi) -  \hat g (\xi) |^2  \over |\xi|^{2s}} \, {\rm d}\xi 
  +  \int_{B_R^c} { | \hat f (\xi) -  \hat g (\xi) |^2  \over |\xi|^{2s}} \, {\rm d}\xi
  \\
  &\le& |f- g|^2_1 \int_{B_R} { {\rm d}\xi  \over |\xi|^{2(s-1)}}
  +  4 \, \int_{B_R^c} {{\rm d}\xi  \over |\xi|^{2s}}
  \\
  &\le& \left|\mathbb S^{d-1}\right| \, R^{d-2s} \, \left(
    \frac{R^2}{(d+2- 2s)} \, |f - g|^2_1 +  \frac{4}{(2s-d)} \right).
\end{eqnarray*}
We optimize this estimate by choosing
\[
R = \left( \frac{4 ( d+2 - 2s)}{(2s-d)} \right)^{\frac12} \,
|f-g|_1 ^{-1}
\]
which yields 
\[
\| f-g \|_{\dot H^{-s}}^2 \le \frac{8 \, \left|\mathbb
    S^{d-1}\right|}{(2s -d)} \, \left(
  \frac {(2s-d)}{4 ( d+2 - 2s)} \right)^{s-\frac{d}2} \, |f-g|_1
^{2s -d}
\]
which concludes the proof of \eqref{estim:H-kd1}.
 
\medskip\noindent{\it Point (iv).} Let us now prove inequality
\eqref{estim:*1s}. We introduce a truncation function 
\[
\chi_R(x) = \chi\left( \frac{x}{R}\right), \quad R>0
\] 
where 
\[
\chi \in C_c^\infty(\R^d), \quad [\chi]_1 \le 1, \ \mbox{ and } \ 0
\le \chi \le 1, \quad \chi \equiv 1 \mbox{ on } B(0,1)
\]
and a mollifier 
\[
\gamma_\eps(x) = \eps^{-d} \, \gamma\left(\frac{x}{\eps}\right), \quad
\eps>0 \qquad \mbox{
  with } \ \gamma(x) = \frac{e^{-\frac{|x|^2}2}}{(2\pi)^{d/2}}.
\]
In particular we have an explicit formula for the Fourier transform of
this mollifier: 
\[
\hat \gamma_\eps (\xi) = \hat \gamma (\eps \, \xi) = e^{-
\eps^2 \, \frac{|\xi|^2}2}.
\]

Fix $\varphi \in W^{1,\infty}(\R^d)$ such that $[\varphi]_1 \le 1$,
$\varphi(0) = 0$ and define 
\[
\varphi_R := \varphi \, \chi_R, \quad \varphi_{R,\eps} = \varphi_R
\ast \gamma_\eps
\]
and write
 \begin{multline*}
 \int_{\R^d} \varphi \, ({\rm d}f-{\rm d}g) \\ = \int_{\R^d} \varphi_{R,\eps} \, ({\rm
   d}f-{\rm d}g) + \int_{\R^d} \left(\varphi_R - \varphi_{R,\eps}\right) \,
 ({\rm d}f-{\rm d}g) + \int_{\R^d} \left(\varphi - \varphi_R \right) \, ({\rm
   d}f-{\rm d}g).
\end{multline*}
For the last term, we have 
\begin{eqnarray}
\label{estim:*1s-1}
  \left|  \int_{\R^d} (\varphi_R- \varphi) \, ({\rm d}f-{\rm d}g) \right| 
  &\le&
   \int_{\R^d} (1- \chi_R) \,| \varphi |\, ({\rm d}f+{\rm d}g)
  \\  \nonumber
 &\le& \int_{B_R^c} [\varphi ]_1 \, {|x|^{k+1} \over R^k} \, ({\rm
   d}f+{\rm d}g) \le 
  \frac{M_{k+1}(f,g)}{R^k}.
\end{eqnarray}

Concerning the second term, we observe that 
$$
\left|\nabla \varphi_R(x) \right| \le \chi \left(\frac{x}{R}\right) +
|\varphi(x)|\, |\nabla (\chi_R)(x)| \le \chi \left(\frac{x}{R}\right) + {|x| \over R} \,
\left|\nabla \chi\left(\frac{x}{R}\right)\right|,
$$
so that
\[
\forall \, q \in [1,\infty], \quad 
\left\|\nabla \varphi_R \right\|_{L^q} \le C(q,d,\chi) \, R^{\frac{d}{q}},
\]
for some constant depending on $q$, $d$ and $\chi$.  Next, using that
\begin{multline*}
\| \varphi_R - \varphi_{R,\eps} \|_\infty \le \| \nabla \varphi_R
\|_\infty \left( \int_{\R^d} \gamma_\eps (x) \, |x| \, {\rm d}x
\right)  \\ = \eps \, \| \nabla
\varphi_R \|_\infty \left( \int_{\R^d} \gamma (x) \, |x| \, {\rm d}x \right) \le
C(q,d,\chi) \, \eps,
\end{multline*}
we find
\begin{equation}\label{estim:*1s-2}
\left| \int_{\R^d} \left(\varphi_R - \varphi_{R,\eps} \right) \, ({\rm d}f  -
  {\rm d}g) \right| \le C(q,d,\chi) \, \eps.
\end{equation}

Concerning the first term,  using Parseval's identity, we have (the
``hat'' denotes as usual the Fourier transform)
\begin{eqnarray}\nonumber
  \left| \int_{\R^d} \varphi_{R,\eps} \, ({\rm d}f-{\rm d}g) \right|
  &=& {1 \over (2\pi)^d} \left| \int_{\R^d} \hat \varphi_R \, \hat \gamma_\eps \,
    \overline{(\hat f - \hat g)} \, {\rm d}\xi \right| 
  \\ \nonumber
  &\le& {1 \over (2\pi)^d} \, \| \nabla \varphi_R \|_{L^1} \, 
  \left\| \frac{\hat f - \hat g}{|\xi|^s} \right\|_{L^\infty}  \,
  \left( \int_{\R^d} |\xi|^{s-1} 
  \, e^{- \eps^2 \, \frac{|\xi|^2}2} \, {\rm d}\xi \right)
  \\  \nonumber
  &\le& C(d,\chi) \, R^{d} \, 
   |f- g|_s  \, \eps^{-(d+s-1)}  
    \\ \label{estim:*1s-3}
  &\le& C \, R^{d} \, \eps^{-(d+s-1)}  \, |f - g|_s. 
\end{eqnarray}

 Gathering \eqref{estim:*1s-1}, \eqref{estim:*1s-2} and
 \eqref{estim:*1s-3}, we get 
 \[ [f - g]_1 ^* \le C(q,d,\chi) \, \left( \frac{M_{k+1}(f,g)}{R^k} +
   \eps +
   R^{d} \, \eps^{-(d+s-1)} \, |f - g|_s \right).
 \]
 This yields \eqref{estim:*1s} by optimizing the parameters $\eps$ 
 and $R$ with 
\[
R = M_{k+1}(f,g) ^{\frac{1}{d+k}} \, |f-g|_s ^{-\frac{1}{d+k}} \,
\eps^{\frac{d+s-1}{d+k}}
\]
and then
\[
\eps = M_{k+1}(f,g) ^{\frac{d}{d+k(d+s)}} \, |f-g|_s
^{\frac{k}{d+k(d+s)}}.
\]
 
\medskip\noindent{\it Point (v).} Let us now prove the inequality
\eqref{estim:W1H-k}.

Let us consider some smooth $\varphi$ such that $[\varphi]_1 \le 1$,
$\varphi(0) = 0$ and let us perform the same decomposition as before:
$$
\int_{\R^d} \varphi \, ({\rm d}f-{\rm d}g) = \int_{\R^d}
\varphi_{R,\eps} \, ({\rm d}f-{\rm d}g) +
\int_{\R^d} \left(\varphi_R - \varphi_{R,\eps}\right) \, ({\rm
  d}f-{\rm d}g) + \int_{\R^d}
\left(\varphi - \varphi_R \right) \, ({\rm d}f-{\rm d}g).
$$ 

The first term is controlled by
\[
\left| \int_{\R^d} \varphi_{R,\eps} \, ({\rm d}f-{\rm d}g) \right| 
= \left| \int_{\R^d} \hat \varphi_{R,\eps} \,
  |\xi|^s \, \frac{(\hat f - \hat g)}{|\xi|^s} \, {\rm d}\xi \right| \le 
\| \varphi_{R,\eps} \|_{\dot H^s} \, \| f - g \|_{\dot H^{-s}}
\]
with 
\begin{eqnarray*}\nonumber
  \|  \varphi_{R,\eps} \|_{\dot H^s} 
  &=&  \left(  \int_{\R^d} |\xi|^2 \, |\widehat{\varphi_R}|^2 \, |\xi|^{2(s-1)} \, |\hat\gamma_\eps|^2 \, {\rm d}\xi \right)^{1/2}
  \\  \nonumber  
  &\le& \left\| \nabla (\varphi_R) \right\|_{L^2} \, 
  \left\| |\xi|^{s-1} \, \hat  \gamma_\eps (\xi) \right\|_{L^\infty}
  \\ 
  &\le& C(s) \, \left\| \nabla (\varphi_R) \right\|_{L^2} \, \eps^{1-s} 
\le  C(d,s) \, R^{\frac{d}{2}} \,  \eps^{-(s-1)}
\end{eqnarray*} 
with the same arguments as above.

The second term and the last term are controlled exactly as in
\eqref{estim:*1s-1} and \eqref{estim:*1s-2}, which yields 
\[ 
[f-g]_1 ^* \le C \, \left( \eps + \frac{M_{k+1}(f,g)}{R^{k}} +
  R^{\frac{d}{2}} \, \eps^{-(s-1)} \, \|f-g\|_{\dot H^{-s}} \right).
\]
We deduce (\ref{estim:W1H-k}) by optimizing the parameters $\eps$ and
$R$ with 
\[
R = M_{k+1}(f,g) ^{\frac{2}{d+2k}} \, \|f-g\|_{\dot H^{-s}}  ^{-\frac{2}{d+2k}} \,
\eps^{\frac{2(s-1)}{d+2k}}
\]
and then
\[
\eps = M_{k+1}(f,g) ^{\frac{d}{d+2 k s}} \, \|f-g\|_{\dot H^{-s}} 
^{\frac{2k}{d+2 k s}}.
\]
\end{proof}

\subsection{Quantitative law of large numbers for measures}
\label{app:W}

Let us recall and extend the definition of the functional $\WW^N _d
(f)$ which was introduced in \eqref{eq:defOmegaN}. For any function 
\[
D : \PP_k(E) \times \PP_k(E) \to \R_+, \quad (f,g) \mapsto D(f,g)
\]
(where $k \ge 0$ is the index of a polynomial weight possibly required for the
correct definition of $D$) such that 
\[
D(f,g) = 0 \ \mbox{ if and only if } \ f=g
\]
it is legitimate to define 
\[
\forall \, f \in \PP_k(\R^d), \quad 
\WW^N_D (f) :=
\left\langle f^{\otimes N}, D\left( \pi^N_E , f\right) \right\rangle
= \int_{E^N} D\left( \mu^N _V, f\right) \,  {\rm d} f^{\otimes N} (v).
\]

For well chosen function $D$, the goal of the next lemma is to
quantity the rate of convergence 
\[
\WW^N_D (f) \xrightarrow[]{N \to +\infty} 0 \qquad \mbox{ in the case}
\ E=\R^d.
\]

\begin{lem}\label{lem:Rachev&W1} 
  We have the following rates for the $\WW$ function:
  \begin{itemize}
  \item[(i)] Let us consider 
    \[
    \forall \, f,g \in \PP_2(\R^d), \quad D_1(f,g) := \left\| f-g  \right\|^2_{\dot H^{-s}}.
    \]
    Then for any $s \in (d/2,d/2+1)$ and $N \ge 1$ there holds
    \begin{equation}\label{estim:RachevHdotk}
      \forall \, f \in \PP_2(\R^d), \quad \WW^N _{D_1} (f) = 
      \int_{\R^{dN}} \left\| \mu^N_V - f \right\|^2_{\dot H^{-s}} \, 
      {\rm d}f^{\otimes N} (V) 
      \le \frac{C}{N}  
    \end{equation}
    for some constant $C$ depending on the second moment of $f$. 
  \item[(ii)] Let us consider 
    \[
    \forall \, f,g \in \PP_2(\R^d), \quad D_2(f,g) := \left\| f-g  \right\|^2_{H^{-s}}.
    \]
    Then for any $s >d/2$ and $N \ge 1$ there holds
    \begin{equation}\label{estim:RachevHk}
      \forall \, f \in \PP_2(\R^d), \quad \WW^N _{D_2} (f) = 
      \int_{\R^{dN}} \left\| \mu^N_V - f \right\|^2_{H^{-s}} \, 
      {\rm d}f^{\otimes N} (V) 
      \le \frac{C}{N}  
    \end{equation}
    for some constant $C$ depending on the second moment of $f$.  

  \item[(iii)] Let us consider 
    \[
    \forall \, f,g \in \PP_1(\R^d), \quad D_3(f,g) := W_1(f,g).
    \]
    Then for any $\eta>0$ there exists $k \ge 1$ such that for any $N
    \ge 1$
    \begin{equation}\label{estim:NousW1}
      \forall \, f \in \PP_k(\R^d), \quad \WW^N _{D_3} (f) = 
      \int_{\R^{dN}} W_1\left(\mu^N _V, f\right) \, 
      {\rm d}f^{\otimes N} (V) 
      \le \frac{C}{N^{\frac{1}{\max\{d,2\} + \eta}}}
    \end{equation}
    for some constant $C$ depending on $\eta$ and the $k$-th moment of
    $f$.

 \item[(iv)] Let us consider 
    \[
    \forall \, f,g \in \PP_2(\R^d), \quad D_4(f,g) := \left( W_2(f,g)\right)^2.
    \]
    Then for any $\eta>0$ there exists $k \ge 2$ such that for any $N
    \ge 1$
    \begin{equation}\label{estim:NousW2}
      \forall \, f \in \PP_k(\R^d), \quad \WW^N _{D_4} (f) = 
      \int_{\R^{dN}} \left( W_2\left(\mu^N _V, f\right) \right)^2 \, 
      {\rm d}f^{\otimes N} (V) 
      \le \frac{C}{N^{\frac{1}{\max\{d,2\} + \eta}}}  
    \end{equation}
    for some constant $C$ depending on $\eta$ and the $k$-th moment of
    $f$. 
  \end{itemize}
\end{lem}
 
 \begin{rems}\label{rem:Rachev} 
  \begin{enumerate} 
\item  Estimate \eqref{estim:NousW2} has to be compared with the 
   following classical estimate (see e.g.  \cite{BookRachev}): for
   any  $N \ge 1$ there holds
  \begin{equation}\label{estim:Rachev2}
    \forall \, f \in \PP_{d+5}(\R^d), \quad 
    \WW^N _{W_2^2} (f) \le \frac{C}{N^{{2 \over d+4}}}
 \end{equation}
 where the constant $C>0$ depends on the $(d+5)$-th moment of $f$.  It
 is worth mentioning that our estimate \eqref{estim:NousW2} improves
 on \eqref{estim:Rachev2} when $d \le 3$ and $k$ is large enough.

 Similarly, if one tries to translate \eqref{estim:Rachev2} into an
 estimate for the $W_1$ distance through 
 a H\"older inequality, it yields for any $N \ge 1$
  \begin{equation*}
   \forall \, f \in \PP_{d+5}(\R^d), \quad   \WW^N _{W_1} (f) \le 
      \frac{C}{N^{ {1 \over d+4}}}
    \end{equation*}
    where the constant $C>0$ depends on the $(d+5)$-th moment of
    $f$. Again observe that for any $d$ this last estimate is weaker
    than \eqref{estim:NousW1} as soon as the probability $f$ belongs
    to $\PP_{k}(\R^d)$ with $k$ large enough. 

  \item When $f, g \in P(\R^d)$ are compactly supported, observe that
    the estimate \eqref{estim:W1H-k} improves into
$$
\forall \, s \ge 1, \quad [f-g]^*_1 \le C \, \| f - g \|_{\dot H^{-s}}^{1/s},
$$
for a constant $C$ depending on $s$ and on a common bound $R$ of
the support of $f$ and $g$. 

If furthermore $d=1$, we can take $s=1$ in order to apply
\eqref{estim:RachevHdotk} in the proof of \eqref{estim:NousW1} below
and we obtain the {\it ``optimal rate''} of convergence in the
functional law of large numbers in Wasserstein distance $W_1$:
$$
\forall \, N \ge 1, \quad   \WW^N _{W_1} (f) \le 
     \frac{C}{\sqrt N}.
$$
In higher dimension $d \ge 2$, the restriction $s > d/2$ means that one
does not produce a better estimate than \eqref{estim:NousW1} by this
line of argument. 

\item As was kindly pointed out by M. Hauray, estimate
  \eqref{estim:NousW2} should also be compared with some estimates in
  \cite{DubricYukich} where the related quantity
$$
\ZZ^N(f) := \int_{\R^{2d N}} 
W_1\left(\mu^N_{V_1},\mu^N_{V_2}\right) \, {\rm d}f^{\otimes N} (V_1)
\, {\rm d}f^{\otimes N} (V_2)
$$
is considered. When $f \in P(\R^d)$ has compact support and $d \ge 3$,
they prove that 
\[
\ZZ^N(f) \le \frac{C}{N^{1/d}}
\]
where the constant depends on the support of $f$. 

Since for any $f, g \in \PP_1(\R^d)$ and for any $\varphi \in
\hbox{Lip}_1(\R^d)$ we have
\begin{eqnarray*}
\int_{\R^{dN}} W_1\left(g,\mu^N_V\right) \, {\rm d}f^{\otimes N}(V) 
&\ge& \int_{\R^{d(N+1)}} \varphi (v) \, \left({\rm d}g  - {\rm d}\mu^N_V\right)(v) \, {\rm d}f^{\otimes N}(V) 
\\
&=& \int_{\R^d} \varphi (v) \, {\rm d}g(v) -  {1 \over N} \sum_{i=1}^N 
\int_{\R^{dN}} \varphi\left(v_i\right) \, {\rm d}f^{\otimes N} (V)
\\
&=& \int_{\R^d} \varphi (v) \, ({\rm d}g - {\rm d}f) (v),
\end{eqnarray*}
we deduce by minimizing in $\varphi$ that 
$$
W_1(f,g) \le \int_{\R^{dN}} W_1\left(g,\mu^N_V\right) \, {\rm d}f^{\otimes N}(V),
$$
and therefore 
\[
\WW^N _{W_1} (f) \le \ZZ^N(f).
\]

As a consequence, when $f \in P(\R^d)$ has compact support and $d \ge
3$ we obtain from this line of argument the stronger estimate 
\[
\WW^N _{W_1} (f) \le \frac{C}{N^{1/d}}. 
\]
It is likely that one could obtain similar estimates to
\eqref{estim:NousW1} by tracking the formula for the constants in the
results of \cite{DubricYukich} and combining them with moment bounds
and some interpolation. 

On the other hand, observe that our estimate \eqref{estim:NousW1} is
{\it almost optimal} in the sense that we can not expect a better
convergence rate than \eqref{estim:NousW1} with $\eta = 0$, as it is
stressed in \cite[Appendix]{Peyre}.
\end{enumerate}
\end{rems}

\begin{proof}[Proof of Lemma~\ref{lem:Rachev&W1}]
 We split the proof into two steps. 

 \medskip\noindent{\sl Step~1.} 
Let us prove \eqref{estim:RachevHdotk} (note that
\eqref{estim:RachevHk} is then readily implied by
\eqref{estim:RachevHdotk}). 
  
Let us fix $f \in \PP_2(\R^d)$. We write in Fourier transform
 \[
 \left(\hat\mu^N_V - \hat f \right)(\xi) = \frac{1}{N} \, \sum_{j=1}^N
 \left(e^{-i \, v_j \cdot \xi } - \hat f(\xi) \right) ,
 \]
 which implies
 \begin{eqnarray*}
   \WW^N_{\| \cdot \|_{\dot H^{-s}}^2} (f) 
   &=&  \int_{\R^{Nd}} \left( \int_{\R^d} 
     \frac{\left|\hat\mu^N_V - \hat f \right|^2}{|\xi|^{2 \, s}} \,
     {\rm d}\xi\right) \, {\rm d}f^{\otimes N} (V) 
   \\
   &=& \frac{1}{N^2} \, \sum_{j_1,j_2=1}^N \int_{\R^{(N+1) d}} 
   \frac{\left(e^{-i \, v_{j_1} \cdot \xi} - \hat f(\xi)\right) \, 
     \left(\overline{ e^{-i \, v_{j_2} \cdot \xi} - \hat
         f(\xi)}\right)}
   {|\xi|^{2 \, s}} \, {\rm d}\xi  \, {\rm d}f^{\otimes N} (V).
 \end{eqnarray*}
 
 Observe then that 
 $$
  \int_{\R^d} (e^{-i \, v_j \cdot \xi}- \hat f(\xi) ) \, {\rm d}f(v_j) = 0,
  \quad j=1, \dots, d,
  $$
  which implies that 
\[
\int_{\R^{(N+1) d}} 
   \frac{\left(e^{-i \, v_{j_1} \cdot \xi} - \hat f(\xi)\right) \, 
     \left(\overline{ e^{-i \, v_{j_2} \cdot \xi} - \hat
         f(\xi)}\right)}
   {|\xi|^{2 \, s}} \, {\rm d}\xi  \, {\rm d}f^{\otimes N} (V) =0
\]
as soon as $j_1 \not = j_2$, and
 \begin{eqnarray*}
 \int_{\R^d} \left| e^{-i \, v \cdot \xi } - \hat f(\xi) \right|^2 \, {\rm d}f(v) 
 &=& \int_{\R^d} \left[ 1 - e^{-i \, v \cdot \xi } \,  
   \overline{\hat f(\xi)} - e^{i \, v \cdot \xi } \, 
   \hat f(\xi) + | \hat f(\xi) |^2 \right]  \, {\rm d}f(v) \\
 &=& 1 - |\hat f(\xi) |^2.
 \end{eqnarray*}

We deduce that 
 \begin{eqnarray*}
   \WW^N_{\| \cdot \|_{\dot H^{-s}}^2} (f) 
   &=&
   \frac{1}{N^2} \, \sum_{j=1}^N \int_{\R^{(N+1) d}} 
   \frac{ \left| e^{-i \, v_j \cdot \xi } - \hat
       f(\xi)\right|^2}{|\xi|^{2 \, s}} \, {\rm d}\xi  \, {\rm d} f^{\otimes N} (V) 
     \\
     &=& \frac{1}{N} \, \int_{\R^{2 d}} \frac{\left| e^{-i \, v \cdot
           \xi } - \hat f(\xi) 
       \right|^2}{|\xi|^{2 \, s}} \, {\rm d}\xi  \, {\rm d}f(v) 
     \\
     &=& \frac{1}{N} \, 
     \int_{\R^{d}} \frac{(1 - |\hat f(\xi) |^2)}{|\xi|^{2 \, s}} \, {\rm d}\xi.
 \end{eqnarray*}
 
 Finally, denoting 
\[
M_2 := \int_{\R^d} \langle v \rangle^2 \, {\rm d}f(v) 
\]
we have 
\[
\hat f(\xi) = 1 +i \, \langle f , v \rangle \cdot \xi + \OO(M_2 \,
|\xi|^2),
\]
and therefore
 \begin{eqnarray*}
   |\hat f(\xi)|^2  
   &=& \left(1 +i \,  \langle f , v \rangle \cdot \xi + \OO(M_2 \,
   |\xi|^2)\right) \, 
   \left(1  - i \,  \langle f , v \rangle \cdot \xi 
     + \overline{\OO(M_2 \, |\xi|^2)}\right) \\
   &=& 1 + \OO\left(M_2 \, |\xi|^2\right), 
 \end{eqnarray*}
which implies
 \begin{eqnarray*}
   \WW^N_{\| \cdot \|_{\dot H^{-s}}^2} (f) 
   &=&
   \frac{1}{N} \, \left( \int_{|\xi| \le 1} \frac{(1 - |\hat f(\xi)
       |^2)}{|\xi|^{2 \, s}} 
     \, {\rm d}\xi +  \int_{|\xi| \ge 1} \frac{(1 - |\hat f(\xi)
       |^2)}{|\xi|^{2 \, s}} 
     \, {\rm d}\xi \right)
   \\
   &=&
   \frac{1}{N} \, \left( \int_{|\xi| \le 1} \frac{M_2}{|\xi|^{2 \,
         (s-1) }} \, 
     {\rm d}\xi +  \int_{|\xi| \ge 1} \frac{1}{|\xi|^{2 \, s}} \, {\rm d}\xi
   \right) \le \frac{C}{N}
 \end{eqnarray*}
 from which (\ref{estim:RachevHdotk}) follows. 
 
 \medskip \noindent{\sl Step~2.} Let us now prove
 \eqref{estim:NousW1}.
 
 We use first (\ref{estim:W1H-k}) in order to get 
 \begin{eqnarray*}
   \WW^N_{W_1} (f)  
   &=& \int_{\R^{dN}} \left[ \mu^N_V - f \right]^*_1 \, {\rm d}f^{\otimes N}(V) \\
   &\le& C \,  \int_{\R^{dN}} \left( M_{k+1} (f) + M_{k+1}\left( \mu^N
       _V \right) \right)^{\frac{d}{d+2ks}} \, 
   \left(  \left\|\mu^N_V - f \right\|_{\dot H^{-s}}^2 \right)^{\frac{k}{d+2ks}} \, 
   {\rm d}f^{\otimes N}(V).   
 \end{eqnarray*}

We then perform a H\"older inequality with exponents 
\[
p = \frac{d+2ks}{k}, \quad p' = \frac{d+2ks}{d+k(2s-1)}
\]
and get 
\begin{multline*}
  \WW^N_{W_1} (f) \le C \, \left( \int_{\R^{dN}} \left( M_{k+1} (f) +
      M_{k+1}\left( \mu^N _V \right) \right)^{\frac{d}{d+k(2s-1)}} \,
    {\rm d}f^{\otimes N}(V) \right)^{\frac{d+k(2s-1)}{d+2ks}} \times
  \\ \left( \int_{\R^{dN}} \left\|\mu^N_V - f \right\|_{\dot H^{-s}}^2
    \, {\rm d}f^{\otimes N}(V) \right)^{\frac{k}{d+2ks}}.
\end{multline*}

Since 
\begin{multline*}
  \int_{\R^{dN}} \left( M_{k+1} (f) + M_{k+1}\left( \mu^N _V \right)
  \right)^{\frac{d}{d+k(2s-1)}} \, {\rm d}f^{\otimes N}(V) \\ \le \int_{\R^{dN}}
  \left( M_{k+1} (f) + M_{k+1}\left( \mu^N _V \right) \right) \,
  {\rm d}f^{\otimes N}(V) \\
  \le M_{k+1}(f) + \int_{\R^{dN}} M_{k+1}\left( \mu^N _V \right) \,
  {\rm d}f^{\otimes N}(V)  \\ \le M_{k+1}(f) + \frac1N \, \sum_{i=1} ^N
  \int_{\R^{dN}} \left\langle v_i \right\rangle^{k+1} \, {\rm d}f^{\otimes
    N}(V) \le 2 \, M_{k+1}(f)
\end{multline*}
we deduce by using \eqref{estim:RachevHdotk} that 
\[
\WW^N_{W_1} (f)  \le \frac{C(f,k)}{N^{\frac{k}{d+2ks}}}
\]
where the constant $C(f,k)$ depends on the $(k+1)$-th moment of $f$. 

This allows to conclude the proof of (\ref{estim:NousW1}) since 
\begin{itemize}
\item if $d=1$ we can take $s=1$ in \eqref{estim:W1H-k} and then $k$
  large enough so that $k/(d+2ks) = 2+\eta$ with some $\eta >0$ as
  small as wanted,
\item if $d \ge 2$ we take $s$ close to $d/2$ and then $k$ large
  enough so that  $k/(d+2ks) = 1/d+\eta$ with some $\eta >0$ as
  small as wanted.
\end{itemize}

Then the estimate \eqref{estim:NousW2} follows from
\eqref{estim:NousW1} with the help of \eqref{estim:W1Wq} and a
H\"older inequality.
\end{proof}

\subsection{Chaotic initial data with prescribed energy and momentum}
\label{subsec:pN0}

In many aspects, the simplest $N$-particle initial data is the
sequence of tensorized initial data $f^{\otimes N}$, $N \ge 1$, where
$f$ is a $1$-particle distribution. This means perfect chaoticity. On
the other hand it has a drawback: since in all applications we shall
use pointwise bounds on the energy of the $N$-particle system (and
also sometimes pointwise higher moment bound as in {\bf (A1)-(iii)}),
this implies for this kind of initial data that $f$ has to be
compactly supported. There is another ``natural'' choice of initial
data, by restricting it to one of the subspaces left invariant by the
dynamics as defined in \eqref{BolSphere}. Without loss of generality
we shall always set $\mathcal M =0$ in this formula in the sequel, and
therefore we shall consider
\begin{equation}
\label{def:SSN-EM} 
\mathcal S^N(\EE) : = \left\{ V \in \R^{dN} \ \mbox{ s. t. } \ \frac1N \,
  \sum_{i=1} ^N |v_i|^2 = \EE, \quad \frac1N \sum_{i=1} ^N v_i = 0 \right\}.
\end{equation}

The drawback is now that an initial data on $\mathcal S^N(\EE)$
\emph{cannot} be perfectly tensorized, and some additional chaoticity
error is paid at initial time. However an advantage of this viewpoint
is that it is simpler to study the asymptotic behavior of both the
$N$-particle and the limit mean-field equation in this setting.
Moreover it has historical value since this approach was introduced by
Kac (see the discussion in \cite[Section 5 ``Distributions having
Boltzmann's property'']{Kac1956}), although in his case there was only
one conservation law, namely the energy, and therefore $\mathcal S^N(\EE)$
was replaced by $\Sp^{N-1}(\sqrt{N})$. Finally in the case of hard
spheres it is easy to check that the relaxation rate degenerates as
$\EE \to 0$ both for the $N$-particle system and the limit equation,
but uniform in time chaoticity can be achieved by avoiding the zero
energy distributions thanks to the restriction to $\SS^N(\EE)$ with
$\EE>0$.

We shall present some results on the construction of chaotic initial
data on $\mathcal S^N(\EE)$, whose proofs are mostly extensions of the
precise statements and estimates recently established in \cite{CCLLV}
on this issue in the setting of Kac on $\Sp^{N-1}(\sqrt N)$. We refer to the
work \cite{kleber} where an extensive study and precise
computations of rates are performed. Without loss of generality we
only consider the case $\MM=0$ for simplicity.

\begin{lem}\label{lem:Bproperty}
Consider $\EE>0$ and an initial data 
\[
f_0 \in \PP_4\left(\R^d\right) \cap L^\infty\left(\R^d\right) 
\]
which fulfills some moment conditions 
\[
M_{m_{\GG_1}}(f) = \left\langle f, m_{\GG_1} \right\rangle < + \infty,
\quad 
M_{m_{\GG_3}}(f) = \left\langle f, m_{\GG_3} \right\rangle < + \infty
\]
for some positive radially symmetric increasing weight functions
$m_{\GG_1}$ and $m_{\GG_3}$, and 
\[
\int_{\R^d} v \, f_0(v) \, {\rm d}v = 0, \qquad 
\int_{\R^d} |v|^2 \, f_0(v) \, {\rm d}v = \EE.
\]

Let us define a non-decreasing sequence $(\alpha_N)_{N \ge 1}$ as
follows:
\begin{itemize}
\item If $f_0$ has compact support 
\begin{equation}\label{cpctsupp}
\mbox{{\em Supp}}f_0 \subset \left\{ v \in \R^d \, , \ |v| \le
  A \right\}
\end{equation}
for some $A>0$, then 
\[
\forall \, N \ge 1, \quad \alpha_N := m_{\GG_3}(A).
\]
\item If $f_0$ has non-compact support, then $(\alpha_N)_{N \ge 1}$
  can be any non-decreasing sequence such that 
\[
\lim_{N \to \infty} \alpha_N = +\infty
\]
(note in particular that this sequence can grow as slowly as wanted). 
\end{itemize}
\medskip

Then there exists 
\[
f^N_0 \in P(\R^{dN}), \quad N \ge 1, 
\]
such that
  \begin{itemize}
  \item[(i)] The sequence $(f^N_0)_{N \ge 1}$ is
    $f_0$-chaotic. 
  \item[(ii)] Its support satisfies $\mbox{{\em Supp}} \, f^N_0 \subset \mathcal
    S^N(\EE)$. 
  \item[(iii)] It satisfies the following integral moment bound based on
  $m_1$:
    \[ 
    \forall \, N \ge 0, \quad \left\langle f^N_0, M_{m_{\GG_1}}^N
    \right\rangle \le C \, \left\langle f_0, m_{\GG_1}\right\rangle
    \]
    where the constant $C>0$ depends on $M^{N\!L}_{0,m_1}$.
   \item[(iv)] It satisfies the following ``support moment bound'':
    \[
    \mbox{{\em Supp}} \,f^N_0 \subset 
    \left\{ V \in \R^{dN};
      \,\, M^N_{m_{\GG_3}} (V) \le \alpha_N \right\}.
    \]
          
      \item[(v)] It satisfies a uniform relative entropy bound
      \begin{equation*} 
      \frac{H(f^N_0 | \gamma^N)}{N} \le C,
      \end{equation*}
      for some constant $C > 0$ (see \eqref{def:RelativH} for notations). 
    \item[(vi)] If furthermore the Fisher information associated to
      $f_0$ is bounded, that is $I(f_0) < \infty$ (see
      \eqref{def:infoFisher} for notations), then $f^N_0$ can be built
      in such a way that it satisfies a uniform relative Fisher
      information bound 
      \begin{equation*}
      I(f^N_0 | \gamma^N) :={1 \over N} \int_{{\SS}^N(\EE)} \left|\nabla
      \ln {{\rm d} f^N_0 \over {\rm d}\gamma^N} \right|^2 \, {\rm d}f^N_0 \le C, 
      \end{equation*}
      for some constant $C > 0$, where the gradient in this formula
      stands for the Riemannian gradient on the manifold $\SS^N(\EE)$.
\end{itemize}
\end{lem} 

\begin{proof}[Proof of Lemma~\ref{lem:Bproperty}]
We aim at defining our initial data $f^N _0$ by conditioning the
tensorized initial data $f_0 ^{\otimes N}$ to $\mathcal S^N(\EE)$: 
\[
f^N_0(V) = \left[f_0^{\otimes N}\right]_{\mathcal S^N(\EE)} := \left( 
{ \prod_{j=1}^N f_0 (v_j) \over F_N\left(\sqrt{N}\right)} \right)
\Bigg|_{\mathcal S^N(\EE)}
\]
with
\[
F^N(r) := \int_{ \mathbb{S}^{dN-1} (r) \cap \{ \sum_{i=1} ^N v_i =0\}}
\left( \prod_{j=1}^N f_0 (v_j) \right) \, {\rm d}\omega.
\]
Such a construction obviously satisfies (ii).

It is proved similarly as in \cite{CCLLV} (see for instance Theorem~9
in this reference) that this conditioned measure is well-defined, and
that it is $f_0$-chaotic, which proves (i).

\begin{rem}
  Among many interesting intermediate steps and other results, it is
  also proved in \cite{CCLLV} the following estimate: assume for
  simplicity that $d=1$ and that $f_0$ has energy $1$, then the
  function
\[
\bar F^N(r) := \frac{F^N(r)}{\gamma^N(r)}
\]
is asymptotically divergent except for $r = \sqrt N$, for which 
\[
\bar F^N\left(\sqrt N\right) \sim_{N \to +\infty} \frac{\sqrt
  2}{\Sigma}
\]
with 
\[
\Sigma = \sqrt{\int_{\R} \left( v^2 - 1\right)^2 \, {\rm d}f(v)}.
\]
(In fact this result was sketched by Kac \cite{Kac1956} but the proof
is made more precise in \cite{CCLLV}). 

This shows in particular that the sequence of chaotic initial data
$f_0^{\otimes N}$, $N \ge 1$, as considered many times in the sequel,
asymptotically concentrates on the Boltzmann spheres $\SS^N(\EE)$.
This manifestation of the central limit theorem explains why
the construction of Kac (to condition to a given energy sphere) is
very natural. It also enlightens why it is possible to expect the kind
of uniform in time propagation of chaos results that we shall prove in
the next sections for such chaotic initial data. 
\end{rem}

Point (iii) is just a consequence of the chaoticity with the test
function $M^N _{m_{\GG_1}}$ (actually an easy truncation and passage
to the limit proceedure is needed in full rigor). 

Concerning point (iv), first if $f_0$ is compactly
supported~\eqref{cpctsupp} we deduce that 
\[
\mbox{Supp} \, f^N_0 \subset
\left\{ V \in \R^{d N}, \,\, M^N_{m_{\GG_3}} (V) \le m_{\GG_3}(A)
  \right\}
\]
and (iv) holds. 

In the non compactly supported case, for any increasing sequence
$(A_k)_{k \ge 1}$ of positive reals (with $A_0$ big enough for the
following to be well-defined) we define
\[
f_{0,k} := \frac{f_0 \, {\bf 1}_{|v| \le A_k}}{f_0\left( \{|v| \le
    A_k\}\right)}.
\]
Using the previous we know that 
\[
f^N_{0,k} := \left[ f_{0,k} ^{\otimes N} \right]_{\SS^N(\EE)}
\]
(conditioning to $\SS^N(\EE)$) is $f_{0,k}$-chaotic. Conditions (ii) and
(iii) will therefore be immediately satisfied. 

We now want to choose a sequence $k_N \to \infty$ such that (iv) is
satisfied and at the same proving chaoticity \emph{towards $f_0$}. 
It is clear that 
\[
\mbox{Supp} \,f^N_{0,k} \subset \left\{ V \in \R^{dN};
\,\, M^N_{m_3} (V) \le m_3(A_k) \right\}.
\]
For any given sequence $(\alpha_N)$ which tends to infinity, we define
$k_N$ in such a way that $m_3(A_{k_N}) \le \alpha_N$ so that $k_N \to
\infty$ when $N \to \infty$. The chaos property is equivalent to the
weak convergence of the $2$-marginal, which can be expressed in
Wasserstein distance for instance: 
\[
W_1\left( \left( f^N_{0,k} \right)_2, f_0 ^{\otimes 2} \right) \le 
W_1\left( \left( f^N_{0,k} \right)_2, f_{0,k} ^{\otimes 2}\right) 
+ W_1\left( f_{0,k} ^{\otimes 2} , f_0 ^{\otimes 2} \right).
\] 
The last term of the RHS converges to zero only depending on $k \to
0$, while the first term in the RHS converges to zero for fixed $k$ as
$N \to 0$ from the previous part of the proof. Therefore, maybe at the
price of a slower increasing sequence $k_N$ we can have both the
support moment condition (iv) and 
\[
W_1\left( \left( f^N_{0,k_N} \right)_2, f_0 ^{\otimes 2} \right)
\xrightarrow[]{N \to 0} 0
\]
which shows the chaoticity and concludes the proof. 

\smallskip
For the proof of (v) and (vi) we refer to \cite{kleber}. 
\end{proof}

\begin{rems}
\begin{enumerate}
\item We note that if one only wants to get rid of the compact support
  requirement in $f_0$ (used for deriving the support bounds on
  $f_0^N$ on the energy and $m_{\GG_3}$), and not necessarily to
  prescribe a given energy, another strategy could have been to simply
  perform the cutoff in the end of the previous proof. In principle it
  could allow to get better information on the rate of
  convergence. However a drawback of this approach is that, in the
  absence of conditioning to an energy sphere, the bound on the
  support of the energy of $f_0 ^N$ shall grow with $N$. In our
  applications it induces a growth in $N$ of the moment bounds that we
  prove \emph{along time} on the $N$-particle system. This growth
  should be matched by the decay of the scheme and a precise optimized
  balance could be searched for. We do not pursue this line of research.
\item Observe that the process of conditioning on the energy sphere is
  obviously compatible with the equilibrium states in the following
  sense: if one denotes by $\gamma$ a centered Gaussian equilibrium of
  the limit equation with energy $\EE$, then
\[
\gamma^N (V) := \left[\gamma^{\otimes N}\right]_{\mathcal S^N(\EE)}
\]
is the uniform measure on $\mathcal S^N(\EE)$, i.e. an equilibrium of the
$N$-particle system.
\end{enumerate}
\end{rems}

Let us also state a refinement of the previous lemma which is needed
for the applications.

\begin{lem}\label{lem:Bpropertybis}
We use the same setting and assumptions as in Lemma~\ref{lem:Bproperty}. 

Then the sequence $(f^N _0)$, $N \ge 1$ of Lemma~\ref{lem:Bproperty} satisfies the more precise
chaoticity estimate: 
    \begin{equation}\label{W1Kac}
      \WW_{W_1} \left( \pi^N _P \left(f_0 ^N\right), f_0 \right) = 
      \int_{\R^{d N}} W_1 \left( \mu^N _V, f_0\right) \, {\rm d}f^N_0(V) 
      = \int_{\SS^N(\EE)} W_1 \left( \mu^N _V, f_0\right) \, {\rm d}f^N_0(V)
      \xrightarrow[]{N \to +\infty} 0
    \end{equation}
with explicit polynomial rate.
  \end{lem}

  \begin{rems}
    \begin{enumerate} 
\item As an easy consequence, for any function $\Theta_a(x)$ such that
\[
\forall \, a >0, \quad \Theta_a(x) \xrightarrow[]{x \to +\infty} 0. 
\] 
(where the parameter $a$ should be thought as keeping track of
dependency of this functional on moment estimates on the
distributions), we have by a diagonal extraction process 
  \begin{equation}\label{W1KacTheta}
      \Theta_{a^N} \left( \WW_{W_1} \left( \pi^N _P \left(f_0
            ^N\right), f_0 \right) \right))  \xrightarrow[]{N \to +\infty} 0 
    \end{equation}
    for some explicit rate in terms of \eqref{W1Kac} and $\Theta$,
    with $a^N = \max \left\{ \alpha_N \, ; \, \left\langle f_0,
        m_{\GG_3} \right\rangle \right\}$.
      \item Using Lemma~\ref{lem:ComparDistances} it would be immediate to
    extend the previous statement to the other weak measure distances
    we have discussed so far.
  \end{enumerate}
\end{rems}

\begin{proof}[Proof of Lemma~\ref{lem:Bpropertybis}]
  Let us give two proofs. The first one is non-constructive and
  inspired from \cite{S6}. 
  First, thanks to the well-known result \cite[Proposition 2.2]{S6}
  and the fact that the sequence $f^N_0$ constructed in
  Lemma~\ref{lem:Bproperty} is $f_0$-chaotic, we deduce that
\[
\pi^N_P f^N_0 \rightharpoonup \delta_{f_0} \ \mbox{ in } \
P\left(P\left(\R^d\right)\right)
\]
(which means convergence when testing against functions in
$C(P(\R^d))$).

Next, thanks to \cite[Theorem 7.12]{VillaniTOT}, \eqref{W1Kac}
boils down to prove the tightness estimate
\begin{eqnarray}\label{eq:tightWW1}
  \lim_{R \to \infty} \sup_{N \in \N^*} \int_{W_1 \left(\rho,f_0 \right) \ge R}
  W_1\left(\rho,f_0\right) \, {\rm d}( \pi^N_P f^N_0)(\rho) = 0.
\end{eqnarray}

Let us prove that it easily follows from the following bound 
\begin{eqnarray}\label{eq:bddWW1}
  \sup_{N \in \N^*} \int_{E^N} \left( W_1\left(\mu^N_V,f_0\right)
  \right)^2 \, 
{\rm d}f^N_0(V) < \infty.
\end{eqnarray}
Indeed \eqref{eq:bddWW1} implies that uniformly in $N \ge 1$ 
\begin{multline*}
\int_{W_1 \left(\rho,f_0 \right) \ge R}
  W_1\left(\rho,f_0\right) \, {\rm d} (\pi^N_P f^N_0)(\rho) = 
\int_{\left\{V \in E^N \, \mbox{ {\scriptsize s.t.} } \, W_1 \left(\mu^N _V,f_0 \right) \ge R\right\}}
  W_1\left(\mu^N _0,f_0\right) \, {\rm d} f^N_0(V)  \\ \le 
\frac1R \, \int_{E^N}
  \left( W_1\left(\mu^N _V,f_0\right)\right)^2 \, {\rm d}f^N_0(V)
   \le \frac{C}{R}
  \xrightarrow[]{R \to \infty} 0
\end{multline*}
which concludes the proof of \eqref{eq:tightWW1}. 

In order to show \eqref{eq:bddWW1}, we infer that from
\cite[Theorem 7.10]{VillaniTOT}
\begin{eqnarray*}
 \left( W_1\left(\mu^N_V,f_0\right)\right)^2 
 &\le& \left\| \mu^N_V - f_0 \right\|_{M^1_1}^2 
\le 2 \left\| \mu^N_V  \right\|_{M^1_1}^2 + 2 \, \left\| f_0 \right\|_{M^1_1}^2  
\\
&\le& 2 \, \left( M^N_1(V) \right)^2 + 2 \, \left\| f_0 \right\|_{M^1_1}^2 
\le  2 \, M^N_2(V)+ 2 \, \left\| f_0 \right\|_{M^1_1}^2,
\end{eqnarray*}
which implies 
$$
\int_{E^N} \left( W_1 \left(\mu^N_V,f_0\right)\right)^2 \, {\rm d}f^N_0(V)
\le 2 \, \left\| f_0 \right\|_{M^1_1}^2 + 2 \, \left\langle f^N_0,M_2 ^N
\right\rangle,
$$
which, together with (ii) in Lemma~\ref{lem:Bproperty}, ends the proof
of \eqref{eq:bddWW1} and then of \eqref{W1Kac}.

Let us now give an alternative explicit argument, even if less
self-contained. From \cite[Theorem~3, (i)]{kleber} we deduce that 
\begin{equation*}
  \forall \, \ell \ge 1, \quad W_1 \left( \Pi_\ell f^N_0, f_0 ^{\otimes
      \ell}\right) \le \frac{C}{\sqrt N}
\end{equation*}
for some explicit constant $C>0$ uniform in $\ell$ and $N$. Then we
use the \cite[Theorem~2.4]{hm} to deduce that 
\begin{equation*}
  \int_{E^N} \left( W_1 \left(\mu^N_V,f_0\right)\right)^2 \, {\rm d}f^N_0(V)
  \le C \, \left( W_1 \left( \Pi_2 f^N_0, f_0 ^{\otimes 2}\right) + \frac1N \right)^{C'}
\end{equation*}
for some explicit constants $C,C'>0$ depending on the energy bound on
$f_0$, which concludes the proof.
\end{proof}


\section{True Maxwell molecules}
\label{sec:BddBoltzmann}
\setcounter{equation}{0}
\setcounter{theo}{0}


\subsection{The model}
\label{sec:modelEBbounded}

Let us consider $E = \R^d$, $d \ge 2$, and a $N$-particle system
undergoing space homogeneous random Boltzmann type collisions
according to a collision kernel 
\[
B(z,\cos \theta) = \Gamma(z) \, b (\cos \theta)
\]
(see Subsection~\ref{sec:introEB}).
More precisely, given a pre-collisional system of velocity variables
\[
V = (v_1, \dots, v_N) \in E^N = (\R^d)^N,
\]
the stochastic process is:
\begin{itemize}
\item[(i)] for any $i'\neq j'$, pick 
  a random time $ T_{\Gamma(|v_{i'}- v_{j'}|)}$ of collision
  accordingly to an exponential law of parameter $\Gamma(|v_{i'}-
  v_{j'}|)$, and then choose the collision time $T_1$ and the
  colliding pair $(v_i,v_j)$ (which is a.s. well-defined) in such a
  way that
  $$
  T_1 = T_{\Gamma(|v_i - v_j|)} := \min_{1 \le i' \neq j' \le N} T_{\Gamma(|v_{i'}-
    v_{j'}|)};
  $$
\item[(ii)] then pick a random unit vector $\sigma \in S^{d-1}$
  according to the law $b(\cos \theta_{ij})$,
  where 
  \[
  \cos \theta_{ij} = \sigma \cdot (v_j-v_i)/|v_j-v_i|;
  \]
\item[(iii)] the new state after collision at time $T_1$ becomes
  \begin{equation}\label{eq:def-vit-coll}
  V^*_{ij} = (v_1, \dots, v^*_i, \dots, v^*_j, \dots , v_N),
\end{equation}
where only velocities labelled $i$ and $j$ have changed, according
  to the rotation
  \begin{equation}\label{vprimvprim*}
    \quad\quad   v^*_i = {v_i + v_j \over 2} 
    + {|v_i - v_j| \, \sigma \over 2}, \quad
    v^*_j= {v_i + v_j \over 2} - {|v_i - v_j| \, \sigma \over 2}.
  \end{equation}
\end{itemize}

\smallskip The associated Markov process 
\[
(\VV_t ^N)_{t \ge 0} \ \mbox{ on } \ (\R^d)^N
\]
is then built by iterating the above construction. 
\medskip

After rescaling time $t \to t/N$ in order that the number of
interactions is of order $\OO(1)$ on finite time interval (see
\cite{spohn}) we denote by $f^N_t$ the law of $\VV_t ^N$ and $S^N_t$ the
associated semigroup. We recall that $G^N$ and $T^N_t$ respectively
denotes the dual generator and dual semigroup, as in the previous
abstract construction.

\medskip
The so-called {\em Master equation} on the law $f^N_t$ is given in
dual form by
\begin{equation}\label{eq:BoltzBddKolmo}
  \partial_t \left\langle f^N_t,\varphi \right\rangle = 
  \left\langle f^N_t, G^N \varphi \right\rangle 
\end{equation}
with 
\begin{equation}\label{defBoltzBddGN}
  \left(G^N\varphi\right) (V) =  \frac{1}{N} \, 
  \sum_{1 \le i < j \le N}  \Gamma\left(|v_i-v_j|\right)
  \, \int_{\mathbb{S}^{d-1}} b(\cos\theta_{ij}) \, \left[\varphi^*_{ij} -
    \varphi\right] \, {\rm d}\sigma
\end{equation}
\[
\mbox{where } \ \varphi^*_{ij}= \varphi \left(V^*_{ij} \right) \ \mbox{ and } \
\varphi = \varphi(V) \in C_b\left(\R^{Nd}\right).
\]

This collision process is invariant under velocities permutations and
satisfies the microscopic conservations of momentum and energy at any
collision time
$$
\sum_{j=1} ^N v^*_j = \sum_{j=1} ^N v_j \quad \mbox{ and } \quad
|V^*|^2 = \sum_{j=1} ^N |v^* _j|^2 = \sum_{j=1} ^N |v_j|^2 = |V|^2.
$$
As a consequence, for any symmetric initial law $f_0^N \in
P_{\mbox{{\scriptsize sym}}}((\R^{d})^N)$ the law $f_t^N$ at later times
is also a symmetric probability. Moreover the evolution conserves
momentum and energy: for any measurable function $ \phi : \R \to \R_+$
\[
\forall \, \alpha = 1, \dots, d, \quad \int_{\R^{dN}} \phi \left(
  \sum_{j=1}^N v_{j,\alpha} \right) \, {\rm d} f_t^N (v) =
\int_{\R^{dN}} \phi \left( \sum_{j=1}^N v_{j,\alpha} \right) \, {\rm
  d} f_0^N (v),
\]
where $(v_{j,\alpha})_{1 \le \alpha \le d}$ denote the components of
$v_j \in \R^d$, and
\begin{equation}\label{eq:preconservationE}
  \int_{\R^{dN}}\phi( |V|^2 ) \, {\rm d}f_t^N (v) = 
  \int_{\R^{dN}} \phi( |V|^2 )  \, {\rm d}f_0^N (v),
\end{equation}
(equalities between possibly infinite non-negative quantities).

\medskip The (expected) limit nonlinear homogeneous Boltzmann
equation is defined by \eqref{el}, \eqref{eq:collop},
\eqref{eq:rel:vit}. The equation generates a nonlinear semigroup
\[
S^{N\! L}_t (f_0) := f_t \ \mbox{ for any } \ f_0 \in \PP_2\left(\R^d\right)
\]
where $\PP_2(\R^d)$ denotes the space of probabilities with bounded
second moment. 

Concerning the Cauchy theory for the limit Boltzmann equation:
\begin{itemize}
\item In the case {\bf (GMM)} (Maxwell molecules with angular cutoff),
  see equation~\eqref{model:gmm} in Subsection~\ref{sec:introEB}, we
  refer to \cite{T1};
\item In the case {\bf (tMM)} (true Maxwell molecules without angular
  cutoff), see equation~\eqref{model:tmm} in
  Subsection~\ref{sec:introEB}, we refer to \cite{TV};
\item In the case {\bf (HS)} of hard spheres, see
  equation~\eqref{model:hs} in Subsection~\ref{sec:introEB}, we refer
  to \cite{MW99} ($L^1(\R^d)$ theory) and \cite{EM,Fo-Mo,Lu-Mouhot}
  ($P(\R^d)$ theory).
\end{itemize}

For these solutions, one has the conservation of momentum and energy
$$
\forall \, t \ge 0, \quad \int_{\R^{d}} v \, {\rm d}f_t (v) = \int_{\R^{d}}v
\, {\rm d}f_0 (v), \quad \int_{\R^{d}} |v|^2 \, {\rm d}f_t (v) =
\int_{\R^{d}}|v|^2 \, {\rm d}f_0 (v).
$$

Observe that the change of variable 
\[
\sigma \in \mathbb{S}^{d-1} \mapsto - \sigma \in \mathbb{S}^{d-1}
\]
maps the domain 
\[
\theta \in [-\pi,\pi/2] \cap [\pi/2,\pi] \ \mbox{ in } \ \theta \in
[-\pi/2,\pi/2].
\]
Therefore without restriction we can consider, for the limit
equation as well as the $N$-particle system, kernel function $b$ such
that $\mbox{Supp} \, b \subset [0,1]$.
We still denote by $b$ the symmetrized version of $b$ by a slight
abuse of notation.

In this section we aim at considering the case of the {\sl Maxwell
  molecules kernel}. We shall indeed make the following general
assumption:
\begin{equation}\label{Maxwelltrue}
  \left\{ 
    \begin{array}{l}
      \Gamma \equiv 1, \quad b \in
      L^\infty_{\mbox{{\tiny loc}}}([0,1)) \vspace{0.3cm} \\ \displaystyle 
      \forall \, \alpha > 0, \quad  
      C_{\alpha}(b) := \int_{\mathbb{S}^{d-1}}  b(\cos \theta)\, \left( 1-\cos
        \theta \right) ^{\frac14 + \alpha} \, 
      {\rm d}\sigma < \infty. 
\end{array}
\right.
\end{equation}

Let us show that Maxwell molecules model~\eqref{model:tmm} enters this
general framework.  Indeed for any positive real function $\psi$ and
any given vector $u \in \R^d$ we have
\[
\int_{\mathbb S^{d-1}} \psi (\hat u \cdot \sigma) \, {\rm d}\sigma =
|\mathbb{S}^{d-2}| \, \int_0^\pi \psi (\cos\theta) \, \sin^{d-2} \,
\theta \, {\rm d}\theta.
\]
Therefore the model~\eqref{model:tmm} satisfies (in dimension
$d=3$) 
\[
b(z) \sim K \, (1-z)^{-5/4} \ \mbox{ as } \ z \to 1,
\]
which hence fulfills (\ref{Maxwelltrue}). This assumption also
trivially includes the Grad's cutoff Maxwell molecules
model~\eqref{model:gmm}. 

\subsection{Statement of the results}
\label{sec:resultEBbounded} 

Our main propagation of chaos estimate result for this model then
states as follows:

\begin{theo}[Maxwell molecules detailed chaos estimates]\label{theo:tMM}
  Assume that the collision kernel $b$ satisfies \eqref{Maxwelltrue}.
  Let us consider a family of $N$-particle initial conditions $f_0 ^N
  \in P_{\mbox{{\scriptsize {\em sym}}}}((\R^d)^N)$, $N \ge 2$, and
  the associated $N$-particle system dynamics
$$
 f^N _t = S^N _t \left(f_0^N \right).
$$
 Let us also consider a $1$-particle initial distribution $f_0 \in \PP_2(\R^d)$  
with zero momentum and energy $\EE \in (0,+\infty)$
  $$
  \int_{\R^d} v \, f_0 \, {\rm d}v = 0, \quad \EE := \int_{\R^d} f_0
  \, {\rm d}v \in (0,\infty),
  $$ 
and the associated solution
  \[
  f_t = S^{N \! L} _t \left(f_0\right)
  \]
 of the limit mean-field equation.

 Then for any $\delta \in (0,1)$, for any $\ell \in \N^*$ and for any
  \[ 
  \varphi = \varphi_1 \otimes \varphi_2 \otimes \dots \otimes \,
  \varphi_\ell \in \FF^{\otimes\ell}, \quad \varphi_i \in \FF, \ i
  =1,\dots, \ell,
  \] 
  where $\FF$ shall be specified below, we have, for some constant
  $C_\delta>0$ (possibly blowing up as $\delta \to 0$) depending only
  on $\delta$, on the collision kernel,  on the size of the support
  and on some moments of $f_0$:

\begin{enumerate}
\item[(i)] \underline{Cases {\bf (GMM)} and {\bf (tMM)}}: Assume that $f^N _0 = f_0 ^{\otimes N}$ 
is a  tensorized $N$-particle initial datum and  that $f_0$ has compact support, and
  take
  $$
  \FF := \left\{ \varphi : \R^d \to \R; \,\, \| \varphi \|_\FF :=
    \int_{\R^d} (1 + |\xi|^4) \, |\hat \varphi (\xi)|\, {\rm d}\xi <
    \infty \right\}.
  $$
  Then we have 
  \begin{eqnarray} \label{eq:cvgBddBE} 
  && \forall \, N \ge 2 \ell, \quad \sup_{t \ge 0} \left| \left
        \langle \left( S^N_t(f_0 ^N) - \left( S^{N\! L}_t(f_0)
          \right)^{\otimes N} \right), \varphi \right\rangle \right|
  \\ \nonumber
  &&\le 
 C_\delta \, \Bigg[ \ell^2 \, \frac{\|\varphi\|_\infty}{N} 
  + 
    {\ell^2 \over N^{1-\delta}} \, 
  \|\varphi\|_{\FF ^2 \otimes (L^\infty)^{\ell-2}} 
  + \ell \, \, \|\varphi\|_{W^{1,\infty} \otimes (L^\infty)^{\ell-1}} \, 
  \WW^N_{W_2} ( f_0)   \Bigg].
\end{eqnarray}

We deduce the following rate of convergence as $N$ goes infinity
by using \eqref{estim:NousW2}-\eqref{estim:Rachev2}:
\begin{equation*}
\sup_{t \ge 0} \left| \left
        \langle \left( S^N_t(f_0 ^N) - \left( S^{N\! L}_t(f_0)
          \right)^{\otimes N} \right), \varphi \right\rangle \right|
\le {C _\delta \, \ell^2 \over N^{\kappa(d,\delta)}} \,
   \| \varphi \|_{\FF^{\otimes\ell}}
 \end{equation*}
 with 
    $$
    \kappa(d,\delta) := 
    \left\{
      \begin{array}{l} 
        {1 \over 4} - \delta \,\,\,\,\hbox{if}\,\,\,\, d \le 2, 
    \vspace{0.3cm} \\ {1 \over 6} - \delta \,\,\,\,\hbox{if}\,\,\,\,  d = 3, 
    \vspace{0.3cm} \\ {1 \over d+4} \,\,\,\,\hbox{if}\,\,\,\, d \ge 4.
  \end{array}
  \right.
  $$  
 This proves the propagation of chaos, uniformly in time and
  with explicit polynomial rates.
\smallskip

\item[(ii)] \underline{Case {\bf (GMM)} with optimal rate for finite
    time}: On a finite time interval $[0,T]$, the following variant is
  available: consider tensorized initial data $f^N _0 = f_0 ^{\otimes
    N}$ for the $N$-particle system and assume that $f_0$ has compact
  support, and take $\FF = H^s$ with $s > d/2$ high enough. Then we
  have
\begin{eqnarray*} 
  &&\sup_{0 \le t \le T} \left| \left
        \langle \left( S^N_t(f_0 ^N) - \left( S^{N\! L}_t(f_0)
          \right)^{\otimes N} \right), \varphi \right\rangle \right|
  \\ \nonumber
  &&\le 
 C_\delta \, \Bigg[ \ell^2 \, \frac{\|\varphi\|_\infty}{N} 
  + C^N_{T,4} \, {C_{\delta,\infty}^\infty \over N^{1-\delta}} \, \ell^2 \, 
  \|\varphi\|_{\FF ^2 \otimes (L^\infty)^{\ell-2}} 
    + \ell \, \, \|\varphi\|_{H^s \otimes (L^\infty)^{\ell-1}} \, 
  \WW^N_{H^{-s}} ( f_0)   \Bigg].
 \end{eqnarray*}

By using \eqref{estim:RachevHdotk}, this proves
\begin{equation*}
\sup_{t \in [0,T]} \left| \left
        \langle \left( S^N_t(f_0 ^N) - \left( S^{N\! L}_t(f_0)
          \right)^{\otimes N} \right), \varphi \right\rangle \right|
\le \ell^2 \, {C_{\delta,T} \over \sqrt N} \,
   \| \varphi \|_{\FF^{\otimes\ell}}
\end{equation*}
with $\FF = H^s$, and where the constant $C_{\delta,T}$ can also
depends on the final time of observation $T$. This proves the
propagation of chaos, on finite time intervals, but with the optimal
rate of the law of large numbers.  \smallskip

\item[(iii)] \underline{Cases {\bf (GMM)} and {\bf (tMM)} conditioned
    to the sphere}: Finally consider $\FF$ as in (i), assume that the
  $1$-particle initial datum $f_0$ belongs to $\PP_6\left(\R^d\right) \cap
  L^\infty\left(\R^d\right)$  and consider the associated $N$-particle initial data $(f_0^N)_{N
    \ge 1}$ constructed in Lemma~\ref{lem:Bproperty} and
  \ref{lem:Bpropertybis} by conditioning to $\SS^N(\EE)$. Then the solution $f^N_t = S^N_t(f_0 ^N)$ has its
  support included in $\SS^N(\EE)$ for all times
  \begin{equation}\label{eq:supportfNtspheres}
    \forall \, t \ge 0, \quad \hbox{{\em Supp}} \, f^N_t \subset \SS^N(\EE) 
  \end{equation}
and we have the estimate
    \begin{eqnarray*}
  &&\quad \sup_{t \ge 0}\left| \left \langle \left( S^N_t(f_0 ^N) - \left(
          S^{N \! L}_t(f_0) \right)^{\otimes N} \right), \varphi 
    \right\rangle \right|  \le
  \\ \nonumber 
  &&\quad 
  C_\delta \, \Bigg[ \ell^2 \, \frac{\|\varphi\|_\infty}{N} 
  + 
    {\ell^2 \over N^{1-\delta}} \, 
  \|\varphi\|_{\FF ^2 \otimes (L^\infty)^{\ell-2}} 
  + \ell \, \, \|\varphi\|_{W^{1,\infty} \otimes (L^\infty)^{\ell-1}} \, 
  \WW^N_{W_2} \left(\pi^N_P f^N_0,\delta_{f_0}\right)\Bigg]
  \end{eqnarray*}
  which goes to zero as $N$ goes infinity with polynomial rate thanks
  to Lemma~\ref{lem:Bpropertybis}. This proves the propagation of
  chaos, uniformly in time and with explicit polynomial rates.
\end{enumerate}
\end{theo}

We now state another version of the propagation of chaos estimate, in
Wasserstein distance, but most importantly which is valid \emph{for
  any number of marginals}, at the price of a possibly worse (but
still constructive) rate. Combined with previous results on the
relaxation of the $N$-particle system we also deduce some estimate of
relaxation to equilibrium \emph{independent of $N$} and, again, for
any number of marginals.

\begin{theo}[Maxwell molecules Wasserstein chaos]
  \label{theo:max-wasserstein}
  We consider the same setting as in Theorem~\ref{theo:tMM}, where the
  initial data are chosen as follows:
  \begin{itemize}
  \item[(a)] either $f_0$ is compactly supported and $f_0 ^N = f_0 ^{\otimes
      N}$,
\smallskip

  \item[(b)] or $ f_0 \in \PP_6 (\R^d) \cap L^\infty(\R^d)$ with
    zero momentum and energy $\EE$, and $f_0 ^N$ is constructed by
    Lemma~\ref{lem:Bproperty} by conditioning to the sphere
    $\SS^N(\EE)$. 
  \end{itemize}

Then we have
\begin{equation}\label{eq:max-wass}
  \forall \, N \ge 1, \ \forall \, 1 \le \ell  \le N, \quad \sup_{t \ge 0} {W_1 \left( \Pi_{\ell} f^N
      _t,  f_t ^{\otimes \ell}  \right) \over \ell} \le \alpha(N)
\end{equation}
for some polynomial rate $\alpha(N) \to 0$ as $N \to \infty$.

Moreover in the case (b) the solution $f^N_t = S^N_t(f_0 ^N)$ has its
support included in $\SS^N(\EE)$ for all times and we have
\begin{equation}\label{eq:max-wass-relax}
  \forall \, N \ge 1, \ \forall \, 1 \le \ell  \le N, \  \forall \, t
  \ge 0, \quad  
    {W_1 \left( \Pi_{\ell} f^N
      _t, \Pi_\ell \left( \gamma^N \right) \right) \over \ell} \le
  \beta(t) 
\end{equation}
for some polynomial rate $\beta(t) \to 0$ as $t \to \infty$, where
$\gamma$ is the centered Gaussian equilibrium with energy $\EE$ and
$\gamma^N$ is the uniform probability measure on $\mathcal S^N
(\EE)$.
\end{theo}

In order to prove Theorem~\ref{theo:tMM}, we shall establish the
assumptions {\bf (A1)-(A2)-(A3)-(A4)-(A5)} of
Theorem~\ref{theo:abstract} with $T < \infty$ or $T=\infty$  for a suitable choice
of functional spaces. The application of the latter theorem then
exactly yields Theorem~\ref{theo:tMM} by following carefully each
constant computed below. Then the proof of
Theorem~\ref{theo:max-wasserstein} will be done in
Subsection~\ref{sec:infchaos}: it is deduced from
Theorem~\ref{theo:tMM} by using Lemma~\ref{lem:ComparDistances}
together with a result from \cite{hm}.

\subsection{Proof of condition (A1)} \label{sec:MaxA1} When the collision kernel
$B$ is bounded the operator $G^N$ is a linear bounded operator on
$C(B_R)$ with $B_R := \{ V \in \R^{dN}; \, |V| \le R \}$ for any $R
\in (0,\infty)$ with an operator norm independent of $R$. As a
consequence, $G^N$ is also well defined and bounded on 
$$
C_{-k,0} ^0 (\R^{dN}) := \left\{ \varphi \in C (\R^{dN}) \quad
  \mbox{s. t. } \quad 
  \frac{\varphi(V)}{|V|^{k}} \to 0 \quad \mbox{as} \quad  |V| \to \infty \right\}
$$
endowed with the norm 
$$
\| \varphi \|_{L^\infty_{-k}} := \sup_{V \in \R^{dN}} | \varphi (V)| \, \langle V
\rangle^{-k}
$$
for any $k \in \R$.  It is also easy (and classical) to verify that
$G^N$ is dissipative in the sense that
\[
\forall \, \varphi \in C_{-k,0} ^0 \left(\R^{dN}\right), \,\, \forall \, \lambda > 0,
\quad \left\|\left(\lambda - G^N\right) \, \phi \right\|_{L^\infty_{-k}} \le \lambda \,
\| \phi \|_{L^\infty_{-k}}.
\]
From the Hille-Yosida theory we deduce that $G^N$ generates a Markov
type semigroup $T^N_t$ on $C_{-k,0} ^0 (\R^{dN})$ and we may also define
$S^N_t$ by duality as a semigroup on $\PP_k(\R^{dN})$.  The nonlinear
semigroup $S^{N\! L}_t$ is also well defined on $\PP_k(\R^d)$, see for
instance \cite{TV,Fo-Mo,EM,Lu-Mouhot}.

\smallskip For the true Maxwell molecules model, the operator $G^N$ is
not bounded anymore and some additional explanations are needed. The
simplest argument is to say that $B$ can be approximated by a
sequence of bounded collision kernels
$$
B_\eps := b_\eps (\cos \theta) \quad \mbox{with} \quad b_\eps \in
L^\infty \quad \mbox{and} \quad b_\eps \nearrow b.
$$
We may then define the associated generator $G^{N,\eps}$, the
associated semigroups $T^{N,\eps}$ on continuous functions and
$S^{N,\eps}_t$ on probabilities, and the nonlinear semigroup
$S^{N\! L,\eps}_t$ on probabilities. We first write estimate
\eqref{eq:cvgBddBE} for any fixes $\eps > 0$. Then since (a) the
right-hand side term in \eqref{eq:cvgBddBE} does not depend on $\eps >
0$ (as a consequence of the estimates established in the proof below),
(b) $S^{N\! L,\eps}(f_0) \wto S^{N\! L}(f_0)$ weakly in $P(\R^d)$, and
(c) $S^{N,\eps}(f_0^N) \wto f^N_t$ weakly in $P(\R^{Nd})$, we can
conclude that \eqref{eq:cvgBddBE} holds for the true Maxwell molecules
model by letting $\eps$ go to $0$. 

Possible other direct arguments (without using approximations) could
be (a) to establish and use stability estimates in Wasserstein
distance of the many-particle equation, or (b) use the {\em core}
$\CC := W^{1,\infty}_{k+2}$ and prove that $\varphi \in
W^{1,\infty}_{k+2}$ implies $G^N \varphi \in C_{k,0}$ (this follows
from an easy decomposition between singular and non-singular angles in
the formula for $G^N$).

\smallskip

Hence the semigroups $T^N _t$ and 
\[
S^N _t = \left( T^N _t \right)^* = T^N _t
\]
are well defined on $C^0 _{-k,0}(\R^{dN})$. Moreover since for
$\varphi \in L^2(\mathcal S^N(\EE))$ we have
\[
\left\langle G^N \varphi, \varphi
\right\rangle_{L^2\left(\mathcal S^N(\EE)\right)} = - {1 \over N} \, \sum_{i,j=
  1}^N 
\int_{\mathcal S^N(\EE)} 
\int_{\mathbb{S}^{d-1}} b(\cos\theta_{ij}) \, \left[\varphi^*_{ij} -
  \varphi\right]^2 \, {\rm d}\sigma \, {\rm d}\gamma^N(v) \le 0, 
\]
it is easily seen by arguing similarly as above that they are
$C_0$-semigroups of contractions on this space $L^2(\mathcal
S^N(\EE))$.


\medskip 
Then it remains to prove bounds on the polynomial moments of the
$N$-particle system. We shall prove the following more general lemma:

\begin{lem}\label{lem:momentsN}
  Consider the collision kernel 
\[
B = |v-v_*|^\gamma \, b(\cos \theta) \ \mbox{ with } \ \gamma = 0
\mbox{ or } 1
\]
and $b \ge 0$ such that
  \[
 C_b :=  \int_0 ^1 b(z) \, (1-z)^2\, {\rm d}z <+\infty.
  \]
  This covers the three cases {\bf (HS)}, {\bf (tMM)} and {\bf (GMM)}.

Assume that the initial datum $f^N_0$ of the $N$-particle system satisfies:
\[
\hbox{{\em Supp}} \, f^N_0 \subset \left\{V \in \R^{Nd}; \,\, M^N_2(V)
  \le \EE_0 \right\} \ \mbox{ and } 
  \left \langle f^N _0, M_k ^N \right \rangle \le C_{0,k} <\infty, \quad k > 2, 
  \]
with  
 \[
 \forall \, V \in \R^{dN}, \quad  M_k ^N(V) = \frac1N \, \sum_{j=1} ^N |v_j|^k.
\]
Then we have
\begin{equation}\label{eq:MaxwellSupportE}
\forall \, t \ge 0, \quad \hbox{\emph{Supp}} \, f^N_t \subset 
\left\{V \in \R^{Nd}; \,\, M^N_2(V) \le \EE_0 \right\}.
\end{equation}
and
\[
\sup_{t \ge 0} \left \langle f^N _t, M_k ^N \right\rangle \le 
\max \left\{ C_{0,k}; \, \bar a_k \right\} 
\]
where $\bar a_k \in (0,\infty)$ depends on $k$ and $\EE_0$. Moreover, when 
  \[
\hbox{{\em Supp}} \, f^N_0 \subset   \SS^N(\EE) =  \left\{V \in \R^{Nd}; \,\, M_2(\mu^N_V) =  \EE, \,\, \langle \mu^N_V , z \rangle = 0 \right\}, 
  \]
 then 
 \begin{equation}\label{eq:MaxwellSupport=E}
\forall \, t \ge 0, \quad \hbox{{\em Supp}} \, f^N_t \subset  \SS^N(\EE) =  
  \left\{V \in \R^{Nd}; \,\, M_2(\mu^N_V) =  \EE, \,\, \langle \mu^N_V , z \rangle = 0 \right\}.
  \end{equation}
\end{lem}

\begin{proof}[Proof of Lemma~\ref{lem:momentsN}]
By using \eqref{eq:preconservationE} with the function of the energy
\[
\phi : \R_+ \to \R, \quad \phi (z) := {\bf 1}_{z > N \, \EE_0}  
\]
and the assumptions on $f^N_0$ we deduce \eqref{eq:MaxwellSupportE}.
On the other hand, by using \eqref{eq:preconservationE} with the functions of the momentum and energy
\[
\phi_\eps: \R^{d+1}Ê\to \R, \quad  
\phi_\eps(z) := {\bf 1}_{|z_0/N-\EE |+|z_1| + ... + |z_d|> \eps}, \,\,\, \forall \, \eps > 0, 
\]
and the assumptions on $f^N_0$ we deduce \eqref{eq:MaxwellSupport=E}.

Next, we write the differential equality on the $k$-th moment:
$$
\frac{{\rm d}}{{\rm d}t} \left\langle f^N _t , \frac1N \, \sum_{j=1} ^N
  |v_j|^{k}\right\rangle = \frac{1}{N^2} \, \sum_{j_1 \neq j_2} ^N
\left\langle f^N _t, \left|v_{j_1}-v_{j_2}\right|^\gamma \,
  \KK\left(v_{j_1},v_{j_2}\right) \right\rangle,
$$
with 
$$
\KK\left(v_{j_1},v_{j_2}\right) = {1 \over 2} 
\int_{\mathbb{S}^{d-1}} b(\theta_{j_1 j_2}) \, \left[ |v_{j_1} ^*|^{k}
  + |v_{j_2} ^*|^{k} - |v_{j_1}|^{k} - |v_{j_2}|^{k} \right] \, {\rm d}\sigma.
$$

We then apply the so-called \emph{Povzner Lemma} proved in \cite[Lemma
2.2]{MW99} (valid for singular collision kernel as in our case) which
implies
$$
\KK (v_{j_1},v_{j_2}) \le C_1 \, \left(|v_{j_1}|^{k-1} \, |v_{j_2}| +
|v_{j_1}| \, |v_{j_2}|^{k-1}\right) - C_2 \, \left(|v_{j_1}|^{k} + |v_{j_2}|^{k}\right)
$$
for some constants $C_1, C_2 \in (0,\infty)$ depending only on $k$ and
$b$.

By using the inequalities $|v_{j_1}-v_{j_2}| \ge |v_{j_1}| -
|v_{j_2}|$ and $|v_{j_1}-v_{j_2}| \ge |v_{j_2}| - |v_{j_1}|$ in order
to estimate the last term when $\gamma =1$, we then deduce
\begin{multline*}
  |v_{j_1}-v_{j_2}|^\gamma \, \KK (v_{j_1},v_{j_2}) 
  \le C_3 \, [ (1+|v_{j_1}|^{k+\gamma-1}) \, (1+ |v_{j_2}|^2) \\ + (1+ |v_{j_1}|^2)
  \, (1+|v_{j_2}|^{k+\gamma-1})] - C_2 \, (|v_{j_1}|^{k+\gamma} +
  |v_{j_2}|^{k+\gamma}),
\end{multline*}
for a constant $C_3$ depending on $C_1$ and $C_2$. 

Using (symmetry hypothesis) that
$$
\forall \, k \ge 0, \quad \left\langle f^N_t , |v_1|^k \right\rangle = \left
  \langle f^N_t, M^N_k \right \rangle,
$$
and  \eqref{eq:MaxwellSupportE}  we get 
\begin{multline*} 
  \frac{{\rm d}}{{\rm d}t} \left\langle f^N_t , |v_1|^{k} \right\rangle \le 2 \,
  C_3 \left\langle f^N_{t} , (1 + M^N_{k+\gamma-1}) \, (1 + M^N_{2})
  \right\rangle - 2 \, C_2 \left\langle f^N_{t} , M_{k+\gamma} \right\rangle \\
  \le 2 \, C_3 \, (1+\EE) \, \left(1 + \left\langle f^N_{t} ,
      |v_1|^{k+\gamma-1} \right\rangle \right) - 2 \, C_2 \left\langle
    f^N_{t} , |v_1|^{k+\gamma} \right\rangle.
\end{multline*}
Using finally H\"older's inequality 
\[
\left\langle f^N_1,|v|^{k-\gamma+1} \right\rangle \le \left\langle
f^N_1,|v|^{k+\gamma} \right\rangle^{(k-\gamma+1)/(k+\gamma)}
\]
we conclude that $y(t) = \langle f^N_t , |v_1|^{k} \rangle$ satisfies a
differential inequality of the following kind
$$
y' \le - K_1 \, y^{\theta_1} +  K_2 \, y^{\theta_2} + K_3
$$
with $\theta_1 \ge 1$ and $\theta_2 < \theta_1$, and for some
constants $K_1$, $K_2$, $K_3 >0$, which concludes the
proof of the lemma by a maximum principle argument. 
\end{proof}

\medskip 
We have
$$
\hbox{Supp} f^N_0 \subset  \Ee_N :=  \{ V \in E^N, \, M^N_2(V) \le \EE_0 \},
$$
with $\EE_0 = A^2$, where $A$ is such that $\hbox{Supp} \, f_0 \subset
B_{\R^d}(0,A)$ under assumptions (i) or (ii) in Theorem~\ref{theo:tMM},
and $\EE_0 = \EE$ under assumption (iii) in Theorem~\ref{theo:tMM}.
Lemma~\ref{lem:momentsN} proves {\bf (A1)-(i)} with the constraint
function ${\bf m}_{\GG_1} : \R^d \to \R_+ \times \R^d$, ${\bf
  m}_{\GG_1} (v) := (|v|^2,v)$ for all $v \in \R^d$ and with the set
of constraints
$$
\RR_{\GG_1} := \{ (r_0,r') \in \R_+ \times \R^d;
\,\, |r'|^2 \le r_0 \le \EE_0 \}.
$$ 
Lemma~\ref{lem:momentsN} also proves {\bf (A1)-(ii)} with 
%
\[
m_{\GG_1}(v) := \langle v\rangle^6 = (1+|v|^2)^3.
\]
Moreover we do not need {\bf (A1)-(iii)} in the present case and we
may take $m_{\GG_3} \equiv 0$. Finally, under assumption (iii) in
Theorem~\ref{theo:tMM}, the support condition
\eqref{eq:supportfNtspheres} is nothing but
\eqref{eq:MaxwellSupport=E} and is also proved by Lemma~\ref{lem:momentsN}.


\subsection{Proof of condition (A2)}
Let us define the space of probability measures 
\[
\mathcal P_{\GG_1} := \left\{ f \in P(\R^d) \, ; \ 
  M_6(f)   <+\infty \right\}, 
\] 
for ${\bf r} \in \RR_{\GG_1}$, ${\bf r} = (r_0,r')$, $r' = (r_1, \dots, r_d)$, 
the constrained space 
\begin{equation*}
  \mathcal P_{\GG_1,{\bf r}} := \left\{ f \in \PP_6(\R^d) \, ;  \ \langle f, |v|^2
  \rangle = r_0, \
  \forall \, i=1, \dots, d, \ \langle f, v_i \rangle = r_i  \right\}
\end{equation*} 
and the vector space
$$
\GG_1 := \left\{ h \in M^1_6(\R^d) \, ; \quad  \forall \, i=1, \dots, d, \
  \langle h, v_i \rangle = \langle h, 1 \rangle = \langle h, |v|^2
  \rangle = 0 \right\}
$$
endowed with the modified Fourier-based norm $\| \cdot \|_{\GG_1} :=
\| \cdot \|_2$ defined in \eqref{eq:defToscanimodif}.  We also define
\[
\BB \mathcal P_{\GG_1,a} := \left\{ f \in \mathcal P_{\GG_1} \, ; \ M_6(f) \le a \right\}
\] 
as well as for any ${\bf r} \in {\RR}_{\GG_1}$
\[
\BB \mathcal P_{\GG_1,a,{\bf r}} := \Bigl\{ f \in \BB \mathcal
P_{\GG_1,a} \, ; \  \langle f, |v|^2 \rangle = r_0, \  \forall \, i=1, \dots, d, \ \langle f, v_i \rangle =
r_i  \Bigr\}.
\] 
These spaces of probability measures are endowed with the distance
$d_{\GG_1}$ associated to the norm $\| \cdot \|_{\GG_1}$.  Observe
that for any $f,g \in \mathcal P_{\GG_1,{\bf r}}$, ${\bf r} \in
\RR_{\GG_1}$, the fact that the two distributions have equal momentum
implies that $d_{\GG_1} (f,g) = \| f-g \|_{\GG_1} = |f-g|_2$, where $
| \cdot|_2$ is the usual Fourier-based norm defined in
\eqref{eq:defToscani}.

\medskip Let us recall the following result proved in
\cite{T1,GabettaTW95,CGT,TV}.  We briefly outline its proof for the
sake of completeness but also, most importantly, since we shall need
to modify it in order to adapt it to our purpose in the next sections.

\begin{lem} \label{lem:contraction} For any $f_0,g_0 \in \PP_2(\R^d)$
  with same momentum (first-order moments), the associated
  solutions $f_t$ and $g_t$ to the Boltzmann equation for Maxwell
  molecules satisfy
\begin{equation}\label{estim:C01TrueMax}
  \sup_{t \ge 0} \left|f_t-g_t\right|_2 \le 
  \left|f_0-g_0\right|_2.
\end{equation}
Moreover, there exists $\bar a \in (0,\infty)$ such that for any $a
\in [\bar a, \infty)$ and any ${\bf r} \in \RR_{\GG_1}$, the
nonlinear semigroup $S^{N\!L}_t$ maps $\BB \mathcal P_{\GG_1,a,{\bf
    r}}$ into $\BB \mathcal P_{\GG_1,a,{\bf r}}$. 
\end{lem}

\begin{proof}[Proof of Lemma~\ref{lem:contraction}] 
  We only prove \eqref{estim:C01TrueMax}. The fact that $S^{N\!L}_t$
  maps $\BB \mathcal P_{\GG_1,a,{\bf r}}$ into $\BB \mathcal
  P_{\GG_1,a,{\bf r}}$ follows from the conservations of momentum and
  energy and higher moment estimates similar to
  Lemma~\ref{lem:momentsN} but on the limit equation. We refer to
  \cite{T1,GabettaTW95,CGT,TV} for this classical results.

We recall Bobylev's identity for Maxwellian collision kernel
(cf. \cite{Bobylev-88})
$$
\FF\left(Q^+(f,g)\right) (\xi) = \hat Q^+ (F,G) (\xi) =: {1 \over 2}
\int_{\mathbb S^{d-1}} b\left(\sigma \cdot \hat\xi\right) \, \left[F^+ \, G^- + F^- \, G^+ \right]\,
{\rm d}\sigma,
$$
with 
\[
F = \hat f, \quad G = \hat g, \quad F^\pm= F\left(\xi^\pm\right), \quad G^\pm=
G\left(\xi^\pm\right), \quad \hat \xi = \frac{\xi}{|\xi|}
\]
and
$$
\xi^+ = {1\over 2} (\xi +  |\xi| \, \sigma),
 \quad
\xi^- = {1\over 2} (\xi -  |\xi| \, \sigma).
$$

With the shorthand notation $D = \hat g - \hat f$, $S = \hat g + \hat
f$, the following equation holds
\begin{equation}\label{eq:BoltzMaxD}
  \partial_t D = \hat Q (S,D) =  \int_{\mathbb S^{d-1}} b
  \left(\sigma \cdot \hat \xi \right) \, 
  \left[ \frac{D^+ \, S^-}{2} + \frac{D^- \, S^+}{2} - D \right] \, {\rm d}\sigma.
\end{equation}
We perform the following cutoff decomposition on the angular collision
kernel:
\[
b = b_K + b_K ^c \ \mbox{ with } \ \int_{\mathbb{S}^{d-1}} b_K \left(
  \sigma \cdot \hat \xi \right) \, {\rm d}\sigma = K, \quad b_K = b \, {\bf
  1}_{|\theta|\ge \delta(K)}
\]
for some well-chosen $\delta(K)$.  As in \cite{TV} observe that
\[
R_K (\xi) = \int_{\mathbb S^{d-1}} b_K ^c
  \left(\sigma \cdot \hat \xi \right) \, 
  \left[ \frac{D^+ \, S^-}{2} + \frac{D^- \, S^+}{2} - D \right] \,
  {\rm d}\sigma 
\]
satisfies
\[
\forall \, \xi \in \R^d, \quad \left| R_K (\xi)\right| \le 
 r_k \, |\xi|^2 \quad \mbox{ where } \quad r_k \xrightarrow[]{K \to \infty} 0
\]
and $r_K$ depends on moments of order $2$ on $d$ and $s$ (hence
bounded by the energy). 

Using that $\| S \|_\infty \le 2$, we deduce in distributional sense
$$
\frac{{\rm d}}{{\rm d}t} {|D| \over |\xi|^2} + K \, {|D| \over |\xi|^2} \le 
\left( \sup_{\xi \in \R^d} {|D| \over |\xi|^2} \right) \, \left( \sup_{\xi \in \R^d}
\int_{S^{d-1}} b_K\left(\sigma \cdot \hat \xi \right) \, \left(
  \left|\hat\xi^+\right|^2 + \left|\hat\xi^-\right|^2 \right) \,
{\rm d}\sigma \right) + r_K
$$
with
$$
\left|\hat\xi^+\right| = 
{1 \over \sqrt{2}} \, \left(1 + \sigma \cdot \hat\xi\right)^{1/2},
\qquad \left|\hat\xi^-\right| = {1 \over \sqrt{2}} \, \left(1 - \sigma \cdot
\hat\xi\right)^{1/2}.
$$
By using 
\[
\left|\hat\xi^+\right|^2 + \left|\hat\xi^-\right|^2 =1,
\]
we deduce
\[
\frac{{\rm d}}{{\rm d}t} {|D| \over |\xi|^2} + K \, {|D| \over |\xi|^2} \le 
K \, \left( \sup_{\xi \in \R^d} {|D| \over |\xi|^2} \right) +r_K
\]
which implies
\[
\sup_{\xi \in \R^d} \frac{|D_t(\xi)|}{|\xi|^2}  \le 
 \sup_{\xi \in \R^d} \frac{|D_0(\xi)|}{|\xi|^2}  + C \, \frac{r_K}{K}
\]
for any value of the cutoff parameter $K$. Therefore by relaxing $K
\to \infty$, we deduce \eqref{estim:C01TrueMax}.
\end{proof}
 
Hence we deduce from the previous lemma that $S^{N\!L}_t$ is Lipschitz
on $\BB \mathcal P_{\GG_1,a,{\bf r}}$ (uniformly in time) and {\bf
  (A2)-(i)} is proved.

\begin{lem}
  \label{lem:estimQtoscani1}
  Consider $f,g \in \PP_{2}(\R^d)$ two probabilities with same momentum
  (first order moments).  Then we have
   \begin{equation}\label{eq:Qtoscani2}
    \left| Q(f,f) \right|_2 \le C \, \left( \int_{\R^d} \left(1+|v|^2\right) \,
      {\rm d}f(v) \right)^2,
  \end{equation}
   \begin{equation}\label{eq:Qbiltoscani2}
    \left| Q(f+g,f-g) \right|_2 \le C \, \left( \int_{\R^d} (1+|v|) \,
      \left( {\rm d}f(v) + {\rm d}g(v) \right) \right) \, 
    \Big( | f-g |_2 + \big|  (f-g) \, v  \big|_1 \Big).
  \end{equation}

  As a consequence, the assumption {\bf (A2)-(ii)} is satisfied in
  the sense that $Q$ is H\"older continuous on $\BB \mathcal
  P_{\GG_1,a,{\bf r}}$ for any $a > 0$, ${\bf r} \in \RR_{\GG_1}$:
  there exists $C > 0$ and $\zeta \in (0,1)$ so that
  $$
  \forall \, f, g \in \BB \mathcal P_{\GG_1,a,{\bf r}}, \quad |Q(f,f) - Q(g,g)|_2 \le C \, |f-g|^\zeta_2.
  $$
\end{lem}
 
\begin{proof}[Proof of Lemma~\ref{lem:estimQtoscani1}] We split the
  proof into two steps.

\noindent{\sl Step 1: proof of \eqref{eq:Qtoscani2}
and \eqref{eq:Qbiltoscani2}. }
  We prove the second inequality \eqref{eq:Qbiltoscani2}. The first
  inequality \eqref{eq:Qtoscani2} then follows immediately by
  writing
  \[
  Q(f,f) = Q(f,f) - Q(\gamma,\gamma) = Q(f-\gamma,f+\gamma)
  \]
  where $\gamma$ is the Maxwellian distribution with same momentum and energy
  as $f$, and then applying \eqref{eq:Qbiltoscani2} with $f-\gamma$ and
  $f+\gamma$.

We write in Fourier:
\bean
\mathcal{F}\left( Q(f+g,f-g) \right) &=& \hat Q(D,S) \\ 
&=& \frac12 \,
\int_{\mathbb{S}^{d-1}} b(\sigma \cdot \hat \xi) \, \left( S(\xi^+) \,
  D(\xi^-) + S(\xi^-) \, D(\xi^+) - 2 \, D(\xi) \right)
\eean
where $\hat Q$ is the Fourier form the symmetrized collision operator
$Q$, which we can rewrite
\[
\frac{\left| \hat Q(D,S) \right|}{|\xi|^2} \le \TT_1 + \TT_2 + \TT_3.
\]

We have 
\[
\TT_1 \le \int_{\mathbb{S}^{d-1}} b(\sigma \cdot \hat \xi) \, 
\left|S(\xi^+)\right| \, \frac{\left| D(\xi^-) \right|}{|\xi^-|^2} \,
\frac{\left| \xi^- \right|^2}{|\xi|^2} \, {\rm d}\sigma 
\le C \, | D |_2
\]
for some constant $C>0$, where we have used 
\[
\frac{\left| \xi^- \right|^2}{|\xi|^2} = (1-\cos \theta)^2
\]
which permits to control the angular singularity of $b$. 

Similarly we compute 
\[
\TT_2 \le \int_{\mathbb{S}^{d-1}} b(\sigma \cdot \hat \xi) \,
\frac{\left|D(\xi^+)\right|}{|\xi^+|} \, \frac{\left| S(\xi^-) -2
  \right|}{|\xi^-|} \, \frac{\left| \xi^- \right|}{|\xi|} \, {\rm d}\sigma
\le C \, | D |_1 \, \left( \int_{\R^d} (1+|v|) \, ( {\rm d}f(v) + {\rm d}g(v) )
\right) 
 \] 
for some constant $C>0$, and 
\begin{multline*}
\TT_3 \le 2 \, \int_{\mathbb{S}^{d-1}} b(\sigma \cdot \hat \xi) \,
\frac{\left| D(\xi^+) - D(\xi) \right|}{|\xi|} \, {\rm d}\sigma \\
\le \int_{\mathbb{S}^{d-1}} b(\sigma \cdot \hat \xi) \,
\frac{|\xi^-|}{|\xi|} \, \int_0 ^1 \frac{|\nabla D (\theta \xi +
  (1-\theta) \xi^+)|}{|\theta \xi +
  (1-\theta) \xi^+|} \, {\rm d}\theta \, {\rm d}\sigma 
\le C \, | (f - g)  \, v |_1
\end{multline*}
for some constant $C>0$. This concludes the proof of
\eqref{eq:Qbiltoscani2} by combining these estimates.

\smallskip
\noindent{\sl Step 2:  proof of \textbf{(A2)-(ii)}.}
First, for $ f,g \in \BB P_{ \GG_1,a,{\bf r}}$ and $1 \le \alpha \le d$, we
have for any $R > 0$
\begin{eqnarray*}
\left|(f-g) \, v_\alpha \right|_1 
&=& \sup_{\xi \in \R^d} {1 \over |\xi|} \left| \int_{\R^d} v_\alpha \,
  \Bigl( e^{-i \, v \cdot \xi} - 1 \Bigr) \, ({\rm d}f-{\rm d}g)(v) \right|
\\
&\le& \sup_{\xi \in \R^d}  \left| \int_{\R^d} \varphi_\xi(v) \, ({\rm
    d}f-{\rm d}g)(v) \right|
+  \sup_{\xi \in \R^d} \left| \int_{\R^d} \psi_\xi(v) \, ({\rm
    d}f-{\rm d}g)(v) \right|
\end{eqnarray*}
with 
$$
\varphi_\xi(v) := {1 \over |\xi|} \, \chi_R  \, v_\alpha \, \Bigl( e^{-i \, v \cdot \xi} - 1 \Bigr), 
\quad
\psi_\xi(v) :=  {1 \over |\xi|} \, (1-\chi_R) \, v_\alpha \, \Bigl( e^{-i \, v \cdot \xi} - 1 \Bigr),
 $$
 and where $\chi_R$ is a truncation function just as in the proof of
 point (iv) in Lemma~\ref{lem:ComparDistances}. 

 Next, we observe that for $R \ge 1$ we have
$$
\forall \, v,\xi \in \R^d, \quad |\nabla_v \varphi_\xi (v) | \le C_1 \, R^2,
\quad
|\psi_\xi (v) | \le |v|^2 \, (1-\chi_R).
$$

Using  the bound on the sixth moment of $f$ and $g$, we deduce
\begin{eqnarray*} 
\left|(f-g) \, v_\alpha \right|_1 
&\le& \inf_{R \ge 1} \Bigl\{ C_1 \, [f-g]^*_1 \, R^2 + C_2 \, {a \over R^4} \Bigr\}
 \\
&\le&   C_3 \, \min \left\{ a^{1/3}\, \left([f-g]^*_1\right)^{2/3}, [f-g]^*_1 \right\} 
\\
&\le&  C_3 ' \, a^{1/3}\, \left([f-g]^*_1\right)^{2/3}.
 \end{eqnarray*}

From \eqref{estim:*1s}, we then obtain  
  $$
 f,g \in \BB P_{ \GG_1,a,{\bf r}}, \quad \left|(f-g) \, v_\alpha \right|_1 
 \le C \, |f-g|_2^\zeta,
 $$
 for some constants $C >0$ and $\zeta \in (0,1)$ depending on $d$ and $a$. 

 Gathering this last estimate with \eqref{eq:Qbiltoscani2} we conclude that for any $ f,g \in \BB P_{ \GG_1,a,{\bf r}}$
\begin{eqnarray*}
\left| Q(f,f) - Q(g,g) \right|_2 &=&  \left| Q(f+g,f-g) \right|_2 
\\
 &\le& C \, M_1(f+g) \, 
    \Big( | f-g |_2 + \big|  f-g  \big|_2^\zeta \Big)
\le C' \,  \big|  f-g  \big|_2^\zeta , 
\end{eqnarray*}
for some constant $C' > 0$ depending on $d$ and $a$. 
\Black
\end{proof}


\subsection{Proof of condition (A3)}\label{subsect:ProofA3tMM}
Let us define $m' _{\GG_1} = \langle v \rangle^4$ and 
\[
\Lambda_1(f) := \left\langle f , m'_{\GG_1} \right\rangle = \left\langle f,
  \langle v \rangle^4 \right\rangle 
\]
for any $f \in \mathcal P_{\GG_1}$.

Let us prove that for any 
\beqn\label{eq:PhiCapC1}
 \Phi \in \bigcap_{  {\bf r} \in \RR_{\GG_1}}
C^{1,\eta}_{\Lambda_1}\left(\mathcal P_{\GG_1,{\bf r}}\right)
\eeqn
we have
\begin{equation*}
  \left\| \left( M^N_{m_{\GG_1}} (V)\right) ^{-1} \, \left( G^N \, \pi^N - \pi^N \, G^\infty \right)  \,
      \Phi \right\|_{L^\infty(\Ee_N)} \le
    {C_{1} \, \EE_0 \over N^\eta} \, \sup_{  {\bf r} \in \RR_{\GG_1} } [ \Phi ]_{C^{1,\eta}_{\Lambda_1} \left(\mathcal P_{\GG_1,{\bf r}}\right)},
\end{equation*}
for some constant $C_1 >0$. 

First, consider velocities $v,v_*, w,w_* \in \R^d$ such that
$$
w = {v+v_* \over 2} + {|v-v_*| \over 2} \, \sigma, \quad w_* = {v+v_*
  \over 2} - {|v-v_*| \over 2} \, \sigma, \quad \sigma \in \mathbb{S}^{d-1}.
$$
Then $\delta_v + \delta_{v_*} - \delta_w - \delta_{w_*} \in \mathcal{I} \mathcal P_{\GG_1}$.
Performing Taylor expansions, we get 
\begin{eqnarray*}
  && e^{i v \cdot \xi} + e^{i v_* \cdot
    \xi} - e^{i w \cdot \xi} - e^{i w_* \cdot \xi} \\
  && = i \, (w-v) \, \xi \, e^{i v \cdot \xi} + \OO\left(|w-v|^2 \, |\xi|^2\right) +
  i \, (w_*-v_*) \, \xi \, e^{i v_* \cdot \xi} + \OO\left(|w_*-v_*|^2 \,
  |\xi|^2\right)
  \\
  && = i \, (w-v) \, \xi \, e^{i v \cdot \xi} + \OO\left(|w-v|^2 \, |\xi|^2\right) \\ 
  && \qquad \qquad \qquad + i
  \, (w_*-v_*) \, \xi \, (e^{i v_ \cdot \xi} + \OO \left( |v-v_*|\, |\xi|\right) +
  \OO\left(|w_*-v_*|^2 \, |\xi|^2\right)
  \\
  && = \OO \left( |v-v_*|^2 \, |\xi|^2 \, \sin \theta/2 \right)
 \end{eqnarray*}   
thanks to the impulsion conservation and the fact that 
 $$
 |w-v| = |w_* - v_*| = |v-v_*| \, \sin \frac{\theta}{2}.
 $$

We hence deduce 
\begin{equation*}
    \left| \delta_v + \delta_{v_*}- \delta_w - \delta_{w_*} \right|_2 
    = \sup_{\xi \in \R^d} \frac{ \left| e^{i v \cdot \xi} + e^{i v_* \cdot \xi} 
          - e^{i w \cdot \xi} - e^{i w_* \cdot  \xi}\right| }{|\xi|^2}
          \le C \, |v-v_*| ^2 \, (1-\cos \theta).
\end{equation*}

As an immediate consequence, for any $V \in \Ee_N$ and $V^*_{ij}$
defined by \eqref{vprimvprim*}, we have
\begin{equation}\label{eq:deltamuNToscani2}
\left|\mu^N_{V^*_{ij}} - \mu^N_V\right|_2 \le {C \over N} \, \left|v_i -
v_j \right|^2 \, (1 - \cos \, \theta_{ij})
\end{equation}
and 
$$
{\bf r}_V := \left( \left\langle \mu^N_V, |z|^2
  \right\rangle, \left\langle \mu^N_V,z_1 \right\rangle, \dots,
  \left\langle \mu^N_V,z_d \right\rangle\right) \in \RR_{\GG_1},
  $$
where $z=(z_1,\dots,z_d) \in \R^d$ has to be understood as the blind
integration variable in the duality bracket. 

Then for given a $\Phi$ satisfying \eqref{eq:PhiCapC1},  for any $V \in \Ee_N$,
and any $1 \le i, j \le N$, 
we set 
$
\phi := D\Phi \left[ \mu_V ^N \right]$,  $  u_{ij} = (v_i-v_j)
$
and we compute:
\begin{eqnarray*}
&&  G^N\left(\Phi \circ  \mu ^N _V\right) 
=  \frac{1}{2N} \, \sum_{i,j=1} ^N 
  \int_{\mathbb{S}^{d-1}} \left[ \Phi\left(\mu^N_{V^*_{ij}}\right) 
    - \Phi\left(\mu^N _V\right) \right] \, b\left(\cos \theta_{ij}\right) \, {\rm d}\sigma 
\\
&&= \frac{1}{2N} \, \sum_{i,j=1} ^N
  \int_{\mathbb{S}^{d-1}} \left\langle \mu^N _{V^*_{ij}}
    - \mu^N _V , \phi \right\rangle \, b\left(\cos \theta_{ij}\right) \, {\rm d}\sigma 
\\
&&+  \frac{[\Phi]_{ C_\Lambda^{1,\eta}(\mathcal P_{\GG_1,{\bf r}_V}) }}{2N} \, 
  \sum_{i,j=1} ^N  \int_{\mathbb{S}^{d-1}} 
  \!\!\! \ \left[ M_{m'_{\GG_1}}\left(\mu^N _{V^*_{ij}}\right)  +
    M_{m'_{\GG_1}} \left( \mu^N _{V}\right) \! \right] \, 
  \OO \! \left( \left| \mu^N _{V^*_{ij}} 
      - \mu^N _V \right|_2^{1+\eta} \right)  \,b\left(\cos \theta_{ij}\right) \, {\rm d}\sigma \\
&&\quad =:  I_1(V) + I _2(V).
\end{eqnarray*}

Concerning the first term $I_1(V)$, thanks to Lemma~\ref{lem:H0}, we
have
\begin{multline*}
  I_1(V) = \frac{1}{2N^2} \, \sum_{i,j=1} ^N
  \int_{\mathbb{S}^{d-1}} b\left(\cos \theta_{ij}\right) \, \left[ \phi(v_i ^*) +
    \phi(v_j^*) - \phi(v_i) - \phi(v_j) \right] \, {\rm d}\sigma \\
  = \frac{1}{2} \, \int_v \int_w 
  \int_{\mathbb{S}^{d-1}} b(\cos \theta) \, \left[ \phi(v^*) +
    \phi(w^*) - \phi(v) - \phi(w) \right] \, {\rm d}\mu^N _V(v) \, 
  {\rm d}\mu^N_V (w) \, {\rm d}\sigma \\
  = \left\langle Q(\mu^N _V, \mu^N _V), \phi \right\rangle = 
  \left(G^\infty\Phi\right)(\mu^N _V).
\end{multline*}

For the second term $I_2(V)$, using estimate
\eqref{eq:deltamuNToscani2} and the following 
inequality which holds for any $k \ge 2$ and any $V \in \Ee^N$ 
\begin{eqnarray}\label{eq:MkmuV*}
  M_{k}\left(\mu^N_{V^*_{ij}}\right) 
  &:=& \frac1N \, \sum_{\ell=1} ^N   \left\langle (V^* _{ij})_\ell \right\rangle^k \\ \nonumber
  &\le& 2^{k/2}
  \, \left( 1+ {1 \over N} \, 
    \left( \left(\sum_{\ell \not= i,j} |v_\ell|^{k} \right) + |v^*_{i}|^{k} +
      |v^*_{j}|^{k} \right) \right)
   \\ \nonumber
  &\le& 2^{k/2} \, \left( 1+ {1 \over N} \, 
  \left( \left( \sum_{\ell \not= i,j} |v_\ell|^{k} \right) 
    +  \left(|v_i|^2 + |v_j|^2\right)^{k/2} \right) \right)
   \\ \nonumber
  &\le& 2^k \, \left( 1+ \frac{1}{N} \,  \sum_{\ell =1} ^N |v_\ell|^{k} \right)   \le 2^k \,
M_{k}\left( \mu^N _V \right) = 2^k \, M^N _{k} (V),
\end{eqnarray}
we deduce, for some constant $C>0$ depending on $k$ and  $\eta \in (0,1]$, 
\begin{eqnarray*}
  \left|I_2(V)\right|  
  &\le& \frac{C}{N^{2+\eta}} \, M^N_{m'_{\GG_1}}(V)  \, 
  [\Phi]_{C^{1,\eta}_\Lambda(\mathcal P_{\GG_1,{\bf r}_V})}
  \\
  &&\quad \times \sum_{i,j= 1}^N   \int_{\mathbb{S}^{d-1}}
    b \left(\cos \theta_{ij}\right)  \, 
    \left(1+|v_i|^2+|v_j|^2\right)^{1+\eta} \, 
    \left(1 - \sigma \cdot \hat u_{ij}\right) \, {\rm d}\sigma   
    \\
   &\le& \frac{C}{N^{\eta}} \, M^N_{m'_{\GG_1}}(V)  \, 
  [\Phi]_{C^{1,\eta}_\Lambda(\mathcal P_{\GG_1,{\bf r}_V})} \, C_b \, M_4^N(V)  . 
  \end{eqnarray*}

  We finally use 
\[
\forall \, V \in \Ee_N, \quad M^N_{m'_{\GG_1}}(V) = M_4^N (V) \le
  M^N_{2}(V)^{1/2} \, M_6^N(V)^{1/2} = \EE_0^{1/2} \,
  M^N_{m_{\GG_1}}(V)^{1/2}
\]
by Cauchy-Schwartz inequality and the energy constraint, which implies
\[
\left|I_2(V)\right|  \le \frac{C_1 \, \EE_0}{N^{\eta}} \, M^N_{m_{\GG_1}}(V)  \, 
  [\Phi]_{C^{1,\eta}_\Lambda(\mathcal P_{\GG_1,{\bf r}_V})}
\]
and concludes the proof.

\subsection{Proof of condition (A4) uniformly in time}
 
Let us consider some $1$-particle initial data $f_0, g_0 \in
\PP_4(\R^d)$ (space of probability measures with bounded fourth moment) and
the associated solutions $f_t$ and $g_t$ to the nonlinear Boltzmann
equation \eqref{el} under the assumption \eqref{Maxwelltrue} as well as  
\[
h_t := \DD^{N\!  L}_t\left[f_0\right](g_0-f_0)
\]
the solution to the linearized Boltzmann equation around $f_t$. Those
solutions are given by
\[
\left\{
\begin{array}{l}
\partial_t f_t = Q(f_t,f_t), \qquad f_{|t=0} = f_0 \vspace{0.3cm} \\
\partial_t g_t = Q(g_t,g_t), \qquad g_{|t=0} = g_0 \vspace{0.3cm} \\
\partial_t h_t = 2 \, Q(h_t,f_t), \qquad h_{|t=0} = h_0 := g_0 - f_0.
\end{array}
\right.
\]

We shall now \emph{expand the limit nonlinear semigroup} in terms
of the initial data, around $f_0$. 
 
\begin{lem}\label{lem:a4maxwellfourier} 
  There exists $\lambda \in (0,\infty)$ and, for any $\eta \in (1/2,1)$,
  there exists $C_\eta >0$ such that for any 
\[
{\bf r} \in \RR_{\GG_1}    \ \mbox{ and } \ f_0, g_0 \in \mathcal P_{\GG_1,{\bf r}} 
\]
we have
\begin{equation}\label{ineq;tMMC01}
  \left| f_t - g_t \right|_2 \le 
  C_\eta \, e^{- (1-\eta) \, \lambda \, t } \, 
   M_4(f_0+g_0)^{\frac12}
  \, 
 \left| f_0 - g_0 \right|_2^\eta, 
\end{equation}
and 
 \begin{equation}\label{ineq;tMMC10}
   \left| h_t \right|_2 \le C_\eta \, e^{- (1-\eta) \, \lambda \, t }
   \, 
   M_4(f_0+g_0)^{\frac12} \, 
   \left| f_0 - g_0 \right|_2^\eta, 
\end{equation}
where we recall that 
\[
\forall \, f \in P(\R^d), \quad M_4(f) := \big\langle f, \langle v \rangle^4 \big\rangle.
\]

\end{lem}

\begin{rem}
  Observe that the decay rate $\lambda$ in this statement is
  uniform in terms of the energy $\EE \ge 0$. 
\end{rem}

\begin{proof}[Proof of Lemma~\ref{lem:a4maxwellfourier}]
We shall proceed in several steps. 

\bigskip
\noindent {\sl Step 1. Estimate in $|\cdot|_4$. } We closely follow
ideas in \cite{MR0075725,T1,GabettaTW95,CGT}.  We shall use the notation 
\[
\MM = \MM_4, \quad \hat \MM = \hat \MM_4,
\]
introduced in Example~\ref{expleFourierGal}, as well as
\[
d:=f-g, \quad s:=f+g 
\]
and 
\[
\tilde d := d - \MM[d], \quad D := \FF(d), \quad S := \FF(s)
\ \mbox{ and } \ \tilde D := \FF(\tilde d) = D - \hat{\mathcal{M}}[d].
\]
The equation satisfied by $\tilde D$ is 
\begin{eqnarray}\label{eq:BoltzTildeD}
  \partial_t \tilde D 
  &=&  \hat Q (D,S) - \partial_t \hat \MM[d]
  \\ \nonumber
  &=&  \hat Q (\tilde D,S) + \left( \hat Q\left(\hat\MM[d],S\right) 
    - \hat\MM[Q(d,s)] \right).
\end{eqnarray}
We infer from \cite{MR0075725,T1} that for any $\alpha \in \N^d$, there exists
some absolute coefficients $(a_{\alpha,\beta})$, $\beta \le \alpha$
(which means $\beta_i \le \alpha_i$ for any $1 \le i \le d$),
depending on the collision kernel $b$ through
\begin{equation}\label{eq:momentsmax}
\int_{\mathbb{S}^{d-1}} b(\cos\theta) \left[\left(v^\alpha\right)' +
  \left(v^\alpha\right)'_* - \left(v^\alpha\right) - \left(v^\alpha\right)_* \right] \,
  {\rm d}\sigma = \sum_{\beta, \, \beta \le \alpha} a_{\alpha,\beta} \, \left(v^\beta\right) \,
  \left(v^{\alpha-\beta}\right)_*
\end{equation}
where $\alpha, \beta \in \N^d$ are \emph{coordinates}
indices and 
\[
v^\alpha := v_1 ^{\alpha_1} \, v_2 ^{\alpha_2} \dots v_d ^{\alpha_d}.
\]
These multi-indices are compared through the usual lexicographical
order, and we use the standart notation 
\[
|\alpha| := \sum_{k=1}^d \alpha_k.
\]

We deduce that 
\[
\forall \, |\alpha|\le 3, \quad \nabla^\alpha _\xi  
\mathcal{\hat M}[Q(d,s)]_{\big|\xi=0} = M_\alpha[Q(d,s)] = \sum_{\beta, \, \beta \le \alpha}
a_{\alpha,\beta} \, M_\beta[d] \, M_{\alpha-\beta}[s]
\]
together with 
\begin{multline*}
  \forall \, |\alpha|\le 3, \quad \nabla^\alpha _\xi \hat
  Q(\mathcal{\hat
    M}[d],S)_{\big|\xi=0} = M_\alpha[Q(\mathcal{M}[d],s)] \\
  = \sum_{\beta, \, \beta \le \alpha} a_{\alpha,\beta} \,
  M_\beta[\mathcal{M}[d]] \, M_{\alpha-\beta}[s] = \sum_{\beta, \,
    \beta \le \alpha} a_{\alpha,\beta} \, M_\beta[d] \, M_{\alpha-\beta}[s]
\end{multline*}
since  
\[
M_\alpha[\mathcal{M}[d]] = M_\alpha[d]
\]
for any $|\alpha| \le 3$ by construction. As a consequence, we get
\begin{equation}\label{eq:MMQds-QMMds}
\forall \, \xi \in \R^d,\quad 
\left| \hat{\mathcal{M}}[Q(d,s)] -
  \hat Q( \hat{\mathcal{M}}[d], S ) \right| 
\le C \, |\xi|^4 \,
\left( \sum_{|\alpha| \le 3} \left|M_\alpha[f-g]\right| \right).
\end{equation}

On the other hand, from \cite[Theorem 8.1]{T1} and its corollary, we
know that there exists some constants $C >0$ and $\lambda_1 >0$ (given
by $\lambda_1 := \min_{\alpha, |\alpha| \le 3} \{ -
a_{\alpha,\alpha}\} \in (0,\infty)$) so that
\begin{equation}\label{eq:MomentsTanaka}
\forall \, t \ge 0, \quad 
\left( \sum_{|\alpha| \le 3} \left|M_\alpha[f_t-g_t]\right| \right) \le 
C \, e^{-\lambda_1 \, t} \left( \sum_{|\alpha| \le 3} \left|M_\alpha[f_0-g_0]\right| \right). 
\end{equation}

We perform the same decomposition on the angular collision kernel
\[
b = b_K + b_K ^c \ \mbox{ with } \ \int_{\mathbb{S}^{d-1}} b_K \left(
  \sigma \cdot \hat \xi \right) \, {\rm d}\sigma = K, \quad b_K = b \, {\bf
  1}_{|\theta|\ge \delta(K)}
\] 
as in the proof of Lemma~\ref{lem:contraction} and use the
straightforward estimate 
\[
R_K(\xi) := \hat Q_{b_K ^c} (\tilde D, S)(\xi)
\]
satisfies 
\[
\forall \, \xi \in \R^d, \quad \left| R_K (\xi) \right| \le  r_K \,
|\xi|^4 \quad \mbox{ where } \ r_K \xrightarrow[]{K \to \infty} 0 
\]
where $Q_{b_K ^c}$ denotes the collision operator associated with the
part $b_K ^c$ of the decomposition of the angular collision kernel,
and where $r_K$ depends on moments of order $4$ on $d$ and $s$.

Then we gather \eqref{eq:BoltzTildeD}, \eqref{eq:MMQds-QMMds} and
\eqref{eq:MomentsTanaka} and we have
\begin{multline*}
  \frac{{\rm d}}{{\rm d}t} {|\tilde D(\xi)| \over |\xi|^4} + K \, {|\tilde D(\xi)|
    \over |\xi|^4} \le \left( \sup_{\xi \in \R^d} {|\tilde D(\xi)|
      \over |\xi|^4} \right) \, \left( \sup_{\xi \in \R^d}
    \int_{\mathbb S^{d-1}} b_K\left(\sigma \cdot \hat \xi \right) \, \left(
      \left|\hat\xi^+\right|^4 + \left|\hat\xi^-\right|^4 \right) \,
    {\rm d}\sigma \right) \\
  + C \, e^{-\lambda_1 \, t} \left( \sum_{|\alpha| \le 3}
    \left|M_\alpha[f_0-g_0]\right| \right) + r_K.
\end{multline*}
Let us compute (the supremum has been droped thanks to the spherical invariance)
\[
\lambda_K := \int_{\mathbb S^{d-1}} b_K\Bigl(\sigma \cdot \hat \xi \Bigr) \, \Bigl(
  \left|\hat\xi^+\right|^4 + \left|\hat\xi^-\right|^4 \Bigr) \,
{\rm d}\sigma = \int_{\mathbb S^{d-1}} b_K\left(\sigma \cdot \hat \xi \right) \,  
 \frac{1+ \left( \sigma \cdot \hat \xi \right)^2}{2}   \,
{\rm d}\sigma,
\]
so that  
\begin{eqnarray*}
  \lambda_K - K& =& -  \int_{\mathbb S^{d-1}} b_K\left(\sigma \cdot \hat \xi
    \right) \,  \frac{1- \left( \sigma \cdot \hat \xi
        \right)^2}{2}  \, {\rm d}\sigma  
        \\
&&\xrightarrow[K \to
  \infty]{} - \int_{\mathbb S^{d-1}} b\left(\sigma \cdot \hat \xi \right) \,
  \frac{1- \left( \sigma \cdot \hat \xi \right)^2}{2}  
  \, {\rm d}\sigma := - \bar\lambda \in (-\infty,0),
\end{eqnarray*}
where in the last step we have used the \eqref{Maxwelltrue}. 

Then, thanks to Gronwall lemma, we get 
\begin{multline*}
  \left( \sup_{\xi \in \R^d} \frac{|\tilde D_t(\xi)|}{|\xi|^4} \right)
  \le e^{(\lambda_K-K) \, t} \, \left( \sup_{\xi \in \R^d}
    \frac{|\tilde D_0(\xi)|}{|\xi|^4} \right)
  \\
  + C_3 \, \left( \sum_{|\alpha| \le 3} \left|M_\alpha[f_0-g_0]\right| \right)
  \, \left( \frac{e^{-\lambda_1 \, t}}{K-\lambda_K-\lambda} -
    \frac{e^{(\lambda_K-K) \, t}}{K-\lambda_K-\lambda} \right) + 
C\, \frac{r_K}{K(K -\lambda_K)}.
\end{multline*}

Therefore, passing to the limit $K \to \infty$ and choosing (without
restriction) $\lambda_2 \in (0,\min\{\lambda_1 ; \bar\lambda\})$, we obtain
\[
\sup_{\xi \in \R^d} \frac{|\tilde D_t(\xi)|}{|\xi|^4} \le C \, e^{ -
  \lambda_2 \, t} \, \left( \sup_{\xi \in \R^d} \frac{|\tilde
    D_0(\xi)|}{|\xi|^4} + \sum_{|\alpha| \le 3} \left|M_\alpha[f_0-g_0]\right|
\right)
\]
from which we conclude thanks to \eqref{eq:MomentsTanaka} and with the
notations of Example~\ref{expleFourierGal}
\begin{equation}\label{eq:decay-d-4}
||| d_t |||_4  \le C \, e^{ - \lambda_2 \, t} \, |||d_0 |||_4 . 
\end{equation}

\bigskip \noindent {\sl Step 2. From $|\cdot|_4$ to $|\cdot|_2$ on the
  difference.} From the preceding step and a straightforward interpolation
argument, we have
\begin{eqnarray}\label{eq:dtTosc2exp-t}
\left| f - g \right|_2 
&\le&
 \left| f - g - \mathcal{M}[f-g] \right|_2 +
C \,   \sum_{|\alpha| \le 3} \left|M_\alpha[f-g]\right|   
\\ \nonumber
&\le& 
\left\| \hat f - \hat g - \hat{\mathcal{M}}[f-g] \right\|_{L^\infty(\R^d)}^{1/2}
\, \left| f - g - \mathcal{M}[f-g] \right|_4 ^{1/2} 
\\ \nonumber
&&+ 
C \, \sum_{|\alpha| \le 3} \left|M_\alpha[f-g]\right|  
\\  \nonumber
&\le& C \, M_4(f_0+g_0) \, e^{ - \frac{\lambda_2}2 \, t}
\end{eqnarray}
where we have used $||| d |||_4 \le C M_4(d)$. 

Then by writing 
$$
\left| f - g \right|_2 \le  \left| f - g \right|_2^\eta \, \left| f - g \right|_2^{1-\eta}
$$
with $\eta \in (1/2,1)$, using Lemma~\ref{lem:contraction} for the
first term of the right-hand side and the previous decay estimate
\eqref{eq:decay-d-4} for the second term, and $(1-\eta)\le 1/2$, we
obtain \eqref{ineq;tMMC01}.

\medskip\noindent {\sl Step 3. From the difference to the linearized
  semigroup.} A similar line of argument imply the same
estimate on $h_t$ as on the difference $(f_t-g_t)$, that is
inequality~\eqref{ineq;tMMC10}. 

Let us briefly sketch the argument. We define
\[
\tilde h := h - \MM[h], \quad H := \FF(h), \quad F := \FF(f)
\ \mbox{ and } \ \tilde H := \FF(\tilde h) = H - \hat{\mathcal{M}}[h].
\]

The equation satisfies by $H$ is 
\begin{equation*}
  \partial_t H = Q(H,F)
\end{equation*}
and arguing exactly as in Lemma~\ref{lem:contraction} one deduces 
\begin{equation}\label{eq:contraction-h}
  \forall \, t \ge 0, \quad  \left| h_t \right|_2 \le \left| h_0
  \right|_2 = \left| f_0 - g_0 \right|_2. 
\end{equation}

Then the equation satisfied by $\tilde H$ is 
\begin{eqnarray*}
  \partial_t \tilde H 
  &=&  \hat Q (H,F) - \partial_t \hat \MM[h]
  \\ \nonumber
  &=&  \hat Q (\tilde H,F) + \left( \hat Q\left(\hat\MM[h],F\right) 
    - \hat\MM[Q(h,f)] \right).
\end{eqnarray*}
We infer from \eqref{eq:momentsmax} again that 
\[
\forall \, |\alpha|\le 3, \quad \nabla^\alpha _\xi  
\mathcal{\hat M}[Q(h,f)]_{\big|\xi=0} = M_\alpha[Q(h,f)] = \sum_{\beta, \, \beta \le \alpha}
a_{\alpha,\beta} \, M_\beta[h] \, M_{\alpha-\beta}[f]
\]
together with 
\begin{multline*}
  \forall \, |\alpha|\le 3, \quad \nabla^\alpha _\xi \hat
  Q(\mathcal{\hat
    M}[h],F)_{\big|\xi=0} = M_\alpha[Q(\mathcal{M}[h],f)] \\
  = \sum_{\beta, \, \beta \le \alpha} a_{\alpha,\beta} \,
  M_\beta[\mathcal{M}[h]] \, M_{\alpha-\beta}[f] = \sum_{\beta, \,
    \beta \le \alpha} a_{\alpha,\beta} \, M_\beta[h] \, M_{\alpha-\beta}[f].
\end{multline*}
As a consequence, we get
\begin{equation*}
\forall \, \xi \in \R^d,\quad 
\left| \hat{\mathcal{M}}[Q(h,f)] -
  \hat Q( \hat{\mathcal{M}}[h], F ) \right| 
\le C \, |\xi|^4 \,
\left( \sum_{|\alpha| \le 3} \left|M_\alpha[h]\right| \right).
\end{equation*}

On the other hand, arguing as in \cite[Theorem 8.1]{T1} and its
corollary on the linearized equation around $f_t$, we deduce that there
exists some constants $C >0$ and $\lambda_1 >0$ (given by $\lambda_1 :=
\min_{\alpha, |\alpha| \le 3} \{ - a_{\alpha,\alpha}\} \in
(0,\infty)$) so that
\begin{equation}\label{eq:MomentsTanakaH}
\forall \, t \ge 0, \quad 
\left( \sum_{|\alpha| \le 3} \left|M_\alpha[h_t]\right| \right) \le 
C \, e^{-\lambda_1 \, t} \left( \sum_{|\alpha| \le 3} \left|M_\alpha[h_0]\right| \right). 
\end{equation}

Then by exactly the same proof as before we deduce
\[
\sup_{\xi \in \R^d} \frac{|\tilde H_t(\xi)|}{|\xi|^4} \le C \, e^{ -
  \lambda_2 \, t} \, \left( \sup_{\xi \in \R^d} \frac{|\tilde
    H_0(\xi)|}{|\xi|^4} + \sum_{|\alpha| \le 3} \left|M_\alpha[h_0]\right|
\right)
\]
for some $\lambda_2 \in (0,\lambda_1)$, from which we conclude
\begin{equation}\label{eq:decay-h-4}
||| h_t |||_4  \le C \, e^{ - \lambda_2 \, t} \, |||d_0 |||_4
\end{equation}
thanks to \eqref{eq:MomentsTanakaH}, and using that $h_0=d_0$. 

Next we write
\begin{eqnarray*}
\left| h \right|_2 
&\le&
 \left|h - \mathcal{M}[h] \right|_2 +
C \,   \sum_{|\alpha| \le 3} \left|M_\alpha[h]\right|   
\\ \nonumber
&\le& 
\left\| \hat h - \hat{\mathcal{M}}[h] \right\|_{L^\infty(\R^d)}^{1/2}
\, \left| h - \mathcal{M}[h] \right|_4 ^{1/2} + 
C \, \sum_{|\alpha| \le 3} \left|M_\alpha[h]\right|  
\\  \nonumber
&\le& C \, M_4(f_0+g_0) \, e^{ - \frac{\lambda_2}2 \, t}
\end{eqnarray*}
and 
$$
\left| h \right|_2 \le  \left| h \right|_2^\eta \, \left| h \right|_2^{1-\eta}
$$
with $\eta \in (1/2,1)$. Using Lemma~\ref{lem:contraction} for the
first term of the right-hand side and the previous decay estimate
\eqref{eq:decay-h-4} for the second term, and $(1-\eta) \le 1/2$, we
obtain \eqref{ineq;tMMC10}.
\end{proof}

We can now consider the \emph{second-order} term in the expansion of
the semigroup. Let us recall that the crucial point here is to prove
that this second-order term is controlled in terms of some power
\emph{strictly greater than $1$} of the initial difference. 

\begin{lem}\label{lem:a4maxwellfourier2}
  There exists $\lambda \in (0,\infty)$ and, for any $\eta \in (1/2,1)$,
  there exists $C_\eta$ such that for any 
  \[
  {\bf r} \in \RR_{\GG_1} \ \mbox{ and } \ f_0, g_0 \in
  \mathcal P_{\GG_1,{\bf r}},
  \]
we have
  $$
  \left| \omega_t \right|_{4} \le C \, e^{- (1-\eta) \, \lambda \,  t} \, 
  M_4(f_0+g_0)^{\frac12} \,   \big|g_0 - f_0\big|_2 ^{1+\eta} 
    $$
  where 
  \[
  \omega_t := g_t - f_t - h_t = S^{N \! L}_t(g_0) - S^{N \! L}_t(f_0)-
  \DD^{N \! L}_t[f_0] (g_0 - f_0).
  \]
\end{lem}

\begin{rem}
  As proved below $\omega_t$ always has vanishing moments up to order
  $3$, which implies that the norm $| \omega_t |_4$ is
  well-defined.  
\end{rem}

\begin{proof}[Proof of Lemma~\ref{lem:a4maxwellfourier2}]
  We consider the angular cutoff decomposition as in Lemma
  \ref{lem:contraction}. 
Consider the error term
$$
\omega := g-f-h, \qquad \Omega := \hat \omega.
$$
which satisfies the evolution equation 
\[
\partial_t \omega_t =  Q\left(\omega_t,f+g\right) - Q^+(h,f-g), \quad
\omega_0 = 0
\] 
and (in Fourier variable) 
\begin{equation*}
  \partial_t \Omega = \hat Q(\Omega,S) - \hat Q^+(H,D).
\end{equation*}

Let us prove that 
\[
\forall \, |\alpha|\le 3, \ \forall \, t \ge 0, \quad M_\alpha[\omega_t] :=
\int_{\R^d} v^\alpha \, {\rm d}\omega_t(v) = 0.
\]

We shall use again the fact that, for Maxwell molecules, the $\alpha$-th moment of
$Q(f_1,f_2)$ is a sum of terms given by product of moments of $f_1$
and $f_2$ whose orders sum to $|\alpha|$, see
equation~\eqref{eq:momentsmax}. 

We obtain
\[
\forall \, |\alpha|\le 3, \quad 
\frac{{\rm d}}{{\rm d}t} M_\alpha[\omega_t]  = \sum_{\beta \le \alpha} 
  a_{\alpha,\beta} \, M_\beta[\omega_t] \, M_{\alpha-\beta}[f_t+g_t]
 + \sum_{\beta \le \alpha} a_{\alpha,\beta} \, M_\beta[h_t] \, M_{\alpha-\beta}[f_t-g_t]
\]
and since
\[
\forall \, |\alpha|\le 1, \quad M_\alpha[h_t]  = M_\alpha[f_t-g_t] = 0,
\]
we deduce 
\[
\forall \, |\alpha|\le 3, \quad 
\frac{{\rm d}}{{\rm d}t} M_\alpha[\omega_t]  = \sum_{\beta \le \alpha} 
  a_{\alpha,\beta} \, M_\beta[\omega_t] \, M_{\alpha-\beta}[f_t+g_t].
\]
This concludes the proof of the claim about the moments of $\omega_t$
since $\omega_0 = 0$.

We now consider the equation in Fourier form
\[
\partial_t \Omega = \hat Q(\Omega,S) - \hat Q^+(H,D)
\]
and we deduce in distributional sense
$$
\left( \frac{{\rm d}}{{\rm d}t} {|\Omega (\xi)| \over |\xi|^4} + K \,
  {|\Omega (\xi)| \over |\xi|^4}\right) \le \TT_1 + \TT_2 + r_K, 
\quad
r_K \xrightarrow[]{K \to \infty} 0
$$
(depending on some moments of order $1$ of $\omega$, $h$, $d$), and
\begin{eqnarray*}
  \TT_1
  &:=& \sup_{\xi \in \R^3} \int_{\mathbb{S}^{d-1}} 
  {b\left(\sigma \cdot \hat\xi\right) \over |\xi|^4} \, 
  \left( \left| \frac{\Omega (\xi^+) \, S (\xi^-)}{2} \right|  
    + \left|\frac{\Omega (\xi^-) \, S
        (\xi^+)}{2} \right| \right) \, {\rm d}\sigma
  \\
  &\le& \sup_{\xi \in \R^3} \int_{\mathbb{S}^{d-1}} 
  b\left(\sigma \cdot \hat\xi\right) \, 
  \left( {\left|\Omega (\xi^+)\right| \over \left|\xi^+\right|^4} \, 
    {\left|\xi^+ \right|^4 \over \left|\xi\right|^4}
    + {\left|\Omega (\xi^-)\right| \over \left|\xi^-\right|^4} \, 
    {\left|\xi^-\right|^4 \over \left|\xi\right|^4} \right) \, {\rm d}\sigma
  \\
  &\le&  \left( \sup_{\xi \in \R^3} 
    {\left|\Omega (\xi)\right|^2 \over \left|\xi\right|^2} \right) \,  
  \left( \sup_{\xi \in \R^3} 
    \int_{\mathbb{S}^{d-1}} b\left(\sigma \cdot \hat\xi\right) \, 
    \left( \left|\hat\xi^+\right|^4 + \left|\hat\xi^-\right|^4 \right)
    \, {\rm d}\sigma \right) \\
  &\le&  \lambda_K \, \left( \sup_{\xi \in \R^3} 
    {\left|\Omega (\xi)\right| \over \left|\xi\right|^4} \right),
\end{eqnarray*}
where $\lambda_K$ was defined in Lemma \ref{lem:contraction}, and
\begin{eqnarray*}
  \TT_2 &:=& {1\over 2} \,  \sup_{\xi \in \R^3} \int_{\mathbb{S}^{d-1}} 
  {b\left(\sigma \cdot \hat\xi\right) \over |\xi|^4} \, 
  \left| H (\xi^+) \, D (\xi^-) +  H (\xi^-) \, 
    D (\xi^+) \right|  \,  {\rm d}\sigma
  \\ 
  &\le& {1 \over 2} \, \sup_{\xi \in \R^3} \int_{\mathbb{S}^{d-1}} 
  b\left(\sigma \cdot \hat\xi\right)
  \, \left(  {| H (\xi^+) | \over |\xi^+|^2} \, {| D (\xi^-) | \over |\xi^-|^2}
    \, {| \xi^- |^2 \over |\xi|^2}  + {| D (\xi^+) |^2 \over |\xi^+|^2} \, 
    {| H (\xi^-) |^2 \over |\xi^-|^2}
    \, {| \xi^- |^2 \over |\xi|^2}  \right) \, {\rm d}\sigma
  \\
  &\le&
  \int_{\mathbb{S}^{d-1}} b\left(\sigma \cdot
    \hat\xi_0\right) \, 
  \left(1 - \sigma \cdot \hat\xi_0\right) \, {\rm d}\sigma \,  \left| h_t \right|_2 \, \left| d_t \right|_2
\le C_b \, |d_0|_2^{1+\eta} \, |d_t|^{1-\eta} 
  \\
  &\le&  
  C \, e^{-(1-\eta) \, \frac{\lambda_2}2 \, t} \, M_4(f_0+g_0)^{1-\eta} \,  \left| d_0 \right|_2 ^{1+\eta}
\end{eqnarray*}
by using the estimates \eqref{eq:contraction-h} and \eqref{eq:dtTosc2exp-t}.

Hence we obtain
$$
\left( \frac{{\rm d}}{{\rm d}t} {|\Omega (\xi)| \over |\xi|^4} + K \, {|\Omega
    (\xi)| \over |\xi|^4}\right) \le \lambda_K \, \left( \sup_{\xi \in \R^3}
  {\left|\Omega (\xi)\right| \over \left|\xi\right|^4} \right) + C \,
e^{-(1-\eta) \, \frac{\lambda_2}2 \, t} \,  M_4(f_0+g_0)^{1-\eta} \, \left| d_0 \right|_2 ^{1+\eta} + r_K.
$$
We then deduce from the Gronwall inequality, relaxing the cutoff
parameter $K$ as in Lemma \ref{lem:a4maxwellfourier} and choosing
without restriction $\lambda>0$ so that $(1-\eta) \lambda \le
\min\{(1-\eta) \lambda_2/2; \bar \lambda\}$, that
\[
\left( \sup_{\xi \in \R^3}
  {\left|\Omega_t (\xi)\right| \over \left|\xi\right|^4} \right) \le C
\, e^{-(1-\eta) \, \lambda \, t} \,  M_4(f_0+g_0)^{1-\eta} \, \left| g_0 - f_0 \right|_2 ^{1+\eta}.
\]
This concludes the proof (using $(1-\eta) \le 1/2$ on the moment exponent). 
\end{proof}

\subsection{Proof of condition (A5) uniformly in time in Wasserstein
  distance}
We know from \cite{T1} that for $f_0$ and $g_0$ with same momentum and
energy one has
$$
\sup_{t \ge 0}W_2 \left( S^{N\! L}_t f_0, S^{N\! L}_t g_0\right) \le W_2 \left(f_0,g_0\right). 
$$
As a consequence, by using 
\[
[ \cdot ]^*_1 = W_1 \le W_2,
\]
we deduce that {\bf (A5)} holds with 
\[
\Theta(x) = x, \quad \FF_3 =
\hbox{Lip}(\R^d) \ \mbox{ and } \ \mathcal P_{\GG_3} = \PP_2(\R^d)
\]
endowed with the distance $d_{\GG_3} = W_2$ and the contraints
corresponding to the momentum and energy. 

By using Theorem~\ref{theo:abstract} whose assumptions have been
proved above, this proves point (i) in Theorem~\ref{theo:tMM} and the
rate follows from the estimate on $\WW_{W_2 ^2}^N (f)$ from
Lemma~\ref{lem:Rachev&W1}.

By using Lemma~\ref{lem:ComparDistances} in order to relate $\WW_{W_2
  ^2} ( \pi^N _P (f_0 ^N), \delta_{f_0} )$ with $\WW_{W_1} ( \pi^N _P (f_0 ^N),
 \delta_{f_0})$ and then Lemma~\ref{lem:Bpropertybis} in order to estimate
$$ 
\WW_{W_1} \left( \pi^N _P \left(f_0 ^N\right), f_0 \right)
\xrightarrow[]{N \to \infty} 0
$$ 
for the sequence of initial data conditioned on the energy sphere
constructed in Lemma~\ref{lem:Bproperty}, we then deduce point (iii)
in Theorem~\ref{theo:tMM}.

\subsection{Proof of condition (A5) with time growing bounds in Sobolev
  norms}

It is also possible (and in fact easier) to prove, in the cutoff case,
that the weak stability holds in negative Sobolev spaces with
\emph{non-uniform-in-time estimates}. 

\begin{lem}\label{lem:stab-max-sob}
  For any $T \ge 0$ and $s > d/2$ there exists $C_{T,s}$ such that for
  any $f_t$, $g_t$ solutions of the Boltzmann equation for Maxwell
  molecules~\eqref{Maxwelltrue} and initial data $f_0$ and $g_0$,
  there holds
   $$
   \sup_{[0,T]} \left\| f_t - g_t \right\|_{H^{-s}} \le C_{T,s}
   \, 
   \left\| f_0 - g_0 \right\|_{H^{-s}} .
   $$
 \end{lem}


\begin{proof}[Sketch the proof of Lemma~\ref{lem:stab-max-sob}] 
  We integrate \eqref{eq:BoltzMaxD} against $D/(1+|\xi|^{2})^s$:
\[
\frac{{\rm d}}{{\rm d}t} \| D \|_{{H}^{-k}} ^2 = { 1 \over 2 } \, \int_\xi
\int_{\mathbb{S}^{d-1}} b\left(\sigma \cdot \hat \xi\right) \, \frac{\left[ D^-
    S^+ D + D^+ S^- \, D - 2 \, 
    |D|^2\right]}{(1+|\xi|^{2})^s} \, {\rm d}\sigma \, {\rm d}\xi
\]
and we use Young's inequality together with the bounds
\[
\left\|S^+\right\|_\infty, \quad \left\| S^{-} \right\|_\infty 
\le \|f+g\|_{M^1}\le 2
\]
to conclude.
\end{proof}

This proves {\bf (A5)} with the alternate choice 
   \[
   \Theta(x) = x, \quad \FF_3 = H^s(\R^d) \ \mbox{ and } \mathcal P_{\GG_3} =
   \PP_2(\R^d)
   \]
   endowed with the distance of the normed space $\GG_3 = H^s(\R^d)$.
   Then point (ii) in Theorem~\ref{theo:tMM} follows from the abstract
   theorem~\ref{theo:abstract} where the (optimal) rate is provided by  the estimate on
$\WW_{\| \cdot \|_{H^{-s}}^2}^N (f)$ from Lemma~\ref{lem:Rachev&W1}.

\subsection{Proof of infinite-dimensional Wasserstein chaos}
\label{sec:infchaos}

We shall prove Theorem~\ref{theo:max-wasserstein} in this subsection.
We only present the proof in the case of assumption (b), since the
case of assumption (a) is similar. Let us proceed in several
steps. Let us emphasize that we do not search for optimality on the
rate functions in this subsection.

\medskip

\noindent
{\bf Step 1: Finite-dimensional Wasserstein chaos.} 
It is immediate that Theorem~\ref{theo:tMM} implies that, under one
of the two possible assumptions on the initial data, for any given
$\ell \ge 1$, one has 
\[
\sup_{t \ge 0} \left\| \Pi_\ell\left[ f^N _t \right] -  f_t ^{\otimes \ell}
\right\|_{H^{-s}} \le \alpha_0(\ell,N)
\]
for some power law rate function $\alpha_0(\ell,N) \to 0$ as $N \to 0$.

Then by using Lemma~\ref{lem:ComparDistances} we deduce that 
 \[
\sup_{t \ge 0} W_1 \left( \Pi_\ell\left[ f^N _t \right], f_t ^{\otimes \ell}
\right) \le \alpha(\ell,N)
\]
for some power law rate function $\alpha(\ell,N) \to 0$ as $N \to 0$.

Note carefully that at this point our rate function still depends on
$\ell$ and in fact a quick look at Theorem~\ref{theo:tMM} shows that
they scale like $\ell^2$, therefore making impossible to choose $\ell
\sim N$. 

\medskip

\noindent
{\bf Step 2: Infinite-dimensional Wasserstein chaos.} 
We shall use here the following result
obtained in \cite{hm}, see also \cite[Th\'eor\`eme 2.1]{MedpX}: for any $f \in P(\R^d)$ and
sequence $f^N \in P_{\mbox{{\tiny sym}}}(\R^d)$ we have 
$$
\forall \, 1 \le \ell \le N, \quad \frac{W_1\left( \Pi_\ell\left[f^N
      \right], f ^{\otimes \ell} \right)}{\ell} \le C \, \left( W_1\left( \Pi_2\left[f^N
      \right], f ^{\otimes 2} \right)^{\alpha_1} +
    \frac{1}{N^{\alpha_2}} \right) 
$$
for some constructive constant $C$, $\alpha_1$, $\alpha_2 >0$. 

By combining this estimate with the previous step we immediately
obtain 
\[
\sup_{1 \le \ell \le N} \sup_{t \ge 0} \frac{W_1 \left(
  \Pi_\ell\left[ f^N _t \right], f_t ^{\otimes \ell} \right)}{\ell} \le
\alpha(N)
\]
for some power law rate function $\alpha(N) \to 0$ as $N \to 0$. This
concludes the proof of \eqref{eq:max-wass}. 
\medskip

\noindent
{\bf Step 3: Relaxation in Wasserstein distance.} 
We shall prove \eqref{eq:max-wass-relax} and then we shall consider here
initial data $f_0 ^N$ constructed by conditioning $f_0 ^{\otimes N}$
to the Boltzmann sphere $\SS^N(\EE)$.  We first write
$$
\frac{W_1 \left( f^N _t , \gamma^N \right)}{N} \le \frac{W_1 \left( f^N
      _t, f_t ^{\otimes N} \right)}{N} + \frac{W_1 \left(
    f_t ^{\otimes N}, \gamma ^{\otimes N} \right)}{N} 
+ \frac{W_1 \left( \gamma ^{\otimes N},
    \gamma^N \right)}{N}.
$$

Since $f^N _t \to \gamma^N$ in $L^2$ and $f_t \to \gamma$ in $L^1$ as
$t \to +\infty$, one can pass to the limit in the Wasserstein distance
and get from the previous step
\[
\frac{W_1 \left( \gamma ^{\otimes N}, \gamma^N \right)}{N} \le
\alpha(N).
\]
Moreover it is immediate that 
\[
\frac{W_1 \left( f_t ^{\otimes N}, \gamma ^{\otimes N} \right)}{N} =
W_1 \left( f_t, \gamma \right).
\]
Finally it was proved in \cite{GabettaTW95,CGT} that under our assumptions on $f_0$ one
has
$$
\left\| \left( f_t - \gamma\right) \, \langle v \rangle \right\|_{L^1}
\le C \, e^{-\lambda_1 \, t} 
$$ 
for some constants $C >0$ and $\lambda_1 >0$ which implies 
$$
W_1 \left( f_t, \gamma\right) \le \left\| \left( f_t - \gamma\right)
  \, \langle v \rangle \right\|_{L^1} \le C \, e^{-\lambda_1 \, t}.
$$ 
Hence, gathering these three estimates, we deduce that 
\begin{equation}\label{estim:Kac1}
\frac{W_1 \left( f^N _t, \gamma^N \right)}{N} \le 2 \, \alpha(N) + C
\, e^{-\lambda_1 \, t}
\end{equation}
for some polynomial rate $\alpha(N) \to 0$ as $N \to +\infty$. 

On the other hand, it was proved in \cite{CarlenGeronimoLoss2008} that
there exists $\lambda_2 > 0$ such that
$$
\forall \, N \ge 1, \,\, \forall \, t \ge 0, \quad \left\| h^N - 1
\right\|_{L^2\left(\SS^N(\EE),\gamma^N\right)} \le e^{-\lambda_2 \, t} \,
\left\| h^N_0 - 1 \right\|_{L^2\left(\SS^N(\EE),\gamma^N\right)},
$$
where $h^N = {\rm d}f^N/{\rm d}\gamma^N$ is the Radon-Nikodym
derivative of $f^N$ with respect to the measure $\gamma^N$ so that
$f^N = h^N \, \gamma^N$. When $f^N_0 = [f_0^{\otimes N}]_{\SS^N(\EE)}$
with $f_0 \in \PP_4(\R^d)$ we easily bound from above the right-hand
side term by
$$
\left\| h^N_0 - 1 \right\|_{L^2\left(\SS^N(\EE),\gamma^N\right)} \le A^N,
$$
where $A = A(f_0) > 1$.  Then by the Cauchy-Schwartz inequality we also have
\begin{equation*}
\left\| h^N -  1 \right\|_{L^2\left(\SS^N(\EE),\gamma^N\right)} 
\ge \left\| h^N -  1 \right\|_{L^1\left(\SS^N(\EE),\gamma^N\right)} 
\end{equation*}
and the Wasserstein distance can be controlled as 
\begin{eqnarray*}
W_1\left(f^N,\gamma^N\right) &=& \sup_{\| \varphi
  \|_{C^{0,1}(\R^{dN})} \le 1} \int_{\R^{dN}} \varphi \, \left({\rm d}f^N -
{\rm d}\gamma^N \right) \\
&=& \sup_{\| \varphi
  \|_{C^{0,1}(\R^{dN})} \le 1} \int_{\R^{dN}} \left( \varphi -
  \varphi(0) \right) \, \left({\rm d}f^N -
{\rm d}\gamma^N \right) \\
&\le& \int_{\R^{dN}}  \left( \sum_{i=1}^N \left|v_i\right|\right) \,
\left|{\rm d}f^N - {\rm d}\gamma^N\right|
\\
&\le& N \, \EE^{1/2} \, \left\| h^N -  1 \right\|_{L^1\left(\SS^N(\EE),\gamma^N\right)}.
\end{eqnarray*}
We hence deduce 
\begin{equation}\label{estim:Kac2}
\forall \, N \ge 1, \ \forall \, t \ge 0, \quad
\frac{W_1\left(f^N _t, \gamma^N\right)}{N} \le A^N \, e^{-\lambda_2 \, t}.
\end{equation}

Finally by combining \eqref{estim:Kac1} when 
\[
N \ge N(t):= \frac{\lambda_2 t}{2 \ln A}, 
\]
and \eqref{estim:Kac2} when $N \le N(t)$, we easily obtain
$$
\forall \, N \ge 1, \ \forall \, t \ge 0, \quad
\frac{W_1\left(f^N _t, \gamma^N\right)}{N} \le \min\left\{
e^{-\frac{\lambda_2}2 t}; \alpha(N(t)) + C \, e^{-\lambda_1 t}\right\}
=: \beta(t) 
$$ 
for some polynomial rate $\beta(t) \to 0$ as $t \to +\infty$, which
concludes the proof of \eqref{eq:max-wass-relax}.

\section{Hard spheres}
\label{sec:hardspheres}
\setcounter{equation}{0}
\setcounter{theo}{0}

\subsection{The model}
\label{sec:modelHS}
The limit equation was introduced in Subsection~\ref{sec:introEB}
and the stochastic model has been already discussed
Subsection~\ref{sec:modelEBbounded}. 

We consider here the case of the Master equation
\eqref{eq:BoltzBddKolmo}, \eqref{defBoltzBddGN} and the limit
nonlinear homogeneous Boltzmann equation \eqref{el},
\eqref{eq:collop}, \eqref{eq:rel:vit} with 
\begin{equation}\label{hypHS}
B(z,\cos \theta) = \Gamma(z) \, b(\cos \theta) = \Gamma(z) = |z|.
\end{equation}

\subsection{Statement of the result}
\label{sec:statementHS}
Our fluctuations estimate result for this model then states as follows:

\begin{theo}[Hard spheres detailed chaos estimates]\label{theo:HS} 
 Assume that the collision kernel $B$ satisfies  \eqref{hypHS}.
Let us consider a family of $N$-particle initial conditions 
$
f_0 ^N \in P_{\mbox{{\scriptsize {\em sym}}}}((\R^d)^N)
$
and the associated $N$-particle system dynamics
$$
 f^N _t = S^N _t \left(f_0^N \right).
$$
Let us also consider a centered $1$-particle initial distribution $f_0 \in
  P(\R^d)$ with energy $\EE \in (0,+\infty)$
  $$
  \int_{\R^d} v \, {\rm d}f_0 = 0, \quad \EE := \int_{\R^d} |v|^2 \, {\rm d}f_0 \in (0,\infty),
  $$
  and the associated solution
  \[
  f_t = S^{N \! L} _t \left(f_0\right)
  \]
 of the limit mean-field equation.
  
  Let us finally fix some $\delta \in (0,1)$. Then we have the
  following results:  
\begin{itemize}
\item[(i)] Suppose that $f_0$ has compact support 
$$
\hbox{{\em Supp}} \, f_0 \subset \left\{ v \in \R^d, \,\, |v| \le A \right\}
$$
and that the $N$-particle initial data are tensorized
$$
\forall \, N \ge 1, \quad f^N _0 = f_0 ^{\otimes N}.
$$
 
Then for any $T \in (0,\infty)$ there are
  \begin{itemize}
  \item some constants $k_1 \ge 2$ depending on
    $\delta$ and $A$;
  \item some constant $C_{\delta,T} >0$ depending on
    $\delta$, $T$ and $A$, and blowing up as $\delta
    \to 1$;
  \item some constant $C_{b,T} >0$ depending on the collision kernel and $T$,
  \end{itemize}
such that for any $\ell \in \N^*$, and for any
  \[ \varphi = \varphi_1 \otimes \varphi_2 \otimes \dots \otimes \,
  \varphi_\ell \in W^{1,\infty}(\R^d)^{\otimes \ell}, \] 
  we have 
  \begin{multline*} 
    \forall \, N \ge 2 \, \ell, \quad 
    \sup_{t \in [0,T]}\left| \left \langle
        \left( S^N_t(f_0 ^N) - \left( S^{N \! L}_t(f_0)
          \right)^{\otimes N} \right), \varphi \right\rangle \right|
    \\ \le  \|\varphi\|_{ W^{1,\infty}(\R^d)^{\otimes
        \ell} } \, 
    \Bigg[ \frac{2 \, \ell^2 }{N} +
    \frac{C_{\delta,T} \, \ell^2 \|f_0\|_{M^1_{k_1}}}{N^{1-\delta}} + \ell \, e^{C_{b,T} \, A} \,
    \theta(N) \Bigg].
  \end{multline*}

  The last term of the right-hand side (which is also the dominant
  error term as $N$ goes to infinity in our estimate) is given by
  \[
  \theta(N) = \frac{C}{(1+ |\ln N|)^{\alpha}} 
  \]
  for some constants $C, \alpha >0$. 

\item[(ii)] Under the same setting but assuming instead for the
  initial datum of the mean-field limit \beqn\label{eq:HSdatumSphere}
  f_0 \in L^\infty\left(\R^d\right) \ \mbox{{\em s.t.}} \ \int_{\R^d}
  e^{z \, |v|} \, {\rm d}f_0(v) < + \infty \eeqn for some $z >0$, and
  taking for the $N$-particle initial data the sequence $(f_0^N)_{N
    \ge 1}$ constructed in Lemma~\ref{lem:Bproperty} and
  \ref{lem:Bpropertybis} by conditioning to the Boltzmann sphere
  $\SS^N(\EE)$, then the solution $f^N_t = S^N_t(f_0 ^N)$ has its
  support included in $\SS^N(\EE)$ for all times
  \begin{equation}\label{eq:supportfNtspheres-HS}
    \forall \, t \ge 0, \quad \hbox{{\em Supp}} \, f^N_t \subset \SS^N(\EE) 
  \end{equation}
and there are
  \begin{itemize}
  \item some constants $k_1 \ge 2$ depending on
    $\delta$ and $\mathcal{E}$;
  \item some constant $C_{\delta} >0$ depending on $\delta$ and
    $\mathcal{E}$, and blowing up as $\delta \to 1$;
  \item some constant $C_{b} >0$ depending on the collision kernel
    and the above exponential moment bound on $f$,
  \end{itemize}
such that for any $\ell \in \N^*$, and for any
  \[ \varphi = \varphi_1 \otimes \varphi_2 \otimes \dots \otimes \,
  \varphi_\ell \in W^{1,\infty}(\R^d)^{\otimes \ell}, \] 
  we have 
  \begin{multline*} 
    \forall \, N \ge 2 \, \ell, \quad 
    \sup_{t \ge 0}\left| \left \langle
        \left( S^N_t(f_0 ^N) - \left( S^{N \! L}_t(f_0)
          \right)^{\otimes N} \right), \varphi \right\rangle \right|
    \\ \le  \|\varphi\|_{ W^{1,\infty}(\R^d)^{\otimes
        \ell} } \, 
    \Bigg[ \frac{2 \, \ell^2 }{N} +
    \frac{C_{\delta} \, \ell^2 \|f_0\|_{M^1_{k_1}}}{N^{1-\delta}} + \ell \, e^{C_{b} \, A} \,
    \theta(N) \Bigg], 
  \end{multline*} 
  with the same estimate on the rate $\theta(N)$ as in (i). This
  proves the propagation of chaos, uniformly in time.
\end{itemize}
\end{theo}
\medskip

We now state again another version of the propagation of chaos
estimate, in Wasserstein distance, but most importantly which is valid
\emph{for any number of marginals}, at the price of a possibly worse
(but still constructive) rate. Combined with previous results on the
relaxation of the $N$-particle system we also deduce some estimate of
relaxation to equilibrium \emph{independent of $N$} and, again, for
any number of marginals.

\begin{theo}[Hard spheres Wasserstein chaos]
  \label{theo:hs-wasserstein}
We consider the same setting as in Theorem~\ref{theo:HS}, where the
  initial data are chosen as follows:
  \begin{itemize}
  \item[(a)] either $f_0$ is compactly supported and $f_0 ^N = f_0 ^{\otimes
      N}$,
  \item[(b)] or $ f_0$ satisfying \eqref{eq:HSdatumSphere}and $f_0 ^N$
    is constructed by Lemma~\ref{lem:Bproperty} by conditioning to the
    Boltzmann sphere $\SS^N(\EE)$.
  \end{itemize}

  Then in the case (a) we have for any $T>0$
\begin{equation}\label{eq:hs-wass-a}
  \forall \, N \ge 1, \ \forall \, 1 \le \ell  \le N, \quad \sup_{t
    \in [0,T]} {W_1 \left( \Pi_{\ell} f^N_t, f_t ^{\otimes \ell}   \right) \over \ell} \le \alpha_T(N)
\end{equation}
for some $\alpha_T(N) \to 0$ as $N \to \infty$ like a power of a
logarithm, and possibly depending on $T$. 

In the case (b) the solution $f^N_t = S^N_t(f_0 ^N)$ has its support
included in $\SS^N(\EE)$ for all times and this estimate can be made
uniform in time:
\begin{equation}\label{eq:hs-wass-b} 
\forall \, N \ge 1, \ \forall \, 1
\le \ell \le N, \quad \sup_{t \ge 0} {W_1 \left( \Pi_{\ell} f^N
    _t,   f_t ^{\otimes \ell}   \right) \over \ell} \le
\alpha(N)
\end{equation}
for some $\alpha(N) \to 0$ as $N \to \infty$ like a power of a
logarithm.  

Moreover, still in the case (b), we have 
\begin{equation}\label{eq:hs-wass-relax}
  \forall \, N \ge 1, \ \forall \, 1 \le \ell  \le N, \  \forall \, t
  \ge 0, \quad  
  {W_1 \left( \Pi_{\ell} f^N
      _t, \Pi_\ell \left( \gamma^N \right) \right) \over \ell} \le
  \beta(t) 
\end{equation}
for some rate $\beta(t) \to 0$ as $t \to \infty$ like a power of
logarithm, where $\gamma$ is the centered Gaussian equilibrium with
energy $\EE$ and $\gamma^N$ is the uniform probability measure on
$\mathcal S^N (\EE)$.
\end{theo}

\bigskip

In order to prove Theorem~\ref{theo:HS}, we shall prove assumptions
{\bf (A1)-(A2)-(A3)-(A4)-(A5)} of Theorem~\ref{theo:abstract} with
$T<\infty$ or $T=\infty$, and with suitable functional spaces. The
application of the latter theorem then exactly yields
Theorem~\ref{theo:HS} by following carefully each constant computed
below.  We fix
\[
\FF_1=\FF_2=C_b(\R^d) \ \mbox{ and } \ \FF_3 = \mbox{Lip}(\R^d).
\]

Then the proof of Theorem~\ref{theo:hs-wasserstein} is deduced from
Theorem~\ref{theo:HS} in a similar way as
Theorem~\ref{theo:max-wasserstein} was deduced from
Theorem~\ref{theo:tMM}, see Subsection~\ref{subsec:wass-hs}.

\subsection{Proof of condition (A1)} 
From the discussion made in section~\ref{sec:MaxA1} we easily
see that for the hard spheres model the operator $G^N$ is bounded from
$C_{-k+1,0}(\R^{dN})$ onto $C_{-k,0}(\R^{dN})$ for any $k \in \R$.
Since $G^N$ is close, dissipative and $C_{-k+1,0}(\R^{dN})$ is dense
in $C_{-k,0}(\R^{dN})$, the Hille-Yosida theory implies that $G^N$
generates a Markov type semigroup $T^N_t$ on $C_{-k,0} (\R^{dN})$ and
we may also define $S^N_t$ by duality as a semigroup on
$\PP_k(\R^{dN})$.  The nonlinear semigroup $S^{N\! L}_t$ is also well
defined on $\PP_k(\R^d)$, $k \ge 2$, see for instance
\cite{Fo-Mo,EM,Lu-Mouhot}.

Lemma~\ref{lem:momentsN} was proved both for Maxwell molecules and
hard spheres. 
It first shows that
$$
\forall \, t \ge 0, \quad \hbox{Supp} \, f^N_t \subset  \mathbb{E}_N := 
\{ÊV \in (\R^d)^N; \, \langle \mu^N_V, {\bf m}_{\GG_1} \rangle \in \RR_{\GG_1} \},
$$
where ${\bf m}_{\GG_1} : \R^d \to \R_+ \times \R^d$, ${\bf m}_{\GG_1}(v) = (|v|^2,v)$ and
\begin{equation}\label{eq:RG1cas1}
\RR_{\GG_1} :=\left\{ {\bf r} = (r_0,r') \in \R_+ \times \R^d, \,\, |r'|^2 \le r_0 \le \EE_0\right\}
\ \mbox{with } \EE_0 = A^2 \ \mbox{ in case (i)},
\end{equation}
and
\begin{equation}\label{eq:RG1cas2}
\RR_{\GG_1} := \{ (\EE_0,0) \} \subset \R_+ \times \R^d, \quad \EE_0 := \EE,  
\ \mbox{ in case (ii)}. 
\end{equation}


It also proves that for any $k \ge 2$,
\[
\sup_{t \ge 0} \left \langle f^N _t, M_k ^N \right\rangle \le C^N_k 
\]
where $C^N_k$ depends on $k$, $\EE_0$, on the collision kernel and on
the initial value 
\[
\langle f^N _0, M_k ^N \rangle
\]
which is uniformly bounded in $N$ in terms of $k$ and $\EE_0$. This
shows that {\bf (A1)-(ii)} holds with $m_1(v) := |v|^{k_1}$ for any
$k_1 \ge 2$. The precise value of $k_1$ shall be chosen in
Section~\ref{sec:HSA3bis}.

As for \textbf{(A1)-(iii)}, we remark that for a given $N$-particle
velocity 
\[
V = (v_1, \dots, v_N) \in \R^{dN},
\]
we have
 $$
 V \in \hbox{Supp} \, f_0^{\otimes N} \ \Longleftrightarrow \ \forall
 \, i=1, \dots, N, \ v_i \in \hbox{Supp} \, f_0
 $$ 
which implies 
$$
\forall \, i=1,\dots, N, \quad m_{\GG_3}(v_i) \le m_{\GG_3}(A) \ \mbox{
  with } \ m_{\GG_3}(v) := e^{a \, |v|}
$$ 
for any constant $a>0$, which shall chosen later on. 

We conclude that
$$
\hbox{Supp} \, f_0^{\otimes N} \subset \left\{ V \in \R^{dN}; \
  M^N_{m_{\GG_3}}(V) \le m_{\GG_3}(A) \right\},
$$
and {\bf (A1)-(iii)} holds for the exponential growing weight
$m_{\GG_3}$. 

\subsection{Proof of condition (A2)} \label{sec:HSA2} For a given $k_1 \ge 2$,
let us define the space of probability measures
\[
\mathcal P_{\GG_1} := \left\{ f \in P(\R^d) \, ; \ 
  M_{k_1}(f)   <+\infty \right\}, 
\]
 the sets of constraints $\RR_{\GG_1}$ given by \eqref{eq:RG1cas1} or \eqref{eq:RG1cas2}, 
the constrained space (for ${\bf r} \in \RR_{\GG_1}$)
\begin{equation*}
  \mathcal P_{\GG_1,{\bf r}} := \left\{ f \in \PP_{k_1}(\R^d) \, ; 
  \ \langle f, |v|^2
  \rangle = r_0, 
   \quad \forall \, i=1, \dots, d, \ \langle f, v_i \rangle = r_i  \right\}, 
\end{equation*}
and the vector space
$$
\GG_1 := \left\{ h \in M^1_{k_1}(\R^d) \, ; \quad \forall \, i=1,
  \dots, d, \ \langle h, v_i \rangle = \langle h, 1 \rangle = \langle
  h, |v|^2 \rangle = 0 \right\}
$$
endowed with the total variation norm $\| \cdot \|_{\GG_1} := \|
\cdot \|_{M^1}$. 
We also define
\[
\BB \mathcal P_{\GG_1,a} := \left\{ f \in \mathcal P_{\GG_1} \, ; \ M_{k_1}(f) \le a \right\}
\] 
as well as for any ${\bf r} \in \RR_{\GG_1}$ the (possibly empty)
bounded constrained space
\[
\BB \mathcal P_{\GG_1,a,{\bf r}} := \Bigl\{ f \in \BB \mathcal
P_{\GG_1,a} \, ; \  \ \langle f, |v|^2 \rangle = r_0, \  \forall \, i=1, \dots, d, \ \langle f, v_i \rangle =
r_i \Bigr\}
\] 
endowed with the distance $d_{\GG_1}$ associated to the norm
$\| \cdot \|_{\GG_1}$.  
\medskip
 
The proof of the assertion {\bf (A2)-(i)} 
is postponed to section~\ref{sec:HSA4}, where we prove in
\eqref{estim:dt} a H\"older continuity of the flow in $\BB
\mathcal P_{\GG_1,a,{\bf r}}$, ${\bf r} \in \RR_{\GG_1}$. 
 \medskip

 Let us prove the assertion {\bf (A2)-(ii)}, that is the fact that the
 operator $Q$ is bounded and H\"older continuous from $\mathcal \BB
 \mathcal P_{\GG_1,a,{\bf r}}$ to $\GG_1$.  For any $f,g \in \mathcal
 \BB \mathcal P_{\GG_1,a,{\bf r}}$ we have
\begin{eqnarray*} \nonumber 
\left\| Q(g,g) - Q(f,f) \right\|_{M^1} &=& 
\left\| Q(g-f,g+f) \right\|_{M^1} \\ \nonumber &\le& 2 \int_{\R^d} \int_{\R^d}
  \int_{\mathbb{S}^{d-1}} b(\cos \theta) \, |v-v_*| \, |f - g | \,
  |f_*+g_*| \, {\rm d}\sigma \, {\rm d}v_* \, {\rm d}v \\ 
  &\le& 2 \,
  (1+a) \, \| b \|_{L^1} \, \| (f - g) \, \langle v \rangle \|_{M^1}.
\end{eqnarray*}
We deduce that 
$$
\| Q(g,g) - Q(f,f) \|_{M^1} \le  2 \, (1+a)^{3/2} \, \| b \|_{L^1} \,
\| f - g \|_{M^1}^{1/2}
$$
which yields 
\[
Q \in C^{0,1/2}(\BB \mathcal P_{\GG_1,a,{\bf r}}; \GG_1)
\]
and also implies that $Q$ is bounded on $\BB \mathcal P_{\GG_1,a,{\bf
    r}}$ since we can choose $g$ to be a Maxwellian distribution, for
which $Q(g,g)=0$.

\subsection{Proof of condition (A3)} \label{sec:HSA3} 

Let us define the weight
\[
\Lambda_1(f) := M_{k_1}(f) = \big\langle f,  \, \langle
v\rangle^{k_1} \big\rangle
\]
(this means that we choose $m'_{\GG_1} (v) = \langle v\rangle^{k_1}$
in assumption {\bf (A3)}). 

We claim that there exists a constant $C_{k_1} >0$ (depending on
$k_1$) such that for any $\eta \in (0,1)$ and any function
\begin{equation}\label{eq:PhiCapC1HS} 
  \Phi \in \bigcap_{{\bf r } \in \RR_{\GG_1}  }
  C_{\Lambda_1}^{1,\eta}(\mathcal P_{\GG_1,{\bf r}};\R),
\end{equation}
we have
\begin{equation}\label{eq:HSA3} 
\forall \, V \in  \mathbb{E}_N, \quad
\left| G^N (\Phi \circ \mu^N_V) - (G^\infty \Phi) (\mu^N_V) \right| 
\le C_{k_1} \, \EE_0 \, 
\left( \sup_{ {\bf r} \in \RR_{\GG_1}  } [ \Phi
  ]_{C^{1,\eta}_{\Lambda_1} (\mathcal P_{\GG_1,{\bf r}})} \right) \, {M^N_{k_1}(V)\over N^\eta}
\end{equation}
where we recall that
\[
M^N_{k_1}(V) := \frac{1}{N} \, \sum_{i=1} ^N \langle v_i \rangle^{k_1} .
\]
This would prove assumption {\bf (A3)} with the rate 
\[
\eps(N) = \frac{C_{k_1}}{N^{\eta}}, \qquad \eta := 1 - \delta.
\]

\medskip

For a given $\Phi$ satisfying \eqref{eq:PhiCapC1HS}, for any $V \in \Ee_N$  let us set
\[
\phi := D\Phi[ \mu_V ^N] 
\]
and remark that 
$$
{\bf r}_V := 
\Big(\left\langle \mu^N_V, |z|^2
\right\rangle, \left\langle \mu^N_V,z_1 \right\rangle, \dots, \left\langle
  \mu^N_V,z_d \right\rangle \Big) \in \RR_{\GG_1} 
$$
where $z=(z_1,\dots,z_d) \in \R^d$ has to understood as the blind
integration variable in the duality bracket. 

We then compute
\begin{eqnarray*}
&&  G^N \left(\Phi \circ  \mu ^N _V\right) 
=  \frac{1}{2N} \, \sum_{i,j=1} ^N |v_i - v_j| \, 
  \int_{\mathbb{S}^{d-1}} \left[ \Phi\left(\mu^N_{V^*_{ij}}\right) 
    - \Phi\left(\mu^N _V\right) \right] \, b(\theta_{ij}) \, {\rm d}\sigma 
\\
&&= \frac{1}{2N} \, \sum_{i,j=1} ^N |v_i - v_j| \,
  \int_{\mathbb{S}^{d-1}} \left\langle \mu^N _{V^*_{ij}}
    - \mu^N _V , \phi \right\rangle \, b(\theta_{ij}) \, {\rm d}\sigma 
\\
&&+  \frac{[\Phi]_{ C_{\Lambda_1} ^{1,\eta}(\mathcal P_{\GG_1,{\bf r}_V}) }}{2N} \, 
  \sum_{i,j=1} ^N |v_i - v_j|  \times \\ 
  && \qquad \qquad \int_{\mathbb{S}^{d-1}} 
  \!\!\! \ \max\left\{ M_{m'_{\GG_1}}\left(\mu^N _{V^*_{ij}}\right) ;\, 
    M_{m'_{\GG_1}}\left( \mu^N _{V}\right) \! \right\} \, 
  \OO \! \left( \left\| \mu^N _{V^*_{ij}} 
      - \mu^N _V \right\|_{M^1}^{1+\eta} \right)  \, {\rm d}\sigma \\
&& =:  I_1(V) + I _2(V).
\end{eqnarray*}

For the first term $I_1(V)$ we argue similarly than in the proof of
{\bf (A3)} for the Maxwell molecules case in
section~\ref{subsect:ProofA3tMM}, and we get
\[
I_1(V) = \left\langle Q\left(\mu^N _V, \mu^N _V\right), \phi
\right\rangle = \left(G^\infty\Phi\right)\left(\mu^N _V\right).
\]

As for the second term $I_2(V)$, using \eqref{eq:MkmuV*} and
$\|\mu^N_V - \mu^N_{V^*_{ij}} \|_{M^1} \le 4/N$,
we deduce 
\begin{eqnarray*}
  |I_2(V)|  
  &\le& C_{k_1} \, {  M^N _{k_1}(V)} \, 
  [\Phi]_{C^{1,\eta}_\Lambda(\mathcal P_{\GG_1,{\bf r}_V})}  \, 
  \left( \frac{1}{2N} \, \sum_{i,j=1} ^N |v_i - v_j| \,
  \left( \frac4N \right)^{1+\eta} \right) \\
  & \le & C_{k_1} \, {  M^N _{k_1}(V)} \, 
 [\Phi]_{C^{1,\eta}_\Lambda(\mathcal P_{\GG_1,{\bf r}_V})}  \,  \left( {1 \over N^\eta} \, 
 \frac{1}{N^2} \, \sum_{i,j=1} ^N \Big( \left\langle v_i \right\rangle + 
    \left\langle v_j \right\rangle \Big) \right) \\
  & \le & \frac{C_{k_1}}{N^\eta} \, M^N _{k_1}(V) \, M^N _{2}(V)  \,
  [\Phi]_{C^{1,\eta}_\Lambda(\mathcal P_{\GG_1,{\bf r}_V})}.
 \end{eqnarray*}
We then use the elementary inequality the energy bound to deduce
\[
|I_2(V)|  \le \frac{C_{k_1} \, \EE_0}{N^\eta} \, M^N _{k_1}(V) \,
  [\Phi]_{C^{1,\eta}_\Lambda(\mathcal P_{\GG_1,{\bf r}_V})}.
\]

We conclude that \eqref{eq:HSA3} holds by combining the two last
estimates on $I_1$ and $I_2$.

\subsection{Proof of condition (A4) with time growing bounds} 
\label{sec:HSA4}

Let us consider some $1$-particle initial data 
\[
f_0, g_0 \in \mathcal P_{\GG_1}.
\]

In a similar way as in the previous section, we then define (under the
assumption \eqref{hypHS} on the collision kernel) the associated
solutions $f_t$ and $g_t$ to the nonlinear Boltzmann equation
\eqref{el}, as well as
\[
h_t := \DD^{N\!  L}_t\left[f_0\right](g_0-f_0)
\]
the solution to the linearized Boltzmann equation around $f_t$. These
solutions are given by
\[
\left\{
\begin{array}{l}
\partial_t f_t = Q(f_t,f_t), \qquad f_{|t=0} = f_0 \vspace{0.3cm} \\
\partial_t g_t = Q(g_t,g_t), \qquad g_{|t=0} = g_0 \vspace{0.3cm} \\
\partial_t h_t = 2 \, Q(f_t,h_t), \qquad h_{|t=0} = h_0 := g_0 - f_0.
\end{array}
\right.
\]
We also define as before
\[
\omega_t := g_t - f_t - h_t.
\]

We shall now again \emph{expand the limit nonlinear semigroup} in terms
of the initial data, around $f_0$. 
The goal is to prove assumption {\bf
  (A4)}.
This imposes the choice of weight 
\[
\Lambda_2 (f) = \Lambda_1(f)^{\frac12} = \sqrt{M_{k_1}(f)}. 
\]

\begin{lem}\label{lem:expansionHS} 
  For any given energy $\EE_0 > 0$ and any $\eta \in (0,1)$ there exists
\begin{itemize}
\item some constant $\bar k_1 \ge 2$ (depending on $\EE_0$ and $\eta$),
\item some constant $C>0$ (depending on $\EE_0$),
\end{itemize}
such that for $k_1 \ge \bar k_1$, for any
\[
{\bf r} \in \RR_{\GG_1}   \  \mbox{ and } \ f_0, g_0 \in \mathcal P_{\GG_1,{\bf r}} 
\]
and for any $t \ge 0$ we have
\begin{eqnarray} \label{estim:dt} && \left\| g_t - f_t \right\|_{M^1 _2}
  \le e^{C \, (1+t)} \, \sqrt{  M_{k_1}(f_0+g_0) } \, \left\| f_0 - g_0 \right\|_{M^1}^{\eta},
  \\ \nonumber \\ \label{estim:ht} && \left\| h_t \right\|_{M^1 _2}
  \le e^{C \, (1+t)}
  \, \sqrt{ M_{k_1}(f_0+g_0) } \, \left\| f_0 - g_0  \right\|_{M^1}^{\eta}, 
  \\ 
  \nonumber \\ \label{estim:dt-ht} && \left\|\omega_t
  \right\|_{M^1 _2} \le e^{C \, (1+t)} \, \sqrt{  
      M_{k_1}(f_0+g_0)  } \, \left\| f_0 - g_0
  \right\|_{M^1}^{1+\eta}.
  \end{eqnarray}
\end{lem}

\begin{proof}[Proof of Lemma~\ref{lem:expansionHS}]
  We proceed in several steps and number the constants for clarity.

Let us define 
$$
\forall \, h \in M^1(\R^d), \quad \| h \|_{M^1_k} := \int_{\R^d}
\langle v \rangle^k \, {\rm d}|h|(v), \quad \| h \|_{M^1_{k,\ell}} :=
\int_{\R^d} \langle v \rangle^k \, (1 + \ln \langle v \rangle)^\ell
\, {\rm d}|h|(v).
$$

\smallskip\noindent{\sl Step 1. The strategy.}
Existence and uniqueness for $f_t$, $g_t$ and $h_t$ is a consequence
of the following important stability argument that we use several
times. This estimate is due to DiBlasio~\cite{DiB74} in a $L^1$
framework, and it has been recently extended to a measure framework in
\cite[Lemma 3.2]{EM} (see also \cite{Fo-Mo} and \cite{Lu-Mouhot} for
other argument of uniqueness for measure solutions of the spatially
homogeneous Boltzmann equation).

Let us sketch the argument for $h$. We first write
\begin{multline}\label{eq:htM12}
\frac{{\rm d}}{{\rm d}t} \int \langle v \rangle^2 \, {\rm d} |h_t|(v) 
\le \int\!\! \! \int\!\! \! \int {\rm d}|h_t|(v) \, {\rm d}f_{t} (v_*)\, |u| \,
b(\theta) \, \Big[ \langle v' \rangle^2 \!+\! \langle v'_* \rangle^2\!
-\! \langle v \rangle^2 \!-\!  \langle v_* \rangle^2 \Big] \, {\rm d}\sigma
\\ + 2 \int\!\!\! \int\!\! \! \int {\rm d} |h_t|(v) \, {\rm d} f_{t} (v_*)\, |u| \,
b(\theta) \, \langle v_* \rangle^2 \, {\rm d}\sigma 
\end{multline}
(this formal computation can be justified by a regularization
proceedure, we refer to \cite{EM} for instance). Since the first
term vanishes, we deduce that
\begin{equation}\label{ineq:htHS}
\frac{{\rm d}}{{\rm d}t} \| h_t \|_{M^1_2} \le C_1 \, \| f \|_{M^1_3} \, \| h_t \|_{M^1_2}
\end{equation}
for some constant $C_1 >0$ only depending on $b$. 

Then in the case when
\begin{equation}\label{bornemomentintegral3}
\|f_s \|_{M^1_3} \in L^1(0,t) \ \mbox{ on some time interval } \ s\in [0,t]
\end{equation}
we may integrate this differential inequality and we deduce that the
solution $h$ to the linear equation $\partial_t h = 2 Q(f_t,h)$ is
unique in $M^1_2$.

More precisely, we have established
\begin{equation}\label{eq:unifh}
\sup_{s \in [0,t]} \left\| h_s \right\|_{M^1_2} \le  \left\| g_0 - f_0 \right\|_{M^1 _2} \, 
\exp \left( C_1 \, \int_0 ^t \left\| f_s \right\|_{M^1 _3} \, {\rm d}s \right), 
\end{equation}
and similar arguments imply
\begin{equation}\label{eq:unifd}
\sup_{s \in [0,t]} \left\| f_s - g_s \right\|_{M^1_2} \le  \left\| g_0 - f_0 \right\|_{M^1 _2} \, 
\exp \left( C_1 \, \int_0 ^t \left\| f_s +g_s \right\|_{M^1 _3} \, {\rm d}s \right).
\end{equation}

It is worth mentioning that one cannot prove
\eqref{bornemomentintegral3} under the sole assumption 
\[
\left\| f_0 \right\|_{M^1_{2}} < \infty
\]
on the initial data since it would contradict the non-uniqueness
result of~\cite{LuW02}. However, as we prove in \eqref{fM13L1t} below,
one may show (thanks to the Povzner inequality, as developped in
\cite{MW99,Lu1999}) that \eqref{bornemomentintegral3} holds as soon as
\[
\| f_0 \|_{M^1 _{2,1}} < \infty.
\]
This will be a key step for establishing \eqref{estim:dt} and
\eqref{estim:ht}.

\medskip Now, our goal is to estimate the $M^1 _2$ norm
of  
\[ 
\omega_t := g_t -f_t - h_t
\]
in terms of $\| g_0 - f_0 \|_{M^1 _2}$.  The measure $\omega_t$ satisfies
the evolution equation:
\[
\partial_t \omega_t = Q(g_t,g_t) - Q(f_t,f_t) - 
Q(h_t,f_t) - Q(f_t,h_t), \quad \omega_0 = 0,
\]
which can be rewritten as
\[
\partial_t \omega_t = Q(\omega_t, f_t+g_t) + Q(h_t,g_t-f_t).
\]
The same arguments as in \eqref{eq:htM12}-\eqref{ineq:htHS} yield the
following differential inequality
\[ 
\frac{{\rm d}}{{\rm d}t} \left\|\omega_t\right\|_{M^1 _2} \le 
C_2 \, \left\|\omega_t\right\|_{M^1 _2} \, 
 \left\|f_t+g_t\right\|_{M^1_3}  + \left\| Q(h_t,g_t-f_t)
\right\|_{M^1_2}, \quad  \left\|\omega_0\right\|_{M^1 _2} = 0,
\]
for some constant $C_2 >0$ depending on $b$. 

We deduce 
\begin{equation*}
 \sup_{s \in [0,t]} \left\|\omega_s\right\|_{M^1 _2} \le  
  \left( \int_0 ^t \left\| Q(h_s,f_s-g_s)
    \right\|_{M^1_2} \, {\rm d}s \right) \, 
  \exp\left(C_2 \, \int_0 ^t \left\|f_s+g_s\right\|_{M^1_3} \, {\rm d}s\right).
\end{equation*}

Since
\begin{multline*}
  \int_0^t  \left\| Q(h_s,f_s-g_s)
  \right\|_{M^1_2} \, {\rm d}s \le
  C_2 \, \left( \sup_{s \in [0,t]} \left\| h_s \right\|_{M^1_2} \right) \,
  \left( \int_0^t \left\| g_s -f_s \right\|_{M^1_3} \, {\rm d}s \right) \\
 + \, C_2 \, \left( \sup_{s \in [0,t]} \left\| g_s - f_s
  \right\|_{M^1_2} \right) \, \left( \int_0^t \| h_s \|_{M^1_3} \, {\rm d}s \right),
\end{multline*}
we deduce from~\eqref{eq:unifh} and \eqref{eq:unifd}
\begin{multline}\label{eq:wcomplete}
 \sup_{s \in [0,t]} \left\| \omega_s \right\|_{M^1_2} \le C_2 
  \, \left\| g_0 - f_0 \right\|_{M^1 _2} \, \exp \left( C_2 \, \int_0 ^t
    \left( \left\| f_s \right\|_{M^1 _3} + \left\| g_s \right\|_{M^1
        _3} \right) \, {\rm d}s  \right)  
  \\
   \times \Bigg[ \left( \int_0 ^t \left\| g_s - f_s \right\|_{M^1
        _3} \, {\rm d}s \right) \, \exp \left( C_1 \, \int_0 ^t \left\| f_s \right\|_{M^1
        _3} \, {\rm d}s \right) \\ + \left( \int_0 ^t \left\| h_s
      \right\|_{M^1 _3} \, {\rm d}s 
    \right) \, \exp \left( C_1 \, \int_0 ^t \left( \left\| f_s \right\|_{M^1 _3} +
      \left\| g_s \right\|_{M^1 _3} \right) \, {\rm d}s \right) \Bigg].
\end{multline}
Hence the problem now reduces to the obtaining of sharp enough time
integral controls over
\[
\left\|f_s\right\|_{M^1_3}, \quad \left\|g_s\right\|_{M^1_3}, \quad 
\left\|f_s - g_s \right\|_{M^1_3} \ \mbox{ and } \ \left\| h_s \right\|_{M^1_3}.
\]

\bigskip\noindent{\sl Step 2. Time integral control of $f$ and $g$
  in $M^1 _3$.}  In this step we prove
\begin{equation}\label{fM13L1t} 
  \int_0^t \left\| f_s
  \right\|_{M^1_{3,\ell-1}}  \, {\rm d}t \le C_3(\EE_0) \, t + C_4 \, \left\| f_0
  \right\|_{M^1_{2,\ell}} \quad \ell = 1, 2, 
\end{equation} 
for the solution $f_t$, where $C_3(\EE_0) >0$ is a constant depending on
the energy, and $C_4 >0$ is a numerical constant. The same estimate
obviously holds for the solution $g_t$. 

The estimates \eqref{fM13L1t} are a consequence of the accurate
version of the Povzner inequality which has been proved in
\cite{MW99,Lu1999}. Indeed it was shown in \cite[Lemma 2.2]{MW99}
that for any function
\[
\Psi : \R^d \to \R, \quad \Psi (v) = \psi (|v|^2) \ \mbox{ with } \
\psi \mbox{ convex,}
\]
the solution $f_t$ to the hard spheres Boltzmann equation satisfies
$$
\frac{{\rm d}}{{\rm d}t} \int_{\R^d} \Psi(v) \, {\rm d}f_t(v) = \int_{\R^d}\!
\int_{\R^d} {\rm d}f_t(v) \, {\rm d}f_t (v_*) \, |v-v_*| \, K_\Psi(v,v_*)
$$
with $K_\Psi = G_\Psi - H_\Psi$, where the term $G_\Psi$ ``behaves
mildly'' (see below) and the term $H_\Psi$ is given by
(see~\cite[formula (2.7)]{MW99})
$$
H_\Psi (v,v_*) = 2\pi \, 
\int_0^{\pi/2} \Big[ \psi \left(|v|^2 \cos^2 \theta + |v_*|^2\sin^2 \theta
  \right) - \cos^2 \theta \, \psi \left(|v|^2\right) - \sin^2 \theta  \, \psi \left(|v_*|^2
  \right) \Big] \, {\rm d}\theta.
$$
Note that $H_\Psi \ge 0$ since its integrand is nonnegative because of
the convexity of $\psi$.

More precisely, in the cases that we are interested with, namely
\[
\Psi(v) = \psi_{2,\ell}(|v|^2) \ \mbox{ with } \ \psi_{k,\ell}(r) = r^{k/2} \,
(\ln \, r)^\ell \ \mbox{ and } \ \ell = 1, 2,
\]
it is established in \cite{MW99} that (with obvious notation)
\begin{equation*}
  \forall \, v,v_* \in \R^d, \quad 
  \left|G_{\psi_{2,\ell}}(v,v_*)\right| \le 
  C_5(\ell) \, \langle v \rangle \, \left(\ln \left\langle v
      \right\rangle\right)^\ell \, 
  \left\langle v_* \right\rangle \, 
  \left(\ln \left\langle v_* \right\rangle \right)^\ell
\end{equation*}
for some constant $C_5(\ell) >0$ depending on $\ell$. 

On the other hand, in the case $\ell = 1$, we compute, with the
help of the the notation $x := \cos^2 \theta$ and $u = |v_*|/|v|$,
\begin{multline*} 
\forall \, x \in [1/4,3/4], \ \forall \,  u \in [0,1/2], \\
\psi_{2,1}
\Big(|v|^2 \cos^2 \theta + |v_*|^2\sin^2 \theta \Big) - \cos^2 \theta \,
\psi_{2,1} \left(|v|^2\right) - \sin^2 \theta \, \psi_{2,1} 
\left(|v_*|^2 \right)=
\\
 = |v|^2 \, \Big[ (1-x) \, \psi_{2,1} \left(u^2\right) 
   + x \, \psi_{2,1} (1) - \psi_{2,1} \left((1-x) \, u^2 
     + x\right) \Big] \ge C_6 \, |v|^2,
\end{multline*} 
for some numerical constant $C_6 > 0$, which only depends on the
strict convexity of the real function $\psi_{2,1}$. We deduce that
there exists a constant $C_7 > 0$ such that
$$
H_{\psi_{2,1}} (v,v_*) \ge C_7 \, |v|^2 \, {\bf 1}_{|v| \ge 2 \, |v_*|}.
$$

Similarly, in the case $\ell = 2$, we have
\begin{multline*}
\forall \, x \in [1/4,3/4], \ \forall \,  u \in [0,1/2], \\
\psi_{2,2} \Big(|v|^2 \cos^2 \theta + |v_*|^2\sin^2 \theta \Big) - \cos^2
\theta \, \psi_{2,2} \left(|v|^2\right) - \sin^2 \theta  \, \psi_{2,2} 
\left(|v_*|^2\right)= \\
= 2 \, |v|^2  \, \ln |v|^2 \,  \left\{ (1-x) \, \psi_{2,1} \left(u^2\right) 
+ x \, \psi_{2,1} (1) - \psi_{2,1} \left((1-x) \, u^2 + x\right)
\right\} \\
+  |v|^2 \,  \Big[ (1-x) \, \psi_{2,2} \left(u^2\right) 
+ x \, \psi_{2,2} (1) - \psi_{2,2} \left((1-x) \, u^2 + x\right) \Big] 
\ge C_8 \,  |v|^2 \, \ln |v|^2,
\end{multline*}
for some constant $C_8 >0$ depending on the strict convexity of
$\psi_{2,1}$ and $\psi_{2,2}$. Hence we obtain for some constant
$C_9 >0$
$$
H_{\psi_{2,2}} (v,v_*) \ge C_9 \, |v|^2 \, \ln |v|^2 \, {\bf
  1}_{|v| \ge 2 \, |v_*|}.
$$

Putting together the estimates obtained on $G_{2,\ell}$ and
$H_{2,\ell}$ we deduce 
\[
|v-v_*| \, K_{2,\ell} \le C_{10} \, |v-v_*| \, \langle v \rangle \, \left\langle v_*
\right\rangle \, \left( \ln \left \langle v \right\rangle \right)^\ell
\, \left( \ln \left \langle v_* \right\rangle\right)^\ell  - C_{11} \,
\left| v-v_* \right| \, |v|^2
\, (\ln |v|)^{\ell-1} \,  {\bf
  1}_{|v| \ge 2 \, |v_*|}
\]
for some constants $C_{10}, C_{11} >0$. Since 
\begin{multline*}
\left| v-v_* \right| \, |v|^2 \, (\ln |v|)^{\ell-1} \, {\bf
  1}_{|v| \ge 2 \, |v_*|} \ge \mbox{Cst.} \, \langle v \rangle^3 \, \left(
  \ln \langle v \rangle \right)^{\ell-1} \, {\bf
  1}_{|v| \ge 2 \, |v_*|} - \mbox{Cst.} \\ 
\ge \mbox{Cst.} \, \langle v \rangle^3 \, \left(
  \ln \langle v \rangle \right)^{\ell-1} - \mbox{Cst.} \, \langle v \rangle^2 \, \left\langle v_*
\right\rangle^2 
\end{multline*}
and 
\begin{multline*}
  |v-v_*| \, \langle v \rangle \, \left\langle v_* \right\rangle \,
  \left( \ln \left \langle v \right\rangle \right)^\ell \, \left( \ln
    \left \langle v_* \right\rangle\right)^\ell \le \langle v
  \rangle^2 \, \left\langle v_* \right\rangle \, \left( \ln \left
      \langle v \right\rangle \right)^\ell \, \left( \ln \left \langle
      v_* \right\rangle\right)^\ell + \langle v \rangle \,
  \left\langle v_* \right\rangle^2 \, \left( \ln \left \langle v
    \right\rangle \right)^\ell
  \, \left( \ln \left \langle v_* \right\rangle\right)^\ell  \\
  \le \mbox{Cst.} \, \langle v \rangle^2 \, \left\langle v_*
  \right\rangle^2 \, \left( \ln \left \langle v \right\rangle
  \right)^\ell + \langle v \rangle^2 \, \left\langle v_*
  \right\rangle^2 \, \left( \ln \left \langle v_*
    \right\rangle\right)^\ell
\end{multline*}
we easily deduce 
\begin{equation}\label{estim:Povnzer} 
|v-v_*| \, K_{2,\ell} \le C_{12}(R) \, \langle v
\rangle^2 \, \langle v_* \rangle^2  +C_{13}(R) \langle v\rangle^3 \, (\ln \langle
v \rangle)^{\ell-1} \, \langle v_* \rangle^2  - C_{14} \, \langle v\rangle^3 \, (\ln \langle
v \rangle)^{\ell-1}
\end{equation} 
for some free cutoff parameter $R>0$, some constant $C_{12}(R) \to
+\infty$ as $R \to +\infty$, $C_{13}(R) \to 0$ as $R \to +\infty$, and
$C_{14} >0$, and we finally obtain the differential inequality
\[
\frac{{\rm d}}{{\rm d}t} \|f_t\|_{M^1_{2,\ell}} \le C_{12}(R) \, (1 +
\EE_0)^2  + C_{13}(R) \, (1+\EE_0) \, M_{3,\ell-1} - C_{14} \, M_{3,\ell-1},
\]
from which \eqref{fM13L1t} follows by choosing $R>0$ large enough. 


\bigskip\noindent{\sl Step 3. Exponential time integral control of
  $f$ and $g$ in $M^1 _3$.}
This step yields a proof of \eqref{estim:dt} and
  \eqref{estim:ht}.

  Let us first prove that
\begin{equation}\label{estim:expfg} 
\forall \, t \ge 0, \quad  
e^{\left(1+ 2C_1+C_2 \right) \, \int_0 ^t \left( 
\left\| f_s \right\|_{M^1 _3} + \left\| g_s \right\|_{M^1 _3} \right) \, {\rm d}s} \le
  C_{15}(\EE_0) \, e^{C_{16} (\EE_0) \, t} \, \left(\max\left\{ M_k(f_0), M_k(g_0)
    \right\} \right)^{\frac16}, 
\end{equation}
for some constants $C_{15}(\EE_0), C_{16}(\EE_0) >0$ depending on the energy $\EE_0$, for
any $k \ge k_{\EE_0}$, with $k_{\EE_0}$ big enough depending on the energy
$\EE_0$.

We shall use the previous step and an interpolation argument. For any
given probability measure
\[
f \in \PP_k(\R^d) \ \mbox{ with } \ \int_{\R^d} |v|^2 \, {\rm d}f(v)  \le \EE_0,
\]
we have for any $a > 2$
\begin{eqnarray*}
  \| f \|_{M^1_{2,1}} &=& \int_{\R^d} \langle v \rangle^2 \, \left(1+
    \frac{\ln(\langle v \rangle^2)}{2} \right) \, \left({\bf 1}_{\langle v \rangle^2 \le a} +
    {\bf 1}_{\langle v \rangle^2 \ge a} \right) \, {\rm d} f(v)
  \\
  &\le& (1+\EE_0) \, \left(1+ \frac{\ln a}{2}\right) + \frac{1}{a} \, \int_{\R^d} \langle v
  \rangle^4 \, \left(1+ \ln\langle v \rangle\right) \, {\rm d} f(v)
  \\
  &\le& (1+\EE_0) \, \left(1+ \frac{\ln a}{2}\right) + \frac{1}{a} \, \| f \|_{M^1_5}
\end{eqnarray*}
where we have used inequality $\ln x \le x-1$ for $x \ge 1$ in the
last step. 

By choosing 
\[
a:= \| f \|_{M^1_5} ^2,
\]
we get
\begin{equation}\label{estim:M121M15} 
\|f\|_{M^1_{2,1}} \le 2 \, (1+\EE_0) \,
\left(1+ \ln \| f \|_{M^1_5}\right).
\end{equation}

\begin{rem}
  Observe here that it was absolutely crucial to be able to control
  the right-hand side of \eqref{fM13L1t} in terms of the $M^1_{2,1}$
  moment, that is only a \emph{logarithmic loss} of moment as compared
  to $M^1 _2$. This is what allows us to control this right-hand side
  in terms of the \emph{logarithm} of a higher moment of $f$, so that
  the exponential in \eqref{estim:expfg} can be controlled in terms of
  some \emph{polynomial} moment of $f$, hence fulfilling the
  requirement on the loss of weight in the stability estimates on the
  semigroup. Recall indeed that the moment associated with the weight
  $\Lambda_1$ has to be \emph{controlled} along the flow of the
  $N$-particle system. And we have not been unable
  to show the propagation of exponential moment bounds for such a
  high-dimension evolution.
\end{rem}

On the other hand, the following elementary H\"older inequality holds
\begin{equation}\label{eq:Holderkell} 
\forall \, k,k' \in \N, \ k' \le k, \ \forall \, f \in M^1_k, \quad 
\| f \|_{M^1_{k'}} \le
\| f \|_{M^1} ^{1-k'/k} \, \| f\|_{M^1_k} ^{k'/k} \le \| f\|_{M^1_k} ^{k'/k}.
\end{equation} 
Then estimate~\eqref{estim:expfg} follows from \eqref{fM13L1t},
\eqref{estim:M121M15} and \eqref{eq:Holderkell} with $k' = 5$ and $k =k_1
\ge 5$ large enough in such a way that
\[
\left( 1+ 2C_1 + C_2 \right) \, C_4 \, 2 \, (1+\EE_0) \, \frac{5}{k} \le \frac16.
\] 
We then deduce \eqref{estim:dt} from \eqref{eq:unifd}, and (similarly)
\eqref{estim:ht} from \eqref{eq:unifh}.


\bigskip\noindent{\sl Step 4. Time integral control on $d$ and $h$.} 
Let us write as before 
\[
d_t := f_t - g_t.
\]
Let us prove
\begin{multline}\label{estim:h&dM131L1T}
\left( \int_0^t 
\left\| d_s \right\|_{M^1_{3}} \, {\rm d}s \right) \ \mbox{ and } \ \left(
\int_0 ^t \left\| h_s \right\|_{M^1_{3}} \, {\rm d}s \right) 
\\ \le 
C_{20}\,  \left\| d_0 \right\|_{M^1 _2} \, e^{C_1 \,
  \int_0 ^t \left( \left\| f_s \right\|_{M^1 _3} + \left\| g_s
    \right\|_{M^1 _3} \right) \, {\rm d}s} \, 
\left( C_3(\EE_0) \, t + C_4 \, \left\| f_0 \right\|_{M^1 _{2,2}} \right)
+ C_{21} \, \left\| d_0 \right\|_{M^1_{2,1}}.
\end{multline}
for some constants $C_{20}, C_{21} >0$ defined later. Performing similar
computations to those leading to \eqref{eq:htM12}, we obtain
\begin{eqnarray*} 
  \frac{{\rm d}}{{\rm d}t} \| h_t \|_{M^1_{2,1}} &\le& \int\!\!
  \int\!\! {\rm d}|h_t|(v) \, {\rm d} f_{t} (v_*)\, \left| v - v_*\right| \, K_{2,1}(v,v_*) \\
  \nonumber && + C_{17} \, \int\!\! \int\!\! \int {\rm d} |h_t|(v) \, {\rm d} f_{t}
  (v_*)\, \left| v-v_* \right| \, \langle v_* \rangle^2 \, \left(1 + \ln \langle v_*
  \rangle \right)
\end{eqnarray*}
for some constant $C_{17} >0$ depending on $b$.  Thanks to the Povzner
inequality \eqref{estim:Povnzer} (with $\ell =1$), we deduce for some
constants $C_{18}, C_{19} > 0$
\[ 
\frac{{\rm d}}{{\rm d}t} \left\| h_t \right\|_{M^1_{2,1}} \le 
C_{18} \, \left\| h_t \right\|_{M^1_{2}} \,
\left\| f_t \right\|_{M^1_{3,1}} - C_{19} \, \left\| h_t \right\|_{M^1_{3}} .
\]
Integrating this differential inequality yields 
\[
\left\| h_t \right\|_{M^1_{2,1}} + C_{19} \, 
\int_0^t \left\| h_s \right\|_{M^1_{3}} \, {\rm d}s \le 
C_{18} \, \left( \sup_{s\in [0,t]} \left\| h_s \right\|_{M^1_{2}} \right) \, 
\left( \int_0^t \left\| f_s \right\|_{M^1_{3,1}} \, {\rm d}s \right)
+ \left\| h_0 \right\|_{M^1_{2,1}}.
\]
Using the previous pointwise control on $\| h_t \|_{M^1_{2}}$
and~\eqref{fM13L1t} (with $\ell = 2$) we get 
\[
\int_0^t \left\| h_s \right\|_{M^1_{3}} \, {\rm d}s \le 
\frac{C_{18}}{C_{19}} \,  \left\| d_0 \right\|_{M^1 _2} \, e^{C_1 \,
  \int_0 ^t \left\| f_s \right\|_{M^1 _3} \, {\rm d}s} \, 
\left( C_3(\EE_0) \, t + C_4 \, \left\| f_0 \right\|_{M^1 _{2,2}} \right)
+ \frac{1}{C_{19}} \, \left\| d_0 \right\|_{M^1_{2,1}}.
\]
Arguing similarly for $d_t$, we deduce~\eqref{estim:h&dM131L1T}.


\bigskip\noindent{\sl Step 5. Conclusion. }  
We first rewrite \eqref{eq:wcomplete} as 
\begin{equation*}
\sup_{s \in [0,t]} \| \omega_s \|_{M^1_2} 
\le C_2 \, \left\| d_0
\right\|_{M^1 _2} \, e^{\left( C_1 + C_2 \right) \, \int_0 ^t \left(
    \left\| f_s \right\|_{M^1 _3} + \left\| g_s \right\|_{M^1 _3}
  \right) \, {\rm d}s} \, \left( \int_0 ^t \left( \left\| d_s \right\|_{M^1
      _3} + \left\| h_s \right\|_{M^1 _3} \right) \, {\rm d}s \right).
\end{equation*}
Then we use the estimate~\eqref{estim:h&dM131L1T} for the last term
and thus obtain
\begin{equation*}
  \sup_{s \in [0,t]} \| \omega_s \|_{M^1_2} 
  \le C_{22} \, \left\| d_0 \right\|_{M^1 _2} \, \left\| d_0
  \right\|_{M^1_{2,1}} \, e^{\left( 2 C_1 + C_2 \right) \, \int_0 ^t
    \left( \left\| f_s \right\|_{M^1 _3} + \left\| g_s \right\|_{M^1
        _3} \right) \, {\rm d}s} \, \left( 1+ C_3(\EE_0) \, t + C_4 \, \left\| f_0 \right\|_{M^1
      _{2,2}} \right)
\end{equation*}
for some constant $C_{22}>0$.  Finally we use
estimate~\eqref{estim:expfg} for the exponential term with $k=k_1$ and
we obtain
\begin{multline*}
\sup_{s \in [0,t]} \| \omega_s \|_{M^1_2} 
  \le C_{22} \, C_{15}(\EE_0) \, \left\| d_0 \right\|_{M^1 _2} \, \left\| d_0
  \right\|_{M^1_{2,1}} \times \\
\times  e^{C_{16}(\EE_0) \, t} \, \left( \max\left\{ M_{k_1}(f_0), M_{k_1}(g_0)
    \right\} \right)^{\frac16} \, \left( 1+ C_3(\EE_0) \, t + C_4 \, \left\| f_0 \right\|_{M^1
      _{2,2}} \right). 
\end{multline*}
Then arguing as in the end of Step 3, for any $\eta \in (0,1)$, using
\eqref{eq:Holderkell} with $k_1$ large enough, we have
\[
\left\| d_0 \right\|_{M^1 _2} \, \left\| d_0 \right\|_{M^1_{2,1}} \le
\left( \max\left\{ M_{k_1}(f_0), M_{k_1}(g_0) \right\}
\right)^{\frac16} \, \left\| d_0 \right\|_{M^1} ^{1+\eta}
\]
and 
\[
\left\| f_0 \right\|_{M^1 _{2,2}} \le \left( \max\left\{ M_{k_1}(f_0), M_{k_1}(g_0)
    \right\} \right)^{\frac16}.
\]

We therefore obtain the desired estimate \eqref{estim:dt-ht} 
$$
\sup_{s \in [0,t]} \left\| \omega_s \right\|_{M^1_2} \le  e^{C\, (1+t)} \, \sqrt{ \max \left\{
    M_{k_1}(f_0),M_{k_1}(g_0) \right\} } \, \left\| f_0 - g_0 \right\|_{M^1}^{1+\eta}
$$
which concludes the proof. 
\end{proof}


\subsection{Proof of condition (A4) uniformly in time} 
\label{sec:HSA3bis}

Let us start from an auxiliary result from \cite{MouhotCMP}. Let us
define the linearized Boltzmann collision operator at $\gamma$
\[
\mathcal L_\gamma(f) := 2 \, Q(\gamma,f)
\]
where 
\begin{equation}\label{eq:maxf}
\gamma =  \frac{e^{-\frac{|v|^2}{2 (\EE/d)}}}{( 2 \pi (\EE/d))^d} 
\end{equation}
is the Maxwellian distribution with zero momentum and energy $\EE>0$.

\begin{theo}[Theorem~1.2 in \cite{MouhotCMP}]\label{theo:CMP} 
  First the linearized Boltzmann semigroup $e^{\mathcal L_\gamma\, t}$
  for hard spheres satisfies
\begin{equation}
  \label{eq:HSdecay}
   \left\| e^{\mathcal L_\gamma \, t} \right\|_{L^1(m_z )} \le C_z \,
  e^{-\lambda \, t}
\end{equation}
where 
\[
m_z(v) := e^{z \, |v|}, \quad z > 0, 
\]
and $\lambda=\lambda(\EE)$ is the optimal rate, given by the first
non-zero eigenvalue of the linearized operator $\mathcal L_\gamma$ in the
smaller space $L^2(\gamma^{-1})$, and $C_z>0$ is an explicit constant
depending on $z$. 

Second the nonlinear Boltzmann semigroup $S^{N \! L} _t$ satisfies
\begin{equation}
  \label{eq:HSdecaybis}
  \left\| S^{N \! L} _t (f_0) - \gamma\right\|_{L^1(m_z )} \le C_{f_0}  \,
  e^{-\lambda \, t}  
\end{equation}
for any $f_0 \in L^1(m_z)$ with zero momentum and energy $\EE>0$, where
$C_{f_0}$ is some constant possibly depending on $z$ and $\| f_0
\|_{L^1(m_z)}$, and $\lambda=\lambda(\EE)$ is the same rate function
as for the linearized operator above.
\end{theo}

Let us now prove \emph{uniform in time} estimate for the expansion of
the limit semigroup in terms of the initial data. 
 
\begin{lem}\label{lem:hsunif}
  For any given energy $\EE > 0$ and $\eta \in (0,1)$, there
  exists
  \begin{itemize}
 \item some constant $\bar k_1 \ge 2$ (depending on $\EE$ and $\eta$),
  \item some constant $C$ (depending on $\EE$),
 \end{itemize}
  such that for any $k_1 \ge \bar k_1$, for any 
\[
f_0, g_0 \in \mathcal P_{\GG_1}  
\]
 satisfying
\[ 
\left\langle
f_0, |v|^2 \right\rangle = \left\langle g_0, |v|^2 \right\rangle = \EE
\quad \mbox{ and } \quad
\forall \, i=1, \dots, d, \quad 
\left\langle f_0, v_i \right\rangle = \left\langle g_0, v_i \right\rangle = 0,
\]
and for any $t \ge 0$, we have
\begin{eqnarray}\label{estim:dt-infty}
&&  \left\| g_t - f_t \right\|_{M^1_2} \le 
  C \, e^{ - {\lambda \over 2} \, t }  \,  
    \sqrt{   M_{k_1}(f_0+g_0)  }\, 
  \left\| g_0 - f_0 \right\|_{M^1 }^{\eta},
\\ \nonumber \\ \label{estim:ht-infty}
&& \left\| h_t \right\|_{M^1_2} \le 
  C \, e^{- {\lambda \over 2} \, t } \,    \sqrt{  M_{k_1}(f_0+g_0) } \, 
  \left\| g_0 - f_0 \right\|_{M^1 }^{\eta}, 
\\ \nonumber \\ \label{estim:dt-ht-infty}
&&  \left\| \omega_t \right\|_{M^1_2} 
    \le C \, e^{- {\lambda \over 2} \, t }  \,
   \sqrt{   M_{k_1}(f_0+g_0)  } \, 
   \| g_0 - f_0 \|_{M^1 } ^{1+\eta}.
\end{eqnarray}
\end{lem}

Note that under the assumption (ii) in Theorem~\ref{theo:HS}, these
estimates imply {\bf (A4)} with $T=+\infty$, $\mathcal P_{\GG_2} =
\mathcal P_{\GG_1}$, since the momentum and energy conditions are
implied by ${\bf r} \in \RR_{\GG_1}$, with $\RR_{\GG_1}$ defined by
\eqref{eq:RG1cas2}.

\begin{rem}
  In the following proof we shall use \emph{moment production bounds}
  on the limit equation. Indeed once stability estimates for small
  times have been secured (as in Lemma~\ref{lem:expansionHS}), one can
  use, for $t \ge T_0 >0$, moments production estimates whose bounds
  only depend on the energy of the solution. This, together with the
  linearized theory in $L^1$ setting with exponential moment bounds of
  Theorem~\ref{theo:CMP}, will be the key to the following proof. 
\end{rem}

\begin{proof}[Proof of Lemma~\ref{lem:hsunif}] 
  In the proof below, we restrict ourself to an initial datum $f_0 \in
  \mathcal P_{\GG_1} \cap L^1(\R^d)$ for the sake of simplicity of the
  presentation, but the proof straightforwardly applies to measures.
  From the result of appearance of exponential moments for measure
  solutions \cite[Theorem~1.2-(b)]{Lu-Mouhot} (see also
  \cite{AlCaGaMo} for another simpler argument in $L^1$, and
  \cite{MM:opus1} for earlier results of appearance of exponential
  moments), there exists some constants $z$, $Z$ (only depending on
  the collision kernel and the energy of the solutions) such that
\begin{equation}\label{control-exp}
\sup_{t \ge 1} \left( \| f_t \|_{M^1_{m_{2z}}}+ \|g_t
  \|_{M^1_{m_{2z}}}+ \| h_t \|_{M^1_{m_{2z}}} \right) \le Z, \qquad
m_{2z}(v) := e^{2 \, z \, |v|}
\end{equation}
(note that the proof in \cite{Lu-Mouhot} applies to the solutions
$f_t$ and $g_t$, however it is straightforward to apply exactly the
same proof to the linearized solution $h_t$ around $f_t$, once
exponential moment is known on $f_t$).

We also know from~\eqref{eq:HSdecaybis} that (maybe by choosing a larger $Z$)
\begin{equation}\label{dec-exp-exp}
\forall \, t \ge 1, \quad \| f_t - \gamma \|_{M^1_{m_{2z}}} + \| g_t - \gamma
\|_{M^1_{m_{2z}}} \le 2 \, Z \, e^{-\lambda \, t}.
\end{equation}

We  write 
\begin{equation*}
  \partial_t (f_t-g_t) = Q(f_t-g_t,f_t+g_t) 
  = \LL_\gamma(f_t-g_t) + Q(f_t-g_t,f_t-\gamma) + Q(f_t-g_t,g_t-\gamma)
\end{equation*}
and, using also \eqref{eq:HSdecay} on the linearized semigroup, we deduce for
\[
u(t) := \left\| f_t-g_t \right\|_{M^1_{m_z}}
\]
the following differential inequality for $t \ge T_0 \ge 1$ and some
constant $C \ge 1$:
\begin{multline*}
u(t) \le C \, e^{-\lambda \, (t-T_0)} \, u(T_0) \\ + C \, \int_{T_0} ^t e^{-\lambda \,
  (t-s)} \, \left( \left\| Q(f_s-g_s,f_s-\gamma) \right\|_{M^1(m_z)} + \left\| Q(f_s-g_s,g_s-\gamma)
\right\|_{M^1(m_z)} \right) \, {\rm d}s
\end{multline*}
(this formal inequality and next ones can easily be justified
rigorously by a regularizing proceedure and using a uniqueness result
for measure solutions such as \cite{Fo-Mo,EM,Lu-Mouhot}). Therefore we
obtain
\begin{multline*} 
  u(t) \le C \, e^{-\lambda \, (t-T_0)} \, u(T_0) \\ + C \, \int_{T_0} ^t
  e^{-\lambda \, (t-s)} \, \left( \left\| f_s-\gamma
    \right\|_{M^1(\langle v \rangle \, m_z)} + \left\| g_s-\gamma
    \right\|_{M^1(\langle v \rangle \, m_z)} \right) \, \left\| f_s -
    g_s \right\|_{M^1(\langle v \rangle \, m_z)}\, {\rm d}s.
\end{multline*}

We then use the control of $M^1(\langle v \rangle \, m_z)$ by
$M^1(m_{2z})$ together with the
controls~\eqref{control-exp}-\eqref{dec-exp-exp}, the decay
control~\eqref{eq:HSdecay} and the estimate
\[
e^{-\lambda \, s -\lambda \, (t-s)} \le e^{-\frac{\lambda}2 \, t -
  \frac{\lambda}2 \, s}.
\]
We get
\[
u(t) \le C \, e^{-\lambda \, (t-T_0)} \, u(T_0) + C \, e^{-\frac{\lambda}2 \, t} \,
\int_{T_0} ^t e^{- \frac{\lambda}2 \, s} \, \left\| f_s - g_s
\right\|_{M^1(\langle v \rangle \, m_z)}\, {\rm d}s.
\]

We then use the following control for any $a>0$:
\begin{eqnarray*} 
\forall \, s \ge T_0, \quad &&\left\| f_s-g_s \right\|_{M^1_{\langle v \rangle \, m_z}}
= \int_{\R^d} \left|f_s-g_s\right| \, \langle v \rangle \, e^{z \, |v|} \, {\rm d}v \\
&&\qquad\le a \int_{|v| \le a} \left|f_s-g_s\right| \, e^{z \, |v|} \, {\rm d}v + e^{-
  z \, a} \int_{|v| \ge a} \left(f_s+g_s\right) \, e^{2 \, z \, |v|} \, {\rm d}v
\\
&&\qquad\le a \, u(s) + e^{- z \, a} \, Z.
\end{eqnarray*}

Hence we get for any $s \ge T_0$: 
\[
\left\| f_s-g_s \right\|_{M^1_{\langle v \rangle \, m_z}} \le 
\left\{ 
\begin{array}{ll}
  u(s) + e^{- z } \, Z \le (1 + Z) \, u(s) \quad \hbox{when}
\quad u(s) \ge 1, \quad \hbox{(choosing } a := 1 \hbox{)}
\vspace{0.3cm} \\
{1 \over z} \, |\ln u(s) | \, u(s) + u(s) \, Z
\quad \hbox{when} \quad u(s) \le 1\quad \hbox{(choosing } - z \, a := \ln u(s) \hbox{)}
\end{array}
\right.
\]
and we deduce 
\[
\forall \, s \ge T_0, \quad \left\| f_s-g_s \right\|_{M^1_{\langle v \rangle \, m_z}} \le K \, u(s) \, \left(1 + (\ln
u(s))_-\right), \qquad K := 1 + {1 \over z} + Z.
\]

Then for any $\delta \in (0,1)$, we have, by choosing $T_0$ large
enough, 
\[
\forall \, t \ge T_0, \quad e^{-\frac{\lambda}2 \, t} \le \delta \, e^{-\frac{\lambda}4 \, t}
\]
and we conclude with the following integral inequality
\begin{equation}\label{eq:integralu}
u(t) \le C \, e^{-\lambda \, (t-T_0)} \, u(T_0) + \delta \, e^{-\frac{\lambda}4 \, t} \,
\int_{T_0} ^t e^{- \frac{\lambda}2 \, s} \,  u_s \, 
\left(1 + \left(\ln u_s\right)_-\right) \, {\rm d}s.
\end{equation}

Let us prove that this integral inequality implies 
\begin{equation}\label{eq:stab-tps-long-HS}
\forall \, t \ge T_0, \quad 
u(t) \le C \, e^{-\frac{\lambda}4 \, t} \, u(T_0)^{1-\delta}.
\end{equation}

Consider the case of equality in~\eqref{eq:integralu}. Then we have
\[
u(t) \ge C \, e^{-\lambda \, (t-T_0)} \, u\left( T_0 \right)  \ge e^{-\lambda \, (t-T_0)} \, u\left( T_0 \right)
\]
and therefore
\[
\left(1 + \left(\ln u_t\right)_-\right) \le \left(1 + 
\left(\ln u(T_0)\right)_- + \lambda \, (t-T_0) \right).
\]
We then have
\begin{multline*}
U(t) := \int_{T_0} ^t e^{- \frac{\lambda}2 \, s} \, u_s \,
\left(1 + \left(\ln u_s\right)_-\right) \, {\rm d}s \\ \le \int_{T_0} ^t e^{-
  \frac{\lambda}2 \, s} \, u_s \, \left(1 + \left(\ln
    u(T_0)\right)_- + \lambda \, (s-T_0) \right) \, {\rm d}s \\ \le \left(3 +
  \left(\ln u(T_0)\right)_- \right) \,  \int_{T_0} ^t e^{-
  \frac{\lambda}4 \, s} \, u_s \, {\rm d}s.
\end{multline*}
By a Gronwall-like argument we can therefore obtain 
\[
u(t) \le C \, e^{-\lambda \, (t-T_0)} \, u(T_0) + C \, \delta \,
e^{-\frac{\lambda}{4} \, t} \, \left(3 + \left( \ln u(T_0)\right)_-
\right) \, u(T_0).
\]
Then thanks to the inequality 
\[
\forall \, x \in (0,1], \quad - (\ln x) \, x \le
\frac{x^{1-\delta}}{\delta} 
\]
we can prove \eqref{eq:stab-tps-long-HS} when $u(T_0) \le 1$, and in
the case when $u(T_0) \ge 1$, we can use \eqref{control-exp} again to
get 
\[
u(T_0) \le ( 2 \, Z)^{\delta} \, u(T_0)^{1-\delta}.
\]
This concludes the proof of the claimed inequality
\eqref{eq:stab-tps-long-HS}.

Then estimate \eqref{estim:dt-infty} follows by choosing $\delta$
small enough (in relation to $\eta$) and then connecting the last
estimate \eqref{eq:stab-tps-long-HS} from time $T_0$ on together with
the previous finite time estimate \eqref{estim:dt} from time $0$ until
time $T_0$. 

Then the estimate \eqref{estim:ht-infty} is proved exactly in the same
way by using the equation
\[
\partial_t h_t = \LL_\gamma(h_t) + Q(h_t,f_t-\gamma)
\]
(which is even simpler than the equation for $f_t-g_t$).

Concerning the estimate \eqref{estim:dt-ht-infty} we start from the equation
\[
\partial_t \omega_t = 2 \, \LL_\gamma(\omega_t) +
Q(\omega_t,f_t-\gamma)+Q(\omega_t,g_t-\gamma) + Q(h_t,d_t).
\]
Then we establish on 
\[
y(t) := \left\| \omega_t \right\|_{M^1_{m_z}}
\]
the following differential inequality
\begin{equation*}
y(t) \le C \, e^{-\lambda \, (t-T_0)} \, y(T_0) + C \, \delta \,
e^{-\frac{\lambda}{4} \, t} \, \left( 1+ \left( \ln y(T_0)\right)_-
\right) \, y(T_0)  + C \, e^{- \frac{\lambda}{2} \, t} \, \left\|
  d_{T_0} \right\|_{M^1_{m_z}} ^{1-\delta} \, \left\| h_{T_0}
\right\|_{M^1_{m_z}} ^{1-\delta}
\end{equation*}
which implies 
\[
y(t) \le C \, e^{- \frac{\lambda}{4} \, t} \, \left(
  y(T_0)^{1-\delta} + \left\|
  d_{T_0} \right\|_{M^1_{m_z}} ^{1-\delta} \, \left\| h_{T_0}
\right\|_{M^1_{m_z}} ^{1-\delta} \right).
\]
Then estimate \eqref{estim:dt-ht-infty} follows by choosing $\delta$
small enough (in relation to $\eta$) and then connecting the last
estimate from time $T_0$ on together with the previous finite time
estimate \eqref{estim:dt-ht} from time $0$ until time $T_0$.
\end{proof}


\subsection{Proof of condition (A5) uniformly in time}  
Let us prove that for any $\bar z, \MM_{\bar z} \in (0,\infty)$ there
exists some continuous function 
\[
\Theta : \R_+ \to \R_+, \quad \Theta(0) = 0,
\]
such that for any $f_0, g_0 \in \PP_{m_{\bar z}}(\R^d)$, $m_{\bar z}(v)
:= e^{\bar z \, |v|}$, with same momentum and energy, and such that
\[
\left\| f_0 \right\|_{M^1_{m_{\bar z}}} \le \MM_{\bar z}, \quad 
\left\| g_0 \right\|_{M^1_{m_{\bar z}}} \le \MM_{\bar z},
\]
 there holds
\begin{equation*} 
\sup_{t \ge 0} W_1\left( S^{N \! L}_t(f_0), \, S^{N \! L}_t(g_0)
\right)  \le \Theta \left( W_1\left( f_0, \, g_0 \right) \right).
\end{equation*}
where $W_1$ stands for the Kantorovich-Rubinstein distance. Let us
write 
\[
W_t := W_1\left( S^{N \! L}_t(f_0), \, S^{N \! L}_t(g_0) \right).
\]

As we shall see, we may choose
\begin{equation}\label{estim:W1dt}
  \Theta (w) := \min \left\{ \bar\Theta , \, \bar \Theta \, e^{1-\left( 1+ |\ln
        w|)\right)^{1/2}}, \, 
   \frac{C_1}{(1+|\ln w|)^{\frac{\lambda}{2K}}} 
  \right\}, \quad \Theta(0)=0,
\end{equation}
for some constants $\bar\Theta$, $C>0$ 
(only depending on $\bar z$ and $\MM_{\bar z}$).

We start off with the inequality 
\begin{equation*}
\forall \, t \ge 0 \quad W_t \le \left\| (f_t  - g_t) |v|
\right\|_{M^1} 
\le  {1\over2} \left\| (f_t  + g_t) \langle v \rangle^2 \right\|_{M^1} 
= 1 + \EE =: \bar \Theta.
\end{equation*}

Let us now improve this inequality for small value of $W_0$. Therefore
we assume without restriction that 
\[
W_0 \le \frac{1}{2}
\]
 in the sequel. 

On the one hand, it has been proved in \cite[Theorem~2.2 and
Corollary~2.3]{Fo-Mo} that
\begin{equation}\label{ineq:dwlogw-} 
W_t \le W_0 + K \,
\int_0 ^t W_s \, \left(1+(\ln W_s)_-\right) \, {\rm d}s,
\end{equation} 
for some constant $K$ (depending on $\bar z$ in the exponential. (To
be more precise, \eqref{ineq:dwlogw-} is proved in the more
complicated case of hard potentials with angular cutoff in
\cite[Theorem~2.2]{Fo-Mo}, but the proof applies to the simpler case
of hard spheres). 

One can then check that the function 
\[
\bar W_t := e^{1-e^{-Kt}} \, \left(W_0 \right)^{e^{-Kt}}
\]
satisfies 
\[
\frac{{\rm d}}{{\rm d}t} \bar W_t = K \, \left( 1- \ln \bar W_t \right) \, \bar W_t, \quad \bar
W_0 = W_0.
\]
Therefore it is a super-solution of the differential inequality
\eqref{ineq:dwlogw-} as long as $W_t \le 1$. It is an easy computation
that this super-solution satisfies
\[
\bar W_t \le 1 \ \mbox{ as long as } \ t \le t_0 := \frac{\ln \left(
    1 + | \ln W_0 | \right)}{K}.
\]
Observe also that $\bar W_t$ is increasing on $t \in [0,t_0]$. 

We then define 
\begin{equation}\label{def:t1}
t_1 := \frac{t_0}{2} = \frac{\ln \left( 1 + | \ln W_0 | \right)}{2K}
\end{equation}
and we deduce the following bound on the solution of
\eqref{ineq:dwlogw-}: 
\begin{equation}\label{controltpspetit}
\forall \, t \in \left[ 0, t_1 \right], \quad 
W_t \le \bar W_t \le \bar W_{t_1} = e^{1-\left( 1+ |\ln
    W_0|)\right)^{1/2}}.
\end{equation}

On the other hand, from~\eqref{eq:HSdecaybis}, there
are constants $\lambda, Z > 0$, $z \in (0,\bar z)$ such that
\begin{equation}\label{decayencore}
\forall \, t \ge 0, \quad \| f_t  - \gamma \|_{L^1_{m_{z}}} +  \| g_t  - \gamma
\|_{L^1_{m_{z}}}   \le  Z \, e^{-\lambda \, t},
\end{equation}
where $\gamma$ stands again for the normalized Maxwellian associated
to $f_0$ and $g_0$.

We deduce from \eqref{decayencore} 
\begin{equation}\label{Wttpslong}
\forall \, t \ge 0, \quad W_t \le C \, e^{-\lambda \, t} 
\end{equation}
for a constant $C>0$.

We then consider times $t \ge t_1$ and we deduce from
\eqref{def:t1} and \eqref{Wttpslong} the following bound from above
\begin{equation}\label{controltpslong}
\forall \, t \ge t_1, \quad W_t \le C \, e^{-\lambda \, t_1} 
= \frac{C}{(1+|\ln W_0|)^{\frac{\lambda}{2K}}}. 
\end{equation}
It is then straightforward to conclude the proof of {\bf (A5)}
uniformly in time for the function \eqref{estim:W1dt} by combining
\eqref{controltpspetit} and \eqref{controltpslong}.
\bigskip

We have proved all the assumptions of
Theorem~\ref{theo:abstract}. 
\begin{itemize}
\item Together with the estimate on $\WW_{W_1}^N (f)$ from
  Lemma~\ref{lem:Rachev&W1}, this concludes the proof of point (i) in
  Theorem~\ref{theo:HS} by using the non-uniform estimates for {\bf
    (A4)}.
\item Then we can conclude the proof of point (ii) in
  Theorem~\ref{theo:HS} by using
\begin{itemize}
\item Lemma~\ref{lem:Bproperty} for the construction of the sequence
  initial data $f^N _0$ which satisfies the required integral and
  support moment bounds, 
\item The previous steps in order to apply Theorem~\ref{theo:abstract},
\item Lemma~\ref{lem:Bpropertybis} in order to estimate
$$ 
\WW_{W_1} \left( \pi^N _P \left(f_0 ^N\right), f_0 \right)
\xrightarrow[]{N \to \infty} 0.
$$ 
\end{itemize}
\end{itemize}

\subsection{Proof of infinite-dimensional Wasserstein chaos}
\label{subsec:wass-hs}

Let us now prove Theorem~\ref{theo:hs-wasserstein}. Its proof is
similar to Theorem~\ref{theo:max-wasserstein}. 

First the proof of \eqref{eq:hs-wass-a} follows from the point (i) in
Theorem~\ref{theo:HS} and \cite[Theorem~1.1]{hm} exactly in a similar
way as we proved that \eqref{eq:max-wass} follows from the point (i)
in Theorem~\ref{theo:tMM} and \cite[Theorem~1.1]{hm}. The proof of
\eqref{eq:hs-wass-b} follows similarly from the point (ii) in
Theorem~\ref{theo:HS} and \cite[Theorem~1.1]{hm}.

Then the proof of \eqref{eq:hs-wass-relax} is also similar to the one
of \eqref{eq:max-wass-relax}, the only difference being that one needs
the following result of lower bound (independent of $N$) on the
spectral gap of the $N$-particle system.

\begin{theo}[\cite{CCL-preprint}]\label{lem:estim-hs-max}
  Consider the operator $L_{HS}$ for the hard spheres $N$-particle
  model with collision kernel $B (v-w) = |v-w|$.  Then for any $\EE>0$
  there is a constant $\lambda>0$ (independent of $N$ but depending on
  $\EE$) such that for any probability $f^N$ on $\SS^N(\EE)$ one has
  \begin{equation*}
    \left\langle L_{HS} f^N, f^N \right\rangle_{L^2\left( \SS^N(\EE)
      \right)} \le - \lambda
    \, \left\| f^N \right\|_{L^2\left( \SS^N \right)}.
  \end{equation*}
where $\SS^N(\EE)$ was defined in \eqref{def:SSN-EM}. 
\end{theo}

Then using Theorem~\ref{lem:estim-hs-max} we deduce that 
$$
\forall \, N \ge 1, \,\, \forall \, t \ge 0, \quad \left\| h^N - 1
\right\|_{L^2\left(\SS^N(\EE),\gamma^N\right)} \le e^{-\lambda \, t} \,
\left\| h^N_0 - 1 \right\|_{L^2\left(\SS^N(\EE),\gamma^N\right)},
$$
where $h^N = {\rm d}f^N/{\rm d}\gamma^N$ is the Radon-Nikodym
derivative of $f^N$ with respect to the measure $\gamma^N$ and the end
of proof of \eqref{eq:hs-wass-relax} is then exactly similar to the
one of \eqref{eq:max-wass-relax} in the previous section.

\section{$H$-theorem and entropic chaos}
\label{sec:h-theorem-entropic}
\setcounter{equation}{0}
\setcounter{theo}{0}

This section is concerned with the $H$-theorem. We answer a question
raised by Kac~\cite{Kac1956} about the derivation of the $H$-theorem.

\subsection{Statement of the results}
\label{sec:results-entropy}

Our main results of this section state as follows:
\begin{theo}\label{theo:entropy} 
  Consider the Boltzmann collision process for Maxwell molecules (with
  or without cutoff) or hard spheres, and some initial data with zero
  momentum and energy $\EE$ satisfying
\[
f_0 \in L^\infty\left(\R^d\right) \ \mbox{ s. t. } \ \int_{\R^d} e^{z
  \, |v|} \, {\rm d}f_0(v) < + \infty
\]
for some $z >0$, and the sequence of $N$-particle initial data
$(f_0^N)_{N \ge 1}$ on $\SS^N(\EE)$ constructed in
Lemma~\ref{lem:Bproperty} and \ref{lem:Bpropertybis}.

Then we have:
\begin{itemize}
\item[(i)] In the case of Maxwell molecules with cut-off and hard
  spheres, if the initial data is entropically chaotic in the
  sense
$$
\frac1N \, H\left( f^N _0 | \gamma^N \right) \xrightarrow[]{N \to
  +\infty} H\left(f_0 | \gamma \right) ,
$$
with 
$$
H\left( f^N _0 | \gamma^N \right) := \int_{\SS^N(\EE)}  h^N_0  \, 
\ln h^N_0 \, {\rm d}\gamma^N(v),
\quad h^N_0 := \frac{{\rm d}f^N_0}{{\rm d}\gamma^N},
$$ 
then the solution is also entropically chaotic for any later time: 
$$
\forall \, t \ge 0, \quad \frac1N \, H\left( f^N _t \Big| \gamma^N \right)
\xrightarrow[]{N \to +\infty} H\left(f_t \big| \gamma \right).
$$
This proves the derivation of the $H$-theorem this context, i.e. the
monotonic decay in time of $H(f_t|\gamma)$, since for any $N
\ge 2$, the functional $H(f^N_t | \gamma^N)$ is monotone decreasing
in time for the Markov process.

\item[(ii)] In the case of Maxwell molecules, and assuming moreover
  that the Fisher information of the initial data $f_0$ is finite: 
  $$
  \int_{\R^d} \frac{\left| \nabla_v f_0 \right|^2}{f_0} \, {\rm d}v <
  +\infty,
  $$
  the following estimate on the relaxation induced by the $H$-theorem
  \emph{uniformly in the number of particles} also holds:
$$ 
\forall \, N \ge 1, \quad \frac{H\left( f^N _t | \gamma^N \right)}{N}
\le \beta(t) 
$$ 
for some polynomial function $\beta(t) \to 0$ as $t \to \infty$. 
\end{itemize}
\end{theo}

\begin{rems}
\begin{enumerate}
\item The assumptions on the initial data could be relaxed to just
  $\PP_6 \cap L^\infty$ as in point (iii) of Theorem~\ref{theo:tMM} in
  the case of Maxwell molecules. However our assumptions allow for a
  unified statement for hard spheres and Maxwell molecules. We do not
  search for optimal statement here, but rather emphasize the
  strategy. 
\item A stronger notion of entropic chaoticity could be 
$$
\frac1N \, H\left( f^N \big| \left[ f^{\otimes N} \right]_{\SS^N(\EE)} \right)
\xrightarrow[]{N \to +\infty} 0.
$$
The propagation of such property is an interesting open question. 
A partial answer is given in \cite[Theorem~25]{kleber}. 
\item The point (ii) holds for the hard spheres conditionally to a
  bound on the Fisher information uniformly in time and in the number
  of particle. However at present, it is an open problem to known
  whether such a bound holds for the many-particle hard spheres jump
  process.
\item In point (ii) we conjecture the better decay rate
$$
\forall \, N \in \N^*, \,\, \forall \, t \ge 0,
\quad 
\frac{H \left( f^N _t | \gamma^N \right)}{N} \le C \, e^{-\lambda \,t}
$$
for some constant $\lambda >0$. 
\end{enumerate}
\end{rems}

\subsection{Propagation of entropic chaos and derivation of the
  $H$-theorem}
\label{sec:prop-entr-chaos}

In this subsection we shall prove the point (i) of
Theorem~\ref{theo:entropy}. Its proof relies on a convexity argument. 

Let us define $h^N := {\rm d}f^N/{\rm d}\gamma^N$ and then  compute 
\begin{multline*}
\frac{{\rm d}}{{\rm d}t} \frac{H \left( f^N _t | \gamma^N \right)}{N} = 
- D^N \left( f^N _t  \right) \\
 := - \frac{1}{2N^2} \,
\int_{\SS^N(\EE)} \sum_{i \not = j} \int_{\mathbb S^{d-1}} \left( h^N _t
  (V_{ij}^*) - h^N _t (V) \right) \, \ln \frac{h^N _t
  (V_{ij} ^*)}{h^N _t (V)} \, B(v_i-v_j,\sigma) \, 
{\rm d}\sigma \, {\rm d} \gamma^N(v)
\end{multline*}
where we recall that $V_{ij} ^*$ was defined in
\eqref{eq:def-vit-coll}, which implies
\begin{equation}\label{eq:Npart-entropy}
  \forall \, t \ge 0, \quad \frac{H \left( f^N _t | \gamma^N
  \right)}{N} + \int_0 ^t  D^N \left( f^N _s \right) \, {\rm d}s =
  \frac{H \left( f^N _0 | \gamma^N \right)}{N}.
\end{equation}
We also note that the same kind of equality is true at the limit (see
e.g. \cite{Lu1999})
$$
\forall \, t \ge 0, \quad H \left( f _t | \gamma
\right) + \int_0 ^t D^\infty \left( f_s \right) \, {\rm d}s =
H \left( f _0 | \gamma \right)
$$
with 
$$
D^\infty \left( f \right) := \frac12 \, \int_{\R^d \times \R^d \times \mathbb
  S^{d-1}} \left( f' f'_* - f f_* \right) \, \ln
\frac{f' f'_*}{f f_*} \, B(v-v_*, \sigma) \, {\rm d}v \, {\rm d}v_* \,
{\rm d}\sigma
$$
(be careful to the factor $1/2$ in our definition of the collision
operator \eqref{eq:collop} when computing the entropy production
functional). 

We then have the following lower semi-continuity property on these
functionals, as a consequence of their convexity property. 

\begin{lem}\label{lem:cvxentropy}
  The many-particle relative entropy and entropy production
  functionals defined above are lower semi-continuous: if the sequence
  $(f^N)_{N \ge 1}$ is $f$-chaotic then
$$
\liminf_{N \to \infty} \frac{H\left( f^N | \gamma^N \right)}{N} \ge H(f
| \gamma)
$$
and 
$$
\liminf_{N \to \infty} \frac{D^N\left( f^N  \right)}{N} \ge
D^{\infty} (f).
$$
\end{lem}

Let us first explain how to conclude the proof of point (i) of
Theorem~\ref{theo:entropy} with this lemma at hand. We first deduce
from \eqref{eq:Npart-entropy} and the entropic chaoticity of the
initial data that 
\[
\forall \, t \ge 0, \quad \frac{H \left( f^N _t | \gamma^N
\right)}{N} + \int_0 ^t D^N \left( f^N _s \right) \, {\rm d}s 
\xrightarrow[]{N \to \infty} H\left( f_0 | \gamma \right) 
= H\left( f_t | \gamma\right) + \int_0 ^t D^\infty \left( f_s \right) \,
{\rm d}s.
\]
Second we use Lemma~\ref{lem:cvxentropy} on the LHS to deduce that 
$$
\forall \, t \ge 0, \quad \liminf_{N \to \infty} \left( \frac{H
  \left( f^N _t | \gamma^N \right)}{N} + \int_0 ^t D^N \left( f^N _s
  \right) \, {\rm d}s \right) \ge H\left( f_t | \gamma\right) + \int_0
^t D^\infty \left( f_s \right) \, {\rm d}s
$$
where each of the limit of the two non-negative terms on the LHS is
greater that the corresponding non-negative term in the RHS. We deduce
from the two last equations that necessarily
$$
\forall \, t \ge 0, \quad \frac{H \left( f^N _t | \gamma^N
\right)}{N} \xrightarrow[]{N \to \infty} H\left( f_t | \gamma \right)
$$
and 
$$
\forall \, t \ge 0, \quad \int_0 ^t  D^N \left( f^N _s \right)
\, {\rm d}s \xrightarrow[]{N \to \infty} \int_0 ^t D^\infty \left( f_s
\right) \, {\rm d}s
$$
which concludes the proof of point (i) of Theorem~\ref{theo:entropy}. 

\begin{proof}[Proof of Lemma~\ref{lem:cvxentropy}]
  These inequalities are consequences of convexity properties.  The
  lower continuity property on the relative entropy on the spheres was
  proved in \cite[Theorem~12]{CCLLV} (actually the proof in this
  reference is performed on the sphere $\Sp^{N-1}$, but extending it
  to the invariant subspaces of our jump processes $\mathcal S^N(\EE)$ is
  straightforward). We refer to \cite{kleber} for a detailed proof of
  the latter.

Let us now prove the inequality for the entropy production
functional $D^N$. Denoting $Z = h^N  (V^*_{12} ) / h^N $, we
 first rewrite thanks to the symmetry of $f^N$ as 
\[
D^N \left( f^N\right) 
=   \frac{N(N-1)}{2N^2} \, \int_{\SS^N(\EE)}
\int_{\mathbb S^{d-1}} J(Z) \,
B(v_1-v_2,\sigma) \, f^N _2 (v_1,v_2) \, 
{\rm d}\sigma \, \frac{{\rm d}f^N (v)}{f^N _2(v_1,v_2)} ,
\]
where $J(z) := (z-1) \, \ln z$ and $f^N _2$ denotes the
$2$-marginal. Since the function $z \mapsto J(z)$ is convex, we can
apply a Jensen inequality according to the variables $v_3, \dots, v_N$
with reference probability measure $f^N / f^N_2$, which yields
$$
D^N \left( f^N\right) \ge \frac{N(N-1)}{2N^2} \, \int_{v_1, v_2 \in
  \R^d} \int_{\mathbb S^{d-1}} J( \bar Z )\,
B(v_1-v_2,\sigma) \, f^N _2 (v_1,v_2) \, {\rm d}\sigma \, {\rm d}v_1 \, {\rm d}v_2
$$
with 
$$
\bar Z(v_1, v_2) := \int_{v_3,\dots, v_N \in \mathcal S^N(v_1,v_2)} Z
\, \frac{{\rm d}f^N (V)}{f^N _2(v_1,v_2)} =
\frac{f^N_2\left((V_{12}^*)_1,(V^*_{12})_2\right)}{f^N _2(v_1,v_2)},
$$ 
where $\mathcal S^N(v_1,v_2) := \{ v_3,\dots, v_N \in E^{N-2}, \,\,
(v_1, \dots, v_N) \in \mathcal S^N\}$.  We therefore deduce a control
from below of the $N$-particle entropy production functional in terms
of the $2$-particle entropy production functional, denoting $(f_2
^N)^*=f^N_2\left((V_{12}^*)_1,(V^*_{12})_2\right)$):
$$
D^N \left( f^N\right) \ge \frac{N(N-1)}{2N^2} \, \int_{v_1, v_2 \in
  \R^d} \int_{\mathbb S^{d-1}} \left( (f^N_2)^*-f^N_2 \right) \, \ln
\frac{(f^N_2)^*}{f^N_2} \, B(v_1-v_2,\sigma) \,
{\rm d}\sigma \, {\rm d}v_1 \, {\rm d}v_2
$$

Finally we take advantage of the convexity of the functional
$$
h(x,y) = (x-y) \, \ln \frac{x}{y} 
$$
which implies that the function 
$$
\left(f_2,g_2\right) \to \int_{v_1,v_2 \in \R^d} \int_{\sigma \in
  \Sp^{d-1}} \left( f_2 - g_2 \right) \, \ln \frac{f_2}{g_2} \,
B(v_1-v_2, \sigma)
$$
is lower semi-continuous for the weak convergence of the $2$-particle
distributions $f_2$ and $g_2$ as proved in \cite[Step~2 of the
proof]{MR1088276}. 

Hence we obtain thanks to the chaoticity of the second marginal
\begin{multline*}
  \liminf_{N \to \infty} D^N \left( f^N \right)  \\ \ge \frac12 \,
  \int_{v, v_* \in \R^d} \int_{\mathbb S^{d-1}} \left(
    f(v')f(v'_*)-f(v)f(v_*) \right) \, \ln
  \frac{f(v')f(v'_*)}{f(v)f(v_*)} \, B(v-v_*,\sigma) \, {\rm d}\sigma \,
  {\rm d}v \, {\rm d}v_* \\ = D^\infty(f)
\end{multline*}
which concludes the proof. 
\end{proof}

\subsection{Many-particle relaxation rate in the $H$-theorem}
\label{sec:relaxation-h-theorem}

In this subsection we shall prove point (ii) in
Theorem~\ref{theo:entropy}. Its proof goes in two steps. First we
shall prove that it follows from an estimate on the Fisher information
thanks to the so-called ``HWI'' interpolation
inequality~\cite{VillaniTOT}. Second we shall prove such a uniform
bound on the Fisher information in the case of Maxwell molecules. Let
us take the opportunity to thank Maxime Hauray who kindly communicated
to us a proof for the latter step. 

Let us define the Fisher informations for the $N$-particle
distribution: 
$$
I\left( f^N \right) := \int_{\R^{dN}} \frac{\left| \nabla f^N
  \right|^2}{f^N} \, {\rm d}v
$$
and
$$
  I\left( f^N |\gamma^N\right) := \int_{\SS^N(\EE)} \frac{\left| \nabla_{\SS^N(\EE)} h^N
  \right|^2}{h^N} \, {\rm d}\gamma^N(v), \,\,\, h^N := {{\rm d}f^N
  \over {\rm d}\gamma^N}
$$
for a probability $f^N$ having a density with respect to the Lebesgue
measure in $\R^{dN}$ and with respect to the measure $\gamma^N$
respectively.  The gradient in that last formula has to be understood as
the usual Riemannian geometry gradient in the manifold $\SS^N(\EE)$.
The tangent space $T\SS^N(\EE)_V$ (of dimension $Nd-2$) at some given point
$V \in \SS^N$ is given by
$$
T \SS^N(\EE) _V = \left\{ W \in \R^{dN} \ \mbox{ s. t. }  \ 
\sum_{i=1} ^N w_i =0 \ \mbox{ and } \  W \, \bot \, V \right\}.
$$

For more informations and other results on the Fisher informations on
$\SS^N(\EE)$
we refer to \cite{MR2247452}. 
We shall prove the following lemma whose proof is inspired from
\cite{MR1646804}). 
\begin{lem}
  \label{lem:fisherN}
  Consider the $N$-particle jump process $(\VV_t ^N)$ for Maxwell
  molecules as defined in Subsection~\ref{sec:modelEBbounded} for $N
  \ge 1$, and some initial law $f^N _0$ with support included in
  $\SS^N(\EE)$ and whose Fisher information is finite $I(f^N
  _0|\gamma^N) <+\infty$ on $\SS^N(\EE)$. Then $f^N_t$ has support
  included in $\SS^N(\EE)$ for later times, and one has the following
  uniform in time bound on the Fisher information of the associated
  law
  $$ \forall \, t \ge
  0, \quad I\left( f^N _t |\gamma^N \right) \le I \left( f^N _0 | \gamma^N \right).
  $$
\end{lem}

\begin{proof}[Proof of Lemma~\ref{lem:fisherN}]
  We shall first consider the case of cutoff Maxwel molecules whose
  collision kernel $b$ is integrable, and for a positive and smooth
  solution $f^N$ on $\SS^N(\EE)$.  These assumptions can be relaxed by a
  mollification argument.  

  It is possible to study directly the estimate to be proved on the
  Boltzmann sphere $\mathcal S^N(\EE)$, however it means that one has to
  consider some local coordinates and a local basis for the tangent
  space. Another simpler method is to take advantage of the fact that
  the dynamics leaves the energy unchanged.

Starting from an initial data $f_0 ^N$ on $\mathcal S^N(\EE_0)$ we consider
the flatened initial data 
$$
\tilde f_0 ^N := \alpha(\EE(V)) \, 
f^N_0 \left( \frac{V \sqrt{N \EE_0}}{\sqrt{N \EE(V)}}\right) \quad \mbox{ with }
\quad \EE(V) = \frac{\sum_{i=1} ^N
  |v_i|^2}{N}. 
$$

Observe that from the conservation of energy and momentum and the
uniqueness of the solutions to the linear master $N$-particle equation
$$
\forall \, t \ge 0, \quad \tilde f_t ^N := \alpha(\EE(V)) 
\, f^N_t \left( \frac{V \sqrt{N \EE_0}}{\sqrt{N\EE(V)}}\right)
$$
where $\tilde f_t ^N$ denotes the solution in $\R^{dN}$ starting from
$\tilde f_0 ^N$. If the function $\alpha$ is regular and compactly
supported, as well as $f_0 ^N$, this produces a smooth solution on
$\R^{dN}$.

Assume that the result on the Fisher information is true in $\R^{dN}$:
$$
I\left( \tilde f^N _t \right) := \int_{\R^{dN}} \frac{\left| \nabla
    \tilde f^N _t \right|^2}{\tilde f^N _t} \, {\rm d}v \le 
\int_{\R^{dN}} \frac{\left| \nabla
    \tilde f^N _0 \right|^2}{\tilde f^N _0} \, {\rm d}v = I\left( \tilde f^N
  _0 \right).
$$
Then we have the orthogonal decomposition of the gradient
locally in terms of radial and ortho-radial directions
$$
\nabla_{\R^{dN}} \tilde f^N _t = \nabla_{\EE} \tilde f^N _t 
+ \nabla_{\mathcal S^N(\EE)} \tilde f^N _t 
= \left( \nabla_{\EE} \ln \alpha 
\right) \, \tilde f^N _t
+ \nabla_{\mathcal S^N(\EE)} \tilde f^N _t 
$$
that we can plug into the Fisher information inequality: 
\begin{multline*}
  I\left( \tilde f^N _t \right) := \left| \nabla_{\EE} \ln \alpha
  \right|^2 + \left( \int_{\R_+} \alpha(\EE) \, (N \EE)^{(dN-1)/2} \,
    {\rm d} \sqrt \EE \right) \, \int_{\mathcal S^N(\EE)} \frac{\left|
      \nabla_{\mathcal S^N(\EE)} h^N _t
    \right|^2}{h^N _t} \, {\rm d}\gamma^N(v) \\
  \le \left| \nabla_{\EE} \ln \alpha \right|^2 + \left( \int_{\R_+}
    \alpha(\EE) \, (N \EE)^{(dN-1)/2} \, {\rm d} \sqrt \EE \right) \,
  \int_{\mathcal S^N(\EE)} \frac{\left| \nabla_{\mathcal S^N(\EE)} h^N
      _0 \right|^2}{h^N _0} \, {\rm d}\gamma^N(v) = I\left( \tilde f^N
    _0 \right).
\end{multline*}
Dropping the terms which do not depend on time we obtain the desired
inequality on $\mathcal S^N(\EE)$. 

Let us now prove the inequality on $\R^{dN}$. Let us first fix some
notation: the $N$-particle solution $f^N _t$ satisfies
\begin{equation*}
  \partial_t f^N   =  \frac1N \, \sum_{i, j=1, i\not= j}^N 
  \int_{\mathbb S^{d-1}} \left(  f^N\left(r_{ij,\sigma}(V)\right) \, 
  b\left(\cos \theta_{ij} \right) \, {\rm d}\sigma   -  f^N(V) \right) \, {\rm d}\sigma
   =:  N B \left( Q^{+,N}(f^N)  - f^N \right) 
\end{equation*}
where we use the following notations. We define
\[
Q^{+,N}(f^N)  := {1 \over N^2}\sum_{i, j=1, i\not= j}^N  Q_{ij}^{+,N} \left(f^N\right),
\qquad
Q_{ij}^{+,N} \left(f^N\right) := \int_{\Sp^{d-1}}
f^N_{ij} \, b\left(\cos \theta_{ij}\right)
\,{\rm d}\sigma,
\] 
\[
\cos \theta_{ij} := \sigma \cdot k_{ij} \ \mbox{ with } \
k_{ij}=(v_i-v_j)/|v_i-v_j| 
\] 
and where we assume that $b$ is even and that
$$
\int_{\Sp^{d-1}} b(\sigma \cdot k) \, {\rm d}\sigma =C_B \ \mbox{ for any } \ k,
\ |k|=1.  $$

For any function $g^N$ on $\R^{dN}$ shall use the shorthand notation
$g_{ij} ^N$ to denote the function $V \mapsto g(r_{ij,\sigma}(V))$, which
depends also implicitly on $\sigma$.  We shall make use of the measure
preserving involution
$$
\Theta_{ij} : \begin{cases}
     \R^n \times \R^n \times \mathbb S^{d-1} & \rightarrow   \R^n \times \R^n \times
\mathbb S^{d-1} \vspace{0.2cm} \\
     (v_i,v_j,\sigma) & \mapsto  (v_i', v_j', \sigma') 
    \end{cases}
$$
where $\sigma'= (v_i-v_j)/|v_i-v_j| =k_{ij}$. 

Finally as in \cite{MR1646804}, we shall use the following endomorphism of
$\R^d$
\begin{eqnarray*}
 M_{\sigma k}(x) & = & (k \cdot \sigma) x - (k\cdot x) \sigma \\
 P_{\sigma k}(x) & = & (\sigma \cdot x) k + M_{\sigma k}
\end{eqnarray*}
and we recall that $\| P_{\sigma k} (x) \| \leq \|x\|$ with equality
only if $x,\sigma,k$ are coplanar.

\medskip

We claim that it is enough to prove that
\begin{equation}\label{IQ+}
I \left( Q^{+,N}(f^N) \right) = I \left( \frac1{N^2} \sum_{i,j=1}^N
  \int_{\Sp^{d-1}} f\left(r_{ij,\sigma}(V)\right) \, b\left(\cos
    \theta_{ij}\right) \,{\rm d}\sigma \right) \leq C_B \, I(f).
\end{equation}
Indeed with this result at hand, we can write for $\varepsilon >0$: 
$$
f^N_{t+\varepsilon} = e^{-N C_B \varepsilon} \, f^N _t + N \, C_B \, \int_0
^\varepsilon e^{N C_B (s-\varepsilon)} \, Q^{+,N} \left(f^N_{t+s}\right)
\, {\rm d}s
$$
and therefore from the convexity of $I$
$$
I\left( f^N_{t+\varepsilon} \right) \le e^{-N C_B \varepsilon} \, I\left(
  f^N _t \right) + \left(1-e^{-N C_B \varepsilon}\right) \, I \left( \int_0
^\varepsilon Q^{+,N} \left(f^N_{t+s}\right)
\, \frac{N \, C_B \, e^{N C_B (s-\varepsilon)}}{\left( 1- e^{-N C_B 
      \varepsilon}\right)}  \, {\rm d}s \right).
$$
Observe that 
$$
\int_0 ^\varepsilon \frac{N \, C_B \, e^{N C_B (s-\varepsilon)}}{\left( 1-
    e^{-N C_B \varepsilon}\right)} \, {\rm d}s =1
$$
and then we can use the convexity of $I$ again to get 
$$
I\left( f^N_{t+\varepsilon} \right) \le e^{-N C_B \varepsilon} \, I\left(
  f^N _t \right) + \int_0
^\varepsilon I\left( Q^{+,N} \left(f^N_{t+s}\right) \right) 
\, N \, C_B \, e^{N C_B (s-\varepsilon)}  \, {\rm d}s.
$$
Finally using the claimed result \eqref{IQ+} we obtain
$$
\frac{I\left(f^N _{t+\varepsilon} \right) - I\left( f^N
    _{t}\right)}{\varepsilon} 
\le - \frac{\left( 1- e^{-N C_B \varepsilon} \right)}{\varepsilon} \,  I\left(
  f^N _t \right) + \frac1\varepsilon \, \int_0 ^\varepsilon I\left( f^N
    _{t+s}\right) \, N \, C_B^2 \, e^{N C_B (s-\varepsilon)} \, {\rm d}s.
$$
Then taking $\varepsilon \to 0$ and using Lebesgue's theorem we deduce
$$
\frac{{\rm d}}{{\rm d}t} I\left( f^N _{t}\right) \le - N \, C_B \, I\left( f^N
  _{t}\right) + N \, C_B  \, I\left( f^N _{t}\right) \le 0
$$
which concludes the proof. 

Let us now focus on the proof of the claim \eqref{IQ+}. Taking
advantage of the convexity of $I$, it is enough to prove
$$
\forall \, i \not = j \in [|1, N|], \quad I\left(Q_{ij} ^{+,N} \left(f^N\right)
\right) \leq C_B \, I\left(f^N\right).
$$

Let us compute each partial derivative of $Q_{ij} ^{+,N} \left( f^N
\right)$.  If $\ell \notin \{ i,j\}$ then the derivative does not act
on the kernel $b$ and we obtain:
\begin{eqnarray*}
  \nabla_{v_\ell} \left( Q_{ij} ^{+,N} \left( f^N \right) \right)  & = & 
  \int_{\Sp^{d-1}} \nabla_{v_\ell} \left(f_{ij} ^N\right)
  \, b\left(\cos \theta_{ij}\right) \, {\rm d}\sigma =
  \int_{\Sp^{d-1}} \left(\nabla_{v_\ell} f^N \right)_{ij} \, 
  b\left(\cos \theta_{ij}\right) \, {\rm d}\sigma \\
  & = & 2 \int_{\Sp^{d-1}} \left(\sqrt{f^N}\right)_{ij} \, \left(\nabla_{v_\ell}
    \sqrt{f^N} \right)_{ij} \, b\left(\cos \theta_{ij}\right) \, {\rm d}\sigma.
\end{eqnarray*}

If $\ell \in \{i,j\}$, then it is slightly more complicated. Without
restriction we perform calculations in the case $\ell=i$.  Let us
first prove the formula
\begin{multline} \label{eq:expgrad} \nabla_{v_i} \left(Q_{ij} ^{+,N} \left(
    f^N \right) \right) \\ = 
  \int_{\Sp^{d-1}} \left[ \left(\nabla_{v_i}
      f^N  \right)_{ij} +  \left( \nabla_{v_j}
      f^N \right)_{ij} + P_{\sigma k} \left( \left( \nabla_{v_i}
      f^N \right)_{ij} - \left(\nabla_{v_j}
      f^N \right)_{ij} \right) \right] \, b\left(\cos \theta_{ij}\right)
  \,{\rm d}\sigma \\
\end{multline}
(the same equality obviously holds where $i$ is replaced by $j$).

Simple computations (see for instance \cite{MR1646804}) yield
\begin{eqnarray*}
  \nabla_{v_i} \left(f_{ij}^N \right) & = & \frac12 \, \left( \left(\nabla_{v_i}
      f^N \right)_{ij} 
    + \left(\nabla_{v_j} f^N \right)_{ij}
  \right) + \frac12 \, \left[ \left( \left(\nabla_{v_i} f^N \right)_{ij} 
      - \left(\nabla_{v_j} f^N \right)_{ij}\right)
    \cdot \sigma \right]\, k_{ij} \,,\\
  \nabla_\sigma \left(f^N _{ij}\right) & = & 
  \frac{|v_i-v_j|}2 \, \left( \left(\nabla_{v_i} f^N \right)_{ij} -
    \left(\nabla_{v_j} f^N \right)_{ij} \right), \\
  \nabla_{v_i} \left[ b\left( \cos \theta_{ij} \right)\right]  &
  = & \frac1{|v_i-v_j|} \, b' \left(\sigma \cdot k_{ij}\right) \, \Pi_{k^\perp} \sigma,
\end{eqnarray*}
where $\Pi_{k^\perp}$ is the projection on the hyperplane
$k^\perp$. Using the first and third equality above, we get
\begin{multline} \label{eq:nablaQ} \nabla_{v_i} \left(Q_{ij} ^{+,N}
    \left( f^N \right) \right) \\ = \frac12 \, \int_{\Sp^{d-1}}
  b\left(\cos \theta_{ij}\right) \, \left(
    \left(\nabla_{v_i} f\right)_{ij} + \left(\nabla_{v_j}
      f\right)_{ij} + \left[ \left( \left(\nabla_{v_i} f \right)_{ij}
        - \left(\nabla_{v_j} f\right)_{ij}\right)
      \cdot \sigma \right]\, k \right) \, {\rm d}\sigma \\
  + \left( \int_{\Sp^{d-1}} b'\left(\cos \theta_{ij} \right)
    \frac{f_{ij}}{|v_i-v_j|} \, \Pi_{k^\perp} \sigma \, {\rm d}\sigma
  \right)
\end{multline}
and we use the following formula of integration by part on the sphere
$\Sp^{d-1}$ (see \cite[Lemma 2]{MR1646804})
\begin{equation*} 
\int_{\Sp^{d-1}} b'\left(\cos
    \theta_{ij}\right) \, F(\sigma) \, \Pi_{k^\perp} \sigma \, {\rm d}\sigma
  = \int_{\Sp^{d-1}} b\left(\cos \theta_{ij}\right) \, M_{\sigma k}
  \left( \nabla_\sigma F(\sigma) \right) \, {\rm d}\sigma
\end{equation*}
and the second equality above to rewrite the term involving $b'$ in
\eqref{eq:nablaQ} into
$$
\frac12 \, \int_{\Sp^{d-1}} b\left(\cos \theta_{ij}\right) \,
M_{\sigma k} \left( \left(\nabla_{v_i} f\right)_{ij} - \left(\nabla_{v_j}
    f\right)_{ij} \right) \, {\rm d}\sigma
$$
Putting all together, we get formula \eqref{eq:expgrad}.

We deduce that for $\ell \not = i,j$ we have by Cauchy-Schwarz
\begin{equation*}
  \left|\nabla_{v_\ell} \left( Q_{ij} ^{+,N}
      \left( f^N \right) \right) \right|^2 \leq 4 \, \left(
    \int_{\Sp^{d-1}} f_{ij} ^N \, b\left(\cos \theta_{ji}\right) \,
    {\rm d}\sigma \right) \, \left( \int_{\Sp^{d-1}} 
    \left|\left(\nabla_{v_\ell} \sqrt{f^N} \right)_{ij}\right|^2 \,
    b\left(\cos \theta_{ij}\right) \, {\rm d}\sigma\right)
\end{equation*}
and therefore
$$
\frac{\left|\nabla_{v_\ell} \left( Q_{ij} ^{+,N}
      \left( f^N \right)\right) \right|^2}{Q_{ij} ^{+,N} \left( f^N
  \right)} \leq 4 \, \int_{\Sp^{d-1}} 
\left|\left(\nabla_{v_\ell} \sqrt{f^N} \right)_{ij}\right|^2 \,
b\left(\cos \theta_{ij}\right) \, {\rm d}\sigma.
$$
Now integrating in $V$ we obtain
\begin{eqnarray*}
  I_\ell \left(Q_{ij} ^{+,N} \left( f^N \right) \right)  & \leq & 4
  \int_{\R^{dN}} \int_{\Sp^{d-1}} 
  \left|\left(\nabla_{v_\ell}
      \sqrt{f^N} \right)_{ij}\right|^2 \, b\left(\cos
    \theta_{ij}\right) \, 
  {\rm d}\sigma \, {\rm d}v \\
  & \leq & 4 \int_{\R^{dN}} \int_{\Sp^{d-1}} 
  \left|\left(\nabla_{v_\ell}
      \sqrt{f^N} \right)\right|^2 \, b\left(\cos \theta_{ij}\right) \, 
  {\rm d}\sigma \, {\rm d}v \\
  & \leq & 4 \, C_B \, \int_{\R^{dN}}   
  \left|\left(\nabla_{v_\ell}
      \sqrt{f^N} \right)\right|^2 \, {\rm d}v  =: I_\ell \left(f^N \right)
\end{eqnarray*}
where we have used Cauchy-Schwartz's inequality and the change of
variable $\Theta_{ij}$, and where $I_\ell$ is defined from the last
line (Fisher information restricted to the $\ell$-th derivative). 

When $\ell=i, j$, we use \eqref{eq:expgrad} to get 
\begin{multline*}
  \left|\nabla_{v_\ell} \left( Q_{ij} ^{+,N}
      \left( f^N \right) \right) \right|^2 \leq \left(
    \int_{\Sp^{d-1}} f ^N \, b\left(\cos \theta_{ji}\right) \,
    {\rm d}\sigma \right) \\ \times \Bigg( \int_{\Sp^{d-1}} 
     \Bigg| \left(\nabla_{v_i}
       \sqrt{ f^N } \right) + \left(\nabla_{v_j}
       \sqrt{ f^N } \right) \\ +  P_{\sigma k}  \,
     \left(  \left( \nabla_{v_j}
       \sqrt{ f^N }\right) - \left( \nabla_{v_j}
       \sqrt{ f^N }\right) \right) \Bigg|^2 \, b\left(\cos \theta_{ij}\right)
   \,{\rm d}\sigma \Bigg)
\end{multline*}
where we have used Cauchy-Schwartz's inequality and the change of
variable $\Theta_{ij}$.

Since for fixed $V$, $P_{\sigma k}$ is odd in $\sigma$ and $b(\cos
\theta_{ij})$ is even in $\sigma$, we have
$$
\int_{\Sp^{d-1}} A \cdot P_{\sigma k}(B) \, {\rm d}\sigma =0
$$
for any functions $A$, $B$ independent of $\sigma$. Using finally that
$P_{\sigma k}$ has norm less than $1$ (for the subordinated norm to
the euclidean norm) we get 
\begin{multline*}
  \left|\nabla_{v_\ell} \left( Q_{ij} ^{+,N}
      \left( f^N \right) \right) \right|^2 \leq 2 \, \left(
    \int_{\Sp^{d-1}} f ^N \, b\left(\cos \theta_{ji}\right) \,
    {\rm d}\sigma \right) \\ \times \Bigg( \int_{\Sp^{d-1}} 
     \left|\nabla_{v_i}
       \sqrt{ f^N } \right|^2 + \left|\nabla_{v_j}
       \sqrt{ f^N } \right|^2 \, b\left(\cos \theta_{ij}\right)
   \,{\rm d}\sigma \Bigg)
\end{multline*}
and therefore 
$$
I_\ell \left(Q_{ij} ^{+,N} \left( f^N \right) \right) \le C_B \, \frac{I_i
  \left(f^N \right)+I_j \left(f^N \right)}{2}.
$$

Finally we end up with 
$$
I\left(Q_{ij} ^{+,N} \left( f^N \right) \right) = C_B \, \sum_{\ell
  =1} ^N I_\ell \left(Q_{ij} ^{+,N} \left( f^N \right) \right) \le C_B
\, \sum_{\ell =1} ^N I_\ell \left(f^N \right) = C_B \, I\left( f^N
\right)
$$
which concludes the proof. 
\end{proof}

Let us now conclude the proof of point (ii) in
Theorem~\ref{theo:entropy}. We make use of the so-called ``HWI''
interpolation inequality on the manifold $\SS^N(\EE)$.  Observe that
$\SS^N(\EE)$ has positive Ricci curvature since it has positive curvature.
Then \cite[Theorem~30.21]{Villani:OTON} implies that
\[
  \frac1N \, H\left( f^N | \gamma^N \right)
    \le \frac{W_2\left( f^N, \gamma^N
    \right)}{\sqrt{N}} \, \sqrt{ \frac{I \left( f^N | \gamma^N
      \right)}{N}}.
\]
We can then use the uniform bound on the Fisher information provided
by Lemma~\ref{lem:fisherN}~(vi) and the bound on the initial data 
to get: 
$$
\frac{I \left( f^N | \gamma^N \right)}{N} \le \frac{I \left( f^N _0 |
    \gamma^N \right)}{N} \le C
$$
for some constant $C >0$ independent of $N$. Moreover
Lemma~\ref{lem:ComparDistances} and the propagation of moments on the
$N$-particle system in Lemma~\ref{lem:momentsN} imply that 
$$
\frac{W_2\left( f^N, \gamma^N
    \right)}{\sqrt{N}} \le C \, \left( \frac{W_1\left( f^N, \gamma^N
    \right)}{N} \right)^\alpha
$$
for some constant $C>0$ and exponent $\alpha >0$ independent of
$N$. Then using Theorem~\ref{theo:max-wasserstein} (case (b)) we
deduce that 
$$
\frac{W_2\left( f^N, \gamma^N
    \right)}{\sqrt{N}} \le C \, \left( \frac{W_1\left( f^N, \gamma^N
    \right)}{N} \right)^\alpha \le \beta(t)
$$
with a polynomial rate $\beta(t) \to 0$ as $t \to +\infty$, which
implies that
$$
\frac1N \, H\left( f^N | \gamma^N \right) \le C \, \beta (t) 
$$
and concludes the proof.

\section{The BBGKY hierarchy method revisited} 
\label{sec:bbgky}
\setcounter{equation}{0}
\setcounter{theo}{0}

The so-called BBGKY hierarchy method (Bogoliubov, Born, Green,
Kirkwood and Yvon) is very popular in physics and mathematics for
studying many-particle systems: see for instance
\cite{ArkerydCI99} where this approach is used for Kac's master
equation for hard spheres, or see, among many other works, the recent
series of papers \cite{MR1869286,MR2257859,MR2276262,MR2680421} where
this approach is used for the derivation of nonlinear mean-field
Schr\"odinger equations in quantum physics. The basic ideas underlying
this approach to mean-field limit could be summarized as:
\begin{itemize}
\item[(i)] Write a BBGKY hierarchy on marginals of the $N$-particle system
  and prove that the $N$-particle system solutions converge to the
  solutions of an ``infinite hierarchy'' when $N$ goes to
  infinity. The proof of this convergence often relies on a
  compactness argument.
\item[(ii)] Prove that solutions to this infinite hierarchy are
  unique, which is the hardest part of this program.
\item[(iii)] Then deduce the propagation of chaos by exhibiting, for any
  chaotic initial data to the infinite hierarchy, a solution to the
  infinite hierarchy obtained by the infinite tensorization of the
  $1$-particle solution to the limit nonlinear mean-field equation.
\end{itemize}

This section revisits this BBGKY hierarchy method, under some
appropriate regularity assumptions on the limit semigroup. We build
a rigorous connection with \emph{statistical solutions} and our
pullback semigroup $T^\infty _t$, by showing (1) how these notions are
included in our functional framework (cf. the abstract semigroup
$T^\infty _t$ defined in Section~\ref{sec:abstract-setting}), and (2)
how to give a proof of uniqueness and propagation of chaos based on
them by using the functional tools we have introduced. We would like
to take the opportunity to mention here the interesting paper
\cite{MR657065} (pointed out to us by Golse) where some key ideas
about the connection between the BBGKY hierarchy and the pullback
semigroup were already presented. 

\subsection{The BBGKY hierarchy} 

Let us recall the master equation of the $N$-particle system
undergoing a Boltzmann collision process (the notion of BBGKY
hierarchy has wider application range, but we shall stick to this
concrete case for clarity), see
\eqref{eq:BoltzBddKolmo}-\eqref{defBoltzBddGN}:
\begin{equation}\label{eq:MasterN}
 \partial_t \left\langle f^N_t,\varphi \right\rangle = 
  \left\langle f^N_t, G^N \varphi \right\rangle 
\end{equation}
with 
\[
  \left(G^N\varphi\right) (V) =  \frac{1}{N} \, 
  \sum_{1 \le i < j \le N}  \Gamma\left(|v_i-v_j|\right)
  \, \int_{\mathbb{S}^{d-1}} b(\cos\theta_{ij}) \, \left[\varphi^*_{ij} -
    \varphi\right] \, {\rm d}\sigma
\]
\[
\mbox{where } \ \varphi^*_{ij}= \varphi \left(V^*_{ij} \right) \ \mbox{ and } \
\varphi = \varphi(V) \in C_b\left(\R^{Nd}\right).
\]

Then the BBGKY hierarchy writes as folows. Let us recall the notation 
\[
f^N_\ell = \Pi_\ell[f^N] = \int_{v_{\ell+1}, \dots , v_N}
{\rm d}f^N\left(v_{\ell+1},\dots, v_N\right)
\]
for the marginals. Then integrating the master equation
\eqref{eq:MasterN} against some test function $\varphi =
\varphi(v_1,\dots,v_\ell)$ depending only on the first $\ell$
variables leads to
\[
\frac{{\rm d}}{{\rm d}t} \left\langle f^N_\ell ,\varphi \right\rangle = 
\left\langle f^N_{\ell +1}, G^N _{\ell+1} (\varphi) \right\rangle
\]
where 
\begin{equation*}
G^N _{\ell+1} (\varphi) := \frac{1}{N} \, 
  \sum_{1 \le i \le \ell, \ 1 \le j \le N, \ i \not= j}  \Gamma\left(|v_i-v_j|\right)
  \, \int_{\mathbb{S}^{d-1}} b(\cos\theta_{ij}) \, \left[\varphi^*_{ij} -
    \varphi\right] \, {\rm d}\sigma.
\end{equation*}

Then with the notation 
$$
\ZZ^N_{ij} := \left\langle f^N, \Gamma\left(|v_i-v_j|\right)
  \, \int_{\mathbb{S}^{d-1}} b(\cos\theta_{ij}) \, \left[\varphi^*_{ij} -
    \varphi\right] \, {\rm d}\sigma \right\rangle
$$
we can futher decompose this sum as
\[
\frac{{\rm d}}{{\rm d}t} \left\langle f^N_\ell ,\varphi \right\rangle
= {1 \over N} \sum_{i,j \le \ell} \ZZ^N_{ij}  + {1 \over N} \sum_{i
  \le \ell < j}  \ZZ^N_{ij} =  {1 \over N} \sum_{i,j \le \ell} \ZZ^N_{ij}  +\OO
\left(\frac{\ell^2}N \right) 
\]
by observing that $\ZZ^N_{ij} = 0$ for $i, j > \ell$. Using the
symmetry of $f^N$ we finally deduce
\begin{equation}\label{eq:hierarchy}
\frac{{\rm d}}{{\rm d}t} \left\langle f^N_\ell ,\varphi \right\rangle
 = \frac{(N-\ell)}{N} \, \left( \sum_{i=1} ^\ell \ZZ^N _{i (\ell +1)}
 \right) + \OO \left(\frac{\ell^2}N \right). 
\end{equation}

We thus end with a series of $N$ coupled equations on the marginals $f^N
_\ell$, where the $\ell$-equation ($\ell \le N-1$) depends on the $f^N
_{\ell+1}$ marginal. 

\subsection{The infinite hierarchy and statistical solutions} 

Assume now that 
\[
\forall \, t \ge 0, \,\, \forall \, \ell \ge 1, \quad f^N _{t\ell} \rightharpoonup \pi_{t\ell} \
\mbox{ in } \ P\left( \R^{d \ell} \right).
\]

Starting from \eqref{eq:hierarchy} we obtain, for
$\varphi=\varphi(v_1,\dots,v_\ell) \in C_b(\R^{d\ell})$ depending only
on the first $\ell$ variables,
\[
\frac{{\rm d}}{{\rm d}t} \left\langle \pi_\ell, \varphi \right\rangle = 
\left\langle \pi_{\ell +1}, G^\infty _{\ell+1}(\varphi) \right\rangle
\]
where $G _{\ell+1}(\varphi) \in C_b(\R^{d(\ell+1)})$ is defined by 
\[
G ^\infty _{\ell+1}(\varphi) := \Gamma\left(|v_i-v_j|\right)
  \, \int_{\mathbb{S}^{d-1}} b(\cos\theta_{ij}) \, \left[\varphi^*_{ij} -
    \varphi\right] \, {\rm d}\sigma.
\]

In a more compact form, we have the following set of linear coupled
evolution equations 
\begin{equation}\label{evolbbgky}
\forall \, \ell \ge 1, \quad \partial_t \pi_\ell = A^\infty _{\ell+1}
\left( \pi_{\ell+1}\right) 
\ \mbox{ with } \  A^\infty _{\ell+1} := \left( G_{\ell+1} ^\infty
\right)^*.
\end{equation}
Since the family of $\ell$-particle probabilities $\pi_{\ell}$ is
symmetric and compatible in the sense that
\[
\forall \, \ell \ge 1, \quad \Pi_\ell \left[ \pi_{\ell+1} \right] = \pi_\ell
\]
(this follows from the construction), we can associate by
Hewitt-Savage's Theorem~\cite{Hewitt-Savage} a unique $\pi \in
P(P(\R^d))$ such that, for any $\ell \ge 1$ and
$\varphi=\varphi(v_1,\dots,v_\ell) \in C_b(\R^{d\ell})$ depending only
on the first $\ell$ variables,
\[
\langle \pi, R_\varphi \rangle = \langle \pi_\ell, \varphi \rangle
\]
and these evolution equations for the $\pi_\ell$ translate into an
evolution equation
\[
\partial_t \pi = A^\infty (\pi) \ \mbox{ on } \ P\left( P\left( \R^d
  \right) \right) 
\]
of \emph{statistical solutions} and the corresponding dual evolution
\begin{equation} \label{dualstat}
\partial_t \Phi = \bar G^\infty \Phi \ \mbox{ on } \ C_b\left( P\left( \R^d
  \right) \right). 
\end{equation}

In order to make this heuristic rigorous at an abstract level, one
needs at least some tightness on the sequence $(f^N _\ell)_{N \ge
  \ell}$ for any $\ell$, and some convergence
\[
G^N_{\ell+1} (\varphi) \to G^\infty_{\ell+1} (\varphi)
\]
on compact subset of $\R^{d(\ell+1)}$. Both are satisfied for
Boltzmann collision processes considered in this paper (note that the
tightness follows from the moment estimates in
Lemma~\ref{lem:momentsN} for instance).

\subsection{Uniqueness of statistical solutions and chaos}
\label{sec:tensor}

We now want, under appropriate abstract assumptions, to identify the
limit evolution \eqref{dualstat} in $C_b(P(\R^d))$ obtained from the
hierarchy, and show that it coincides with the the pullback evolution
semigroup $T^\infty _t$ introduced in
Subsection~\ref{sec:mean-field-limiting}. Meanwhile we shall prove
that the statistical solutions to the infinite hierarchy are unique,
and hence prove the propagation of chaos, without any rate, but also
under weaker assumptions than previously. For the sake of clarity we
do not include weights nor constraints in the following theorem, but
it can easily be extended in this direction in a similar way as in
Theorem~\ref{theo:abstract}.  Our aim here is rather the conceptual
presentation of the method.  As a consequence, our result only applies
(straightforwardly) to the {\bf (GMM)} model. The {\bf (HS)} model and
the {\bf (tMM)} model could be handled in a similar way by using an
extended version of the theorem including weights and constraints.
\medskip

We make the following assumptions:
\medskip

\begin{quote}
\begin{center}{\bf (A1') Assumptions on the $N$-particle system.} 
\end{center}
\smallskip

$G^N$ and $T^N_t$ are well defined on $C_b(E^N)$ and invariant under
permutation, and the associated solutions $f^N _t$ satisfy:
$$ 
\forall \, t \ge 0,\, \forall \, \ell \ge 1, \quad \hbox{the sequence}
\ (\Pi_\ell f^N _{t})_{N \ge
  \ell} \ \hbox{ is tight in } \  \mathcal P_{\GG_1}(E) ^{\otimes \ell}
$$
where $\GG_1$ is a Banach space and $\mathcal P_{\GG_1}(E)$ is defined
in Definitions~\ref{defGG1}-\ref{defGG2}, and is associated to a
weight function
$m_{\GG_1}$ 
and endowed with the metric induced from $\GG_1$.
\end{quote}
\bigskip

\begin{quote}
  \begin{center} {\bf (A2') Assumptions for the existence of the
      statistical and pullback semigroups.}
\end{center}
\smallskip

For some $\delta \in (0,1]$ and some $\bar a\in
(0,\infty)$ we assume that for any $a\in (\bar a,\infty)$:
  \begin{itemize}
  \item[(i)] The equation \eqref{eq:limit} generates a semigroup 
    \[
    S^{N\! L}_t : \BB
    \mathcal P_{\GG_1,a} \to \BB \mathcal P_{\GG_1,a}
    \]
    which is $\delta$-H\"oder continuous locally uniformly in time, in
    the sense that for any $\tau \in (0,\infty)$ there exists $ C_\tau
    \in (0,\infty)$ such that
    \begin{equation*}
      \forall \, f,g
      \in \BB \mathcal P_{\GG_1,a}, \quad \sup_{t \in [0,\tau]}
      \left\| S^{N\! L}_t f -
      S^{N\! L}_t g \right\|_{\GG_1} \le C_\tau \, \, \| f - g \|^\delta_{\GG_1}.
     \end{equation*}
   \item[(ii)] The application $Q$ is bounded and $\delta$-H\"older continuous from
    $\BB \mathcal P_{\GG_1,a}$ into $\GG_1$.
 \item[(iii)] $E$ is a locally compact Polish space 
   and there is $\FF_1$ in duality with $\GG_1$ such that $\FF_1$ is dense in
    $C_b(E)$ in the sense of uniform convergence on any compact set.
%
\end{itemize}
\end{quote}
\bigskip

\begin{quote}
\begin{center}
{\bf (A3') Convergence of the generators.} 
\end{center}
\smallskip

For any fixed $\ell \in \N^*$ and any $\varphi \in
C_b(E^\ell)$, the sequence 
\[
G^N_{\ell+1} (\varphi) \in C_b\left(E^{\ell+1}\right)
\ \mbox{ satisfies } \ G^N_{\ell+1} \varphi \xrightarrow[]{N \to
  \infty} G^\infty_{\ell+1} \varphi
\]
uniformly on compact sets, where $G^\infty_{\ell+1} \varphi$ satisfies
the following {\it compatibility binary derivation structure}: for any
$ \varphi = \varphi_1 \otimes \, \dots \, \otimes \varphi_\ell \in
C_b(E)^{\otimes\ell}$ and any $V = (v_1, \dots , v_{\ell+1}) \in
E^{\ell+1}$
\begin{equation}\label{eq:compatibiliteGinftyell2} 
G^\infty_{\ell+1} (\varphi) (V) = \sum_{i=1}^\ell \left( \prod_{j \not=i} \varphi_j(v_j) \right) \,
Q^*(\varphi_i)\left(v_i,v_{\ell+1}\right)
\end{equation}
where $Q^*$ is related to $Q$ through the duality relation 
$$
 \forall \, f \in \mathcal P_{\GG_1}, \,\, \forall \, \psi \in C_b(E),
\quad \left\langle Q(f,f),\psi\right\rangle = \left\langle f \otimes f, Q^*(\psi) \right\rangle.
$$
\end{quote}
\bigskip

\begin{rem}
  The identity \eqref{eq:compatibiliteGinftyell2} is called {\it
    compatibility binary derivation structure} for the following
  reasons: {\it compatibility} since it is a
  natural condition in order that any solution $f_t$ to the nonlinear
  Boltzmann provides a tensorized solution to the BBGKY hierarchy
  \eqref{evolbbgky}. Indeed, considering such a solution 
$$ 
f_t \in C\left(\R_+;\mathcal P_{\GG_1}(E)\right)  \ \hbox{ and } \ 
 \varphi = \varphi_1 \otimes \, \dots \, \otimes \varphi_\ell \in
 C_b(E)^{\otimes\ell}
$$ 
we compute
\begin{eqnarray*}
  \frac{{\rm d}}{{\rm d}t} \left\langle f_t^{\otimes \ell}, \varphi \right\rangle 
  &=& \sum_{i=1}^\ell \left( \prod_{j\not=i} \left\langle f_t, \varphi_j
  \right\rangle \right) 
\frac{{\rm d}}{{\rm d}t} \langle f_t, \varphi_i \rangle 
  = \sum_{i=1}^\ell \left( \prod_{j\not=i} \left\langle f_t, \varphi_j
  \right\rangle \right) 
\left\langle Q\left(f_t,f_t\right), \varphi_i \right\rangle 
  \\
  &=& \sum_{i=1}^\ell \left( \prod_{j\not=i} \left\langle f_t, \varphi_j
  \right\rangle \right) 
  \left\langle f_t \otimes f_t, Q^*\left(\varphi_i\right)\right\rangle 
  = \left\langle f_t^{\otimes \ell+1}, G^\infty_{\ell+1} \varphi \right\rangle.
\end{eqnarray*}
The word {\it binary} refers to the fact that $G^\infty _{\ell+1}$
decomposes in function acting on {\it one} variable and adding {\it
  one} variable, which corresponds to the binary nature of the
collisions. Finally the word {\it derivation} refers to the fact that
the following distributivity property holds
$$
G^\infty_{\ell+1} \left( \varphi \otimes \psi \right) = 
G^\infty_{\ell+1} \left( \varphi \right) \otimes \psi + 
\varphi \otimes G^\infty_{\ell+1} \left( \psi \right).
$$
Let us mention that this distributivity property is at the basis of
the original combinatorial proof of Kac \cite{Kac1956} of propagation
of chaos for the simplified Boltzmann-Kac equation. This structure
assumption is also partly inspired from \cite{McKean1967}.
\end{rem}

\bigskip

\begin{quote}
\begin{center}
{\bf (A4') Differential stability of the limit semigroup.} 
\end{center}  
\smallskip

We consider some Banach space $\GG_2 \supset \GG_1$ (where $\GG_1$ was
defined in {\bf (A2)}) and the corresponding space of probability
measures $\mathcal P_{\GG_2}(E)$ (see
Definitions~\ref{defGG1}-\ref{defGG2}) with the weight function
$m_{\GG_2}$ 
and endowed with the metric induced from $\GG_2$.

We assume that the flow $S^{N \! L}_t$ is $UC^1(\mathcal
P_{\GG_1},\mathcal P_{\GG_2}) \cap UC^{0,1/2}(\mathcal
P_{\GG_1},\mathcal P_{\GG_2})$ for any $t \ge 0$, where  $UC^{0,1/2}$ is  the space
of functions $\SS$ satisfying 
\eqref{eq:devdist1}Êwith $\Omega_c$ such that $\Omega_c(s)/s^{1/2} \to 0$ as
$s\to0$. 
\end{quote}
\bigskip

Thanks to {\bf (A2')}, we know from Lemma~\ref{lem:H0} that
for any $\Phi \in UC_b(\mathcal P_{\GG_1},\R)$ we may define the
$C_0$-semigroup $T^\infty _t [\Phi] \in UC_b(\mathcal P_{\GG_1})$ by
$$
T^\infty _t  [\Phi] (f) = \Phi\left(S^{N\!L}_t f\right),
$$
and so that $\Phi_t = T^\infty _t [\Phi]$ satisfies the equation
\begin{equation*}
\partial_t \Phi = G^\infty [\Phi]
\end{equation*}
with a generator $G^\infty$ which is a closed operator on
$UC_b(\mathcal P_{\GG_1})$ and has domain $\hbox{Dom}(G^\infty)$
which contains 
$UC^1(\mathcal P_{\GG_1})$, and is defined by
$$
G^\infty [\Phi](f) 
= \left\langle Q(f,f), D \Phi(f)
\right\rangle_{\mathcal P_{\GG_1},C_b(\mathcal P_{\GG_1})}.
$$
The evolution corresponds to the following dual evolution equation 
\begin{equation}\label{evolpush}
\frac{{\rm d}}{{\rm d}t} \left\langle \pi_t, \Phi \right\rangle = \left\langle
  \pi_t, G^\infty [\Phi] \right\rangle.
\end{equation}

Our goal is to prove first that the evolution equations
\eqref{evolbbgky} and \eqref{evolpush} are identical (or in other
words that the generator $\bar G^\infty$ introduced for the
hierarchy is well-defined and equal to $G^\infty$), and second and
most importantly that the solution to these equations is given by the
characteristics method
\begin{equation}\label{eq:defbarpit} 
\forall \, \Phi \in UC_b
(\mathcal P_{\GG_1};\R), \quad \left\langle \pi_t,\Phi \right\rangle = \left\langle
\pi_0 , T^\infty _t \Phi \right\rangle.
\end{equation}

Let us explain why the relation \eqref{eq:defbarpit} indeed defines
uniquely a probability evolution $\bar\pi_t \in P(\mathcal P_{\GG_1})$. 
For any $\ell \in \N^*$ we define 
$$
\varphi \in \FF^{\otimes\ell} \mapsto \left\langle \pi^\ell_t,\varphi
\right\rangle := \left\langle \pi_0 , T^\infty _t R^\ell_\varphi \right\rangle.
$$
That is a positive linear form on $\FF^{\otimes \ell}$. Thanks to {\bf
  (A2')-(iii)}, the Stone-Weierstrass density theorem and the
Markov-Riesz representation theorem, we conclude that $\pi^\ell_t$ is
well defined as a element of $\mathcal P_{\GG_1}(E) ^{\otimes \ell}$. Since now the sequence
$(\pi^\ell_t)$ is symmetric and compatible, the Hewitt-Savage
representation theorem implies that there exists a unique probability
measure $\bar\pi_t \in P(\mathcal P_{\GG_1})$ such that for any $\varphi \in
\FF_1 ^{\otimes\ell}$
 \begin{equation}\label{eq:defbbgkybis}
 \langle \bar\pi_t,R^\ell_\varphi \rangle := \langle \pi_0 ,  T^\infty _t R^\ell_\varphi \rangle.
\end{equation}

\begin{theo}\label{theo:BBGKYuniq} 
  Under the asumptions {\bf (A1')-(A2')-(A3')-(A4')}, for any initial
  datum $\pi_0 \in P(\mathcal P_{\GG_1})$, the flow $\bar\pi_t$ defined from
  \eqref{eq:defbbgkybis} is the unique solution in $C([0,\infty);
  P(\mathcal P_{\GG_1}))$ to the infinite hierarchy evolution \eqref{evolbbgky}
  starting from $\pi_0$.

  Moreover, if $\pi_0 = \delta_{f_0}$ with $f_0 \in P(E)$ then $\pi_t
  = \delta_{f_t}$ for any $t \ge 0$, with $f_t := S^{N \! L}_t
  f_0$. As a consequence we deduce that if $f^N_0$ is
  $f_0$-chaotic, then $f^N_t$ is $S^{N\!L}_t f_0$-chaotic. More
  generally, if $f^N_0$ converges to $\pi_0$ then $f^N _t$ converge to
  $\bar \pi_t$, which is the associated statistical solution.
\end{theo}

\begin{proof}[Proof of Theorem~\ref{theo:BBGKYuniq}]
We shall proceed in several steps. 
\bigskip

\noindent {\sl Step 1: Propagation of Dirac mass structure.} Let us
recall that Hewitt-Savage's theorem~\cite{Hewitt-Savage} implies that
for any $\pi \in P(P(E))$ there exists a unique sequence $(\pi^\ell)
\in P(E^\ell)$ such that
$$
\forall \, \varphi \in \left( C_b\left( E \right) \right)^{\otimes \ell}, \quad 
\left\langle \pi^\ell , \varphi \right\rangle 
= \left\langle \pi, R^\ell_\varphi \right\rangle.
$$
As a consequence,  if $\pi_0= \delta_{f_0}$ and $\bar \pi$ satisfies 
\eqref{eq:defbbgkybis}, then  
\begin{eqnarray*}
\left\langle \bar\pi_{t\ell}, \varphi \right\rangle
&=& \left\langle \bar \pi_t, R^\ell_\varphi \right\rangle 
= \left\langle \pi_0, T^\infty_t R^\ell_\varphi \right\rangle =
T^\infty_t \left[R^\ell_\varphi \right] \left(f_0\right) \\
&=& R^\ell_\varphi \left(S^{N \! L}_t f_0\right)  
=  \left\langle S^{N \! L}_t \, f_0,\varphi_1 \right\rangle \, \dots
\,  \left\langle S^{N \! L}_t \, f_0,\varphi_\ell \right\rangle,
\end{eqnarray*}
which means that $\bar\pi_{t\ell} = f^{\otimes\ell}_t$, or
equivalently $\bar\pi_t = \delta_{f_t}$.

\bigskip\noindent {\sl Step 2: Equivalence between \eqref{evolbbgky}
  and \eqref{evolpush}.} 

First let us assume \eqref{evolpush} and prove
\eqref{evolbbgky}. We start with the following observation. Consider $f \in \mathcal P_{\GG_1}(E)$ and $\varphi \in
\FF^{\otimes \ell}$. Then we have $R^\ell_\varphi \in C^{1,1}
(\mathcal P_{\GG_1}(E))$ and we deduce from \eqref{eq:compatibiliteGinftyell2}
that
\begin{equation*}
  \left\langle f^{\otimes \ell+1},
    G^\infty_{\ell+1} \varphi \right\rangle = 
  \left\langle Q(f,f), DR^\ell_\varphi (f)\right\rangle = G^\infty
  \left[R^\ell _\varphi\right](f)
\end{equation*}
which means
\[
R ^{\ell+1} _{G^\infty_{\ell+1} \varphi} = G^\infty \left[R^\ell _\varphi\right].
\]

Then, using Hewitt-Savage's Theorem again, \eqref{evolpush} implies that
\begin{equation}\label{transition}
\frac{{\rm d}}{{\rm d}t} \left\langle \pi_{t\ell}, \varphi \right\rangle
= \frac{{\rm d}}{{\rm d}t} \left\langle  \pi_{t}, R^\ell _\varphi
\right\rangle = 
\left\langle  \pi_{t}, G^\infty \left[R^\ell _\varphi\right] \right\rangle
= \left\langle  \pi_{t}, R^{\ell+1}
  _{G^\infty_{\ell+1}[\varphi]} \right\rangle  
= \left\langle \pi_{t(\ell+1)}, G^\infty_{\ell+1}
  [\varphi] \right\rangle 
\end{equation}
which means that $\pi_t$ satisfies \eqref{evolbbgky}.

Assume conversely that $\pi_t$ satisfies \eqref{evolbbgky} and let us
prove \eqref{evolpush}. One needs to prove that one can recover any
$\Phi \in UC^{1} (\mathcal P_{\GG_1}(E))$ from the previous equation
\eqref{transition}.

Therefore consider $\Phi \in UC^{1} (\mathcal P_{\GG_1}(E))$ and let us define
$$
\varphi = \left(\pi^\ell_C \Phi\right) (V) = \Phi
\left(\mu^\ell_V\right), \quad V = \left(v_1,\dots,v_\ell\right)
$$ 
and let us write \eqref{transition} for this choice of $\varphi$:
$$
\frac{{\rm d}}{{\rm d}t} \left\langle \pi_{t}, R^\ell _{\pi^\ell_C \Phi}
\right\rangle = \left\langle \pi_{t}, G^\infty \left[R^\ell
    _{\pi^\ell_C \Phi}\right] \right\rangle = \left\langle \pi_{t},
  R^{\ell+1} _{G^\infty _{\ell+1} \left[\pi^\ell_C \Phi\right]}
\right\rangle.
$$

Then, on the one hand, for any $f \in \mathcal P_{\GG_1}(E)$
\begin{equation*}
R ^\ell _{\pi^\ell_C\Phi} (f) = \int_{E^\ell}
\Phi\left(\mu^\ell_V\right) \,  {\rm d}f^{\otimes \ell} (V) 
\xrightarrow[]{\ell \to \infty} \Phi(f)
\end{equation*}
by the law of large numbers. 

\smallskip
On the other hand, for any $f \in \mathcal P_{\GG_1}(E)$, we have 
\begin{eqnarray*}
R^{\ell+1}_{ G^\infty_{\ell+1}  \left[\pi^\ell _C \Phi\right] }  (f) 
&=&
\left\langle f^{\otimes \ell+1}, G^\infty_{\ell+1}  \left(\pi^\ell_C \Phi\right) \right\rangle
\\
&=&
\left\langle D R^\ell_{ \pi^\ell _C \Phi} (f) , Q(f,f) \right\rangle
\\
&=&
\sum_{i=1}^\ell \int_{E^\ell} \Phi\left(\mu^\ell _V\right) \, {\rm d} Q(f,f)(v_i)
\prod_{j \not=i}  {\rm d}f (v_j).
\end{eqnarray*}
For any given $i = 1, \dots, \ell$, we define 
\[
\phi^{\ell-1}_{V_i} = D\Phi\left(\mu^{\ell-1}_{V_i}\right) ,\quad V_i
:= \left(v_1, \dots, v_{i-1}, v_{i+1}, \dots, v_\ell\right)
\]
and we write
$$
\Phi\left(\mu^\ell _V\right) = \Phi \left(\mu^{\ell-1}_{V_i}\right) +
\left\langle \phi^{\ell-1}_{V_i}, \mu^\ell _V - \mu^{\ell-1}_{V_i}
\right\rangle + \OO \left( \Omega \left( \left\| \mu^{\ell-1}_{V_i} - \mu^\ell _V
  \right\| \right)  \right),
$$
where $\Omega$ is the modulus of differentiability of $\Phi$ as introduced in Definition~\ref{def:Holdercalculus}. 
Observing that 
$$
\mu^\ell _V - \mu^{\ell-1}_{V_i} = {1 \over \ell} \, \delta_{v_i} -
\sum_{j\not=i} {1 \over \ell \, (\ell-1)} \, \delta_{v_j}
$$
and that $\langle Q(f,f), 1 \rangle = 0$ from assumption {\bf
  (A2')-(ii)}, we find
\begin{eqnarray*}
  R^{\ell+1}_{ G^\infty_{\ell+1}  (\pi^\ell _C\Phi) }  (f) 
  &=&
  \sum_{i=1}^\ell \int_{E^\ell} \left(  {1 \over \ell}
    \,\phi^{\ell-1}_{V_i}(v_i) + \OO \left( \Omega(\ell^{-1})
    \right)\right) \, 
  {\rm d} Q(f,f)(v_i) \prod_{j\not=i}  {\rm d}f (v_j)
  \\
  &=&
  \sum_{i=1}^{\ell-1} \int_{E^{\ell-1}}  {1 \over \ell-1} \,
  \left\langle Q(f,f) , \phi^{\ell-1}_{V} \right\rangle \,
  {\rm d}f^{\otimes(\ell-1)}  (V) 
  + \OO \left( \ell \,  \Omega(\ell^{-1}) \right)
  \\
  &=&  \int_{E^{\ell-1}}   \left\langle Q(f,f), 
   D\Phi\left(\mu^{\ell-1}_V\right) \right\rangle \, 
 {\rm d}f^{\otimes(\ell-1)} (V) 
 + \OO \left( \ell \,  \Omega(\ell^{-1}) \right)
  \\
  &&\!\!\!\!\!\!\!\!\!\!\!\! \mathop{\longrightarrow}_{\ell \to\infty}
  \,  
  \left\langle D\Phi (f) ,   Q(f,f) \right\rangle
\end{eqnarray*}
by the law of large numbers again. This implies \eqref{evolpush}. 

\bigskip\noindent {\sl Step 3: Uniqueness.} Let us prove that any
solution of \eqref{evolbbgky}-\eqref{evolpush} satisfies the
characteristics equation \eqref{eq:defbbgkybis}, or in other words
that $\pi_t = \bar \pi_t$. This shall imply uniqueness since we have
already seen that the the solution to \eqref{eq:defbbgkybis} is
unique.

The main observation here is that for any $\Phi \in UC^1(\mathcal
P_{\GG_1}(E))$ if we define $\Phi_t := T^\infty_t \Phi$,  thanks
to assumption {\bf (A4') } and a straightforward extended version of
Lemma~\ref{lem:H0}, we have
\[
\forall \, t \ge 0, \quad \Phi_t \in UC^1(\mathcal P_{\GG_1}(E)) \subset
\hbox{Dom}(G^\infty).
\]

Then since 
\[
\tau \in [0,t] \mapsto \left\langle \pi_\tau, \Phi_{t-\tau}
\right\rangle
\]
is $C^1$ from the fact that $\Phi_{t-\tau} \in UC^1(\mathcal P_{\GG_1}(E))$
belongs to the domain of $G^\infty$ for any $\tau$, we compute
\[
{{\rm d} \over {\rm d}\tau} \left\langle \pi_\tau, \Phi_{t-\tau}
\right\rangle = \left\langle \pi_\tau, G^\infty\left[\Phi_{t-\tau}\right]
\right\rangle - \left\langle \pi_\tau, G^\infty\left[\Phi_{t-\tau}\right]
\right\rangle =0
\]
and we deduce that 
\[
\left\langle \pi_t, \Phi_{0}
\right\rangle = \left\langle \pi_0, \Phi_{t}
\right\rangle
\]
which proves that $\pi_t = \bar \pi_t$ satisfies
\eqref{eq:defbbgkybis}, and concludes the proof.
\end{proof}

\subsection{A remark on stationary statistical solutions} 

As we have seen:
\begin{itemize}
\item The chaoticity of a sequence of symmetric $N$-particle
  distributions $f^N \in P(E^N)$, $N \ge 1$ is equivalent to the fact
  that the associated $\pi \in P(P(E))$ is a Dirac at some $f_0 \in
  P(E)$: $\pi =\delta_{f_0}$. Hence, in view of Hewitt-Savage's
  theorem, non-chaoticity can be reframed as saying that $\pi$ is a
  superposition of \emph{several}, instead of one, chaotic states.
\item We have recalled the result in \cite{Kac1956,CCLLV} stating that
  a chaotic (tensorized) sequence is asymptotically concentrated on
  the energy sphere, which is an effect of the law of large numbers. 
\item The $N$-particle dynamics leaves the energy spheres
  invariant and relaxes on each energy spheres to the uniform
  measure. This is a consequence of the energy conservation laws: at
  the level of the particle system, the dynamics is layered
  according to the value of this conservation law.
\end{itemize}
One deduces from these considerations that there is room for
\emph{non-chaotic stationary states} of the $N$-particle system,
namely superposition of \emph{several} stationary states on different
energy spheres. Let us make this more precise. 

\begin{lem}\label{sec6:SolStat} 
  There exists a non-chaotic stationary solutions to the statistical
  Boltzmann equation. In other words, there exists $\pi \in
  P(P(\R^d))$ such that $\pi \not= \delta_p$ for some $p \in P(\R^d)$
  and $A_{\ell+1} ^\infty (\pi_{\ell+1}) = 0$ for any $\ell \in \N$.
\end{lem}

\begin{proof}[Proof of Lemma~\ref{sec6:SolStat}] It is clear that any
  function on the form
$$
V \in \R^{d(\ell+1)} \, \mapsto \, \pi_{\ell+1} (V) = \phi (|V|^2)
$$
is a stationary solution for the equation \eqref{evolbbgky}, that is
$A_{\ell+1} (\pi_{\ell+1}) = 0$ for any $\ell \ge 1$. Now we define,
with $d=1$ for the sake of simplicity, the sequence
$$
\forall \, \ell \ge 1, \quad V \in \R^\ell \mapsto \pi_\ell (V) =
{c_\ell \over (1+ |V|^2)^{m+\ell/2}}
$$
where the sequence of positive constants $c_\ell$ is inductively
constructed in the following way. 
\begin{itemize}
\item First $c_1$ is chosen in a unique way so that $\pi_1$ is a probability measure. 
\item Then, once $c_1, \dots, c_\ell$ are constructed, $c_{\ell+1}$ is
  constructed so that $\Pi_\ell [\pi_{\ell+1}] = \pi_\ell$, which
  means
\[
\forall \, V \in \R^\ell, \quad \int_{v_* \in \R} {c_{\ell+1} \over \left( 1 + |V|^2
    +|v_*|^2\right)^{m + \ell/2 +1/2}} \, {\rm d}v_* =  {c_{\ell} \over
  \left( 1 + |V|^2 \right)^{m + \ell/2}}.
\]
This is always possible since 
\begin{multline*}
\int_{v_* \in \R} {c_{\ell+1} \over \left( 1 + |V|^2
    +|v_*|^2\right)^{m + \ell/2 +1/2}} \, {\rm d}v_*
\\ = {c_{\ell+1} \over
  \left( 1 + |V|^2 \right)^{m + \ell/2+1/2}} \, \int_{v_* \in \R} {1 \over \left( 1 
    +{|v_*|^2 \over \left( 1 + |V|^2 \right)} \right)^{m + \ell/2
    +1/2}} \, {\rm d}v_* 
\\ = {c_{\ell+1} \over
  \left( 1 + |V|^2 \right)^{m + \ell/2}} \, \int_{v_* \in \R} {1 \over \left( 1 
    +|v_*|^2 \right)^{m + \ell/2
    +1/2}} \, {\rm d}v_* 
\end{multline*}
which concludes the induction. 
\end{itemize}

We then deduce that the sequence $\pi_\ell$, $\ell \ge 1$, satisfies
\eqref{evolbbgky} since every terms only depends on the energy, and
also satisfies the compatibility condition $\Pi_\ell [\pi_{\ell+1}] =
\pi_\ell$. This concludes the proof.
\end{proof}

\appendix

\bibliographystyle{acm}
\bibliography{./meanfield}


\signsm \signcm 

\end{document}